\documentclass[a4paper,10pt]{article} 

\usepackage{fancybox} 
\usepackage{pifont} 

\usepackage{fancybox} 

\usepackage{amsmath} 
\usepackage[applemac]{inputenc}
\usepackage{amsfonts}
\usepackage{amsthm}
\usepackage{color}
\usepackage[dvipsnames]{xcolor}
\usepackage{mathrsfs} 

\newcommand{\mysection}{\setcounter{equation}{0} \section}

\renewcommand{\d}{\mathbf{d}}
\newcommand{\F}{\mathcal{F}}  
\renewcommand{\P}{\mathbb{P}}

\newcommand{\E}{\mathbb{E}}

\newcommand{\N}{\mathbb{N}}  
  
\newcommand{\T}{\mathbb{T}}  

\newcommand{\mW}{\mathcal{W}}  

\newcommand{\R}{\mathbb{R}}  
\newcommand{\I}{\mathbb{I}}

\newtheorem{THM}{Theorem}   
\newtheorem{REM}{Remark}

\newtheorem{PROP}[THM]{Proposition}

\newtheorem{lem}[THM]{Lemma}

\def\1{\mbox{1\hspace{-0.25em}l}}

\newcommand \A[1]{{\bf (#1)}}

\def\leftB{[\![}
\def\rightB{]\!]}
\def\btheta{{\boldsymbol{\theta}}}

\def\bxi{{\boldsymbol{\xi}}}
\def\bzeta{{\boldsymbol{\zeta}}}

\def\bXi{{\boldsymbol{\Xi}}}

\def\det{{{\rm det}}}
\def\X{{\mathbf{X}}}

\def\x{{\mathbf{x}}}
\def\u{{\mathbf{u}}}
\def\U{{\mathbf{U}}}

\def\z{{\mathbf{z}}}

\def\y{{\mathbf{y}}}
\def\z{{\mathbf{z}}}

\def\m{{\mathbf{m}}}

\def\0{{\mathbf{0}}}

\def\gF{{\mathbf{F}}}
\def\gP{{\mathbf{P}}}

\def\K{{\mathbf{K}}}
\def\gR{{\mathbf{R}}}

\def\tP{{\tilde P}}

\def\p{{\partial}}

\title{
	\textcolor{black}{\textbf{Strong regularization by Brownian noise propagating through a weak H\"ormander structure}}}
\author{\textbf{Paul-\'Eric Chaudru de Raynal}\footnote{Univ. Grenoble Alpes, Univ. Savoie Mont Blanc, CNRS, LAMA, 73000 Chamb\'ery, France. pe.deraynal@univ-smb.fr},
	\  \textbf{Igor Honor\'e}\footnote{Laboratoire de Mod\'elisation Math\'ematique d'Evry (LaMME), Universit\'e d'Evry Val d'Essonne, 23 Boulevard de France 91037 Evry. igor.honore@univ-evry.fr}, \textbf{and
		St\'ephane Menozzi}\footnote{Laboratoire de Mod\'elisation Math\'ematique d'Evry (LaMME), Universit\'e d'Evry Val d'Essonne, 23 Boulevard de France 91037 Evry, France and Laboratory of Stochastic Analysis, HSE,
		Shabolovka 31, Moscow, Russian Federation. stephane.menozzi@univ-evry.fr}}

\begin{document}
	\maketitle

	\begin{abstract}
		We establish strong uniqueness for a class of degenerate SDEs \textcolor{black}{of weak H\"ormander type}  under suitable H\"older regularity conditions  for the associated drift term. 
		Our approach relies on the Zvonkin transform which requires to exhibit good smoothing properties of the underlying parabolic PDE with \textit{rough}, here H\"older, drift coefficients and source term.
		Such regularizing effects are established through a perturbation technique (forward parametrix approach) which also heavily relies on appropriate duality properties on Besov spaces.
		
		For the method employed, we exhibit some sharp thresholds on the H\"older exponents for the strong uniqueness to hold.
	\end{abstract}

	\mysection{Introduction}
	
	\subsection{Statement of the problem} 
	In this work, we aim at establishing a strong well-posedness result outside the classical Cauchy-Lipschitz framework for the following degenerate Stochastic Differential Equation (SDE) of Kolmogorov type:
	\begin{equation}
	\label{strong_SYST}
	\begin{array}{l}
	\displaystyle d\X_t^1 = \gF_1(t,\X_t^1,\dots,\X_t^n) dt + \sigma(t,\X_{t}^1,\dots,\X_{t}^n) dW_t,
	\\
	\displaystyle d\X_t^2 = \gF_2(t,\X_t^1,\dots,\X_t^n) dt,
	\\
	\displaystyle d\X_t^3 = \gF_3(t,\X_t^2,\dots,\X_t^n) dt,
	\\
	\vdots
	\\
	\displaystyle d\X_t^n = \gF_n(t,\X_t^{n-1},\X_t^n) dt,
	\end{array}
	\quad t \geq 0
	,
	\end{equation}
	where $(W_t)_{t \geq 0}$ stands for a $d$-dimensional Brownian motion on some filtered probability space  $(\Omega,\textcolor{black}{\F},(\F_t)_{t\ge 0},\P) $ and for all $i\in \leftB 1,n\rightB $\footnote{We will use throughout the paper the notation $\leftB \cdot,\cdot\rightB $ for integer intervals.}, $t\ge 0 $ the component $\X_t^i $ is $\R^d $-valued as well (i.e. $\X_t\in \R^{nd} $). We suppose that the $(\gF_i)_{i\in \leftB 2,n\rightB}$ satisfy a kind of weak H\"ormander condition, i.e. the matrices $\big(D_{{\textcolor{black}{\x}}_{i-1}}\gF_i(t,\cdot)\big)_{i\in \leftB 2,n\rightB} $ have full rank. However, the coefficients $ ({\textcolor{black}{\gF}}_i)_{i\in \leftB 2,n\rightB}$ can be rather \textit{rough} in their other entries, namely, H\"older continuous. We assume as well that the diffusion coefficient $\sigma $ is bounded \textcolor{black}{from above and below} and spatially Lipschitz continuous.\\
	
	For a system of Ordinary Differential Equation (ODE) it may be a real challenge to prove the well-posedness outside the Lipschitz framework (see \emph{e.g.} \cite{dipe:lion:89}) and, as shown by Peano's example, uniqueness may fail as soon as the drift of the system of interest is only H\"older continuous. For an SDE, the story is rather different since the presence of the noise may allow to restore well-posedness. Such a phenomenon, called \emph{regularization by noise} (see the Saint Flour Lecture notes of Flandoli \cite{flandoli_random_2011} and the references therein for an \textcolor{black}{overview of} the topic), has been widely studied since the pioneering  works \textcolor{black}{of Zvonkin \cite{zvonkin_transformation_1974} and Veretennikov \cite{veretennikov_strong_1980} who establish, respectively in the scalar and multidimensional setting, strong well-posedness for non-degenerate Brownian SDEs} \textcolor{black}{with bounded and measurable drift}.  \textcolor{black}{We recall that a \textit{strong} solution is \textcolor{black}{adapted  to the Brownian filtration generated by the driving noise and that  
	 \textit{non-degenerate} means that the noise has the same dimension as the underlying system on which it acts}.}
	
	Let us mention, among others, \textcolor{black}{and still within the \textit{non-degenerate} setting}, the works of Krylov and R\"ockner \cite{kryl:rock:05} ($L_q-L_p$ drift), Zhang \cite{zhang_well-posedness_2010} ($L_q-L_p$ drift and weakly Lipschitz diffusion matrix), Fedrizzi and Flandoli \cite{fedr:flan:11} ($L_q-L_p$ and H\"older drift). \textcolor{black}{We also mention Flandoli \textit{et al.} \cite{flandoli_well-posedness_2010} and \textcolor{black}{Beck \textit{et al.} \cite{beck:flan:gubi:maur:19} for connections with stochastic transport equations}}.\\ 
	
	The crucial assumption, shared by all the aforementioned results, is the non-degeneracy condition on the noise added in the considered system. A possible approach to relax this \textcolor{black}{hypothesis} was proposed by Veretennikov in \cite{veretennikov_stochastic_1983}, where the author extended the result in \cite{veretennikov_strong_1980} to some specific case of the considered chain \eqref{strong_SYST} for $n=2$. In comparison with our setting, the author does not impose any non-degeneracy condition on $D_{\x_{1}}\gF_2(t,\cdot)$. The to price pay is anyhow that all coefficients (i.e. with the notations of \eqref{strong_SYST} $\gF_1, \gF_2,\sigma $) need to be twice continuously differentiable functions with bounded derivatives w.r.t. the degenerate component, meaning that no regularization by noise is investigated in the degenerate direction. More generally, it is useless to expect a generalization of the previous results without any additional assumption: we can benefit from the regularization by noise phenomenon only in the directions submitted to the noise.\\


	In our current framework, the non-degeneracy assumption on the Jacobian $\big(D_{\x_{i-1}}\gF_i(t,\cdot)\big)$, $i\in \leftB 2,n\rightB$ precisely allows the noise to propagate through the chain passing from the $i^{\rm th}$  to the $(i+1)^{\rm th}$ level thanks to the drift, hence leading to a propagation of the noise in the whole considered space. The main idea is then to take advantage of this particular propagation, known as \emph{weak H\"ormander setting} (in reference to the work of H\"ormander on hypoelliptic differential operator \cite{hormander_hypoelliptic_1967}), to restore strong well-posedness under our current H\"older framework. This feature has already been considered in the literature for the system \eqref{strong_SYST} in the \textcolor{black}{kinetic case (i.e. when $n=2$}), see the works of Chaudru de Raynal \cite{chau:17}, Wang and Zhang \cite{wang:zhan:16}, Fedrizzi \textcolor{black}{\textit{et al.}} \cite{fedr:flan:prio:vove:17}, Zhang \cite{zhan:16}. In any cases, in addition to the weak H\"ormander structure, the regularity of the drift w.r.t. the second space \textcolor{black}{(and hence degenerate)} argument is required to be of regularity index superior or equal (depending on the work) to $2/3$ (\textcolor{black}{usually called} critical H\"older index or critical weak differentiation index).  As a generalization of these results, we prove in this paper that strong well-posedness holds as soon as each drift component $\gF_i$ is $\beta_j$-H\"older continuous in the $j^{\rm th}$ variable for some $\beta_j \in \Big((2j-2)/(2j-1), 1\Big]$ \textcolor{black}{when $i\le j $}, so that we recover the critical index mentioned above when $j=2$. \textcolor{black}{We refer to Section \ref{strong_SUB_STRAT_PDE} for a thorough discussion about those facts}.\\

	\subsection{Notations, assumptions and main result}
	\textbf{Some notations.} We will denote by a bold letter, \textcolor{black}{ e.g. $\x$ or $\y $}, any element of $\R^{nd} $, writing as well $\x=(\x_1,\textcolor{black}{\hdots},\x_n) $ where for $i\in \textcolor{black}{\leftB 1,n\rightB},\ \x_i\in \R^d $.
	For practical purpose we will be led in our analysis to consider \textit{subcomponents} of a vector $\x\in \R^{nd} $. Namely, for any $\textcolor{black}{1}\le i\le j\le n $ and $\x\in \R^{nd} $, we introduce the notation $\x_{i:j}:=\left( \x_i,\textcolor{black}{\hdots},\x_j\right)$.
	Accordingly,  we write the drift as the mapping  
	
	\begin{eqnarray*}
		(s,\x)\in \R^+\times \R^{nd} \mapsto \gF(s,\x)&=&\left(\begin{array}{c}\gF_1(s,\x)\\ \vdots\\
			\gF_n(s,\x)
		\end{array}
		\right)
		\textcolor{black}{:=}\left(\begin{array}{c}\gF_1(s,\x)\\ \gF_2(s,\x)\\ \gF_3(s,\x_{2:n})\\ \vdots\\ \gF_n(s,\x_{n-1:n})\end{array}\right),
	\end{eqnarray*}
	from the specific structure of the drift appearing in \eqref{strong_SYST}.\\

	\textcolor{black}{ 
		For $f \in C^1(\R^{nd}, \R^k)$, $k \in \{1,d\}$,  we denote for each $i \in \leftB 1,n\rightB$, by $D_{\x_i}f(\x)$ the Jacobian matrix of the derivative of $f$ w.r.t. 
		its $\R^d$-valued variable $\x_i $.
		As shortened form, and when no ambiguity is possible, we also write for all $\x,\y \in \R^{nd}$, $D_{\x_i}f(\x)=D_if(\x)$ and $D_{\y_i}f(\y)=D_if(\y)$. Also, if $k=1$ we denote by ${\mathbf D}f(\x)=(D_1 f(\x)\textcolor{black}{,\textcolor{black}{\hdots} ,}D_n f(\x))^* $, \textcolor{black}{where ``$*$'' stands for the transpose}, the full gradient of the function $f$ at point $\x$. 
		}

	Let $f : \R^{nd} \to \R^k$ and $\beta:=(\beta_1,\textcolor{black}{\hdots}, \beta_n) \in (0,1]^n $ be a multi-index. We say that $f$ is uniformly $\beta$-H\"older continuous if for each $j \in \leftB 1,n \rightB $
	\begin{equation}\label{strong_HOLDER_MODULUS}
	[\textcolor{black}{(f)}_j]_{\beta_j}:=\!\!\!\!\sup_{(\z_{1:j-1},\z_{j+1:n})\in \R^{(n-1)d},\ z\neq z', (z,z')\in (\R^d)^2} \!\!\!\! \frac{|\textcolor{black}{f (
			\z_{1:j-1},z, \z_{j+1:n})\!-\!f(
			\z_{1:j-1},z', \z_{j+1:n})}|}{|z-z'|^{\beta_j}}\! < +\infty.
	\end{equation}
	
	\textcolor{black}{ For a smooth function ${\mathbf \Psi}:[0,T]\times \R^{nd} \rightarrow \R^{nd}$, where $T>0 $ is a fixed given time, writing for $(t,\x)\in [0,T]\times \R^{nd} $, ${\mathbf \Psi}(t,\x)=\big({\mathbf \Psi}_1(t,\x), \textcolor{black}{\hdots},{\mathbf \Psi}_n(t,\x) \big)$ where for each $i\in \leftB 1, n\rightB, {\mathbf \Psi}_i $ is $\R^d $ valued, we denote by 
		\begin{equation}
		\label{strong_NORMES}
		\|{\mathbf D} {\mathbf \Psi}\|_\infty:=\sum_{i=1}^n \sup_{(t,\x)\in [0,T]\times \R^{nd}} |\!|\!|{\mathbf  D}{\mathbf \Psi}_i(t,\x)|\!|\!|, \ \|{\mathbf D} (D_1 {\mathbf \Psi})\|_\infty:=\sum_{i=1}^n \sup_{(t,\x)\in [0,T]\times \R^{nd}} |\!|\!|{\mathbf  D} (D_1{ \mathbf \Psi}_i)(t,\x)|\!|\!|,
		\end{equation}
		where in the above equation $|\!|\!| \cdot \ \!|\!|\!| $ stands for  a tensor norm in the appropriate corresponding dimension. Precisely, $ {\mathbf  D}{\mathbf \Psi}_i(t,\x)\in \R^{nd}\otimes \R^d$ and ${\mathbf  D} (D_1{ \mathbf \Psi}_i)(t,\x)\in \R^{nd}\otimes \R^d \otimes \R^d $}.\\
	
	\noindent\textbf{Assumptions\textcolor{black}{.}} We will assume throughout the paper that the following conditions hold.
	\begin{trivlist}
		\item[\A{ML}] The coefficients $\gF$ and $\sigma$ are measurable in time \textcolor{black}{and $\gF(t,\mathbf{0})$ is bounded}. Also, the diffusion coefficient $\sigma$ is uniformly Lipschitz continuous in space, uniformly in time, i.e. there exists $\kappa>0  $ s.t. for all $t\ge 0,\ (\x,\x')\in  (\R^{nd})^2 $:
		$$|\sigma(t,\x)-\sigma(t,\x')|\le \kappa |\x-\x'| .$$ 
		\item[\A{UE}] The diffusion matrix $ a:=\sigma \sigma^* $ is uniformly elliptic and bounded, uniformly in time, i.e. there exists $ \Lambda \ge 1$ s.t. for all $t\ge 0,\ (\x,\zeta)\in  \R^{nd}\times \R^d $:
		$$\Lambda^{-1}|\zeta|^2 \le \langle a(t,\x)\zeta,\zeta\rangle \le \Lambda|\zeta|^2.$$
		
		\item[\A{${\mathbf T}_{\beta} $}] For all $j\in\leftB 1, n \rightB $, the functions $(\gF_i)_{i\in \leftB 1,\textcolor{black}{j}\rightB} $ are uniformly $\beta_j $-H\"older continuous in the $j^{{\rm th}} $ spatial variable with  $\beta_j\in \textcolor{black}{\big(\textcolor{black}{(}2j-2\textcolor{black}{)}/\textcolor{black}{(}2j-1\textcolor{black}{)}},1\big] $, uniformly w.r.t. the other \textcolor{black}{spatial} variables of $\gF_i $ and in time. In particular,
		there exists a finite constant 	\textcolor{black}{$C_\beta>0$} s.t. \textcolor{black}{with the notations of \eqref{strong_HOLDER_MODULUS}}, $$ \max_{\textcolor{black}{i\le j}\in \leftB 1,n\rightB^2} \sup_{s\in [0,T]}[(\gF_i)_j(s,\cdot)]_{\beta_j}\le C_\beta.$$
		
		\item[\A{${\mathbf H}_{\eta} $}] For all $i\in \leftB 2,n\rightB $, there exists a closed convex subset ${\mathcal E}_{i-1} \subset GL_{d}(\R)$
		(set of invertible $d \times d$  matrices)
		s.t., for all
		$t \geq 0$ and $(\x_{i-1},\dots,\x_n) \in \R^{(n-i+2)d}$, $D_{\x_{i-1}} \textcolor{black}{\gF_i}(t,\x_{i-1},\dots,\x_n)
		\in {\mathcal E}_{i-1}$.
		For example, ${\mathcal E}_{i-1}$ 
		may be a closed ball
		included in $GL_{ d}(\R)$ \textcolor{black}{the latter being an open set}. Moreover, $D_{\x_{i-1}}\gF_i $ is \textcolor{black}{ $\eta$
		}-H\"older continuous w.r.t. $\x_{i-1} $ uniformly in $\x_{i:n} $ and time. \textcolor{black}{We also assume without loss of generality that  $\eta \in \Big(0, \inf_{j\in \leftB 2,n\rightB} \big\{ \beta_j-\textcolor{black}{(2j-2)/(2j-1)}\big\}\Big)$, i.e. $ \eta$ is meant to be \textit{small}. }
	\end{trivlist}
	From now on, we will say that assumption $\A{A}$ is in force provided that \A{ML}, \A{UE}, \A{T${}_\beta $} \textcolor{black}{and} \A{H${}_\eta$} hold.\\
	
	\noindent \textbf{Main result.}
	The main result of this work is the following theorem.
	\begin{THM}[Strong uniqueness for the degenerate system \eqref{strong_SYST}]\label{strong_MAIN_RESULT}
		Under \A{A} there exists a unique strong solution to system \eqref{strong_SYST}. 
	\end{THM}
	
	\begin{REM}\label{strong_REM_KRYLOV}
		Still in comparison with the results obtained in the non-degenerate cases, and especially the one of Krylov and R\"ockner \cite{kryl:rock:05}, we do not tackle the case of drifts in $L_q-L_p$ w.r.t. the first (and then non-degenerate) variable. This is only to keep our result as clear as possible and to concentrate on the novelty of the approach we use here. We are anyhow confident that these specific drifts could be handled. Indeed, all the intermediate results needed to perform the analysis in that setting seem 
		to be already available. We refer to subsection \ref{strong_SUB_STRAT_PDE} for further details.
	\end{REM}
	
	\subsection{\textcolor{black}{Related discussions and perspectives}}
	One may wonder if the thresholds in \A{T$ {}_\beta$} are sharp. To investigate this question, we should recast our result \textcolor{black}{within the framework of the} \emph{regularization by noise phenomenon}\textcolor{black}{, see \emph{e.g.} the  lecture notes \textcolor{black}{by Flandoli} \cite{flandoli_random_2011}}. Such phenomenon has been considered for ODEs perturbed by many noise classes (in particular $\alpha$-stable, $\alpha \in (0,2]$ and fractional Brownian motion with Hurst index $H$ in $(0,1)$) and has given rise to a large literature. Before going further, we feel that it is important to \textcolor{black}{specify} what ``\emph{regularization}'' means when speaking about \emph{regularization by noise}. \textcolor{black}{Hereafter, we say that a noise ``regularizes'' an ill posed system if the resulting perturbed equation has a unique solution.} Within our probabilistic setting, we can \textcolor{black}{thus} distinguish three types of \emph{regularization\textcolor{black}{s}}: \textcolor{black}{the \emph{weak regularization}, for which the law of the solution is concerned; the \emph{strong regularization} and \emph{path by path regularization} \textcolor{black}{which relate to the path of the solution}. While the difference between \emph{weak} and \emph{path by path} or \emph{strong} well-posedness is clear, we refer to the recent work of Shaposhnikov and Wresch \cite{sha:wre:20} for the more subtle differences between strong and path by path uniqueness (let us only mention that path by path uniqueness implies strong uniqueness while the converse is not true).}
	
	Also, in order to illustrate the following discussion on the thresholds in \A{T$ {}_\beta$}, we introduce, for a given \textcolor{black}{scalar $\alpha$-stable\footnote{For simplicity reasons, we restrain our considerations to rotationally invariant stable processes \textcolor{black}{with generator $1/2$ the usual fractional Laplacian}.}, $\alpha \in (0,2]$ or $H$-fractional Brownian, $H \in (0,1)$} noise $\mW$ \textcolor{black}{of self similarity index $\gamma>0 $}, the following \textcolor{black}{simplified} 
	version of \eqref{strong_SYST}
	\begin{equation}\label{SYSTPEANOPER}
	\textcolor{black}{d}\X_t={\mathbf A}\X_t dt+B d\mW_t+\textcolor{black}{\gP}(\X_t)dt,\ {\mathbf A}=\left( \begin{array}{ccccc} 
	0                             & \cdots & \cdots & \cdots &0\\
	1& 0        & \cdots& 0 & \vdots\\
	0                              & 1&0        &\vdots &\vdots\\
	\vdots                              & \ddots &\ddots        &0 & \vdots\\
	0                               &   \cdots          & 0&    1               &0
	\end{array}\right),\quad \X_0=\mathbf {0},
	\end{equation}
	\textcolor{black}{where $d=1$\textcolor{black}{, $\gP = (\gP_1,\hdots, \gP_n)^*$} and for each $i\in \leftB 1,n \rightB $},
	$\textcolor{black}{\gP_i(\x)} \textcolor{black}{= \sum_{j=i}^n \textcolor{black}{P}_{i}^j(\x)}$ with  $\textcolor{black}{P}_{i}^j:=c_{i,j}   {\rm sgn}(\x_j)\big(|\x_j|^{\beta_i^j}\textcolor{black}{\wedge 1\big)}
	$ \textcolor{black}{with $c_{i,j}\in \R $  
		and $B=(1,0,\textcolor{black}{\hdots},0)^*=(1,{\mathbf 0}_{1,n-1})^*$ }
	\textcolor{black}{where ``$*$'' stands for the transpose}. \textcolor{black}{The dynamics \eqref{SYSTPEANOPER} can be viewed as a perturbation of  $\textcolor{black}{d}\bar \X_t={\mathbf A}\bar \X_t dt+B d\mW_t $, \textcolor{black}{$\bar \X_0=0$}, corresponding to the noise and its iterated integrals, i.e. $\bar \X_t=(\mW_t,\int_0^t\mW_s ds,\textcolor{black}{\hdots},\int_0^t dt_1\int_0^{t_1} dt_2\textcolor{black}{\hdots} \int_{0}^{t_{n-1}}dt_n \mW_{t_n}  )^{\textcolor{black}{*}}$, by a Peano type drift $\gP $. With the notations of \eqref{strong_SYST}, the full drift of \eqref{SYSTPEANOPER}  writes $\gF(t,\x)=\gF(\x)=\mathbf A \x+\gP(\x) $}.
	
	\textcolor{black}{Note that the above truncation \textcolor{black}{in $\gP$} is only introduced to avoid technical considerations in the following discussion (simple statement of parabolic bootstrap results). 
		However, our approach allows to consider general non-linear unbounded drifts satisfying \A{T$ {}_\beta$} and \A{H$ {}_\eta$}}. 
	\textcolor{black}{We emphasize that,} when $n=1$, setting \textcolor{black}{$\beta_1^1 = :\beta$, $c_{1,1}=1$}, we obtain the following dynamics
	\begin{equation}\label{eq:peano}
	\textcolor{black}{d}\textcolor{black}{{\X}}_t =   {\rm sgn}(\X_t) \big(|\X_t|^{\beta}\textcolor{black}{\wedge 1} \big)
	dt + d\mW_t,\quad \textcolor{black}{\X}_0=0,\\
	\end{equation}
	known as \textcolor{black}{the} \textcolor{black}{(localized version of the)} 
	stochastic Peano example and when $n=2$, $\textcolor{black}{\gP_1\equiv 0, c_{2,2}=1}$, we get that
	\begin{equation}\label{eq:peano:kinetic}
	\textcolor{black}{d} \X_t^1 = d\mW_t,\quad \textcolor{black}{d}\textcolor{black}{{\X}}^2_t =   {\rm sgn}(\X_t^2) \big(|\X_t^{\textcolor{black}{2}}|^{\beta_2^2}\textcolor{black}{ \wedge 1} \big)
	dt + d\left\{\int_0^t ds \mW_s \right\} ,\quad \X_0^1=\X_0^2=0,
	\end{equation}
	which is the kinetic version of the above stochastic Peano example.\\
	
	\noindent\textbf{Weak regularization for \textcolor{black}{the} stochastic Peano example \eqref{eq:peano}.} Let us start \textcolor{black}{with} the \emph{weak regularization by noise}, as it seems to us that it is the most understood and understandable setting. Weak well-posedness relies on \textcolor{black}{the} well-posedness of the martingale formulation for the system which \textcolor{black}{itself} relies, in turn, on a good theory for the associated PDE (see e.g. \cite{stro:vara:79}). Roughly speaking, it consists in showing that the transport term for the associated PDE is \textcolor{black}{somehow} a negligible perturbation of the equation. This can be seen, at least heuristically, through scaling argument\textcolor{black}{s}. Consider indeed the dynamics \eqref{eq:peano} and assume  that $\mW$ is an $\alpha$-stable process. The \textcolor{black}{\textit{formal}} associated 
	PDE
	writes
	
	$$\partial_t u(t,\x)+\langle {\rm sgn}(\x)(|\x|^\beta \textcolor{black}{\wedge 1})
	, \mathbf {D}u(t,\x)\rangle+\frac 12 \Delta^{\frac \alpha 2} u(t,\x)=0,$$
	\textcolor{black}{where $\Delta^{ \alpha/2} $ stands for the usual fractional Laplacian}. 
	Introduce \textcolor{black}{for $\lambda>0 $} the corresponding rescaled \textcolor{black}{function} $u_\lambda( t,\x)\textcolor{black}{:=}u(\lambda t, \lambda^{1/\alpha} \x) $ (scaling reflecting the parabolic scale for $t$ and $\x$). We get that \textcolor{black}{$u_\lambda $ satisfies the equation}:
	
	$$\partial_t u_\lambda (t,\x)+\langle \lambda^{1-\frac 1\alpha}{\rm sgn}(\x)(|\lambda^{\frac1\alpha}\x|^\beta \textcolor{black}{\wedge 1})
	, \mathbf {D}u_\lambda(t,\x)\rangle+\frac 12 \Delta^{\frac \alpha2} u_\lambda(t,\x)=0,$$
	so that the terms associated with the principal part of the partial differential operator in the above PDE, namely $\partial_t u_\lambda$ and $\Delta^{\alpha/2} u_\lambda(t,\x) $, are \textit{comparable}. On the other hand, \textcolor{black}{when $\x$ belongs to compact subsets of $\R$}:
	\begin{itemize}
		\item if $ \beta> 1-\alpha$, the scaled drift coefficient goes to zero with $\lambda $ and the smoothing effect of the principal part of the partial differential operator dominates;
		\item if $\beta=1-\alpha $, the scaled drift coefficient stays at a macro scale and the rescaled drift has the same order as the principal part of the partial differential operator (critical case);
		\item otherwise, the drift explodes when $\lambda$ goes to zero.
	\end{itemize}
	
	In other words, \textcolor{black}{a sufficient condition on the exponent $\beta$ to ensure that the drift is a negligible perturbation of the equation} is
	\begin{equation*}
	\beta>1-\alpha.
	\end{equation*}
	This \textcolor{black}{rule} can also \textcolor{black}{be} obtained by analyzing the path associated \textcolor{black}{with} the solution of \eqref{eq:peano}, see e.g. the work \cite{delarue_transition_2014} and the counter-examples to weak uniqueness in \cite{chau:16,chau:meno:17}. From \textcolor{black}{these} pathwise analysis, one can formally generalize the above heuristic an\textcolor{black}{d} get that,
	\textcolor{black}{for the system \eqref{eq:peano} driven by a general $\gamma $ self-similar noise $\mathcal W $},
	the condition
	\begin{equation}\label{weak_rule}
	\beta>1-\frac 1\gamma,
	\end{equation}
	is needed for the weak well-posedness to hold. When $\gamma=1/\alpha$, $\alpha$ in $(0,1]$ we refer to e.g. \cite{chau:meno:prio:20,chau:meno:prio:00} for results in that direction and to \cite{delarue_rough_2015,canni:chouk:18,flandoli_multidimensional_2017,ZZ17,chau:meno:stable,ling:zhao:00} for results related to the case where $\alpha$ in $(1,2]$ so that the above rule allows $\beta$ to take negative values, \textcolor{black}{the drift being then a distribution}.\\

	\noindent\textbf{Thresholds for weak well-posedness and thresholds in \A{T$ {}_\beta$}.} Let us come back to the more tricky system \eqref{SYSTPEANOPER} and let us consider first the kinetic version \eqref{eq:peano:kinetic} with $\mW=W$. In this last case, the associated PDE writes
	$$\partial_t u(t,\x_1,\x_2)+\langle \x_1, D_{\x_2}u(t,\x_1,\x_2)\rangle + \langle {\rm sgn}(\x_2)(|\x_2|^{\beta_2^2}\textcolor{black}{\wedge 1})
	, D_{\x_2}u(t,\x_1,\x_2)\rangle +\frac 12 \Delta_{\x_1} u(t,\x_1,\x_2)=0.$$
	We introduce the corresponding rescaled \textcolor{black}{function} $u_\lambda( t,\x)\textcolor{black}{:=}u(\lambda t, \lambda^{1/2} \x_1, \lambda^{3/2}\x_2) $ (scaling reflecting the usual parabolic scale for $t$ and $\x_1$ and the fast regime associated with the integral of the Brownian motion) \textcolor{black}{which solves the equation}
	$$\partial_t u_\lambda (t,\x_1,\x_2)+ 
	\langle 
	\textcolor{black}{\x_1}, D_{\x_2}u_\lambda(t,\x_1,\x_2)\rangle+\lambda^{-\frac 12}\langle {\rm sgn}(\x_2)(|\lambda^{\frac 32}\x_2|^{\beta_2^2}\textcolor{black}{\wedge 1})
	, D_{\x_2} u_\lambda(t,\x_1,\x_2)\rangle+\frac 12 \Delta_{\x_1} u_\lambda(t,\x_1,\x_2)=0
	,$$
	so that, again, the terms $\partial_t u_\lambda $, $\Delta_{\x_1} u_\lambda(t,\x_1,\x_2) $ \textcolor{black}{and $ \langle 
		\textcolor{black}{\x_1}, D_{\x_2}u_\lambda(t,\x_1,\x_2)\rangle$ associated with the principal part of the partial differential operator in the above PDE} are \textit{comparable}. On the other hand\textcolor{black}{, when $\x_2$ belongs to compact subsets of $\R$}:
	\begin{itemize}
		\item if $ \textcolor{black}{\beta_2^2}> 1/3 $, the scaled drift \textcolor{black}{term $ \lambda^{-1/2}\langle {\rm sgn}(\x_2)(|\lambda^{3/2}\x_2|^{\beta_2^2}\textcolor{black}{\wedge 1})
			, D_{\x_2} u_\lambda(t,\x_1,\x_2)\rangle$} goes to zero with $\lambda $, and the smoothing effect of the principal part of the partial differential operator dominates;
		\item if $\textcolor{black}{\beta_2^2}=1/3$, the \textcolor{black}{previous} scaled drift  coefficient stays at a macro scale and \textcolor{black}{has} the same order as the principal part of the partial differential operator (critical case);
		\item otherwise, the drift explodes when $\lambda$ goes to zero.
	\end{itemize}
	Observe that $1/3 = 1-1/(3/2)$ and that $3/2$ is precisely the self similarity index of $\int_0^{\cdot} W_s ds$, so that we recover the rule \eqref{weak_rule}. This threshold has also been shown to be (almost) sharp through counter-examples in \cite{chau:16}. \textcolor{black}{In comparison}, the threshold \textcolor{black}{introduced} in  \A{T$ {}_\beta$} gives $\textcolor{black}{\beta_2^2=}\beta_2 > 2/3 = 1 - 1/[2\times(3/2)] > 1/3$ so that, in this case, we \textcolor{black}{\textit{lose}} a factor 2 w.r.t. the thresholds predicted by \eqref{weak_rule}.\\
	
	Reproducing then the above analysis \textcolor{black}{for} the PDE \textcolor{black}{formally} associated with \eqref{SYSTPEANOPER}\textcolor{black}{,} we obtain that the \textcolor{black}{degenerate part of the scaled drift, i.e. $\sum_{i=2}^n \sum_{j=i}^n \lambda^{-i+3/2} \langle {\rm sgn}(\x_j)(|\lambda^{j-1/2}\x_j|^{\beta_i^j}\wedge 1),D_{\x_i}\rangle$,}  explode\textcolor{black}{s} when $ \beta_i^j <(2i-3)/(2j-1)$\textcolor{black}{, $i\le j$ in $\leftB 1, n\rightB^2$}. These thresholds have also been shown to be (almost) sharp in \cite{chau:meno:17} through counter-examples. On the other hand, along the diagonal of the system (i.e. for \textcolor{black}{the} indexes $\beta_j^j$), this gives that weak uniqueness fails as soon as $\beta_j^j <(2j-3)/(2j-1) = 1-1/(j-1/2)$ where $(j-1/2)$ \textcolor{black}{again} precisely \textcolor{black}{corresponds to the} self similarity index of \textcolor{black}{the} $(j-1)^{\rm{th}}$ iterated in time integral of the Brownian motion. Still \textcolor{black}{in comparison}, we assumed in \A{T$ {}_\beta$} that each $\textcolor{black}{\beta_j^j = }\beta_j$ \textcolor{black}{is} (strictly) greater than $(2j-2)/(2j-1) = 1-1/[2\times (j-1/2)]$. Therefore, we \textcolor{black}{lose} this factor 2 for each variable of the chain (note further that we also \textcolor{black}{lose} the dependence of the thresholds w.r.t. the level of the chain, \textcolor{black}{i.e. $\textcolor{black}{\beta_i^j}=\beta_j,\ i\in \leftB 1,n\rightB $}).
	%
	%
	%
	%
	%
	%
	We are hence not able to prove strong well-posedness for \eqref{eq:peano} as soon as $\beta \le 1-1/(2\gamma)$, $\gamma = n+1/2$, $n \in \mathbb N^*$. One may thus wonder if this factor 2 is \emph{the price to pay to pass from the weak to strong regularization phenomenon.}\\

	\textcolor{black}{\noindent\textbf{Strong regularization.} \textcolor{black}{The} first results on strong regularization go back to the pioneering works of Zvonkin \cite{zvonkin_transformation_1974} and Veretennikov \cite{veretennikov_strong_1980} where \eqref{eq:peano} with $\mW=W$ is shown to be well-posed in a strong sense as soon as $\beta \ge 0$ (the result holds therein for any bounded and measurable drift). This result has then been extended in the seminal work of Krylov and R\"ockner \cite{kryl:rock:05} for $L_q-L_p $ singular drifts (see also \cite{zhang_well-posedness_2010} and \cite{fedr:flan:11}). When $\mW$ is a pure jump process ($\alpha \in \textcolor{black}{[}1,2)$), Priola showed in \cite{priola_pathwise_2012} that strong well-posedness holds for any $\beta >1-\alpha/2$ for \textcolor{black}{bounded H\"older drifts} and Chen \emph{et al.} \textcolor{black}{obtained} in \cite{CZZ17} \textcolor{black}{the same condition} for $\alpha \in (0,1)$. 
	\textcolor{black}{When} $\mW$ is a fractional brownian motion with Hurst parameter $H>1/2$, it has been shown by Nualart and Ouknine \cite{NUALART2002103} that \eqref{eq:peano} is well-posed as soon as $\beta >1-1/(2H)$. This last result relies on \textcolor{black}{the} Girsanov transform and \textcolor{black}{the} Yamada\textcolor{black}{-}Watanabe Theorem to deduce that strong existence and uniqueness in law \textcolor{black}{give} strong uniqueness. Putting those results together and denoting by $\gamma$ the self similarity index of $\mW$,
		we thus get (at least in \textcolor{black}{the} stable and fractional setting) that strong well-posedness holds if $\beta > 1-1/(2\gamma)$, $\gamma \in (0,2]$ and that the ``critical'' case is attained for $\gamma=1/2$. Again, a factor 2 is lost in comparison with \eqref{weak_rule} and these results exhibit the same type of thresholds as our\textcolor{black}{s}.\\}

	\textcolor{black}{In the Markovian setting}, a good manner to understand \textcolor{black}{how} such a factor \textcolor{black}{(and thus threshold)} \textcolor{black}{appears} consists in investigating the smoothing properties of the \textcolor{black}{underlying}  PDE. It is indeed important to notice that \textcolor{black}{most of} the aforementioned results \textcolor{black}{concerning strong uniqueness}
	\textcolor{black}{are based} on the same technology\textcolor{black}{:} 
	\textcolor{black}{the Zvonkin transform of the SDE}. This is precisely where PDE\textcolor{black}{s} come into play.\\
	
	\noindent\textbf{A primer about Strong regularization and \textcolor{black}{the} Zvonkin method.} \textcolor{black}{Let us focus on the \textcolor{black}{Markovian (stable)} case.}
	The main idea consists in rewriting the dynamics \eqref{eq:peano} as
	\begin{eqnarray}\label{THE_ZVONK_GEN}
	\X_t &=& \u(t,\X_t) + \X_0 - \u(0,\X_0) - M_{0,t}(\alpha,\u,\X) + \mW_t,
	\end{eqnarray}
	with
	\begin{equation}\label{def_de_m_alpha}
	M_{0,t}(\alpha,\u,\X) \textcolor{black}{:=} \left\lbrace\begin{array}{llll}
	\displaystyle \int_0^t 
	 \mathbf {D} \u(s,\X_s) \cdot  d\mW_s,\quad \text{ if }\alpha=2;\\ 
	\displaystyle \int_0^t \int_{\R \backslash\{0\} }\textcolor{black}{\tilde N(dr,dz)}  \{\u(s,\X_{s^-}+z) - \u(s,\X_{s^-}) \},\quad \text{if }\alpha<2,
	\end{array}
	\right.
	\end{equation}
	where $\tilde N$ is the compensated Poisson measure associated with $\mW$ and where \textcolor{black}{$\u$} is the solution of 
	\begin{equation}\label{zvonk:peano}
	\partial_t \u(t,\x)+\langle {\rm sgn}(\x)(|\x|^\beta \textcolor{black}{\wedge 1}), {\mathbf D}\u(t,\x)\rangle+\frac 12 \Delta^{\frac \alpha2} \u(t,\x)= {\rm sgn}(\x)(|\x|^\beta\textcolor{black}{\wedge 1}),\quad \u(T,\cdot)=0.
	\end{equation}
	\textcolor{black}{Equation \eqref{THE_ZVONK_GEN} then follows} from the It\^o formula \textcolor{black}{provided $\u$ is smooth enough}. This transformation allows to get \textcolor{black}{rid} of the \emph{bad drift} in the equation and to replace it by the solution of a parabolic PDE which \textcolor{black}{benefits} from the smoothing effect associated with the generator of the noise $\mW$. The price to pay is that the diffusion matrix in the ``martingale'' part is now \textcolor{black}{modified},
	as one adds a stochastic integral involving the derivative or a perturbation of \textcolor{black}{$\u$}.
	
	For this new equation to be well-posed in any dimension\footnote{Note indeed that the scalar case induces very specific features, see \emph{e.g.} \cite{bass:chen:03,gradinaru_existence_2013,athreya2020,chau:meno:stable}.},
	it is commonly assumed that \textcolor{black}{the integrand} in the stochastic integral $M(\alpha,\textcolor{black}{\u},\X)$ must be Lipschitz continuous \textcolor{black}{in the spatial variable} (in order to apply a stability type argument based on martingale and Gr\"onwall inequalities). \textcolor{black}{This roughly means that for any $t$ in $[0,T]$, the gradient of $u(t,\cdot)$ must be Lipschitz continuous, uniformly in $t$, in the case $\alpha=2$. In the case $\alpha <2$, it follows from the interpolation type Lemma 4.1 in \cite{priola_pathwise_2012} that a sufficient condition is that for any $t$ in $[0,T]$, $ u(t,\cdot)$ must belong to $C^{1+\eta}$, where $C^{1+\eta}$ stands for the usual H\"older space see \emph{e.g.} \cite{kryl:96}, with $\eta >\alpha/2$. This last condition comes from integrability purposes in order to compensate the lack of integrability of the L\'evy measure associated with the pure jump process around 0 when applying \textcolor{black}{martingale inequalities} (see  \eqref{prio_modif} below as well). 
		From the corresponding parabolic bootstrap (Schauder estimates for  \eqref{zvonk:peano}, \textcolor{black}{see e.g. \cite{friedman_partial_2008,kryl:prio:10} for $\alpha=2$ or \cite{chau:meno:prio:20,miku:prag:14} in the \textcolor{black}{pure} jump case}), we could expect at best that $\| u(t,\cdot) \|_{C^{\alpha+\beta}} \le \| {\rm sgn}(\cdot)\textcolor{black}{(}|\cdot|^\beta \wedge 1\textcolor{black}{)}\|_{C^\beta}$}. \textcolor{black}{Hence}, we obtain that $\beta$ must satisfies $\alpha+\beta \textcolor{black}{>}1+\alpha/2 \Leftrightarrow \beta >1-\alpha/2$. In other words, denoting again by $\gamma$ the self similarity index of the noise, we get that strong well-posedness holds for 
	\begin{equation}\label{strong_thresh}
	\beta > 1-\frac 1{2\gamma}.
	\end{equation}
	\textcolor{black}{These are precisely the thresholds obtained in the literature on strong regularization and \textcolor{black}{we will say that} they are thus (almost) \emph{sharp w.r.t. the methodology}. Let us mention that in the critical case in \eqref{strong_thresh} (with an equality therein) when the noise is a Brownian motion, the Lipschitz property of the gradient of the solution holds in a suitable $L_p$ space only (this is the so-called Calder\'on-Zygmund estimate for the case of a bounded and measurable drift and its ad hoc version from \cite{kry:01} in the $L_q-L_p$ framework of Krylov and R\"ockner). Such type of estimates require non-trivial techniques from harmonic analysis to handle singular integrals, which differ significantly from the Schauder type control adopted here.}\\

	\noindent\textbf{The thresholds in \A{T$ {}_\beta$} are derived from the Zvonkin method.} Let us \textcolor{black}{bring to light how the Zvonkin transform naturally leads to the thresholds in \A{T$ {}_\beta$} for our simple system \eqref{SYSTPEANOPER}. The corresponding process $\mathbf X $ writes} 
	\begin{equation}\label{zvonk:proto}
	\X_t  = \u(t,\X_t) + \x_0 - \u(0,\x_0)  \textcolor{black}{- \int_0^t {\mathbf D}\u(s,\X_s) B dW_s+ W_t},
	\end{equation}
	where $\u = (\u_1,\ldots,\u_n)^*$ and each $\textcolor{black}{\u_i}$ solves
	\begin{equation}\label{EDP_deg}
	\partial_t \u_i (t,\x)+\langle \underbrace{\textcolor{black}{{\mathbf A}\x} +\mathbf {P}(\x)}_{\textcolor{black}{=\gF(\x)}}, {\mathbf D}\u_i(t,\x)\rangle+\frac 12 \Delta_{\x_1} \u_i(t,\x)= \mathbf {F}_i(\x).
	\end{equation}
	In order to have \textcolor{black}{a Lipschitz continuous in space integrand in the stochastic integral} associated with the Zvonkin transform \eqref{zvonk:proto}, we need, due to the particular structure of the embedding matrix $B$ (recall that $B=(1,{\mathbf 0}_{1,n-1})^*$), the gradient $D_1\u$ to be Lipschitz in all \textcolor{black}{the spatial} direction\textcolor{black}{s}.
	\textcolor{black}{We thus ask each component of the function $\mathbf u$ to \textcolor{black}{have} the same regularity (namely their gradient in the non-degenerate direction must be Lipschitz w.r.t. all variables) and this is the reason why the corresponding thresholds do not depend on the level of the chain, as opposed to the weak thresholds. Accordingly, from now on, we denote by $\beta_j$ the \textcolor{black}{regularity index} of any component $\mathbf F_i$, $i$ in $\leftB 1,n\rightB$ w.r.t. the $j^{\rm {th}}$ variable, $j$ in $\leftB i,n\rightB$.}
	
	Note that the main particularity of PDE \eqref{EDP_deg} \textcolor{black}{comes from} the different scales at which each component of the PDE evolves and \textcolor{black}{the unboundedness of the source terms $(\gF_i)_{i\in \leftB 1,n\rightB} $ coming from both the Zvonkin transform and the weak H\"ormander setting. \textcolor{black}{Especially}, we cannot expect the solutions $(\u_i)_{i\in \leftB 1,n\rightB} $ to be bounded}. The associated parabolic bootstrap is thus more tricky than in \eqref{zvonk:peano}. In the companion work \cite{chau:hono:meno:18}, \textcolor{black}{assuming that the source term in the right\textcolor{black}{-}hand side of \eqref{EDP_deg} is bounded and $\beta_j $-H\"older w.r.t. 
		the $j^{\rm th }$-variable}, we derived  Schauder estimates and proved that for any $i$ in $\leftB1,n\rightB$, any $j$ in $\leftB 1,n\rightB$, for fixed $(\x_{1:j-1},\x_{j+1:n})$ in $\R^{n-1}$ and $t$ in $[0,T]$, the map 
	\begin{equation}\label{def:proj}
	(\u_i)_j(t,\cdot) : \R \ni z_j \mapsto \u_i(t,\x_{1:j-1},z_j,\x_{j+1:n})\in \R
	\end{equation}
	belongs to $C^{2/(2j-1) + \beta_j}$ \textcolor{black}{ for any $\beta_j$ in $\big(0,1/[2j-1]\big)$,  uniformly w.r.t. $t$ and $(\x_{1:j-1},\x_{j+1:n})$. This means that the smoothing effect of the \textcolor{black}{hypoelliptic operator $\p_t + \Delta_{\x_1} + \langle {\mathbf A\mathbf x}, \mathbf D\rangle$} yields a regularity gain of order $2/[2j-1]$ in the $j^{\rm{th}}$ variable. In other words, the smoothing effect decreases as one moves away from the source of noise. This particular feature is reminiscent from the weak H\"ormander like structure of \eqref{EDP_deg}. Especially, one formally understand\textcolor{black}{s} that, as $(\u_i)_j(t,\cdot)$ belongs to $C^{2/(2j-1) + \beta_j}$, $(D_1\u_i)_j$ belongs to $C^{1/(2j-1) + \beta_j}$\textcolor{black}{,} i.e. \textcolor{black}{ for homogeneity reasons,} one differentiation w.r.t. the non-degenerate variable induces a loss of H\"older regularity of order $1/[2j-1]$ w.r.t. the $j^{\rm {th}}$ degenerate variable. 
		\textcolor{black}{In the current work, we first manage to extend such type of estimates to unbounded sources}. 
		\textcolor{black}{
			\textcolor{black}{As} we cannot expect anymore the solutions $ (\u_i)_{i\in \leftB 1,n\rightB}$ to be bounded, but to have linear growth \textcolor{black}{(and consequently to have bounded gradients), we} therefore specifically state the parabolic bootstrap in terms of usual H\"older spaces on the gradients. Namely, we manage below to prove that for any $\beta_j$ in $\big(\textcolor{black}{(}2j-2\textcolor{black}{)}/\textcolor{black}{(}2j-1\textcolor{black}{)},1\big)$, the map $(\textcolor{black}{D_1}\u_i)_j(t,\cdot)$ belongs uniformly to $C^{1/(2j-1) + \beta_j - \varepsilon}$ for any $0<\varepsilon<<1$ (see Section \ref{OU} for details).
		} 
		\textcolor{black}{T}o obtain \textcolor{black}{the Lipschitz control in all \textcolor{black}{the spatial} direction\textcolor{black}{s} of the gradient $D_1\u$}, we therefore need $1/(2j-1) + \beta_j > 1 \Leftrightarrow \beta_j > (2j-2)/(2j-1)$, which is precisely the thresholds assumed in \A{T$ {}_\beta$}. }  \\

	Thus, one may conclude that the thresholds in \A{T$ {}_\beta$} are the one deriving from the Zvonkin method combined with the Schauder type approach. Let us eventually mention that other \textcolor{black}{a}uthors \textcolor{black}{precisely} recover \textcolor{black}{those} thresholds in the case $n=2$, see \emph{e.g.}  \cite{chau:17,wang:zhan:16,fedr:flan:prio:vove:17}. Note also that, therein, the critical  case ($\beta_2=2/3$) for the degenerate component is left open. As already mentioned, such framework requires to use significantly different techniques that are out of the scope of the present work as well.\\

	\noindent\textbf{Path by path regularization, associated thresholds and Zvonkin thresholds.} To the best of our \textcolor{black}{knowledge, the} first result on \emph{path by path} regularization go\textcolor{black}{es} back to the work of Davie \cite{da:07} and has \textcolor{black}{then} been formalized through \textcolor{black}{the works} of Flandoli \cite{flandoli_random_2011,flan:11}. The main particularity of that approach relies on the fact that the system \eqref{eq:peano} is considered ``$\omega$ by $\omega$'', so that it \textcolor{black}{somehow} goes beyond the probabilistic framework. This setting thus matches the rough path perspective for SDE\textcolor{black}{s}. Based on these considerations, Catellier and Gubinelli proposed in \cite{catellier_averaging_2012} a systematic study of fractional Brownian perturbation\textcolor{black}{s} of ODEs, for any \textcolor{black}{H}urst parameter $H$ in $(0,1)$. Therein, they prove in particular that \eqref{eq:peano} is well-posed as soon as $\beta >1-1/(2H)$. Therefore, the \textcolor{black}{a}uthors obtained the same thresholds as the one required for strong well-posedness in the literature although the approach, especially for \cite{catellier_averaging_2012}, differs significantly from the \textcolor{black}{PDE  based trick of the} Zvonkin method. \textcolor{black}{Furthermore}, it seems  that, \emph{modulo} an additional work, \textcolor{black}{the} Zvonkin \textcolor{black}{transform} allows to recover, \textcolor{black}{in the Markovian framework}, path by path results. We refer \emph{e.g.} to Shaposhnikov \cite{shaposhnikov_2016}, who revisited the  result by Davie, and to Priola \cite{priola2018,priola2019davies}, \textcolor{black}{who extended the Davie result to the $\alpha$-stable, $\alpha$ in $\textcolor{black}{[}1,2)$ and the degenerate kinetic framework}, for recent works in that direction.  
	
	\textcolor{black}{From the previous considerations, it is in fact} not clear that \textcolor{black}{the} Zvonkin thresholds \textcolor{black}{derive} from the PDE approach. \textcolor{black}{We rather feel that they are actually related to the type of well-posedness considered} and that the difference with the ``weak thresholds'' is indeed a price to pay to pass from weak to strong or path by path uniqueness.\\

	\noindent\textbf{Conclusions and extension to ``pure jump'' noise.} In view of the previous discussions, it appears that our thresholds are sharp w.r.t. the method employed, and almost sharp w.r.t. the existing literature (at least the one we know) on regularization by noise. In this perspective, our result roughly says (in particular) that a Brownian \textcolor{black}{type} noise $\mW$ of self similarity index  $\gamma = 1/2+n$, $n$ in $\textcolor{black}{\mathbb N}$ restores strong well-posedness for ODEs whose drift \textcolor{black}{has} regularity index $\beta > 1-1/(2\gamma)$. As the range $\gamma=H$ in $(0,1)$ is covered by \cite{catellier_averaging_2012} and the range corresponding formally to the case $H=0$ has been investigated recently in \cite{harang2020cinfinity}, one may wonder if the range $\gamma = 1/\alpha+n$, $n$ in $\textcolor{black}{\mathbb N}$ and $\alpha$ in $(0,2)$ could be attainable as well.\\
	
	\textcolor{black}{\textcolor{black}{
	The main point to implement the Zvonkin transform consists in establishing parabolic bootstrap results on the underlying PDE. To this end, the idea is to expand the associated differential operator around  the generator of a \textit{suitable} \textit{proxy process}.  This is the so-called parametrix type expansion that will be here performed at order one. The \textit{proxy process} needs somehow to  resemble the initial process to be investigated and to be \textit{well understood}.
	 For instance, for the considered example \eqref{SYSTPEANOPER}, the natural \textit{proxy process} corresponds to $\bar \X $ (degenerate Ornstein-Uhlenbeck type Gaussian process) introduced after \eqref{SYSTPEANOPER}. By \textit{well understood,} we mean that the proxy process admits smooth marginal densities, which together with their derivatives,  satisfy appropriate heat kernel estimates. 
	 These estimates turn out to be} crucial as our methodology relies on duality arguments between Besov spaces, the underlying H\"older spaces being viewed as Besov spaces. Especially, because we will use the so-called thermic characterization of Besov spaces (see \emph{e.g.} Chapter 2.6.4 in \cite{trie:83}) with an underlying heat kernel somehow compatible with the one of the \textit{proxy}. 
		This is precisely why we feel that the methodology provide\textcolor{black}{d} here is robust enough to handle the case of a large class of symmetric non-degenerate pure jump noises ``$\mW=Z^\alpha$'', $\alpha$ in $(0,2)$,  up to a modification of the associated \textit{proxy} and heat kernels, see \emph{e.g.} \cite{huan:meno:15,HUANG2019162,mar:20}.
	}
	
	Let us try to understand what are the main steps by restricting ourselves to the (already discussed) prototype system \eqref{SYSTPEANOPER}.
	\textcolor{black}{In this case, the Zvonkin transform associated with \eqref{SYSTPEANOPER} writes
		\begin{equation}\label{zvonk:proto:stable}
		\X_t  = \u(t,\X_t) + \X_0 - \u(0,\X_0)  \textcolor{black}{-} \int_0^t \int_{\R\backslash\{0\}}\textcolor{black}{\tilde N_s(ds,dz)} (\u(s,\X_{s^-}+ Bz) - \u(s,\X_{s^-})) +\textcolor{black}{ Z_t^\alpha,} 
		\end{equation}
		where $\u = (\u_1,\ldots,\u_n)^*$ and each $\textcolor{black}{\u_i}$ solves the Integro-Partial Differential Equation (IPDE in short) associated with \eqref{strong_SYST} i.e. the Laplacian therein is replaced by a fractional Laplacian $\Delta^{\alpha/2}_{\x_1}$. As already discussed, we need the \textcolor{black}{integrand in the} stochastic integral in \eqref{zvonk:proto:stable} to be Lipschitz \textcolor{black}{w.r.t. the spatial variable}. Yet again, a sufficient condition is given by a modification of the interpolation type Lemma 4.1 in \cite{priola_pathwise_2012}: we have for $\x$, $\y$ in $\R^{\textcolor{black}{n}}$ and $z$ in $\R$, $|z| \le 1$:
		\begin{equation}\label{prio_modif}
		|\u(t,\x+Bz) - \u(t,\x) -\u(t,\y+Bz) + \u(t,\y)| \le C |z|^{\eta} [(\mathbf D\u)_1]_\eta |\x-\y|,
		\end{equation}
		where $[(\mathbf D\u)_1]_\eta$ denotes the $\eta$-H\"older modulus w.r.t. the variable $\x_1$ of the full gradient $\mathbf D \u$. We therefore have to control this $\eta$-H\"older modulus.
	}
	
	In \cite{mar:20}, Marino proved Schauder estimate\textcolor{black}{s} for (more general versions of) the IPDE associated with \textcolor{black}{\eqref{SYSTPEANOPER}}. He \textcolor{black}{established} therein that for any $j$ in $\leftB 1,n\rightB$, for fixed $(\x_{1:j-1},\x_{j+1:n})$ in $\R^{n-1}$ and $t$ in $[0,T]$, \textcolor{black}{the map $(\u_i)_j(t,\cdot)$ \textcolor{black}{belongs to} $C^{\alpha/[1+\alpha(j-1)]+\beta_j}$ uniformly in $t$ and $(\x_{1:j-1},\x_{j+1:n})$, provided that \textcolor{black}{for each $j$ in $\leftB 1,n\rightB$: $\alpha > \beta_j$ (this condition comes from integrability purpose\textcolor{black}{s}), $\alpha+ \beta_j >1$ and $\beta_j$ in $\big(0,1/\textcolor{black}{(}1+\alpha(j-1)\textcolor{black}{)}\big)$}. This means that the smoothing effect of the hypoelliptic operator $\p_t + (1/2)\Delta_{\x_1}^{\alpha/2 } +  \langle \mathbf A{\mathbf x}, {\mathbf D} \rangle$ yields a regularity gain of order $\alpha/[1+\alpha(j-1)]$ in the $j^{\rm{th}}$ variable. We thus understand (at least formally), that for each $j$ in $\leftB 1,n \rightB$, $D_j\textcolor{black}{\big(\u(t,\cdot)\big)_j}$
		belongs to $C^{\alpha/[1+\alpha(j-1)]+\beta_j-1}$. Note that, due to homogeneity, controlling the the $\eta$-H\"older modulus of $D_j\textcolor{black}{\u(t,\cdot)}$ w.r.t. the non-degenerate variable amounts to control the $\eta/[1+\alpha(j-1)]$-H\"older modulus of $\textcolor{black}{\big(D_j\u(t,\cdot)\big)_j}$ i.e. a $\eta$-H\"older regularity w.r.t. the non-degenerate variable corresponds to a $\eta/[1+\alpha(j-1)]$-H\"older regularity w.r.t. the $j^{{\rm{th}}}$ variable. According to the associated parabolic bootstrap, we thus need the following condition to be fulfilled:
		$$\frac{\alpha}{1+\alpha(j-1)} + \beta_j - 1\textcolor{black}{\ge \frac{\eta}{1+\alpha(j-1)}} > \frac{\alpha/2}{1+\alpha(j-1)} \Leftrightarrow \beta_j > \frac{1+\alpha(j-3/2)}{1+\alpha(j-1)},$$
		recalling that, for integrability purposes coming from the L\'evy measure associated with the pure jump process, $\eta$ must again be strictly greater than $\alpha/2$.
		These are precisely the thresholds obtained in the recent work \cite{HAO2020139} in the kinetic case (i.e. when $n=2$). Following our method, we should be able to extend in some sense the Schauder control in order to obtain that for any $\beta_j$ in $\big(\textcolor{black}{(}1+\alpha(j-3/2)\textcolor{black}{)}/\textcolor{black}{(}1+\alpha(j-1)\textcolor{black}{)},1\big)$ the map $(\u_i)_j(t,\cdot)$ belongs uniformly to $C^{\alpha/[1+\alpha(j-1)]+\beta_j - \varepsilon}$ for any $0<\varepsilon<<1$, so that the result would follow.}

	\subsection{Proof of the main result: Zvonkin Transform and smoothing properties of the PDE associated with \eqref{strong_SYST}}\label{strong_SUB_ZVONK}
	
	We emphasize that under our assumptions, it follows from \cite{chau:meno:17} that \eqref{strong_SYST} is well-posed in the weak sense. Hence, from \textcolor{black}{the} Yamada-Watanabe theorem it is sufficient to prove that strong \textcolor{black}{(or pathwise)} uniqueness holds to prove strong well-posedness. \textcolor{black}{As explained above,} our main strategy rests upon the Zvonkin transform initiated by Zvonkin in \cite{zvonkin_transformation_1974} 
	which heavily relies on the \textcolor{black}{connection} between SDEs and PDEs.
	%
	%
	We rewrite \eqref{strong_SYST} in the shortened form 
	\begin{equation*}
	d{\mathbf X}_t = {\mathbf F}(t,{\mathbf X}_t) dt+ B \sigma(t,{\mathbf X}_{t}) dW_t,
	\end{equation*}
	where we recall that ${\mathbf F}=(\gF_1,\dots,\gF_n)$ is an $\R^{nd}$-valued function and $B$ is the embedding matrix from $\R^d$ into $\R^{nd}$, \textcolor{black}{i.e. $B= ({\mathbf I}_{d , d} , {\0}_{d,d}, \dots, \0_{d,d})^*=({\mathbf I}_{d , d} , {\0}_{d,(n-1)d})^*$, \textcolor{black}{where $ {\mathbf I}_{d , d} , {\0}_{d,d}$  respectively denote the identity matrix and the matrix with zero entries in $\R^d\otimes \R^d $}}. For all $\varphi \in C_0^2 (\R^{nd},\R^{}) $ and $(t,\x)\in [0,T]\times \R^{nd} $ let
	\begin{equation}
	\label{strong_GENERATOR}
	L_t\varphi(\x)=\langle \gF(t,\x),{\mathbf D} \varphi(t,\x)\rangle+\frac 12 {\rm Tr}\Big(a(t,\x)D_{\x_1}^2\varphi (\x) \Big),
	\end{equation}
	where $a=\sigma\sigma^*$\textcolor{black}{,} denote the generator associated with \eqref{strong_SYST}. We then \textit{formally} associate the SDE \eqref{strong_SYST} with the following systems of PDEs:
	\begin{equation}
	\label{strong_GEN_PDE}
	\begin{cases}
	(\partial_t \u_i+L_t \u_i)(t,\x)= \gF_i (t,\x), \ (t,\x)\in [0,T)\times \R^{nd},\\
	\u_i(T,\x)=\0_{d},\qquad i\in \leftB 1,n\rightB.
	\end{cases}
	\end{equation}
	\begin{REM}\label{nota_F}
		Note that above we adopted the following convention for notational convenience: for each $i$ in $\leftB 3,n\rightB$, all $t$ in $[0,T]$, $\x$ in $\R^{nd}$, $\gF_i(t,\x) := \gF_i(t,\x_{i-1:n})$ i.e. the independence of each map $\gF_i(t,\cdot)$ with respect to the first $i-2$ components of the vector $\x=(\x_1,\hdots,\x_n) \in \R^{nd}$ is implicitly assumed. 
	\end{REM}
	Denote by $\U=(\u_1,\textcolor{black}{\hdots},\u_n)$ its global solution. 
	Let now  $({\mathbf F}^m)_{m \geq 0}, (a^m)_{m \geq 0}$ denote two  sequences of mollified coefficients satisfying assumption \A{A} uniformly in $m$ that are infinitely differentiable functions with bounded derivatives of \textcolor{black}{all order\textcolor{black}{s}} greater than 1 for ${\mathbf F}^m$, and converging in supremum norm to $({\mathbf F},a)$ (such sequences are easily obtained from \cite{chau:17}). Then, for each $m$, \textcolor{black}{the regularized systems of PDEs associated with \eqref{strong_GEN_PDE} write}:
	\begin{equation}
	\label{strong_GEN_PDE_MOLL}
	\begin{cases}
	(\partial_t \u_i^m+L_t^m \u_i^m)(t,\x)= \gF_i^m (t,\x), \ (t,\x)\in [0,T)\times \R^{nd},\\
	\u_i^m(T,\x)=\0_{d},\qquad i\in \leftB 1,n\rightB,
	\end{cases}
	\end{equation}
	where $L_t^m$ is obtained from \eqref{strong_GENERATOR} replacing $\gF$ by $\gF^m $ and $a $ by $a^m$.

	The above system \eqref{strong_GEN_PDE_MOLL} is well-posed and admits a unique smooth solution $\U^m=(\u_1^m,\textcolor{black}{\hdots},\textcolor{black}{\u_n^m})$. \textcolor{black}{This can be derived from the Feynman-Kac representation formula, which holds independently of the degeneracy for smooth coefficients with linear growth, and stochastic flow techniques 
		(see e.g. Gikhman and Skorokhod \cite{gikh:skor:69} or Talay and Tubaro \cite{tala:tuba:90})}. Hence, applying It\^o's formula, one easily deduces that
	\begin{equation}\label{strong_ZVONK_TR}
	\int_0^t ds{\mathbf F}(s,\X_s)  = -\U^m(0,\x) + \U^m(t,\X_t) - \int_0^t \mathbf{D}\U^m(s,\X_s) B \sigma(s,{\mathbf X}_{s})d W_s + \int_0^t \mathcal{R}^m_s(\X_s)ds ,
	\end{equation} 
	where
	\begin{equation}
	\begin{split}
	\mathcal{R}^m_s(\X_s) := &\textcolor{black}{[} \mathbf{F}(s,\X_s) - \mathbf{F}^m(s,\X_s) \textcolor{black}{]} -  (L_s - L_s^m)\U^m(s,\X_s)\\
	\textcolor{black}{=} & \textcolor{black}{ \textcolor{black}{[} \mathbf{F}(s,\X_s) - \mathbf{F}^m(s,\X_s)\textcolor{black}{]} -  (\gF - \gF^m)(s,\X_s)\cdot {\mathbf D}\U^m(s,\X_s)}\\
	&\textcolor{black}{-\frac 12  {\rm Tr }\big((a - a^m)(s,\X_s) D_{\x_1}^2\U^m(s,\X_s)\big)}.
	\end{split}\label{EXP_REMAINDER}\tag{R${}_m $}
	\end{equation}
	\textcolor{black}{This representation \eqref{strong_ZVONK_TR} is the \emph{Zvonkin Transform}} discussed above, up to a remainder.
	Then, the main idea consists in taking advantage of the regularization properties of the operator \textcolor{black}{$L^m$ (uniformly in $m$)} and expect that the solutions $\U^m$, $m\geq 0$ will  be smoother than the source term ${\mathbf F}$ so that the right\textcolor{black}{-}hand side of \eqref{strong_ZVONK_TR} is smoother than the integrand of the left\textcolor{black}{-}hand side of the consider\textcolor{black}{ed} equation. In other words, we are looking for a \emph{good regularization theory} for the PDE \eqref{strong_GEN_PDE_MOLL} uniformly w.r.t. the mollification argument. This \emph{good regularization theory} is summarized in the following crucial result whose proof is, in fact, the main subject of this work and is postponed to Section \ref{strong_SEC_PROOF_PDE}.
	
	\begin{THM}\label{strong_MAIN_RESULT_PDE}
		\textcolor{black}{For $T>0$ small enough\footnote{By ``small enough" we mean that there exists a time $\mathcal{T}>0$ depending \textcolor{black}{only} on known parameters in \A{A} s.t. for any $T\le {\mathcal T} $ the statement of the theorem holds.}}, there exists a constant $C_T := C_T(\A{A})>0$ 
		satisfying $C_T \to 0$ when $T\to 0$ such that for every $m \geq 0$, the solution $\U^m$ satisfies\textcolor{black}{,} \textcolor{black}{with the notation of \eqref{strong_NORMES}}:
		\begin{equation}\label{strong_GEN_ESTI_PDE}
		\|{\mathbf D} \U^m\|_{\infty} + \|{\mathbf D} (D_1\U^m )\|_{\infty} \leq C_T.
		\end{equation}
	\end{THM}
	\begin{REM}[On well-posedness of the initial PDE \eqref{strong_GEN_PDE}]
		\textcolor{black}{We also point out that, from the uniformity in $m$  in the previous theorem, we could also derive some regularizing properties for the system \eqref{strong_GEN_PDE}  through appropriate compactness arguments. Indeed, as it will appear in the proof of this result, we are in fact able to control uniformly the H\"older moduli of the gradients and of the second order derivatives w.r.t. the non-degenerate direction (see Lemmas \ref{strong_LEMME_CTR_HOLDER} and \textcolor{black}{\ref{Lemme_Taylor_reverse}}). These controls precisely allow to derive, through the Arzel\`a-Ascoli theorem, a well-posedness result for equation \eqref{strong_GEN_PDE} under the sole assumption \A{A} as well as the above gradient estimates.}
	\end{REM}
	
	\textcolor{black}{Theorem \ref{strong_MAIN_RESULT_PDE}, which is a consequence of Theorem \ref{strong_THM_DER_PDE} below, is the key to prove our main result for strong uniqueness.}
	\begin{proof}[\textcolor{black}{Proof of Theorem \ref{strong_MAIN_RESULT}}]
		Let now $\X$ and $\X'$ be t\textcolor{black}{w}o solutions of \eqref{strong_SYST}. Using \textcolor{black}{the} representation \eqref{strong_ZVONK_TR} to express the difference of the \textit{bad} drifts in terms of the function $\U^m$ and its derivative up to a remainder, we write:
		\begin{eqnarray*}
			&&\X_t - \X_t' 
			\nonumber \\
			&=&  \U^m(t,\X_t) -\U^m(t,\X_t')-\int_0^t \left[ {\mathbf D\U^m(s,\X_s)}B\sigma(s,\X_s)-{\mathbf D\U^m(s,\X_s')}B\sigma(s,\X_s') \right]dW_s   \\
			&& +\Big[{\mathcal R}_t^m(\X)-{\mathcal R}_t^m(\X')\Big]+\int_0^t B\left[\sigma(s,\X_s)-\sigma(s,\X_s') \right] dW_s.
		\end{eqnarray*}

		Take then \textcolor{black}{the} supremum \textcolor{black}{in time} of the square of the difference. Passing to the expectation,  \textcolor{black}{ a convexity inequality} then leads to the following estimate:
		\begin{eqnarray*}
			&&\E \left[\sup_{t \leq T}|\X_t - \X_t' |^2\right] 
			\nonumber \\
			&\leq& \textcolor{black}{5}\Big( \E\left[ \sup_{t \leq T} |\U^m(t,\X_t) -\U^m(t,\X_t')|^2   \right] 
			+ \E \left[\int_0^Td s \left|\left[\mathbf{D}_{}\U^mB\right](s,\X_s)-\left[{\mathbf D}\U^mB\right](s,\X'_s)\right|^2 \left\|\sigma\right\|_{\infty}^2  \right]\\
			&&  +  \E \left[\int_0^Td s (\left\| {\mathbf D}\U^mB \right\|_{\infty} + \mathbf{1}) \left|\left[\sigma(s,\X_s) - \sigma(s,\X'_s)\right]\right|^2  \right] + 2 T\|\mathcal{R}_\cdot^m(\cdot)\|_{\infty}^{\textcolor{black}{2}}\Big).
		\end{eqnarray*}
		\textcolor{black}{Note that thanks to the particular structure of $B$ one has $ {\mathbf D}\U^mB = (\textcolor{black}{D_{1}}\U^m,{\mathbf 0}_{d,d},\ldots, {\mathbf 0}_{d,d})^*$. Hence, }thanks to Theorem \ref{strong_MAIN_RESULT_PDE} and \textcolor{black}{Gr\"onwall's lemma}, 
		there exists  $\bar C_T:=\bar C_T(C_T,\sigma,n,d,T)$ satisfying $\bar C_T\to 0$ when $T$ goes to $0$ such that 
		\begin{eqnarray}
		\E \left[\sup_{t \leq T}|\X_t - \X_t' |^2\right] \leq \bar C_T \E\left[ \sup_{t \leq T} |\X_t -\X_t'|^2   \right]  +\textcolor{black}{10} T \|\mathcal{R}_\cdot^m(\cdot)\|_{\infty}^{\textcolor{black}{2}}.
		\end{eqnarray}
		Letting $m\to + \infty$, \textcolor{black}{since $\|a-a^m\|_\infty $ and $\|\gF-\gF^m\|_\infty $ tend to $0$, it readily follows from \eqref{EXP_REMAINDER} and the bound \eqref{strong_GEN_ESTI_PDE} in Theorem \ref{strong_MAIN_RESULT_PDE} that $ \|\mathcal{R}_\cdot^m(\cdot)\|_{\infty} \underset{m}{\rightarrow}0$. Hence}, choosing $T$ small enough so that $\bar C_T \leq 1/2$, we deduce that strong uniqueness holds  on a sufficiently small time interval. Iterating this procedure in time gives the result on $\R^+$ \textcolor{black}{from usual Markov arguments involving the regular versions of conditional expectations, see e.g. \cite{stro:vara:79}}.
	\end{proof}

	\subsection{Regularization properties of the underlying PDE \eqref{strong_GEN_PDE}: strategy of \textcolor{black}{the} proof and primer}\label{strong_SUB_STRAT_PDE}
	As mentioned above, the regularization properties of the PDE \eqref{strong_GEN_PDE_MOLL} given by estimate \eqref{strong_GEN_ESTI_PDE} in Theorem \ref{strong_MAIN_RESULT_PDE} are the core of this work. \textcolor{black}{Smoothing properties of linear partial differential operators of second order with non-degenerate diffusion matrix have been widely studied in the literature for bounded H\"older coefficients. In that setting, the estimates of Theorem  \ref{strong_MAIN_RESULT_PDE} are well known, see \emph{e.g.} the book of Friedman \cite{frie:64} or Bass \cite{bass:97}. For unbounded drift and source terms, those estimates have been recently established in \cite{kryl:prio:10}.} In our case, the story is rather different since the diffusion matrix $Ba$ of the system is totally degenerate in the directions $2$ to $n$. However, as we already emphasized, the non-degeneracy condition assumed on the family of \textcolor{black}{Jacobian matrices} $(D_{\x_{i-1}}\gF_i)_{i \in \leftB 2,n\rightB}$ allows the noise to propagate in the 
\textcolor{black}{indicated}	directions thanks to the drift. It can be viewed as a weak type \textcolor{black}{of} H\"ormander condition. Under such a condition, the operator $L^m$ with mollified coefficients is said to be hypoelliptic\footnote{Pay attention that this is not the case for $L$ whose coefficients do not have the required smoothness in \A{T${}_\beta $} to compute the corresponding Lie brackets.} and it is well known that hypoelliptic differential operators also have some smoothing properties (see the seminal work of H\"ormander \cite{hormander_hypoelliptic_1967} or, for a probabilistic viewpoint, the \textcolor{black}{monograph} of Stroock \cite{stroock_partial_2008}). The tricky point in our \emph{weak H\"ormander setting} is that the pointwise gradient estimates \eqref{strong_GEN_ESTI_PDE} of Theorem \ref{strong_MAIN_RESULT_PDE} had, to the best of our knowledge, not \textcolor{black}{been established yet}, \textcolor{black}{even though}  such a setting has already been considered by several authors (see \emph{e.g.} Delarue and Menozzi \cite{dela:meno:10} for density estimates, Menozzi \cite{meno:10}, \cite{meno:17} and Priola \cite{prio:15} for the martingale problem and also \textcolor{black}{Bramanti \textit{et al.}} \cite{bram:cupi:lanc:prio:09}, \cite{bram:cupi:lanc:prio:13} for related $L^p$ estimates and Bramanti and Zhu \cite{bram:zhu:11} for the VMO framework). We can mention the work of Lorenzi \cite{lore:05} which gives gradient estimates in the degenerate kinetic like case ($n=2$ in our framework) when the diffusion coefficient is sufficiently smooth and the drift linear. We  point out that our main estimate in Theorem \ref{strong_MAIN_RESULT_PDE} needs precisely to be uniform w.r.t. the mollification parameter and therefore does not depend on the smoothness of $\gF^m,a^m$, but only on known parameters appearing in \A{A}. Again, this is what would also allow to transfer those bounds to equation \eqref{strong_GEN_PDE} from a suitable compactness argument, extending well known results for non-degenerate diffusions with H\"older coefficients to the current degenerate setting. \\
	
	To prove this result our main strategy rests upon the \emph{parametrix approach} see \emph{e.g.} the work of McKean and Singer \cite{mcke:sing:67} or the book of Friedman \cite{frie:64}. Roughly speaking, \textcolor{black}{it is a perturbative argument consisting in expanding the operator $L^m$  around a \emph{good proxy}}, usually denoted by $\tilde L^m$ (we keep here the super-scripts in $m$ to emphasize that the perturbative technique we perform will concern the system \eqref{strong_GEN_PDE_MOLL} with mollified coefficients).  \textcolor{black}{In our setting, the term \emph{good proxy} relates} to the fact that the operator $\tilde L^m$ is the generator of the ``closest'' Gaussian approximation  $\tilde \X^m$ of $\X^m$ which has generator $L^m$. In our case, such a process is well known and is the linearized (with respect to the source of noise) version of \eqref{strong_SYST} whose coefficients are frozen along the curve $(\btheta_{s,t}^m)_{s\in [t,T]}$ that solves the deterministic counterpart of \eqref{strong_SYST} with mollified coefficients (\emph{i.e.} with $\sigma^m \equiv \mathbf{0}_{d,d}$) namely, $\dot \btheta_{s,t}^m = {\mathbf F}^m(s,\btheta_{s,t}^m)$. This \textcolor{black}{\textit{proxy}} process \textcolor{black}{$\tilde \X^m $} may be seen as a (non-linear) generalization of the so-called Kolmogorov example \cite{kolm:33} and we refer the reader to the work of Delarue and Menozzi \cite{dela:meno:10} and Menozzi \cite{meno:10} for more explanations. Having this \textit{proxy} at hand, the parametrix procedure consists in deriving the desired estimates for the \textit{proxy} and control the expansion error.

	In \cite{chau:meno:17}, Chaudru de Raynal and Menozzi successfully used this approach in its \emph{backward} form to prove weak well-posedness of \eqref{strong_SYST} under less restrictive assumptions (the critical thresholds for the H\"older exponents being smaller as indicated above). In that case, the curve along which the system is frozen for the \textit{proxy} is the solution of the \emph{backward} deterministic counterpart of \eqref{strong_SYST}). This \emph{backward} approach is very suitable when investigating the martingale problem associated with our main system since it allows to control subtly the expansion error associating precisely the coefficients $\gF_i $ with their corresponding differentiation operator $D_i $ and does not require any mollification of the coefficients. Unfortunately, when trying to obtain estimates on the derivatives of the solutions of the PDE, the \emph{backward} approach is not \textcolor{black}{convenient} since the corresponding \textit{proxy} does not provide an exact density and this fact does not allow to benefit from cancellation techniques which are very helpful in this context (see paragraph below).

	Hence, our parametrix approach will be of \emph{forward}\footnote{Meaning that the freezing curve $\btheta$ solves the corresponding ODE associated  with \eqref{strong_SYST}  in a forward form.} form as done in the work of Chaudru de Raynal \cite{chau:17}. This is, in fact, a non-trivial generalization of the approach developed in the aforementioned paper where the strong well-posedness of \eqref{strong_SYST} is obtained when $n=2$. Indeed, the strategy used in \cite{chau:17} is not adapted to this general case because of some subtle phenomena appearing only when $n\geq 3$. In particular, the singularities appearing when considering the remainder term of the parametrix were in \cite{chau:17} equilibrated \textit{at hand} through elementary cancellation arguments, whereas the current approach takes advantage of the full-force duality results between Besov spaces (see Sections \ref{strong_HINT_BESOV_1} and \ref{strong_SECTION_CTR_SENSI} below). This forward perturbative approach has also been successfully used in \cite{chau:hono:meno:18} to establish some weaker regularization properties of the PDE \eqref{strong_GEN_PDE} through appropriate Schauder estimates.\\
	
	\subsubsection{Regularizing properties of the degenerate Ornstein-Uhlenbeck \textit{proxy}}\label{OU}
	When exploiting such a \emph{forward} parametrix approach, a good primer to understand what could be, at \textcolor{black}{best,}
	\textcolor{black}{expected}, consists in investigating the regularization properties of the \textit{proxy} operator $\tilde L$. To be as succinct as possible, let us consider the case where $(\tilde L_t)_{t\ge 0}$ is the generator of a degenerate Ornstein-Uhlenbeck process $(\tilde \X_t)_{t\ge 0} $ with dynamics: 
	\begin{equation}
	\label{strong_OU_DYN}
	\textcolor{black}{d\tilde \X_t= \mathbf A_t \tilde \X_t dt + Bd W_t},
	\end{equation}
	where ${\mathbf A}_t$ is the $nd \times nd$ matrix with sub-diagonal blocks $({\mathbf a}_{i,i-1}(t))_{i \in \leftB2,n\rightB}$ of size $d \times d$ and ${\mathbf 0}_{d ,d}$ elsewhere. \textcolor{black}{In particular, 
		\begin{equation}\label{strong_def_A}
		{\mathbf A}_t=\left (\begin{array}{ccccc} \0_{d,d} & \cdots & \cdots &\cdots  & \0_{d,d}\\
		{\mathbf a}_{2,1} (t)& \0_{d,d} &\cdots &\cdots &\0_{d,d}\\
		\0_{d,d} & {\mathbf a}_{3,2}(t)&\0_{d,d} & \cdots &\0_{d,d}\\
		\vdots &  \0_{d,d}                     & \ddots & \textcolor{black}{\ddots}&\textcolor{black}{\vdots}\\
		\0_{d,d} &\cdots &  \0_{d,d}   &  {\mathbf a}_{n,n-1}(t)& \0_{d,d}
		\end{array}\right) .
		\end{equation}} The entries $( {\mathbf a}_{i,i-1}(t))_{i\in \leftB 2,n\rightB} $ are \textcolor{black}{uniformly in time} non-degenerate elements of $\R^d\otimes \R^d $ (which expresses the weak H\"ormander condition). The corresponding generator  $\tilde L_t $ writes for any $\varphi\in C_0^2(\R^{nd},\R) $:
	$$\tilde L_t\varphi(\x)=\langle {\mathbf A}_t\x, {\mathbf D} \varphi(\x)\rangle +\frac 12 \Delta_{\x_1} \varphi(\x).$$

	In such a case,  
	each component $\tilde \u_i$, $i\in \leftB1,n\rightB$ of the solution $\tilde \U$ of the corresponding system of PDEs
	\textcolor{black}{\begin{equation}
		\label{strong_GEN_PDE_A}
		\begin{cases}
		(\partial_t + \tilde L_t)\tilde \U(t,\x)= \gF (t,\x), \ (t,\x)\in [0,T)\times \R^{nd},\\
		\tilde \U(T,\x)=\0_{{nd}},
		\end{cases}
		\end{equation}}
	where $\gF $ is a non-linear (non-mollified) source satisfying \A{T${}_\beta $}, writes
	\begin{equation}\label{strong_OU_SOURCE}
	\tilde \u_i(t,\x) = -\int_t^T \textcolor{black}{ds} \int_{\R^{nd}} d\y \gF_i(s,\y) \textcolor{black}{\tilde  p(t,s,\x,\y)}.
	\end{equation}
	Above, $\textcolor{black}{\tilde p}$  
	stands for the transition density \textcolor{black}{of the Gaussian process $(\textcolor{black}{\tilde \X_v})_{v\ge 0} $ with dynamics \textcolor{black}{\eqref{strong_OU_DYN}.}}
	Using the resolvent associated with $(\mathbf A_v)_{v\in [t,s]} $, i.e. $\partial_s \tilde \gR_{s,t}=\mathbf A_s \tilde \gR_{s,t}, \ \tilde \gR_{t,t}={\mathbf I}_{nd,nd}  $, the above equation can be explicitly integrated. Precisely, for a fixed starting point $\x $ at time $t$:
	\begin{equation}\label{strong_INTEGRATED_OU}
	\tilde \X_v=\tilde \gR_{v,t}\x+\int_t^v \tilde \gR_{v,u}B dW_u.
	\end{equation}
	Hence, \textcolor{black}{at time $s>t $} the \textcolor{black}{covariance matrix of the random variable $\tilde \X_s$ writes  as}  $\tilde \K_{s,t}:=\int_t^s du \tilde \gR_{s,u}BB^* \tilde \gR_{s,u}^*  $.
	From \eqref{strong_INTEGRATED_OU}, the density at time $v=s$ and at the spatial point $\y$ therefore writes:
	\begin{equation}\label{strong_DENS_OU}
	\tilde p(t,s,\x,\y)=\frac{1}{(2\pi)^{\frac {nd}2}{\rm det}(\textcolor{black}{\tilde \K_{s,t}})^{\frac 12}}\exp\left( -\frac 12\Big\langle (\textcolor{black}{\tilde \K_{s,t}})^{-1}\big(\tilde \gR_{s,t}\x-\y\big), \tilde \gR_{s,t}\x-\y\Big\rangle\right).
	\end{equation}
	Note that the resolvent also appears in \eqref{strong_INTEGRATED_OU} and in the density. Since the drift in \eqref{strong_OU_DYN} \textcolor{black}{is} unbounded, the term $\tilde \gR_{s,t}\x $  actually corresponds to the transport of the initial condition $\x $ through the associated deterministic and linear differential system. It is well known, see e.g. \cite{dela:meno:10} and Section \ref{strong_SEC_GAUSS_DENS} below, that the covariance $\tilde \K_{s,t}$ enjoys what we will call a \textit{good scaling property}. Precisely, for a given $T>0$ there exists $C:=C\big((\mathbf A_v)_{v\in [0,T]},T\big)\ge 1$ s.t. for any $\bxi \in \R^{nd} $,
	\begin{equation}
	\label{strong_GSP_OU}
	C^{-1} (s-t)^{-1}|\T_{s-t}\bxi|^2\le \langle \tilde \K_{s,t}\bxi,\bxi\rangle \le C (s-t)^{-1}|\T_{s-t}\bxi|^2,
	\end{equation}
	where for any $u>0$, we denote by $\T_u $ the \emph{intrinsic} scale matrix:
	\begin{equation}
	\label{strong_DEF_T_ALPHA}
	\T_u=\left( \begin{array}{cccc}
	u\mathbf I_{d,d}& \0_{d, d}& \cdots& \0_{d, d}\\
	\0_{d, d}   &u^2 \mathbf I_{d, d}&\0_{d,d}& \vdots\\
	\vdots & \ddots&\ddots & \vdots\\
	\0_{d, d}& \cdots & \0_{d,d}& u^{n}\mathbf I_{d , d}
	\end{array}\right).
	\end{equation}
	Importantly, the \textit{good scaling property} stated in \eqref{strong_GSP_OU} indicates that, for a given initial time $t$ and for each $i\in \leftB 1,n\rightB$, each  $\R^{d} $-valued component $\tilde \X_s^i $ has typical fluctuations of order $(s-t)^{i- 1/2} $ which \textcolor{black}{correspond} to those of the $(i-1)^{{\rm th}} $ iterated integrals of the Brownian motion. Accordingly, we derive that the frozen density $\tilde p $ also satisfies the bound

	\begin{equation*}
	\tilde p(t,s,\x,\y)\le \frac{C}{(s-t)^{\frac{n^2d} 2}}\exp\left(-C^{-1}(s-t)|\T_{s-t}^{-1}(\tilde \gR_{s,t}\x-\y)|^2 \right) =:\bar C \hat p_{C^{-1}}(t,s,\x,\y),
	\end{equation*}
	\textcolor{black}{for some $\bar C:= \bar C(C)$ such that $\int_{\R^{nd}}d\y \hat p_{C^{-1}}(t,s,\x,\y)=1$}. Similarly, the derivatives of $\tilde p $ will be bounded by a density of the form $\hat p_{C^{-1}}$ up to an additional multiplicative \textcolor{black}{contribution} reflecting the time-singularities associated with the differentiation index. 
	Precisely, there exists $\bar C$ s.t. for any  $l\in \leftB 1,n\rightB$,  $r\in \{0,1\}$:
	\begin{eqnarray}
	|D_{\x_l} D_{\x_1}^r\tilde p(t,s,\x,\y)|&\le& \frac{\textcolor{black}{C}}{(s-t)^{\frac{n^2d} 2+(l-\frac 12)+\frac r2}}\exp\left(-C^{-1}(s-t)|\T_{s-t}^{-1}(\tilde \gR_{s,t}\x-\y)|^2 \right) \notag\\
	&\le& \frac{\bar C}{(s-t)^{(l-\frac 12)+\frac r2}} \hat p_{C^{-1}}(t,s,\x,\y),\label{strong_CTR_DER_OU}
	\end{eqnarray}
	\textcolor{black}{up to a modification of $\bar C$ for the last inequality.} We refer to the proof of Proposition \ref{strong_THE_PROP} for a complete version of this statement. 

	To prove estimate \eqref{strong_GEN_ESTI_PDE} of Theorem \ref{strong_MAIN_RESULT_PDE} for the current system \eqref{strong_GEN_PDE_A}, it follows from the specific structure of the matrix $B$ that we have to estimate for any $l \in \leftB 1,n\rightB$ the quantities $D_{\x_l} D_{\x_1}^r \tilde \u_i(t,\x)$, $r\in \{0,1\}$. From \eqref{strong_CTR_DER_OU}, we thus have 
	\begin{equation}\label{strong_FORME_U_SIMPLE}
	|D_{\x_l} D_{\x_1}^r \tilde \u_i(t,\x)|  \leq  \bar C \int_t^T ds \int_{\R^{nd}} d\y |\gF_i(s,\y)| (s-t)^{-(l-\textcolor{black}{\frac12})-\textcolor{black}{\frac r2}}\hat p_{C^{-1}}(t,s,\x,\y).
	\end{equation}
	We now face two problems: first the $\gF_i$ are unbounded, second the above time singularity is, as is, not integrable. Let us consider the worst case \emph{i.e.} when $r=1$. To smoothen the time singularity, the main idea consists in using the regularity of the source term $\gF_i$ exploiting precisely the fact that, once integrated through the variables $\y_l$ to $\y_n$, the transition density $\tilde p $ does not depend on the variable $\x_l$ anymore. \textcolor{black}{This is due to the structure of $\mathbf A$ in \eqref{strong_def_A}, which in particular yields that the resolvent $(\tilde \gR_{s,t})_{0\le t\le s\le T} $ is lower triangular}. 
	Precisely, denoting for conciseness by $\btheta_{s,t}(\x)=\tilde \gR_{s,t} \x $ (which is coherent with the notation below when handling non-linear flows), we write:
	$$\int_{\R^{(l-1)d}} \!\!d\y_{1:l-1} \gF_i(t,\y_{1},\textcolor{black}{\hdots},\y_{l-1},\btheta_{s,t}^{l}(\x),\textcolor{black}{\hdots}, \btheta_{s,t}^{n}(\x)) D_{\x_1}D_{\x_l} \int_{\R^{(n-(l-1))d}}\!\!d\y_{l:n} \tilde p(t,s,\x,\y)=0\textcolor{black}{.}$$
	\textcolor{black}{This} is what \textcolor{black}{will be called} a \emph{cancellation (or centering) argument} in the following.
	When using this property, we obtain that 
	\begin{eqnarray*} 
		\!\!&\!\!\!\!&\!\!|D_{\x_1}D_{\x_l}\tilde \u_i(t,\x)| \\
		\!\!&\!\!= \!\!&\!\!  \bigg|\int_t^T ds \int_{\R^{nd}} d\y \left(\gF_i(s,\y)-\gF_i(t,\y_{1},\textcolor{black}{\hdots},\y_{l-1},\btheta_{s,t}^{l}(\x),\textcolor{black}{\hdots}, \btheta_{s,t}^{n}(\x))\right) D_{\x_1}D_{\x_l} \tilde p(t,s,\x,\y)
		\bigg|.
	\end{eqnarray*} 
	We thus obtain from \eqref{strong_CTR_DER_OU}:
	\begin{eqnarray*} 
		&&|D_{\x_1}D_{\x_l}\tilde \u_i(t,\x)| \\
		&\leq & \bar C \int_t^T ds \int_{\R^{nd}} d\y \left|\gF_i(s,\y)-\gF_i(t,\y_{1},\textcolor{black}{\hdots},\y_{l-1},\btheta_{s,t}^{l}(\x),\textcolor{black}{\hdots}, \btheta_{s,t}^{n}(\x))\right| 
		(s-t)^{-\textcolor{black}{(l-\frac 12)}-\textcolor{black}{\frac12}}\hat p_{C^{-1}}(t,s,\x,\y).
	\end{eqnarray*}
	Then, using the regularity assumed of $\gF_i$, which satisfies \A{T$_\beta$}, we get that for some \textcolor{black}{constants $\bar C, C$} (which possibly change from line to line)
	\begin{equation} 
	\label{ABSORB}\tag{Abs}
	\begin{split}
	&|D_{\x_1}D_{\x_l}\tilde \u_i(t,\x)|
	\\
	\leq & \textcolor{black}{\bar  C}\int_t^Tds  \int_{\R^{nd}} d\y \sum_{j=l}^n |\y_j-\btheta^j_{s,t}(\x)|^{\beta_j} (s-t)^{-(l-\textcolor{black}{\frac 12})-\textcolor{black}{\frac12}}\hat p_{C^{-1}}(t,s,\x,\y)\\
	\leq &\textcolor{black}{\bar C}\int_t^T ds \int_{\R^{nd}} d\y \sum_{j=l}^n \left|\frac{\y_j-\btheta_{s,t}^j(\x)}{(s-t)^{\textcolor{black}{j-\frac12}}}\right|^{\beta_j} (s-t)^{\beta_j(j-\textcolor{black}{\frac12})}
	(s-t)^{-(l-\textcolor{black}{\frac 12})-\textcolor{black}{\frac 12}}\hat p_{C^{-1}}(t,s,\x,\y)\\
	\leq & \textcolor{black}{\bar C}\int_t^T ds \int_{\R^{nd}} d\y \sum_{j=l}^n (s-t)^{-(l-\textcolor{black}{\frac 12})-\textcolor{black}{\frac 12}+\beta_j(j-\textcolor{black}{\frac 12})}\hat p_{C^{-1}}(t,s,\x,\y),
	\end{split}
	\end{equation}
	\textcolor{black}{thanks to Remark \ref{rembelow} \textcolor{black}{below} for the derivation of the last inequality. Note that the above term is} integrable only if for each $j \in \leftB l,n\rightB$, $-(l-1/2)-1/2+\beta_j(j-1/2)>-1\iff \beta_j>\big( (2l-2)/(2j-1) \big)$. This condition actually  holds if for any $i \in \leftB 1,n\rightB$,  $\beta_i > \big( (2i-2)/(2i-1) \big)$ which is exactly the infimum assumed in \A{T$_\beta$}. As we can see, there is no hope to obtain better thresholds with such a strategy. This is the reason why we said that these thresholds are \emph{almost sharp} for the approach used here.\\

	\textcolor{black}{
		\begin{REM}[On the absorption of spatial differences by the \textit{proxy} density]\label{rembelow}
			From a technical viewpoint, we also insist here that in the above inequalities in \eqref{ABSORB}, we explicitly made the contribution \textcolor{black}{$(s-t)^{-(j-1/2)} |\y_j-\btheta_{s,t}^j(\x)|$}
appear since it precisely corresponds to the absolute value of the $j^{th} $ entry of the vector $(t-s)^{1/2}\T_{t-s}^{-1}(\btheta_{s,t}(\x)-\y) $ which also precisely appears with the current notation $\btheta_{s,t}(\x)=\tilde \gR_{s,t}\x $, in its square norm, in the off-diagonal exponential bound of \eqref{strong_CTR_DER_OU}. This is why for the last inequality this contribution disappeared, i.e. it can be absorbed by the exponential (again up to a modification of the concentration constant $C$ which is modified at most a finite number of times).
		\end{REM}
	}

	\subsubsection{Back to the perturbative analysis}\label{strong_HINT_BESOV_1}
	Let us now briefly explain what happens when one wants to control the approximation error in the forward parametrix expansion. \textcolor{black}{We now come back} to our general setting and \textcolor{black}{denote by $\tilde p^m$} the transition density of \textcolor{black}{a \textit{suitable} Gaussian \textit{proxy} process $ \tilde \X^m$ with generator $\tilde L^m $}. \textcolor{black}{Observe that equation \eqref{strong_GEN_PDE_MOLL} can be equivalently rewritten as:}
	\textcolor{black}{\begin{equation*}
		\label{strong_GEN_PDE_MOLL_PROXY}
		\begin{cases}
		(\partial_t \u_i^m+\tilde L_t^m \u_i^m)(t,\x)= \gF_i^m (t,\x)-(L_t^m-\tilde L_t^m) \u_i^m(t,\x), \ (t,\x)\in [0,T)\times \R^{nd},\\
		\u_i^m(T,\x)=\0_{d},\qquad i\in \leftB 1,n\rightB.
		\end{cases}
		\end{equation*}
	}
	\textcolor{black}{Since $\u_i^m(t,\x) $ is smooth (see the arguments following \eqref{strong_GEN_PDE_MOLL}), so is the term $(L_t^m-\tilde L_t^m) \u_i^m(t,\x)$ which appears above as part of the source}.

	\textcolor{black}{ Hence, we derive from the above equation and the Feynman-Kac formula  the following representation}. For each regularized component $\u_i^m$, $i\in \leftB 1,n\rightB$ of our solution $\U^m$ \textcolor{black}{of the systems \eqref{strong_GEN_PDE_MOLL}} \textcolor{black}{it holds that} for any $(t,\x) \in [0,T] \times \R^{nd}$ 
	\begin{equation}\label{strong_repPDEintro}
	\u^m_i(t,\x) = \int_t^T ds\int_{\R^{nd}} d\y \Big\{-\gF^m_i(s,\y) + (L_s^m-\tilde L_s^m)\u_i^m(s,\y)\Big\}  \tilde p^m(t,s,\x,\y).
	\end{equation}
	\textcolor{black}{This corresponds to the so-called first order parametrix expansion of the initial equation \eqref{strong_GEN_PDE_MOLL} around the \textit{proxy} generator $\tilde L^m $}.
	Above, the additional term in the right\textcolor{black}{-}hand side \textcolor{black}{is}, in comparison with \eqref{strong_OU_SOURCE}, precisely the approximation error due to the parametrix expansion. It thus appears that the solution has an implicit representation which makes  its derivatives themselves appear. Hence, when differentiating  the above representation to derive the estimate \eqref{strong_GEN_ESTI_PDE} in Theorem \ref{strong_MAIN_RESULT_PDE}, we obtain bounds that depend themselves on the derivatives of the solution. We then have to estimate each derivative appearing in the right\textcolor{black}{-}hand side and use a \emph{circular argument}. Namely, when differentiating $\u^m_i(t,\x) $, we will obtain the required estimate  provided the multiplicative constants associated with the terms $\|{\mathbf D} {\u_i}^m\|_\infty  $ and $\|D_1{\mathbf D}{\u_i}^m\|_\infty  $, that will appear in the corresponding upper-bound for the above right\textcolor{black}{-}hand side, are small enough (see also Section 2 of \cite{chau:17} and Section \ref{strong_SECTION_CTR_SENSI} below for details).

	Moreover, as we have already seen, in order to smoothen the time singularity appearing when we apply a cross differentiation operator in the $l^{{\rm th}}$ and $1^{{\rm st}}$ direction to the term $\int_t^T ds \int_{\R^{nd}} d\y(L_s^m-\tilde L_s^m)\u_i^m(s,\y)  \tilde p^m(t,s,\x,\y) $ corresponding to the approximation error, we will have to \emph{center this term around the derivatives of the solution} itself (in the sense given in the above discussion). This procedure allows us, thanks to Taylor expansions, to weaken the singularities and provides  integrable (in time) terms. The dramatic point is that, when doing so, our bound involves the cross derivatives $D_\ell 
	\textcolor{black}{{D_j}\u_i^m}$, $\ell\textcolor{black}{,j} \in \leftB \textcolor{black}{1},n\rightB$ whose control in supremum norm is, as suggested by the discussion done in the explicit case of a simple degenerate Ornstein-Uhlenbeck process, definitely out of reach \textcolor{black}{as soon as $\ell>1$}. \textcolor{black}{In fact, as suggested by the results in \cite{chau:17}, the only thing we could hope is that the gradient in the degenerate directions viewed as a function of the degenerate variables, i.e. ${\mathbf D}_{2:n}\u_i^m(t,\x_1,\cdot):=\big(D_2 \u_i^m(t,\x_1,\cdot), \textcolor{black}{\hdots}, D_n\u_i^m(t,\x_1,\cdot)\big)^{\textcolor{black}{*}}$ for any $(t,\x_1)\in [0,T]\times \R^d $, belongs to an appropriate \textcolor{black}{anisotropic} H\"older space. Note importantly that such spaces can as well be viewed as particular cases of anisotropic Besov spaces with corresponding positive regularity indexes. The main idea is thus to use an integration by parts argument in order to rebalance one differentiation operator from the solution to the remaining terms coming from the differentiation of \eqref{strong_repPDEintro}, which in particular involve the coefficients of the operator $L^m-\tilde L^m$ and contain the time singularities coming from the derivatives of the frozen Gaussian kernel $\tilde p^m $. As the coefficients of the operator $L^m-\tilde L^m$ are assumed to be only H\"older continuous, their \textcolor{black}{generalized} derivative should belong to some \textcolor{black}{anisotropic} Besov space\textcolor{black}{s} of negative regularity index, strictly bigger than $-1$. To tackle this problem, our main idea, in order to balance the lack of differentiation property of the remaining terms, consists in putting precisely in duality the gradient ${\mathbf D}_{2:n}\u_i^m(t,\x_1,\cdot)$, belonging to an \textcolor{black}{anisotropic} \textcolor{black}{inhomogeneous} Besov space with positive regularity exponent 
		and the remaining terms, belonging to an \textcolor{black}{anisotropic} \textcolor{black}{inhomogeneous} Besov space with negative regularity exponent.}
	We refer to the proof of the main Theorem \ref{strong_MAIN_RESULT_PDE} in Section \ref{strong_SEC_PROOF_PDE} for details and to Proposition 3.6 in the book of Lemari\'e-Rieusset \cite{lemar:02} for duality results on Besov spaces.
	
	We are thus led to control on the one hand the Besov norm with \textcolor{black}{positive} exponent \textcolor{black}{(or equivalently the H\"older norm)} of the derivatives of the solution, see Lemma \ref{strong_LEMME_CTR_HOLDER}, and on the other hand the Besov norm with \textcolor{black}{negative} exponent of the remaining terms in \eqref{strong_repPDEintro} (involving the coefficients of the operator $L^m-\tilde L^m$), see Lemma \ref{strong_LEMME_CTR_BESOV}. The first control 
\textcolor{black}{(H\"older norm)} is crucial and appears to be quite delicate. Indeed, due to the implicit representation \eqref{strong_repPDEintro}, this estimate also involves \textcolor{black}{H\"older} norms of the full gradient $\mathbf D u^m_i$. This again reflects the \emph{circular} nature of the arguments needed to derive the result.\\


	Let us close this discussion coming back to Remark \ref{strong_REM_KRYLOV}. As we emphasized, in comparison with the non-degenerate result, Theorem \ref{strong_MAIN_RESULT} should hold assuming that the drift $\gF_1$ belongs to a suitable  $L_q-L_p$ space w.r.t. time and the non-degenerate variable $\x_1$. We are convinced that this is the case but we deliberately decide  not to tackle this setting in order to keep this work shorter and more coherent. Indeed, in this case, the difficulty comes from the estimate on the second order derivative in the non-degenerate direction of the first component of the solution $\U^m$, namely $D_{\x_1}D_{\x_1}\u^m_1$ (which is a part of the main estimate \eqref{strong_GEN_ESTI_PDE} in Theorem \ref{strong_MAIN_RESULT_PDE}). The point is to establish for this quantity an $L_q-L_p$ control.  This cannot be derived from the previously described approach and requires harmonic analysis techniques (see also \textcolor{black}{\cite{kry:01}}). The main problem to establish the estimate is mainly due to  the source term, which is actually $\gF_1$. To prove it, the main idea consists in exploiting the results of Menozzi \cite{meno:17} (where such an estimate is proved under the assumption that the drift is Lipschitz) through the tools developed in \cite{chau:meno:17} (\emph{backward} parametrix approach for drift ${\mathbf F}$ whose first component may be in $L_q-L_p$ and the other ones in H\"older spaces). Then, the Zvonkin transform should also be tuned a little bit following the strategy developed by Veretennikov (see e.g.\cite{veretennikov_strong_1980} and \cite{ fedr:flan:prio:vove:17}). Such a program would surely \textcolor{black}{toughen} our paper without adding any surprising result and we prefer to focus on the novelty of the approach based on duality results for Besov spaces and the generalization of the strong uniqueness result to the whole chain (\emph{i.e.} to any arbitrary $n\geq 1$) rather than drowning the reader into additional technical considerations.
	
	
	\mysection{Perturbation techniques for the PDE: proof of Theorem \ref{strong_MAIN_RESULT_PDE}}\label{strong_SEC_PROOF_PDE}
	In order to keep the notations as clear as possible, we forget the superscript $m$ standing for the mollifying procedure and we suppose that the following assumptions hold:
	\\
	
	\textbf{Assumption \A{AM}}. We say that assumption \A{AM} holds if  the assumptions gathered in \A{A} hold true and the coefficients $\gF$, $a$ are infinitely differentiable functions with bounded derivatives of all order\textcolor{black}{s} for $a$ and  greater than 1 for the coefficient $\gF$.\\
	
	\textcolor{black}{\textbf{About the constants.} We importantly specify that all the constants appearing below \textbf{do not} depend on the (omitted) smoothing parameter $m$, but only on known parameters \textcolor{black}{from} assumption \A{A} and possibly \textcolor{black}{on} $T>0$}.\\ 
	
	In the whole section, we consider a fixed final time $T>0$ which is meant to be \textit{small}, i.e. $T \ll 1 $. Let us consider for this section a generic PDE with generator corresponding to \eqref{strong_GENERATOR} and \textcolor{black}{scalar} source $f$ having the same H\"older regularity than the drift terms in \eqref{strong_SYST} (i.e. \textcolor{black}{the scalar function $f$ below can be any of the entries of the $\R^d $-valued  $(\gF_i)_{i\in \leftB 1, n\rightB}$ in the dynamics \eqref{strong_SYST}}). Namely, we concentrate on 
	\begin{equation}
	\label{strong_GENERIC_PDE}
	\begin{cases}
	(\partial_t u+L_t u)(t,\x)=-f(t,\x), \ (t,\x)\in [0,T)\times \R^{nd},\\
	u(T,\x)=0,
	\end{cases}
	\end{equation}
	where $(L_t)_{t \geq 0}$ is defined in \eqref{strong_GENERATOR} and stands for the generator associated with \eqref{strong_SYST} when the coefficients are smooth.

	The key result to prove strong uniqueness for the SDE \eqref{strong_SYST} is actually the following theorem. 
	\begin{THM}[Pointwise bounds for the derivatives of the PDE \eqref{strong_GENERIC_PDE}]\label{strong_THM_DER_PDE}
		There exists $\gamma:=\gamma(\A{A})>0 $ and $C:=C(\A{A})>0$ s.t. 
		\begin{equation}\label{strong_BD_PDE}
		\|{\mathbf D} u\|_\infty+\|{\mathbf D}(D_{1} u)\|_\infty\le C T^\gamma,
		\end{equation}
		with obvious extension of the definition in \eqref{strong_NORMES} to the current scalar case.
	\end{THM}
	
	The proof of Theorem \ref{strong_THM_DER_PDE} is performed in Section \ref{strong_SECTION_CTR_SENSI} through the forward parametrix approach consisting in considering a suitable \textit{proxy} semi-group around which the initial solution of \eqref{strong_GENERIC_PDE} can be expanded. To this end we first investigate in Section \ref{strong_SEC_GAUSS_DENS} below the linearized Gaussian process deriving from the dynamics in \eqref{strong_SYST} which will provide the suitable model for the parametrix.
	
	\textcolor{black}{\begin{proof}[Proof of Theorem \ref{strong_MAIN_RESULT_PDE}]\textcolor{black}{Theorem \ref{strong_MAIN_RESULT_PDE} is then a direct consequence of Theorem \ref{strong_THM_DER_PDE}}.
	\textcolor{black}{Precisely, recalling that  $\U=(\u_1,\textcolor{black}{\hdots} ,\u_n)$ where for $i\in \leftB 1,n\rightB $, $\u_i $ solves \eqref{strong_GEN_PDE}, we derive that $\u_i^j,\ j\in \leftB 1,d\rightB $ solves \eqref{strong_GENERIC_PDE} with $f(t,\x)=-\gF_i^j(t,\x) $. Theorem \ref{strong_MAIN_RESULT_PDE} thus follows in \textcolor{black}{whole} generality applying  Theorem \ref{strong_THM_DER_PDE} to each component of the solution of the systems \eqref{strong_GEN_PDE}. 
	}\end{proof}}

	%

	\subsection{Gaussian \textit{proxy} and associated controls }

	\label{strong_SEC_GAUSS_DENS}
	
	\subsubsection{Linearization of the dynamics}
	
	Fix some freezing points $(\tau,\bxi) \in [0,T]\times \R^{nd}$. For fixed initial conditions $(t,\x)\in [0,T]\times \R^{nd}$, a natural linearization associated with the mollified version of \eqref{strong_SYST} writes 
	\begin{eqnarray} 
	d\tilde \X_v^{(\tau,\bxi)}&=&[\gF(v,\btheta_{v,\tau}(\bxi))+ D\gF(v,\btheta_{v,\tau}(\bxi))(\tilde \X_v^{(\tau,\bxi)}-\btheta_{v,\tau}(\bxi))]dv +B\sigma(v,\btheta_{v,\tau}(\bxi)) dW_v,\ \forall v\in  [t,s],\nonumber\\
	\tilde \X_t^{(\tau,\bxi)}&=&\x, \label{strong_FROZ_MOL_FOR_NO_SUPER_SCRIPTS}
	\end{eqnarray}
	where 
	\begin{equation}
	\label{strong_DYN_DET_SMOOTH}
	\dot \btheta_{v,\tau}(\bxi)=\gF(v,\btheta_{v,\tau}(\bxi)),\ v\in [0,T],\  \btheta_{\tau,\tau}(\bxi)=\bxi,
	\end{equation}
	and ${D \gF(v,\cdot)} $ denotes the subdiagonal of the Jacobian matrix ${{\mathbf D} \gF(v,\cdot)} $. Namely, for $\z\in \R^{nd}$: 
	$$D\gF(v,\z)=\left (\begin{array}{ccccc}\0_{d,d} & \cdots & \cdots &\cdots  & \0_{d,d}\\
	D_{\z_1}\gF_{2}(v,\z) & \0_{d,d} &\cdots &\cdots &\0_{d,d}\\
	\0_{d,d} & D_{\z_2} \gF_{3}(v,\textcolor{black}{\z_{2:n}})& \0_{d,d}& \0_{d,d} &\vdots\\
	\vdots &  \0_{d,d}                     & \ddots & \textcolor{black}{\ddots} & \textcolor{black}{\vdots}\\
	\0_{d,d} &\cdots &     \0_{d,d}      & D_{\z_{n-1}}\gF_{n}(v,\z_{n-1},\z_n) & \0_{d,d}
	\end{array}\right) .
	$$ 
	\textcolor{black}{In the following, we will often refer to the Gaussian process $(\tilde \X_v^{(\tau,\bxi)})_{v\ge t}$ introduced in \eqref{strong_FROZ_MOL_FOR_NO_SUPER_SCRIPTS} as the \textit{proxy} process. This terminology comes from the fact that it is a natural, well controlled object, meant to locally approximate the original dynamics in \eqref{strong_SYST}}.
	
	We explicitly integrate \eqref{strong_FROZ_MOL_FOR_NO_SUPER_SCRIPTS} to obtain for any $v\in [t,s] $:
	\begin{align}
	\label{strong_INTEGRATED_FLOW}
	&\tilde  \X_v^{(\tau,\bxi)}\notag
	\\
	&=\textcolor{black}{\tilde\gR}^{(\tau,\bxi)}(v,t)\x+ \!\int_t^v \!\!du \textcolor{black}{\tilde \gR}^{(\tau,\bxi)}(v,u)\Big( \gF(u,\btheta_{u,\tau}(\bxi))-D\gF(u,\btheta_{u,\tau}(\bxi))\btheta_{u,\tau}(\bxi)\Big )
	+\!\int_t^v \!\! \textcolor{black}{\tilde \gR}^{(\tau,\bxi)}(v,u)B\sigma(u,\btheta_{u,\tau}(\bxi)) dW_u \notag
	\\
	&=:\m_{v,t}^{(\tau,\bxi)}(\x)+\int_t^v \textcolor{black}{\tilde \gR}^{(\tau,\bxi)}(v,u)B\sigma(u,\btheta_{u,\tau}(\bxi)) dW_u,
	\end{align}
	where $(\tilde \gR^{(\tau,\bxi)}(v,u))_{t\le u,v \le s} $ stands for the resolvent associated with the partial gradients $(D {\mathbf F}(v,\btheta_{v,\tau}(\bxi)))_{ v\in [t,s]} $ which satisfies for $v\in [t,s]$:
	\begin{equation}
	\begin{split}
	\partial_{v}\tilde \gR^{(\tau,\bxi)}(v,t)&={ D \gF}(v,\btheta_{v,\tau}(\bxi))\tilde \gR^{(\tau,\bxi)}(v,t) , 
	\ \tilde \gR^{(\tau,\bxi)}(t,t)=\mathbf I_{nd\times nd}.
	\end{split}
	\label{strong_DYN_RES_FORWARD}
	\end{equation}
	Note in particular that since the partial gradients are subdiagonal $ {\rm det}(\tilde \gR^{(\tau,\bxi)}(v,t))=1$.\
	
	Also, for $v\in [t,s]$, we recall that  $\m_{v,t}^{(\tau,\bxi)}(\x) $ stands for the mean of $\tilde  \X_v^{(\tau,\bxi)} $ and  corresponds as well  to the solution of \eqref{strong_FROZ_MOL_FOR_NO_SUPER_SCRIPTS} when $\sigma =0 $ \textcolor{black}{and} the  starting point is $\x $. 
	Importantly, we point out that $\x \in \R^{nd}\mapsto \m_{v,t}^{(\tau,\bxi)}(\x)$ is \textbf{affine}  w.r.t. the starting point $\x$. Precisely, for $\x,\x'\in \R^{nd} $:
	\begin{equation}
	\label{strong_AFFINE_FLOW}
	\m_{v,t}^{(\tau,\bxi)}(\x+\x')=\tilde \gR^{(\tau,\bxi)}(v,t)\x'+ \m_{v,t}^{(\tau,\bxi)}(\x).
	\end{equation}
	\textcolor{black}{
		It is also useful to note that, since from \eqref{strong_FROZ_MOL_FOR_NO_SUPER_SCRIPTS}
		$$d \m_{v,t}^{(\tau,\bxi)}(\x) =[\gF(v,\btheta_{v,\tau}(\bxi))+ D\gF(v,\btheta_{v,\tau}(\bxi))(\m_{v,t}^{(\tau,\bxi)}(\x)-\btheta_{v,\tau}(\bxi))]dv,\ \ \m_{t,t}^{(\tau,\bxi)}(\x)=\x,$$
		it holds from  \eqref{strong_DYN_DET_SMOOTH} that for $(\tau,\bxi)=(t,\x) $:
		\begin{eqnarray*}
			d \big( \m_{v,t}^{(t,\x)}(\x)-\btheta_{v,t}(\x)\big)&=&[\gF(v,\btheta_{v,t}(\x))+ D\gF(v,\btheta_{v,t}(\x))(\m_{v,t}^{(t,\x)}(\x)-\btheta_{v,t}(\x))]dv-\gF(v,\btheta_{v,t}(\x))dv\\
			&=&[D\gF(v,\btheta_{v,t}(\x))(\m_{v,t}^{(t,\x)}(\x)-\btheta_{v,t}(\x))]dv,\\
			\m_{t,t}^{(t,\x)}(\x)-\btheta_{t,t}(\x)&=&0.
		\end{eqnarray*}
		Therefore, the Gr\"onwall lemma yields:
		\begin{equation}\label{EQUIV_LIN_FLOW_GOOD_FREEZING} \tag{LF}
		\m_{v,t}^{(\tau,\bxi)}(\x)|_{(\tau,\bxi)=(t,\x)}=\m_{v,t}^{(t,\x)}(\x)=\btheta_{v,t}(\x).
		\end{equation}
		Namely, when freezing the parameters at the initial time $t$ and the starting spatial point $\x$, the linearized flow $\m_{v,t}^{(t,\x)}(\x) $ and the non-linear one $\btheta_{v,t}(\x) $ coincide.
	} 
	
	From the non-degeneracy of $ \sigma$ and the H\"ormander like condition, the Gaussian process defined by \eqref{strong_INTEGRATED_FLOW} admits a density $\tilde p^{(\tau,\bxi)}(t,s,\x,\cdot) $ which is suitably controlled (see Proposition \ref{strong_THE_PROP} below and for instance \cite{dela:meno:10}, \cite{chau:meno:17}).

	We first give in the next proposition a key estimate on the covariance matrix associated with \eqref{strong_INTEGRATED_FLOW} and its properties w.r.t. a suitable scaling of the system. 
	\begin{PROP}[Good Scaling Properties of the Covariance Matrix]
		\label{strong_PROP_SCALE_COV}
		The covariance matrix of $\tilde  \X_v^{(\tau,\bxi)}$ in \eqref{strong_INTEGRATED_FLOW} writes:
		$$\tilde \K_{v,t}^{(\tau,\bxi)}:=\int_t^vdu \textcolor{black}{\tilde  \gR}^{(\tau,\bxi)}(v,u) Ba(u,\btheta_{u,\tau}(\bxi))B^*\textcolor{black}{\tilde \gR}^{(\tau,\bxi)}(v,u)^* . 
		$$
		Uniformly in $(\tau,\bxi)\in [0,T]\times \R^{nd} $ and $s\in [0,T] $, it satisfies a \textit{good scaling property} in the sense of Definition 3.2 in \cite{dela:meno:10} (see also Proposition 3.4 of that reference). That is, for any fixed $T>0$, there exists $C_{\ref{strong_GSP}}:=C_{\ref{strong_GSP}}(\A{A},T)\ge 1$ s.t. for all $0\le t<v \le s\le T $, for any $(\tau,\bxi) \in [0,T]\times \R^{nd} $:
		\begin{equation}
		\label{strong_GSP}
		\forall \bzeta \in \R^{nd},\   C_{\ref{strong_GSP}}^{-1} (v-t)^{-1}|\T_{v-t} \bzeta|^2\le \langle  \tilde \K_{v,t}^{(\tau,\bxi)}\bzeta,\bzeta\rangle \le C_{\ref{strong_GSP}} (v-t)^{-1}|\T_{v-t} \bzeta|^2,
		\end{equation} 
		where we again use the notation introduced in  \eqref{strong_DEF_T_ALPHA} \textcolor{black}{for} the scaling matrix $\T_{v-t} $.
		
	\end{PROP}
	\textcolor{black}{Under \A{A}}, Proposition \ref{strong_PROP_SCALE_COV} readily follows from \textcolor{black}{Proposition 3.4}  in \cite{dela:meno:10} \textcolor{black}{(see also the scaling Lemma 3.6 therein). A complete proof is provided in Appendix \ref{PROOF_PROP_4}} \textcolor{black}{below}.
	
	\textcolor{black}{
		\begin{REM}[On some important consequences of the good scaling property]\label{CONS_GSP}
			We state here some rather direct yet important controls that follow from Proposition \ref{strong_PROP_SCALE_COV}.
			\begin{trivlist}
				\item[-] Setting for any $s\in (t,T] $,
				\begin{equation}\label{SCALING}\tag{S}
				\widehat {{\tilde {\mathbf K}_1^{(\tau,\bxi),s,t}}}:=(s-t)\T_{s-t}^{-1} \tilde {\mathbf K}_{s,t}^{(\tau,\bxi)}\T_{s-t}^{-1},
				\end{equation}
				it follows from \eqref{strong_GSP} that for any $\bzeta\in \R^{nd} $:
				\begin{equation}\label{S_ND}\tag{S${}_{ND}$}
				C_{\ref{strong_GSP}}^{-1}|\bzeta|^2 \le \langle \widehat {{\tilde {\mathbf K}_1^{(\tau,\bxi),s,t}}} \bzeta, \bzeta\rangle\le 
				C_{\ref{strong_GSP}}|\bzeta|^2 \Longrightarrow \exists \bar C:=\bar C(\A{A},T)\ge 1,\ \bar C^{-1}\le \det(\widehat {{\tilde {\mathbf K}_1^{(\tau,\bxi),s,t}}})\le \bar C.
				\end{equation}
				\item[-] Equations \eqref{SCALING} and \eqref{S_ND} then readily give that:
				\begin{equation}\label{DET}\tag{D}
				\bar C^{-1} (s-t)^{n^2d}  \le \det(\tilde {\mathbf K}_{s,t}^{(\tau,\bxi)})\le \bar C (s-t)^{n^2d},
				\end{equation}
				as well as, for any $\bzeta\in \R^{nd} $,
				\begin{equation}\label{BD_M_1}
				\tag{S${}_{ND}^{{\rm Inv}}$}
				C_{\ref{strong_GSP}}^{-1}|\bzeta|^2 \le \langle (\widehat {{\tilde {\mathbf K}_1^{(\tau,\bxi),s,t}}})^{-1} \bzeta, \bzeta\rangle\le 
				C_{\ref{strong_GSP}}|\bzeta|^2\Rightarrow C_{\ref{strong_GSP}}^{-1}(s-t)|\T_{s-t}^{-1}\bzeta|^2\le \langle (\tilde {\mathbf K}_{s,t}^{(\tau,\bxi)})^{-1}\bzeta,\bzeta\rangle \le C_{\ref{strong_GSP}} (s-t)|\T_{s-t}^{-1}\bzeta|^2.
				\end{equation}
				This last bound will be of particular relevance to control the tails of the Gaussian density of $\tilde \X_{s}^{(\tau,\bxi)} $.
			\end{trivlist}
		\end{REM}
	}

	We now state some important density bounds for the linearized model.

	\begin{PROP}[Density of the linearized dynamics]\label{strong_THE_PROP}
		Under \A{A}, we have that, for any $s\in (t,T]$ the random variable $\tilde  \X_s^{(\tau,\bxi)} $ in \eqref{strong_INTEGRATED_FLOW} admits a Gaussian density $\tilde p^{(\tau,\bxi)}(t,s,\x,\cdot) $ which writes for any $\y \in \R^{nd}$:
		\begin{equation}\label{strong_CORRESP}
		\tilde p^{(\tau,\bxi)}(t,s,\x,\y):=\frac{1}{(2\pi)^{\frac{nd}2}\det(\tilde \K_{s,t}^{(\tau,\bxi)})^{\frac 12}}\exp\left( \!-\frac 12 \left\langle \!(\tilde \K_{s,t}^{(\tau,\bxi)})^{-1} (\m_{s,t}^{(\tau,\bxi)}(\x)-\y),\m_{s,t}^{(\tau,\bxi)}(\x)-\y \!\right\rangle\!\right)\!\!,
		\end{equation}
		with $\tilde \K_{s,t}^{(\tau,\bxi)} $ as in Proposition \ref{strong_PROP_SCALE_COV}.
		Also, there exist $C:=C(\A{A},T)>0$ and $\bar C := \bar C(C)$ s.t. for all $l\in \leftB 1,n \rightB $, $\textcolor{black}{q},\ r\in \{0,1\} $, 
		we have: 
		\begin{eqnarray}
		|D_{\x_l}^{\textcolor{black}{q}}D_{\x_1}^r\tilde p^{(\tau,\bxi)}(t,s,\x,\y)|&\le& \frac{C}{(s-t)^{\frac{n^2d}2+(l-\frac{1}{2})\textcolor{black}{q}+\frac r2 }}\exp\left(-C^{-1}(s-t) \big|\T_{s-t}^{-1}\big(\m_{s,t}^{(\tau,\bxi)}(\x)-\y\big)\big|^2\right)\notag\\
		&=:&\frac{\bar C}{(s-t)^{(l-\frac{1}{2})\textcolor{black}{q}+\frac r2}}\hat p_{C^{-1}}^{(\tau,\bxi)}(t,s,\x,\y).\label{strong_CTR_GRAD}
		\end{eqnarray}
		
	\end{PROP}
	
	\begin{proof} 
		Expression \eqref{strong_CORRESP} readily follows from \textcolor{black}{\eqref{strong_INTEGRATED_FLOW}}. The control \eqref{strong_CTR_GRAD} \textcolor{black}{is} then a direct consequence of  Proposition \ref{strong_PROP_SCALE_COV} for \textcolor{black}{$q=r=0 $ (upper bound of the density)}.
		Differentiating w.r.t. $\x$ recalling from \eqref{strong_AFFINE_FLOW} that $\x\mapsto 
		\m_{s,t}^{(\tau,\bxi)}(\x) $ is affine yields:
		\begin{eqnarray}
		\label{strong_EXPLICIT_DERIVATIVES}
		D_{\x_j} \tilde p^{(\tau,\bxi)}(t,s,\x,\y)=- \Big[ \big[{\tilde  \gR}^{(\tau,\bxi)}(s,t)\big]^* (\tilde \K_{s,t}^{(\tau,\bxi)})^{-1} (\m_{s,t}^{(\tau,\bxi)}(\x)-\y) \Big]_j \tilde p^{(\tau,\bxi)}(t,s,\x,\y),
		\end{eqnarray}
		\textcolor{black}{where we denoted here for a vector $\z=(\z_1,\textcolor{black}{\hdots},\z_n)\in \R^{nd},
			\ [\z]_j =\z_j$ for readibility}.
		The point is now to use scaling arguments.
		We can first rewrite \textcolor{black}{with the notations of \eqref{SCALING}}:
		\begin{equation}\label{strong_FIRST_SCALING}
		\big[{\tilde  \gR}^{(\tau,\bxi)}(s,t)\big]^* (\tilde \K_{s,t}^{(\tau,\bxi)})^{-1}=(s-t) \big[{\tilde  \gR}^{(\tau,\bxi)}(s,t)\big]^* \T_{s-t}^{-1} ( \widehat { \tilde {\K{}}_1^{\textcolor{black}{(\tau,\bxi)},s,t}})^{-1}\T_{s-t}^{-1}.
		\end{equation}
		\textcolor{black}{Put it differently,
			$\widehat {\tilde {\K}_1^{(\tau,\bxi),s,t}} $ can also be seen as} the covariance matrix   at time $1$ of the rescaled process $\big((s-t)^{\frac 12}\T_{s-t}^{-1}\tilde \X_{t+v(s-t)}^{t,\x,\textcolor{black}{(\tau,\bxi)}}\big)_{v\in [0,1]} $. From the good-scaling property of Proposition \ref{strong_PROP_SCALE_COV}, \textcolor{black}{it was already observed in \eqref{S_ND} that
			$\widehat { \tilde \K_1^{(\tau,\bxi),s,t}} $ is 
			a non-degenerate bounded matrix}.
		\textcolor{black}{ From \eqref{strong_DYN_RES_FORWARD}, a similar scaling argument yields that the resolvent $\mathbf R^{(\tau,\xi)}(s,t) $ can also be written as:}
		\begin{equation}\label{strong_RESCALED_RES}
		[{\tilde  \gR}^{(\tau,\bxi)}(s,t)]^*=\T_{s-t}^{-1}\Big[ \widehat{\tilde  \gR}^{(\tau,\bxi),s,t}(1,0)\Big]^*\T_{s-t},
		\end{equation}
		where again $\widehat{\tilde  \gR}^{(\tau,\bxi),s,t}(1,0)$ is the \textcolor{black}{resolvent} at time 1 of  the rescaled system \textcolor{black}{associated with \eqref{strong_DYN_RES_FORWARD}, i.e. for any $v\in [0,1],\ \T_{s-t}[{\tilde  \gR}^{(\tau,\bxi)}(t+v(s-t),t)]^*\T_{s-t}^{-1}=\big[ \widehat{\tilde  \gR}^{(\tau,\bxi),s,t}(v,0)\big]^*$}. 
		From the analysis performed in Lemma 5.1 in \cite{huan:meno:15} (see also the proof of \textcolor{black}{Propositions 3.3 and 3.4} in \cite{dela:meno:10}) one derives that there exists $\hat C_1$ s.t.  for any $\bzeta \in \R^{nd}, |\big[\widehat{\tilde  \gR}^{(\tau,\bxi),s,t}(1,0)\big]^* \bzeta|\le\hat C_1 |\bzeta| $. \textcolor{black}{We refer to  Appendix \ref{RES_DENS_SUM} for a complete proof of this last bound (see equation \eqref{THE_PROOF_BD_SCALED_RES})}. Equations \eqref{strong_EXPLICIT_DERIVATIVES}, \eqref{strong_FIRST_SCALING} and \eqref{strong_RESCALED_RES} therefore yield:
		\begin{eqnarray*}
			&&|D_{\x_j} \tilde p^{(\tau,\bxi)}(t,s,\x,\y)|\\
			&\le&  (s-t)^{-j+\frac 12} \Bigg|\bigg( \Big[ \widehat{\tilde  \gR}^{(\tau,\bxi),s,t}(1,0)\Big]^* \big(\widehat {\tilde  \K_1^{s,t}}\big)^{-1}\big((s-t)^{\frac 12} \T_{s-t}^{-1}(\m_{s,t}^{(\tau,\bxi)}(\x)-\y)\big) \bigg)_j\Bigg| \tilde p^{(\tau,\bxi)}(t,s,\x,\y)  \\
			&\le& C (s-t)^{-j+\frac 12}   (s-t)^{\frac 12} | \T_{s-t}^{-1}(\m_{s,t}^{(\tau,\bxi)}(\x)-\y)| \tilde p^{(\tau,\bxi)}(t,s,\x,\y).
		\end{eqnarray*}
		From the explicit expression \eqref{strong_CORRESP}, Proposition \ref{strong_PROP_SCALE_COV} and the above equation, we eventually derive:
		\begin{eqnarray*}
			&&|D_{\x_j} \tilde p^{(\tau,\bxi)}(t,s,\x,\y)|
			\nonumber \\
			&\le& \frac{C}{(s-t)^{j-\frac 12 }}\Big((s-t)^{\frac 12} | \T_{s-t}^{-1}(\m_{s,t}^{(\tau,\bxi)}(\x)-\y)| \Big) \frac{1}{(s-t)^{\frac{n^2d} 2}}\exp \left(-C^{-1} (s-t) | \T_{s-t}^{-1}(\m_{s,t}^{(\tau,\bxi)}(\x)-\y)|^2\right)\\
			&\le & \frac{\bar C}{(s-t)^{j-\frac 12}} \hat p_{C^{-1}}^{(\tau,\bxi)}(t,s,\x,\y),
		\end{eqnarray*}
		which gives the statement for one partial derivative. The controls on the higher order derivatives are obtained similarly (see e.g. the proof of Lemma 5.5 of \cite{dela:meno:10} for the bounds on $D_{\x_1}^2 \tilde p^{(\tau,\bxi)}(t,s,\x,\y) $).
	\end{proof}
	\textcolor{black}{
		\begin{REM}\label{REM_WITH_FLOWS_LIN_AND_NON}
			When the freezing couple $(\tau,\bxi) $ corresponds to the couple $(t,\x) $, where $t$ is the starting time and $\x $ the starting position of $\tilde \X^{(\tau,\bxi)} $, we importantly derive from \eqref{EQUIV_LIN_FLOW_GOOD_FREEZING} that \eqref{strong_CTR_GRAD} can be specified as follows:
			\begin{equation}
			|D_{\x_l}^{\textcolor{black}{q}}D_{\x_1}^r\tilde p^{(\tau,\bxi)}(t,s,\x,\y)|_{(\tau,\bxi)=(t,\x)}\le \frac{C}{(s-t)^{\frac{n^2d}2+(l-\frac{1}{2})\textcolor{black}{q}+\frac r2 }}\exp\left(-C^{-1}(s-t) \big|\T_{s-t}^{-1}\big(\btheta_{s,t}(\x)-\y\big)\big|^2\right).
			\label{strong_CTR_GRAD_WGPOINTS}\tag{G${}_{NL}$}
			\end{equation}
		\end{REM}
	}
	
	Now, let us \textcolor{black}{state} a useful control involving the previous Gaussian kernel which will be exploited in some cancellation techniques.
	\begin{PROP} \label{strong_Prop_moment_D2_tilde_p}
		For all $k \in \leftB 1,n \rightB$, $0 \leq t \leq s \leq T$, $(\x,\bxi) \in \R^{nd} \times \R^{nd}$, and $  
		\textcolor{black}{z}\in  \R^d$ the following identity \textcolor{black}{holds}:
		\begin{equation} \label{strong_eq_D2moment2}
		\int_{\R^{nd}}d \y 
		D_{\x_k} \tilde p^{(\tau,\bxi)}(t,s,\x,\y)  \Big \langle 
		z,(\y-\m^{(\tau,\bxi)}_{s,t}(\x))_k \Big\rangle = \textcolor{black}{z
		}.
		\end{equation}
	\end{PROP}
	\begin{proof}
		From Proposition \ref{strong_THE_PROP}, we have $ \int_{\R^{nd}} \tilde p^{(\tau,\bxi)}(t,s,\x,\y) (\y-\m^{(\tau,\bxi)}_{s,t}(\x))_k d \y={\mathbf 0}_d$.
		Differentiating this expression w.r.t. $\x_k$ and using the Leibniz formula (recalling as well the identity  \eqref{strong_AFFINE_FLOW} which yields $D_{\x_k}[\m_{s,t}^{(\tau,\bxi)}(\x)]_k=(\tilde \gR^{(\tau,\bxi)}(s,t))_{k,k}={\mathbf I}_{d,d} $) gives  \eqref{strong_eq_D2moment2}.
	\end{proof}
	
	\subsubsection{\textcolor{black}{Density and} \textcolor{black}{a}ssociated inhomogeneous semi-group: \textcolor{black}{regularization properties}}
	
	For fixed  $(t,\x)\in [0,T]\times \R^{nd} $, we give in this section some important properties concerning the regularization effects in time of the density $\tilde p^{(\tau,\bxi)}(t,s,\x,\cdot), \ s\in (t,T] $ and the associated semi-group. From now on, for notational simplicity, we will write with a slight abuse of notation $\tilde p^{\bxi}(t,s,\x,\y):=\tilde p^{(t,\bxi)}(t,s,\x,\y) $, i.e. we omit the freezing parameter in time when it corresponds to the considered starting time\textcolor{black}{.}

	\paragraph{Density.} The following result will be thoroughly used to derive some key bounds of the sensitivity analysis performed in Section \ref{strong_SECTION_CTR_SENSI}.
	\begin{lem}[Regularization effects for \textcolor{black}{the density}]\label{strong_LEMME_DENS} 
		
		There exists $C:=C(\A{A},T)$ s.t.  for all $\gamma_1,\gamma_2\in (0,1] $, $\ell \in \{0,1\},\ l \in \leftB 1,n\rightB,\  j,k\in \leftB 1,n\rightB$,
		\begin{equation} \label{strong_SMOOTHING_EFFECT_SIMPLE} 
		\int_{\R^{nd}}d\y \Big|  D_{\x_l}D_{\x_1}^\ell \tilde p^\bxi(s,t,\x,\y)\Big|_{\bxi=\x}\big|\big(\y-\btheta_{s,t}(\x)\big)_j \big|^{\gamma_1}\big|\big(\y-\btheta_{s,t}(\x)\big)_k \big|^{\gamma_2} \le C  (s-t)^{-\frac \ell 2-(l-\frac 12) +\gamma_1(j-\frac 12) +\gamma_2(k-\frac 12)}.
		\end{equation}
	\end{lem}
	\begin{proof}
		Equation \eqref{strong_SMOOTHING_EFFECT_SIMPLE} actually follows from \eqref{strong_CTR_GRAD_WGPOINTS} observing that:
		\begin{eqnarray}
		&&\int_{\R^{nd}}d\y \Big|  D_{\x_l}D_{\x_1}^\ell \tilde p^\bxi(s,t,\x,\y)\Big|_{\bxi=\x}\big|\big(\y-\btheta_{s,t}(\x)\big)_j \big|^{\gamma_1}\big|\big(\y-\btheta_{s,t}(\x)\big)_k \big|^{\gamma_2} \notag\\
		&\le& \int_{\R^{nd}}d\y \frac{C}{(s-t)^{\frac{n^2d}2+(l-\frac{1}{2})+\frac \ell2 }}\exp\left(-C^{-1}(s-t) \big|\T_{s-t}^{-1}\big(\btheta_{s,t}(\x)-\y\big)\big|^2\right) \notag\\
		&&\times \Big( \frac{|(\btheta_{s,t}(\x)-\y)_j|}{(s-t)^{j-\frac 12}}\Big)^{\gamma_1} (s-t)^{\gamma_1(j-\frac 12)}\times \Big( \frac{|(\btheta_{s,t}(\x)-\y)_k|}{(s-t)^{k-\frac 12}}\Big)^{\gamma_2} (s-t)^{\gamma_2(k-\frac 12)},\label{PREAL_SMOOTHING}
		\end{eqnarray}
		where we normalized the $j^{\rm th}$ and  $k^{\rm th} $ entries at their corresponding time scales in order to absorb them in the off-diagonal bound of the density. Namely, since:
		$$\Big( \frac{|(\btheta_{s,t}(\x)-\y)_j|}{(s-t)^{j-\frac 12}}\Big)^{\gamma_1}\Big( \frac{|(\btheta_{s,t}(\x)-\y)_k|}{(s-t)^{k-\frac 12}}\Big)^{\gamma_2}\le |(s-t)^{\frac 12}(\T_{s-t}^{-1}(\btheta_{s,t}(\x)-\y  ))|^{\gamma_1+\gamma_2} ,$$
		we eventually derive from \eqref{PREAL_SMOOTHING} that up to a modification of $C$, which can be chosen uniformly  w.r.t. $\gamma_1,\gamma_2\in(0,1] $ :
		\begin{eqnarray*}
			&& \int_{\R^{nd}} d\y \Big|D_{\x_l}D_{\x_1}^\ell   \tilde p^\bxi(s,t,\x,\y)\Big|_{\bxi=\x}\big|\big(\y-\btheta_{s,t}(\x)\big)_j \big|^{\gamma_1}\Big|_{\bxi=\x}\big|\big(\y-\btheta_{s,t}(\x)\big)_k \big|^{\gamma_2}\\
			&\le& \int_{\R^{nd}}d\y \frac{C}{(s-t)^{\frac{n^2d}2+(l-\frac{1}{2})+\frac \ell2 }}\exp\left(-C^{-1}(s-t) \big|\T_{s-t}^{-1}\big(\btheta_{s,t}(\x)-\y\big)\big|^2\right)  (s-t)^{\gamma_1(j-\frac 12)+\gamma_2(k-\frac 12)}\notag\\
			&\le& C (s-t)^{-\frac \ell 2-(l-\frac 12) +\gamma_1(j-\frac 12) +\gamma_2(k-\frac 12)},
		\end{eqnarray*}
		which precisely gives \eqref{strong_SMOOTHING_EFFECT_SIMPLE}.
	\end{proof}

	\paragraph{Semi-group.} Fix $t\in [0,T]$, $\bxi\in \R^{nd} $. With the notations of the previous paragraph, we introduce the following inhomogeneous semi-group associated with \eqref{strong_FROZ_MOL_FOR_NO_SUPER_SCRIPTS} for $\tau=t $. Namely, for all $s\in (t,T] $, $g \in  B_{{\rm lin}} (\R^{nd},\R)$ (space of measurable functions with linear growth), $\x\in \R^{nd} $:
	\begin{equation}
	\label{strong_DEF_SG_INHOMO}
	\tilde P_{s,t}^\bxi g(\x):=\int_{\R^{nd}} d\y \tilde p^{(t,\bxi)}(t,s,\x,\y) g(\y).
	\end{equation}
	One can derive from Proposition \ref{strong_THE_PROP} the following important \textcolor{black}{centering} and regularization result.
	\begin{lem}[Centering and Regularization effects for the  inhomogeneous  semi-group]\label{strong_LEMME_SG} 
		
		Let $\textcolor{black}{f:\R^{nd}\rightarrow \R}$ be a $\vartheta $-H\"older continuous functions where $\vartheta:=(\vartheta_1,\textcolor{black}{\hdots}, \vartheta_n) \in (0,1]^n $ is a multi-index and for $i \in \leftB 1,n \rightB $, $\vartheta_i $ stands for the H\"older regularity of $f$ in the variable $\x_i$. 

		\begin{trivlist}
			\item[-] Centering arguments. For all $l\in \leftB 1,n \rightB, k\le l $, $0\le t<s\le T,\ \x,\bxi\in \R^{nd} $ and any \textcolor{black}{$\mathbf a\in \R^{nd} $, it holds that}:
			\begin{equation}\label{strong_CENTERING_PARTIAL}
			\textcolor{black}{D_{\x_l} \tilde P_{s,t}^\bxi \Big( f(\cdot_{1,k-1},\mathbf a_{k:n})\Big)(\x) = 0}.
			\end{equation}
			
			\item[-] As particular cases of the previous property, we have that
			there exists \textcolor{black}{$C:=C(\A{A},T)$} s.t. for all $l\in \leftB 1,n\rightB$, $\x,\bxi\in \R^{nd} $:
			\begin{eqnarray}\label{strong_SMOOTHING_EFFECT}
			\textcolor{black}{|D_{\x_1}D_{\x_l} \tilde P_{s,t}^\bxi f(\x)|}&\le& C \sum_{j=l}^n [ f_{j}(s,\cdot)]_{\vartheta_j}  (s-t)^{-l + \vartheta_j(j-\frac 12)},\\
			\textcolor{black}{|D_{\x_l} \tilde P_{s,t}^\bxi f(\x)|} &\le& C \sum_{j=l}^n [ f_{j}(s,\cdot)]_{\vartheta_j}  (s-t)^{-(l-\frac 12) + \vartheta_j(j-\frac 12)},\notag
			\end{eqnarray}
			\textcolor{black}{with the notations introduced in \eqref{strong_HOLDER_MODULUS}}.
			
		\end{trivlist}
	\end{lem}
	
	\begin{proof}
		
		
		Centering arguments like \eqref{strong_CENTERING_PARTIAL} will be a crucial tool in the analysis below. To justify such an  identity, write:
		\textcolor{black}{
			\begin{eqnarray*}
				\tilde P_{s,t}^\bxi \Big(f(\cdot_{1:k-1}, \mathbf a_{k:n})\Big)(\x)&=&\int_{\R^{nd}} d\y\tilde p^\bxi (t,s,\x,\y) f(\y_{1,k-1},\mathbf a_{k:n})\\
				&=& \int_{\R^{nd}} d\y \tilde p^\bxi (t,s,\x,\y+\m_{s,t}^{(t,\bxi)}(\x)) f\big(\y_{1,k-1}+(\m_{s,t}^{(t,\bxi)}(\x))_{1:k-1},\mathbf a_{k:n}\big).
			\end{eqnarray*}
			Note now from \eqref{strong_CORRESP} that $\tilde p^\bxi (t,s,\x,\y+\m_{s,t}^{(t,\bxi)}(\x))$ does not depend on $\x$. Observe as well that $(\m_{s,t}^{(t,\bxi)}(\x))_{1:k-1} $ does not depend on  $\x_l $. Thus, the right\textcolor{black}{-}hand side of the above equality does not depend on $\x_l $, which implies \eqref{strong_CENTERING_PARTIAL}.}
		
		Let us now prove \eqref{strong_SMOOTHING_EFFECT}. The idea is to use first a centering argument w.r.t. 
		 the variables $l$ to $n$. \textcolor{black}{Namely, from \eqref{strong_CENTERING_PARTIAL} and Proposition \ref{strong_THE_PROP}, it holds that for any $\mathbf a\in \R^{nd} $, 
			\begin{eqnarray*}
				\left| D_{\x_1}D_{\x_l} \left[\tP_{s,t}^\bxi f (s,\cdot)\right](\x) \right| 
				&=&
				\left| D_{\x_1}D_{\x_l} \left[\tP_{s,t}^\bxi \Big(f(s,\cdot)-f(s,\cdot_{1:l-1}, \mathbf a_{l:n} )\Big) \right](\x) \right|\\
				&
				\le 
				& 
				\textcolor{black}{\bar C}   (s-t)^{-(l-1/2)-1/2} \int_{\R^{nd}}d\y \hat p_{C^{-1}}^\bxi(t,s,\x,\y)\Big|f(s,\y)-f(s,\y_{1:l-1},\mathbf a_{l:n}) \Big|  \\
				&\leq & \textcolor{black}{\bar C}  (s-t)^{-(l-1/2)-1/2} \sum_{j=l}^n [ f_{j}(s,\cdot))]_{\vartheta_j} \int_{\R^{nd}} d\y \hat p_{C^{-1}}^\bxi(t,s,\x,\y)\Big|\y_j-\mathbf a_{j}\Big|^{\vartheta_j}.
			\end{eqnarray*}
			Taking now $\mathbf a=\m_{s,t}^{(t,\bxi)}(\x) $, we derive from \eqref{strong_CTR_GRAD} that:}
		\textcolor{black}{
			\begin{align*}
			\left| D_{\x_1}D_{\x_l} \left[\tP_{s,t}^\bxi f (s,\cdot)\right](\x) \right| 
			&\leq  \textcolor{black}{\bar C}   (s-t)^{-(l-1/2)-1/2} \sum_{j=l}^n [ f_{j}(s,\cdot))]_{\vartheta_j} \int_{\R^{nd}}d\y \hat p_{C^{-1}}^\bxi(t,s,\x,\y)\Big|(\y-\m_{s,t}^{(t,\bxi)}(\x))_j\Big|^{\vartheta_j} \\
			&\le C\sum_{j=l}^n [f_j(s,\cdot)]_{\vartheta_j}(s-t)^{-l+\vartheta_j(j-\frac 12)},
			\end{align*}
			which gives the first bound in \eqref{strong_SMOOTHING_EFFECT}}.
		The control for $|D_{\x_l} \tilde P_{s,t}^\bxi f(\x)| 
		$ is derived similarly.
	\end{proof}
	
	We state in the lemma below a useful control to obtain through Lemma \ref{strong_LEMME_SG} some smoothing effects for the degenerate part of the operator. The statement readily \textcolor{black}{follows} from \A{${\mathbf T}_{\beta} $} \textcolor{black}{and \A{H${}_\eta$}}.
	\begin{lem}
		\label{strong_SMOOTH_OPERATOR_EFFECTS}
		From the smoothness assumption on the drift coefficient in \A{${\mathbf T}_{\beta} $}, there exists $C:=C(\A{A})$ s.t. for all $\ell\in \leftB 2,n\rightB $, $k\in \textcolor{black}{\leftB \ell,n+1\rightB}$, and for any $(s,\x,\textcolor{black}{\mathbf a})\in [0,T]\times (\R^{nd})^2 $:
		\begin{eqnarray*}
			\!\!&\!\!\!\!&\!\!\bigg|  \Big(\gF_{\ell}(s,\y_{\textcolor{black}{\ell-1}:k-1},\textcolor{black}{\mathbf a_{k:n}})-\gF_{\ell}(s,\textcolor{black}{\mathbf a_{\ell-1:n}})
			-D_{{\ell-1}}\gF_{\ell}(s,\textcolor{black}{\mathbf a_{\ell-1:n}})\big(\y-\textcolor{black}{\mathbf a}\big)_{\ell-1}\Big)\bigg|\\
			\!\!&\!\!\leq\!\!&\!\! C\Big\{  \sum_{j=\ell}^{k-1}\big\{ [(\gF_{\ell})_j(s,\cdot)]_{\beta_j}|(\y - \textcolor{black}{\mathbf a})_j|^{\beta_j} \big\}\!+\! [(D_{{\ell-1}}\gF_\ell)_{\ell-1}(s,\cdot))]_\eta|(\y - \textcolor{black}{\mathbf a})_{\ell-1}|^{1+\eta}\Big\},
		\end{eqnarray*}
		\textcolor{black}{with the convention that for $k=n+1 $, $\gF_\ell(s,\y_{\ell-1:\textcolor{black}{k-1}},\mathbf{a}_{k:n})=\gF_\ell(s,\y_{\ell-1:n}) $}.
		
	\end{lem} 
	
	\subsection{Control of the sensitivities: proof of Theorem \ref{strong_THM_DER_PDE} }\label{strong_SECTION_CTR_SENSI}
	To prove Theorem \ref{strong_THM_DER_PDE}, the idea is to expand the solution of the PDE with regularized coefficients around a suitable \textit{proxy}, as explained in Section \ref{strong_SUB_STRAT_PDE} \textcolor{black}{(see the connection between equations \eqref{strong_GEN_PDE_MOLL} and \eqref{strong_repPDEintro})}. The \textit{proxy} used here is the Gaussian process introduced in Section \ref{strong_SEC_GAUSS_DENS} for a suitable freezing parameter $ \bxi$ to be specified later on \textcolor{black}{ and whose generator is given by, for all $\varphi \in C_0^2 (\R^{nd},\R^{}) $ and $(t,\x)\in [0,T]\times \R^{nd} $,
\begin{eqnarray*}
 \tilde L_t^{\bxi} \varphi(\x)&:=&
\langle \gF(t,\btheta_{s,t}(\bxi))+ D \gF(t,\btheta_{s,t}(\bxi)) (\x-\btheta_{s,t}(\bxi)), D_\x \varphi(\x)\rangle \\
&&+ \frac 12 {\rm tr}(\sigma\sigma^*(t,\btheta_{s,t}(\bxi) ) D_{\x_1}^2 \varphi(\x) ).\label{DEF_TILDE_L}
\end{eqnarray*}}
Then, the Duhamel formula (or first order parametrix expansion) yields:	
	\begin{eqnarray}
	\label{strong_eq:repPDE}
	&&u(t,\x)
	\nonumber \\
	&=& \int_t^{T} ds \left[\tP_{s,t}^\bxi f(s,\cdot) \right](\x) + \int_t^{T} d s \left[ \tP_{s,t}^\bxi\left(( L_s-\tilde{ L}_s^{\bxi} ) u\right)(s,\cdot)\right](\x) \notag\\
	&=& \int_t^{T} ds \left[\tP_{s,t}^\bxi f (s,\cdot)\right](\x) +\int_t^{T} ds\left[ \tP_{s,t}^\bxi\left(\big\langle\gF_1(s,\cdot)-\gF_1(s,\btheta_{s,t}(\bxi)), D_{1}u(s,\cdot)\big\rangle\right)\right](\x)\notag\\
	&& + \int_t^{T} ds\left[ \tP_{s,t}^\bxi\left(\frac{1}{2}{\rm Tr}\left[a(s,\cdot)-a(s,\btheta_{s,t}(\bxi)) D_{1}^2u(s,\cdot)\right]\right)\right](\x)  \notag\\
	&&+ \int_t^{T} ds\Bigg[ \tP_{s,t}^\bxi\Bigg(\sum_{i=2}^n \big\langle(\gF_i(s,\cdot)-\gF_i(s,\btheta_{s,t}(\bxi))
	-D_{{i-1}
	}\gF_i(s,\btheta_{s,t}(\bxi))(\cdot - \btheta_{s,t}(\bxi))_{i-1}), D_{i}u(s,\cdot)\big\rangle\Bigg)\Bigg](\x),\notag \\
	\end{eqnarray}
	for any $\bxi$ in $\R^{nd}$\textcolor{black}{, according to the notations introduced in Remark \ref{nota_F} for the entries $(\gF_i)_{i \in \leftB 1,n\rightB}$ of $\gF$}.\\
	
	To establish \eqref{strong_BD_PDE} we need to differentiate the above expression w.r.t.  \textcolor{black}{$(\x_l)_{l\in \leftB 1,n \rightB}$  and then w.r.t. $\x_1 $}. Differentiating first this expression w.r.t. $\x_l, l\in \leftB 1,n\rightB$ we obtain:
	\begin{eqnarray}
	&&D_{\x_l}u(t,\x)\notag\\
	&=& \int_t^{T} d s D_{\x_l} \left[\tP_{s,t}^\bxi f(s,\cdot) \right](\x) + \int_t^{T} ds D_{\x_l}  \left[ \tP_{s,t}^\bxi\left(( L_s-\tilde{ L}_s^{\bxi} ) u\right)(s,\cdot)\right](\x) \notag\\
	&=& \Bigg\{ \!\int_t^{T} d s D_{\x_l}  \left[\tP_{s,t}^\bxi f (s,\cdot)\right](\x) 
	+\int_t^{T} ds D_{\x_l}\left[ \tP_{s,t}^\bxi\left(\big\langle(\gF_1(s,\cdot)-\gF_1(s,\btheta_{s,t}(\bxi))), D_{1}u(s,\cdot)\big\rangle\right)\right](\x)
	\notag\\
	&& + \int_t^{T}ds D_{\x_l} \left[ \tP_{s,t}^\bxi\left(\frac{1}{2}{\rm Tr}\left[\Big(a(s,\cdot)-a(s,\btheta_{s,t}(\bxi))\Big) D_{1}^2u(s,\cdot)\right]\right)\right](\x) \Bigg\}\notag\\
	&&+\Bigg\{ \! \int_t^{T} \! ds D_{\x_l}  \Bigg[ \tP_{s,t}^\bxi\bigg(\sum_{i=2}^n \big\langle(\gF_i(s,\cdot)-\gF_i(s,\btheta_{s,t}(\bxi))
	-D_{{i-1}}\gF_i(s,\btheta_{s,t}(\bxi))(\cdot - \btheta_{s,t}(\bxi))_{i-1}), D_{i}u(s,\cdot)\big\rangle\bigg)\Bigg](\x) \Bigg\}\notag\\
	&=:&\textcolor{black}{\Bigg \{ }\int_t^T ds \sum_{k=1}^3  H_{l,k}^\bxi(s,\x) \textcolor{black}{\Bigg \}} +
	\textcolor{black}{\Bigg \{ }\int_t^T ds I_l^\bxi(s,\x) \textcolor{black}{\Bigg \} }
	=:
	\int_t^T ds H_l^\bxi(s,\x) +\int_t^T ds I_l^\bxi(s,\x). \label{strong_GROS_DECOUP}
	\end{eqnarray}
	\textcolor{black}{The term $H_l^\bxi(s,\x) $ gathers the sensitivities w.r.t. $\x_l $ of the source term and the non-degenerate part of the difference of the operators, whereas  $ I_l^\bxi(s,\x)$ precisely gathers the sensitivities w.r.t. $\x_l $ of the degenerate part of the difference of the operators.}
	We will now start from the representation \eqref{strong_GROS_DECOUP} which we will again differentiate w.r.t. the non-degenerate variable $\x_1$ in order to prove the estimates of Theorem \ref{strong_THM_DER_PDE} concerning the second order derivatives which are the trickiest ones.  Indeed, as it has been succinctly explained in Section \ref{strong_SUB_STRAT_PDE}, when differentiating the kernel associated with the frozen semi-group defined by \eqref{strong_DEF_SG_INHOMO} we generate an \emph{a priori} not integrable time singularity which then needs to be smoothen\textcolor{black}{ed} by using, among others, tools developed in Lemma \ref{strong_LEMME_SG} (centering or cancellation arguments). The worst case then corresponds to the higher order of differentiation, namely $ D_{\x_l}\big( D_{\x_1} u(t,\x)\big)$ which, as suggested by Proposition \ref{strong_THE_PROP}, generates a time singularity of order $1/2 + (l-1/2)$ in the time integrand of the r.h.s. of \eqref{strong_GROS_DECOUP}. We then only concentrate on this term and omit the proof of the statement concerning the boundedness of the gradient $ D_{\x_l} u(t,\x)$ which could be shown more directly. 
	
	The proof will be divided into two parts: we first handle the \textcolor{black}{source} and non-degenerate part of the operator (\emph{i.e.} the estimate for $D_{\x_1}H_l^{\textcolor{black}{\bxi}}(s,\x)$) and then the degenerate part (\emph{i.e.} the estimate for $D_{\x_1} I_l^{\textcolor{black}{\bxi}}(s,\x)$) which is a bit more involved.\\

	\paragraph{Source term and non-degenerate part of the operator: estimates for $ \textcolor{black}{(D_{\x_1}H_l^\bxi(s,\x))_{\bxi=\x}}$.} We first focus 
	on  the source term and the derivatives w.r.t. the non-degenerate variable $\x_1$, 
	\textcolor{black}{three first terms in the r.h.s. denoted by $H_{l,1}^{\textcolor{black}{\bxi}}(s,\x),H_{l,2}^{\textcolor{black}{\bxi}}(s,\x),H_{l,3}^{\textcolor{black}{\bxi}}(s,\x) $ in  \eqref{strong_GROS_DECOUP}, \textcolor{black}{taking $\bxi=\x$ after the additional differentiation}}.
	
	For each $l\in \leftB 1,n\rightB $, one readily derives from \textcolor{black}{\eqref{strong_SMOOTHING_EFFECT} in Lemma \ref{strong_LEMME_SG} 
	} 
	that for the source term\footnote{Observe that for this contribution, from \eqref{strong_SMOOTHING_EFFECT} the bound would hold for any freezing parameter $\bxi $. We choose here to take $\bxi=\x $ for the compatibility with the other terms $(D_{\x_1}H_{l,\textcolor{black}{k}}^{\textcolor{black}{\bxi}}(s,\x) )_{\textcolor{black}{k}\in \{2,3\} }$ for which this specific choice is indeed needed.}:
	\begin{equation}
	\left| D_{\x_1} H_{l,1}^{\textcolor{black}{\bxi}}(s,\x)\right|_{\textcolor{black}{\bxi=\x}}= \left|  D_{\x_1}D_{\x_l} \left[\tP_{s,t}^\bxi f (s,\cdot)\right](\x) \right|\bigg|_{\bxi=\x}  \leq  C \sum_{j=l}^n [ f_{j}(s,\cdot)]_{\beta_j}  (s-t)^{-l + \beta_j(j-\textcolor{black}{\frac12})}.\label{strong_CTR_DX1_HL_1}
	\end{equation}
	Those terms are integrable in time as soon as 
	\begin{equation}\label{strong_LES_SEUILS}
	\beta_j(j-\frac 12) - l> -1,\quad j\in \leftB l,n\rightB \Leftarrow \beta_j \in \left( \frac{2j-2}{2j-1}\textcolor{black}{,}1\right].
	\end{equation}

	Similarly, from \eqref{strong_eq:repPDE}, \eqref{strong_GROS_DECOUP}, for the drift associated with the non-degenerate part, we first rewrite from the centering properties of Lemma \ref{strong_LEMME_SG}:
	\begin{eqnarray*}
		&&\left| D_{\x_1} H_{l,2}^{\textcolor{black}{\bxi}}(s,\x)\right|_{\textcolor{black}{\bxi=\x}}
		\nonumber \\
		&=&\left| D_{\textcolor{black}{\x_1}} D_{\textcolor{black}{\x_l}}  \left[ \tP_{s,t}^\bxi\left(\big\langle\gF_1(s,\cdot)-\gF_1(s,\btheta_{s,t}(\bxi)), D_{1}u(s,\cdot)\big\rangle\right)\right](\x)\right|\bigg|_{\bxi=\x}\\
		&\le & \left| D_{\textcolor{black}{\x_1}} D_{\textcolor{black}{\x_l}} \left[ \tP_{s,t}^\bxi\left(\big\langle\gF_1(s,\cdot)-\gF_1(s,\cdot_{1:l-1},\btheta_{s,t}^{l:n}(\bxi)), D_{1}u(s,\cdot)\big\rangle\right)\right](\x)\right|\bigg|_{\bxi=\x}\\
		&& + \left| D_{\textcolor{black}{\x_1}} D_{\textcolor{black}{\x_l}} \left[ \tP_{s,t}^\bxi\left(\big\langle\gF_1(s,\cdot_{1:l-1},\btheta_{s,t}^{l:n}(\bxi))-\gF_1(s,\btheta_{s,t}(\bxi))
		\big(D_{1}u(s,\cdot)-D_1u(s,\cdot_{1:l-1},\btheta_{s,t}^{l:n}(\bxi))\big)\big\rangle\right)\right](\x)\right|\bigg|_{\bxi=\x}.
	\end{eqnarray*}
	\textcolor{black}{Expanding the semi-group yields
		\begin{eqnarray*}
			\!\!&\!\!&\!\!\left| D_{\x_1} H_{l,2}^\bxi(s,\x)\right|_{\bxi=\x}
			\\
			\!\!&\!\!\le\!\!&\!\! \int_{\R^{nd}}d\y|D_{\textcolor{black}{\x_1}} D_{\textcolor{black}{\x_l}} \tilde p^\bxi(t,s,\x,\y) |_{\bxi=\x}\ |\gF_1(s,\y)-\gF_1(s,\y_{1:l-1},\btheta_{s,t}^{l:n}(\x))|\ |D_1 u(s,\y)|d\y\notag\\
			\!\!&\!\!&\!\! \!+\int_{\R^{nd}} \!\!d\y |D_{\textcolor{black}{\x_1}} D_{\textcolor{black}{\x_l}} \tilde p^\bxi(t,s,\x,\y) |_{\bxi=\x}\ \big |\gF_1(s,\y_{1:l-1},\btheta_{s,t}^{l:n}(\x))-\gF_1(s,\btheta_{s,t}(\x))\big |
			\big |D_1 u(s,\y)-D_1 u(s,\y_{1:l-1},\btheta_{s,t}^{l:n}(\x))\big |.
		\end{eqnarray*}
			\textcolor{black}{Hence,}
				\begin{eqnarray*}\label{THE_PREAL_LAST_PIV}
				\!\!&\!\!&\!\!
							\textcolor{black}{\left| D_{\x_1} H_{l,2}^\bxi(s,\x)\right|_{\bxi=\x}}
				\\
			 \notag\\
			\!\!&\!\!\le\!\! &\!\!\bigg( \sum_{j=l}^n \|D_1 u\|_\infty [(\gF_1)_j(s,\cdot)]_{\beta_j}\int_{\R^{nd}}d\y|D_{\x_l} D_{\x_1} \tilde p^\bxi(s,t,\x,\y)|_{\bxi=\x}|(\btheta_{s,t}(\x)-\y)_j|^{\beta_j}\notag\\
			\!\!&\!\!+\!\!&\!\!\sum_{j=l}^{n} \sum_{k=1}^{l-1} \| D_1 D_j u\|_{\infty}  [(\gF_1)_k(s,\cdot)]_{\beta_k} \int_{\R^{nd}}d\y|D_{\x_l} D_{\x_1} \tilde p^\bxi(s,t,\x,\y)|_{\bxi=\x}|(\btheta_{s,t}(\x)-\y)_k|^{\beta_k}|(\btheta_{s,t}(\x)-\y)_j|\bigg).\notag
		\end{eqnarray*}
	}
	
	\textcolor{black}{We then conclude from equation \eqref{strong_SMOOTHING_EFFECT_SIMPLE} of Lemma \ref{strong_LEMME_DENS}, taking $\gamma_1=\beta_j,\gamma_2=0 $ for the first terms and $\gamma_1=1,\gamma_2=\beta_k $ for the second ones, that}
	\begin{eqnarray}
	\!\!&\!\!&\!\!\left| D_{\x_1} H_{l,2}^{\textcolor{black}{\bxi}}(s,\x)\right|_{\textcolor{black}{\bxi=\x}}\label{strong_CTR_DX1_HL_2}
	\\
	\!\!&\!\!\leq \!\!&\!\! C \bigg( \sum_{j=l}^n \|D_1 u\|_\infty [(\gF_1)_j(s,\cdot)]_{\beta_j}(s-t)^{-l+\beta_j(j-\frac 12)}
	+ \sum_{j=l}^{n} \sum_{k=1}^{l-1} \| D_1 D_j u\|_{\infty}  [(\gF_1)_k(s,\cdot)]_{\beta_k} (s-t)^{-l+\beta_k(k-\frac 12)+(j-\frac 12)}\bigg),\notag
	\end{eqnarray}
	leading precisely to the same integrability thresholds of equation \eqref{strong_LES_SEUILS} and assumption \A{T${}_\beta$} (as for the source term). The idea behind this control is crucial.  We first handle, with the sole H\"older properties of \textcolor{black}{the} drift and the supremum norm of $D_1 u$, the variables which are at a good smoothing scale w.r.t. the induced singularity. For the remaining term, which exhibits for the drift non-sufficient smoothing effects, we then additionally exploit a cancellation argument involving the gradient of the solution itself, which consequently makes the cross derivatives appear.
	
	Eventually, we get for the diffusive part:
	\begin{eqnarray}
	\left| D_{\x_1} H_{l,3}^{\textcolor{black}{\bxi}}(s,\x)\right|_{\textcolor{black}{\bxi=\x}}
	&=&\left| D_{\x_1} D_{\x_l} \left[ \tP_{s,t}^\bxi\left(\frac{1}{2}{\rm Tr}\left[\big(a(s,\cdot)-a(s,\btheta_{s,t}(\bxi))\big) D_{1}^2u(s,\cdot)\right]\right)\right](\x)\right|\Bigg|_{\bxi=\x}
	\notag\\
	&\textcolor{black}{\le}& \left| D_{\x_1} D_{\x_l}  \left[ \tP_{s,t}^\bxi\left(\frac{1}{2}{\rm Tr}\left[\big(a(s,\cdot)-a(s,\cdot_{1:l-1},\btheta_{s,t}^{l:n}(\bxi)) \big)D_{1}^2u(s,\cdot)\right]\right)\right](\x)\right| \Bigg|_{\bxi=\x}
	\notag\\
	&& + \left| D_{\x_1} D_{\x_l}  \left[ \tP_{s,t}^\bxi\left(\frac{1}{2}{\rm Tr}\left[\big(a(s,\cdot_{1:l-1},\btheta_{s,t}^{l:n}(\bxi))-a(s,\btheta_{s,t}(\xi))\big) D_{1}^2u(s,\cdot)\right]\right)\right](\x)\right|\Bigg|_{\bxi=\x}
	\notag\\
	&=:&\left| D_{\x_1} H_{l,31}^{\textcolor{black}{\bxi}}(s,\x)\right|_{\textcolor{black}{\bxi=\x}}+\left| D_{\x_1} H_{l,32}^{\textcolor{black}{\bxi}}(s,\x)\right|_{\textcolor{black}{\bxi=\x}}. 
	\label{strong_CTR_DX1_HL_3}
	\end{eqnarray}
	The term $|D_{\x_1} H_{l,31}^{\textcolor{black}{\bxi}}(s,\x)|_{\textcolor{black}{\bxi=\x}} $ is already centered at the \textit{appropriate} scales, i.e. from variables $l$ to $n$. \textcolor{black}{Recalling that $a$ is Lipschitz continuous, write:
		\begin{eqnarray*}
			|D_{\x_1} H_{l,31}^{\textcolor{black}{\bxi}}(s,\x)|_{\textcolor{black}{\bxi=\x}} &=&\left|   \int_{\R^d}d\y D_{\x_1} D_{\x_l} \tilde p^\bxi(t,s,\x,\y)\frac{1}{2}{\rm Tr}\left[\big(a(s,\y)-a(s,\y_{1:l-1},\btheta_{s,t}^{l:n}(\bxi)) \big)D_{1}^2u(s,\y)\right]\right| \Bigg|_{\bxi=\x}\notag\\
			&\le & \frac 12\int_{\R^d}d\y |D_{\x_1} D_{\x_l} \tilde p^\bxi(t,s,\x,\y)|_{\bxi=\x}\ |\big(a(s,\y)-a(s,\y_{1:l-1},\btheta_{s,t}^{l:n}(\x)) |\ |D_{1}^2u(s,\y)|\notag\\
			&\le & \frac 12[a(s,\cdot)]_1\|D_1^2u\|_\infty \int_{\R^d}d\y |D_{\x_1} D_{\x_l} \tilde p^\bxi(t,s,\x,\y)|_{\bxi=\x}\ |\y_{l:n}-\btheta_{s,t}^{l:n}(\x) |.
	\end{eqnarray*} }

	\textcolor{black}{We can apply  \eqref{strong_SMOOTHING_EFFECT_SIMPLE} of Lemma \ref{strong_LEMME_DENS} with $\gamma_1=1,\gamma_2=0 $ to any of the entries $|\y_{j}-\btheta_{s,t}^{j}(\x)|,\ j\in \leftB l,n\rightB $ to obtain:}
	\begin{equation}\label{strong_CTR_DX1_HL_31}
	|D_{\x_1} H_{l,31}^{\textcolor{black}{\bxi}}(s,\x)|_{\textcolor{black}{\bxi=\x}}\le C\|D_1^2 u\|_\infty [a(s,\cdot)]_1 \sum_{j=l}^{n}(s-t)^{-l+ (j-\frac 12)}
	\le C\|D_1^2 u\|_\infty[a(s,\cdot)]_1 (s-t)^{-\frac 12},
	\end{equation}
	which does not give a \textit{critical} contribution w.r.t. the previously exhibited thresholds in \eqref{strong_LES_SEUILS} and \A{T${}_\beta $}. 
	For the contribution $|D_{\x_1} H_{l,32}^{\textcolor{black}{\bxi}}(s,\x)|_{\textcolor{black}{\bxi=\x}} $ we use, in the same spirit as for $|D_{\x_1} H_{l,2}^{\textcolor{black}{\bxi}}(s,\x)|_{\textcolor{black}{\bxi=\x}}$, a centering argument and an integration  by parts to obtain:
	\begin{eqnarray*}
		&&|D_{\x_1} H_{l,32}^{\textcolor{black}{\bxi}}(s,\x)|_{\textcolor{black}{\bxi=\x}}
		\nonumber \\
		&=&\Bigg| D_{\textcolor{black}{\x_1}} D_{\textcolor{black}{\x_l}} \Big[ \tP_{s,t}^\bxi\Big(\frac{1}{2}{\rm Tr}\big[\big(a(s,\cdot_{1:l-1},\btheta_{s,t}^{l:n}(\bxi))-a(s,\btheta_{s,t}(\bxi))\big)
		(D_{1}^2u(s,\cdot)-D_1^2u(s,\cdot_{1:l-1},\btheta_{s,t}^{l:n}(\bxi) )) \big] \Big)\Big](\x)\bigg| \Bigg|_{\bxi=\x} \\
		&=&\Bigg| \frac 12 D_{\textcolor{black}{\x_1}}D_{\textcolor{black}{\x_l}}\Bigg(\sum_{j=1}^d   \int_{\R^{nd}}d\y \Big\langle\partial_{\y_1^j}\Big( \tilde p^\bxi (t,s,\x,\y)\big(a_{j\cdot}(s,\y_{1:l-1},\btheta_{s,t}^{l:n}(\bxi))-a_{j\cdot}(s,\btheta_{s,t}(\bxi))\big) \Big), 
		\\
		&& 
		\Big(D_{1}u(s,\y)-D_1u(s,\y_{1:l-1},\btheta_{s,t}^{l:n}(\bxi) )\Big)\Big\rangle \Bigg)\bigg| \Bigg|_{\bxi=\x},
	\end{eqnarray*}
	where $\y_1=(\y_1^1,\textcolor{black}{\hdots},\y_1^d) $, $\textcolor{black}{\partial_{\y_1^j}} $ denotes the derivative w.r.t. 
	the $j^{{\rm th}} $ scalar entry of the non-degenerate variable $\y_1$ and $ a_{j\cdot}$ denotes the $j^{{\rm th}}$ row of the diffusion matrix $a$. We therefore derive from \textcolor{black}{the proof of Lemma \ref{strong_LEMME_DENS}} and the smoothness of $a$ (\textcolor{black}{by the Rademacher theorem, $a$ is differentiable almost everywhere}):
	\begin{eqnarray}
	&&|D_{\x_1} H_{l,32}^{\textcolor{black}{\bxi}}(s,\x)|_{\textcolor{black}{\bxi=\x}}
	\notag \\
	&\le&C\|{\mathbf D}D_1 u(s,\cdot)\|_\infty\int_{\R^{nd}} \frac{d\y}{(s-t)^l} \hat p_{C^{-1}}^{\x}(t,s,\x,\y)
	\Big(\frac{|(\y-\btheta_{s,t}(\textcolor{black}{\x}))_{1:l-1}|}{(s-t)^{\frac 12}} +1\Big)[a(s,\cdot)]_1 |(\y-\btheta_{s,t}(\textcolor{black}{\x}))_{l:n}|\notag\\
	&\le& C\|{\mathbf D}D_1 u(s,\cdot)\|_\infty (s-t)^{-\frac 12}.\label{strong_CTR_DX1_HL_32}
	\end{eqnarray}
	With the notations of \eqref{strong_GROS_DECOUP}, 
	plugging \eqref{strong_CTR_DX1_HL_31}, 
	\eqref{strong_CTR_DX1_HL_32} into \eqref{strong_CTR_DX1_HL_3} and together with \eqref{strong_CTR_DX1_HL_2}, \eqref{strong_CTR_DX1_HL_1},
	we eventually derive that there exists $\delta:=\delta(\A{A})>0 $:
	\begin{equation}\label{strong_LE_CTR_SUR_H_1L}
	\Bigg|\int_t^T ds D_{\x_1}H_l^{\textcolor{black}{\bxi}}(s,\x)\Bigg|_{\textcolor{black}{\bxi=\x}}\le CT^{\delta}(\|D_1 u\|_\infty+
	\|{\mathbf D} D_1u\|_\infty ).
	\end{equation}

	\paragraph{Degenerate part of the operator\textcolor{black}{: estimates for $(D_{\x_1} I_l^{\textcolor{black}{\bxi}}(s,\x))_{\textcolor{black}{\bxi=\x}}$.}} These are the most delicate terms to handle. Restarting from \eqref{strong_GROS_DECOUP}, we first write for any $l\in \leftB 2,n\rightB $:
	\begin{eqnarray*}
		&&|D_{\x_1} I_l^{\textcolor{black}{\bxi}}(s,\x)|_{\textcolor{black}{\bxi=\x}}\notag\\
		\!&\!\!\!=\!\!&\!\!\!\bigg | D_{\textcolor{black}{\x_1}} D_{\textcolor{black}{\x_l}} \bigg[\tP_{s,t}^\bxi  \bigg (\sum_{i=2}^n \big\langle(\gF_i(s,\cdot)-\gF_i(s,\btheta_{s,t}(\bxi))-D_{i-1}\gF_i(s,\btheta_{s,t}(\bxi))(\cdot - \btheta_{s,t}(\bxi))_{i-1}), 
		D_{i}u(s,\cdot) \big\rangle\bigg )\bigg](\x) \bigg| \Bigg|_{\bxi=\x}, 
	\end{eqnarray*}
	which readily yields that
	\begin{eqnarray}
	\!\!&\!\!&\!\! |D_{\x_1} I_l^{\textcolor{black}{\bxi}}(s,\x)|_{\textcolor{black}{\bxi}=\x}\notag\\
	\!&\!\!\!\leq\! \!&\!\!\! \bigg| D_{\x_1} D_{\x_l} \bigg[\tP_{s,t}^\bxi \bigg(\sum_{i=l+1}^{n} \big\langle(\gF_i(s,\cdot)-\gF_i(s,\btheta_{s,t}(\bxi))-D_{i-1}\gF_i(s,\btheta_{s,t}(\bxi))(\cdot - \btheta_{s,t}(\bxi))_{i-1}), 
	D_{i}u(s,\cdot)\big\rangle\bigg)\bigg](\x) \bigg|\Bigg|_{\bxi=\x}  \notag\\
	\!\!&\!\!&\!\!+\bigg| D_{\x_1} D_{\x_l} \bigg[\tP_{s,t}^\bxi \bigg(\sum_{i=2}^{l} \big\langle\gF_i(s,\cdot)-\gF_i(s,\cdot_{1:l-1},\btheta_{s,t}^{l:n}(\bxi)), 
	D_{i}u(s,\cdot)\big\rangle\bigg)\bigg](\x) \bigg| \Bigg|_{\bxi=\x}\notag \\
	\!\!&\!\!&\!\!+ \bigg| D_{\x_1} D_{\x_l} \bigg [\tP_{s,t}^\bxi \bigg(\sum_{i=2}^{l
	} \big\langle(\gF_i(s,\cdot_{1:l-1},\btheta^{l:n}_{s,t}(\bxi))-\gF_i(\btheta_{s,t}(\bxi))
	\notag\\
	\!\!&\!\!&\!\! \phantom{bou!} \hspace*{.2cm}
	-D_{i-1}\gF_i(s,\btheta_{s,t}(\bxi))(\cdot - \btheta_{s,t}(\bxi))_{i-1}), D_{i}u(s,\cdot)\big\rangle\bigg)\bigg](\x) \bigg| \Bigg|_{\bxi=\x}\notag\\
	\!\!&\!\!=:\!\! &\!\! |D_{\x_1 }I_{\textcolor{black}{l,1}}^{\textcolor{black}{\bxi}}(s,\x)|_{\textcolor{black}{\bxi=\x}} + |D_{\x_1 } I_{\textcolor{black}{l,2}}^{\textcolor{black}{\bxi}}(s,\x)|_{\textcolor{black}{\bxi=\x}} +|D_{\x_1 } I_{\textcolor{black}{l,3}}^{\textcolor{black}{\bxi}}(s,\x)|_{\textcolor{black}{\bxi}=\x}
	.\label{strong_DECOUP_I}
	\end{eqnarray}
	\noindent\textcolor{black}{$\bullet$ \emph{Control of $|D_{\x_1 }I_{\textcolor{black}{l,1}}^{\textcolor{black}{\bxi}}(s,\x)|_{\textcolor{black}{\bxi=\x}}$ and $|D_{\x_1 }I_{\textcolor{black}{l,2}}^{\textcolor{black}{\bxi}}(s,\x)|_{\textcolor{black}{\bxi=\x}}$.}} We emphasize that the integrands $|D_{\x_1 }I_{\textcolor{black}{l,1}}^{\textcolor{black}{\bxi}}(s,\x)|_{\textcolor{black}{\bxi=\x}}$ and $|D_{\x_1 }I_{\textcolor{black}{l,2}}^{\textcolor{black}{\bxi}}(s,\x)|_{\textcolor{black}{\bxi=\x}}$ are already designed to smoothen the time singularities generated by the cross differentiation of the inhomogeneous semi-group w.r.t. the variables $\x_l$ and $\x_1$. Indeed, 
	\textcolor{black}{\textcolor{black}{w}rite:
		\begin{eqnarray*}
			\!\!&\!\!\!\!&\!\!|D_{\x_1 }I_{\textcolor{black}{l,1}}^{\bxi}(s,\x) |_{\bxi=\x}
		\nonumber \\
			\!\!&\!\! \le\!\!  &\!\! \int_{\R^{nd}}d\y| D_{\x_1} D_{\x_l} \tilde p^\bxi(s,t,\x,\y)
			|_{\bxi=\x}
			\notag\\
			\!\!&\!\!\!\!&\!\!\hspace*{10pt}\times 
			\sum_{i=l+1}^{n} \Big| \big\langle(\gF_i(s,\y)-\gF_i(s,\btheta_{s,t}(\x))-D_{i-1}\gF_i(s,\btheta_{s,t}(\x))(\y - \btheta_{s,t}(\x))_{i-1}) ,
			D_{i}u(s,\y)\big\rangle\Big|\notag\\
			\!\!&\!\!\le\!\! &\!\! \| D_{l+1:n} u\|_\infty\int_{\R^{nd}}\!\!d\y| D_{\x_1} D_{\x_l} \tilde p^\bxi(s,t,\x,\y)
			|_{\bxi=\x}
			\textcolor{black}{\sum_{i=l+1}^{n}} 
			\Big|\gF_i(s,\y)-\gF_i(s,\btheta_{s,t}(\x))-D_{i-1}\gF_i(s,\btheta_{s,t}(\x))(\y - \btheta_{s,t}(\x))_{i-1}\Big|,
			\label{PREAL_SECOND_ASSERT}
		\end{eqnarray*}
		\textcolor{black}{where $D_{l+1:n}u:=\big(D_{l+1}u,\textcolor{black}{\hdots}, D_n u \big)$}.
	}
	\textcolor{black}{Applying now Lemma \ref{strong_SMOOTH_OPERATOR_EFFECTS} with $\mathbf a=\btheta_{s,t}(\x) $, $\ell=i $, $k=n+1 $ yields:
		\begin{eqnarray*}
			&&\sum_{i=l+1}^n\Big|\gF_i(s,\y)-\gF_i(s,\btheta_{s,t}(\x))-D_{i-1}\gF_i(s,\btheta_{s,t}(\x))(\y - \btheta_{s,t}(\textcolor{black}{\x}))_{i-1})\Big|\notag\\
		&	\le& C \sum_{j=l+1}^n\Big( [(\gF_{\textcolor{black}{l+1:j}}(s,\cdot))_j]_{\beta_j} |(\btheta_{s,t}(\x)-\y)_j|^{\beta_j}
			+[(D_{j-1}\gF_j(s,\cdot))_{j-1}]_\eta  |(\btheta_{s,t}(\x)-\y)_{j-1}|^{1+\eta}  \Big),
		\end{eqnarray*}
			 \textcolor{black}{using the notation $[(\gF_{l+1:j})_j(s,\cdot)]_{\beta_j}:=\max_{i\in \leftB l+1,j\rightB} [(\gF_{i})_j(s,\cdot)]_{\beta_j}$
for the last inequality}.
	We finally derive from Lemma \ref{strong_LEMME_DENS}}:
	\begin{eqnarray}
	\!\!&\!\!\!\!&\!\!
	|D_{\x_1 }I_{\textcolor{black}{l,1}}^{\textcolor{black}{\bxi}}(s,\x) |_{\textcolor{black}{\bxi=\x}}\label{strong_CTR_DEBUT_I_11l}\\
	\!\!&\!\!\leq \!\!&\!\! C\|D_{l+1:n}u\|_{\infty}\!\! \sum_{j=l+1}^n\!\!\! \left\{
		[(\gF_{\textcolor{black}{l+1:j}}\textcolor{black}{(s,\cdot)})_j]_{\beta_j}   (s\!-\!t)^{-l+\beta_j(j-\textcolor{black}{\frac 12})} 
		\!+\!
		[(D_{j-1}\gF_j\textcolor{black}{(s,\cdot)})_{j-1}]_\eta (s\!-\!t)^{-l+(1+\eta)(j-\textcolor{black}{\frac 32})} 
	\right\}.\notag
	\end{eqnarray}
	\textcolor{black}{Since in the above contribution $j\ge l+1 $, we have on the one hand  $-l+(1+\eta)(j-3/2)\ge -l+(1+\eta)(l+1-3/2)\ge -1/2 $. On the other hand,  from assumption $\A{T_{\beta}} $, $\beta_j\ge (2j-2)/(2j-1) $ and
		$-l+\beta_j(j- 1/2)\ge -l +(2j-2)/(2j-1) \times (j-1/2)=-l+j-1\ge 0$. Thus, none of the associated exponent is critical. We have either integrable singularities or no singularities at all}.
	
	
	\textcolor{black}{Write now for the other term, 
		\begin{eqnarray*}
			&&|D_{\x_1 } I_{\textcolor{black}{l,2}}^{\textcolor{black}{\bxi}}(s,\x)|_{\textcolor{black}{\bxi}=\x} 
			\nonumber \\
			\!\!&\!\! =\!\! &\!\! \left| D_{\x_1} D_{\x_l} \left[\tilde P_{s,t}^\bxi \left(\sum_{i=2}^{l} \big\langle\gF_i(s,\cdot)-\gF_i(s,\cdot_{1:l-1},\btheta^{l:n}_{s,t}(\bxi)), D_{i}u(s,\cdot)\big\rangle\right)\right](\x) \right| \Bigg|_{\bxi=\x}\notag\\ 
			&\le & \sum_{i=2}^l\int_{\R^{nd}}d\y |D_{\x_1} D_{\x_l} \tilde p^\bxi(s,t,\x,\y)|_{\bxi=\x} |\gF_i(s,\y)-\gF_i(s,\y_{1:l-1},\btheta^{l:n}_{s,t}(\x))| |D_i u(s,\y)|\\
			&\le &  \|D_{{2:l}}u\|_{\infty}\sum_{i=2}^l\sum_{j=l}^n\int_{\R^{nd}}d\y|D_{\x_1} D_{\x_l} \tilde p^\bxi(s,t,\x,\y)|_{\bxi=\x}
			[(\gF_i)_j(s,\cdot)]_{\beta_j} |(\btheta_{s,t}(\x)-\y)_j|^{\beta_j} \\
			&\le &C\|D_{{2:l}}u\|_{\infty} \sum_{j=l}^n [(\gF_{\textcolor{black}{2:j}})_j(s,\cdot)]_{\beta_j}\textcolor{black}{(s-t)^{-l
			}}
			\int_{\R^{nd}}d\y|D_{\x_1} D_{\x_l} \tilde p^\bxi(s,t,\x,\y)|_{\bxi=\x}
			|(\btheta_{s,t}(\x)-\y)_j|^{\beta_j} .
		\end{eqnarray*}
	}
	\textcolor{black}{Hence, Lemma \ref{strong_LEMME_DENS} also yields} 
	\begin{eqnarray}
	|D_{\x_1 } I_{\textcolor{black}{l,2}}^{\textcolor{black}{\bxi}}(s,\x)|_{\textcolor{black}{\bxi=\x}} 
	\!\!&\!\!\leq \!\!&\!\! C \textcolor{black}{\|D_{{2:l}}u\|_{\infty}}  \sum_{j=l}^n  [(\textcolor{black}{\gF_{\textcolor{black}{2:j}}})_j(s,\cdot)]_{\beta_j} (s-t)^{-l+\beta_j(j-\frac 12)},
	\label{strong_CTR_DEBUT_I_12l}
	\end{eqnarray}
	and those terms are again integrable as soon as the thresholds of \A{T${}_\beta$} hold.\\
	
	\noindent\emph{\textcolor{black}{$\bullet$ Control of $|D_{\x_1} I_{\textcolor{black}{l,3}}^{\textcolor{black}{\bxi}}\textcolor{black}{(s,\x)}|_{\textcolor{black}{\bxi=\x}}$.}} It hence remains to control the terms in $|D_{\x_1} I_{\textcolor{black}{l,3}}^{\textcolor{black}{\bxi}}\textcolor{black}{(s,\x)}|_{\textcolor{black}{\bxi=\x}}$. 
	These terms are the tricky ones since they are, \emph{a priori}, not designed to smoothen the time singularities generated by the cross differentiation. Observe indeed that, if one tries to reproduce the above calculations, we obtain from \textcolor{black}{Lemmas \ref{strong_LEMME_DENS} and \ref{strong_SMOOTH_OPERATOR_EFFECTS}},   that
	\begin{equation}\label{impasse}\tag{I}
	\begin{split}
	&|D_{\x_1 } I_{\textcolor{black}{l,3}}^{\textcolor{black}{\bxi}}(s,\x)|_{\textcolor{black}{\bxi}=\x} 
	\\
	=\,  &\Bigg| D_{\x_1} D_{\x_l} \Bigg[\tP_{s,t}^\bxi \bigg(\sum_{i=2}^{l
	} \big\langle(\gF_i(s,\cdot_{1:l-1},\btheta^{l:n}_{s,t}(\bxi))-\gF_i(\btheta_{s,t}(\bxi))
	\notag\\
	&\phantom{bou!}\hspace*{.2cm}
	-D_{i-1}\gF_i(s,\btheta_{s,t}(\bxi))(\cdot - \btheta_{s,t}(\bxi))_{i-1}), D_{i}u(s,\cdot)\big\rangle\bigg)\Bigg](\x) \Bigg| \Bigg|_{\bxi=\x}\notag\\ 
	\leq\, & C \|D_{{2:l}}u\|_{\infty} (s-t)^{-l}
	\Big( \sum_{j=2}^l  [(\gF_{2:j})_j(s,\cdot)]_{\beta_j} (s-t)^{(j-\frac 12)\beta_j} +[(D_{j-1}\gF_j)_{j-1}(s,\cdot)]_\eta (s-t)^{(1+\eta)(j-\frac 32)} \Big)\notag\\
	\leq\, &
	C \|D_{{2:l}}u\|_{\infty} (s-t)^{-l} 
	\Big(   (s-t)^{\frac 32\beta_2} + (s-t)^{\frac 12(1+\eta)} \Big)\notag,
	\end{split}
	\end{equation}
	up to a modification of $C$ and \textcolor{black}{recalling that $T$ is small}.
	This leads, as soon as $l\ge 2$, to a time singularity which is not integrable. Indeed, $(1+\eta)/2<1 $ (recall that $\eta $ is meant to be small). To overcome this problem, the idea consists in writing, thanks to \textcolor{black}{the cancellation properties of} Lemma \ref{strong_LEMME_SG}, 
	\begin{eqnarray}
	\!\!&\!\!\!\!&\!\!|D_{\x_1}I_{\textcolor{black}{l,3}}^{\textcolor{black}{\bxi}}(s,\x)|_{\textcolor{black}{\bxi=\x}} 
	\notag\\
	\!\!&\!\! =\!\! &\!\! \bigg| \!D_{\x_1} D_{\x_l} \!\Big[\tP_{s,t}^\bxi \Big( \! \sum_{i=2}^{l
	}  \! \big\langle \!(\gF_i(s,\cdot_{1:l-1},\btheta^{l:n}_{s,t}(\bxi))\!-\!\gF_i(s,\btheta_{s,t}(\bxi)) \!-\! D_{i-1} \gF_i(s,\btheta_{s,t}(\bxi))\big(\!\cdot\!-\btheta_{s,t}(\bxi)\big)_{i-1}),\notag\\
	\!\!&\!\!\!\!&\!\! \big(D_{i}u(s,\cdot)-D_{i}u(s,\cdot_{1:l-1},\btheta_{s,t}^{l:n}(\bxi))\big)\big\rangle \Big) \Big](\x) \bigg| \Bigg|_{\bxi=\x}\label{strong_CTR_P1_I3L_1}
	\end{eqnarray}
	and to take advantage of the additional smoothing effect from the solution of the regularized PDE itself through the above contribution
	$D_{i}u(s,\cdot)-D_{i}u(s,\cdot_{1:l-1},\btheta_{s,t}^{l:n}(\bxi))$, $i\in \leftB 2, l\rightB$. This was the strategy implemented in \cite{chau:17} which, unfortunately, cannot be repeated as this in our general framework. 
	Roughly speaking, 
	to control for \textcolor{black}{any $k$ in $\leftB l,n \rightB$, for} fixed $(\y_{1:\textcolor{black}{k}-1},\y_{\textcolor{black}{k}+1:n})\in \R^{(n-1)d}$ \textcolor{black}{ and any $0\le s < t \le T$,} the $\alpha_k^i$-H\"older modulus of the partial application (see \eqref{def:proj}) $\y_{\textcolor{black}{k}}\mapsto (D_{\textcolor{black}{i}} u)_{\textcolor{black}{k}}(s,\y_{\textcolor{black}{k}})$,
	we need to control the $\alpha_k^i$-H\"older modulus of the map $\textcolor{black}{\y_{\textcolor{black}{k}}\mapsto (D_i\tilde G^{\bxi} f)_k}(s,\y_k)$, where $\tilde G^\bxi f(s,\y):=\int_{s}^T dv \tilde P_{v,s}^\bxi f(v,\y) $ is the Green kernel involving the source $f$ in the indicated equation. 
	\textcolor{black}{When doing so, we \textcolor{black}{are} face\textcolor{black}{d} to the fact that we should only consider indexes $k\le i $ in order to apply cancellation arguments.
		The only corresponding index in \eqref{strong_CTR_P1_I3L_1} is thus $ i=k=l$.}
	\textcolor{black}{When $i=l$ and  $k$ lies in $\leftB l+1,n\rightB$, we \textcolor{black}{can handle the corresponding terms} 
		by interpolating the $\alpha_k^l$-H\"older modulus of the partial application $(D_l u)_k$ from the $\alpha_l^l$-H\"older modulus of the partial application $(D_l u)_l $, see Lemma \ref{Lemme_Taylor_reverse} below.
		Otherwise, i.e. when $i<l$, we expand the gradients in \eqref{strong_CTR_P1_I3L_1} with the Taylor formula. \textcolor{black}{Namely, using the Fubini theorem} we write: 
	}
	%
	%
	\begin{eqnarray}
	\!\!&\!\!\!\!&\!\!|D_{\x_1}I_{\textcolor{black}{l,3}}^{\textcolor{black}{\bxi}}(s,\x)|_{\textcolor{black}{\bxi=\x}} \notag\\
	\!\!&\!\! \textcolor{black}{\le } \!\!&\!\!\Bigg| \sum_{i=2}^{\textcolor{black}{l-1}
	}  \sum_{k=l}^n \int_0^1 d\lambda  D_{\x_1} D_{\x_l} \bigg[\tP_{s,t}^\bxi \bigg(  \bigg\langle \Big(\gF_i(s,\cdot_{1:l-1
	},\btheta^{l:n}_{s,t}(\bxi))-\gF_i(s,\btheta_{s,t}(\bxi))
	- D_{i-1}\gF_i(s,\btheta_{s,t}(\bxi))\big(\cdot-\btheta_{s,t}(\bxi)\big)_{i-1}\Big),\notag\\
	\!\!&\!\!\!\!&\!\! D_{k}\textcolor{black}{D_i}u(s,\cdot_{1:l-1},\btheta_{s,t}^{l:n}(\bxi)+\lambda(\cdot_{l:n}-\btheta_{s,t}^{l:n}(\bxi)))(\cdot-\btheta_{s,t}(\bxi))_{k}   \bigg\rangle\bigg)\bigg](\x) \Bigg|_{\textcolor{black}{\bxi=\x}}\notag\\
	\!\!&\!\!\!\!&\!\!\textcolor{black}{+\Bigg|\!D_{\x_1} D_{\x_l} \!\bigg[\tP_{s,t}^\bxi \bigg( \!   \! \bigg\langle \!\Big (\gF_{\textcolor{black}{l}}(s,\cdot_{1:l-1},\btheta^{l:n}_{s,t}(\bxi))\!-\!\gF_{\textcolor{black}{l}}(s,\btheta_{s,t}(\bxi)) \!-\! D_{\textcolor{black}{l}-1} \gF_{\textcolor{black}{l}}(s,\btheta_{s,t}(\bxi))\big(\!\cdot\!-\btheta_{s,t}(\bxi)\big)_{\textcolor{black}{l}-1} \Big ),}
	\notag\\
	\!\!&\!\!\!\!&\!\! \textcolor{black}{\big(D_{\textcolor{black}{l}}u(s,\cdot)-D_{\textcolor{black}{l}}u(s,\cdot_{1:l-1},\btheta_{s,t}^{l:n}(\bxi))\big)\bigg\rangle \bigg) \bigg](\x) \bigg| \Bigg|_{\bxi=\x}}
	\nonumber \\
	\!\!&\!\!=:\!\!&\!\! 
	\textcolor{black}{
		|D_{\x_1}I_{\textcolor{black}{l,31}}^{\textcolor{black}{\bxi}}(s,\x)|_{\textcolor{black}{\bxi=\x}}+|D_{\x_1}I_{\textcolor{black}{l,32}}^{\textcolor{black}{\bxi}}(s,\x)|_{\textcolor{black}{\bxi=\x}}}.
	\label{strong_JUST_BEFORE_DUALITY}
	\end{eqnarray}
	As already underlined, the term $|D_{\x_1}I_{\textcolor{black}{l,32}}^{\textcolor{black}{\bxi}},(s,\x)|_{\textcolor{black}{\bxi=\x}}$ is handled thanks to interpolation type argument. It thus remains to deal with the term $|D_{\x_1}I_{\textcolor{black}{l,31}}^{\textcolor{black}{\bxi}}(s,\x)|_{\textcolor{black}{\bxi=\x}}$. \textcolor{black}{Our main idea consists \textcolor{black}{first} in switching the differential operators acting on the map $u$ therein and then \textcolor{black}{in rebalancing} the differential operator w.r.t. the ``less degenerate'' direction (namely $D_i$) on the remaining terms in the integrand \textcolor{black}{through an integration by parts}.} To do so, we introduce now for all $i\in\leftB 2,l\textcolor{black}{-1}
	\rightB, k\in \leftB l,n\rightB$, $(\y_{1:i-1},\y_{i+1:n})\in \R^{(n-1)d},(t,\x)\in[0,T]\times \R^{nd} $, \textcolor{black}{$s\in (t,T] $} \textcolor{black}{ and $ \textcolor{black}{\y_i} \in \R^ d$, the function:
		\begin{eqnarray}\label{strong_GROSSE_DEF}
		&&
		\Psi_{i,(l,1),k}^{(s,\y_{1:i-1},\y_{i+1:n}),(t,\x)}(\y_i)
		\notag\\
		&:=&
		\bigg[D_{\x_1} D_{\x_l} \tilde p^{\bxi}(t,s,\x,\y)
		\\
		&& \otimes \Big(  (\gF_i(\textcolor{black}{s},\y_{1:l-1},\btheta^{l:n}_{s,t}(\bxi))-\gF_i(\textcolor{black}{s},\btheta_{s,t}(\bxi))
		D_{i-1}\gF_i(\textcolor{black}{s},\btheta_{s,t}(\bxi))\big(\y-\btheta_{s,t}(\bxi)\big)_{i-1})\Big)((\textcolor{black}{\y}-\btheta_{s,t}(\bxi))_{k})^*\bigg]_{\textcolor{black}{\bxi=\x}},\nonumber 
		\end{eqnarray}}
where the subscript $\textcolor{black}{(l,1)}$ in $\Psi_{i,(l,1),k}^{(s,\y_{1:i-1},\y_{i+1:n}),(t,\x)} $ is here to indicate the differentiation w.r.t. $D_{\x_l}D_{\x_1}$ acting on the frozen density. 
Pay attention that the above function is $(\R^d)^{\otimes 4} $-valued. With these notations at hand, \textcolor{black}{we write for $|D_{\x_1}I_{\textcolor{black}{l,31}}^{\textcolor{black}{\bxi}}(s,\x)|_{\textcolor{black}{\bxi=\x}} $}:
	\begin{eqnarray*}
		|D_{\x_1}I_{\textcolor{black}{l,31}}^{\textcolor{black}{\bxi}}(s,\x)|_{\textcolor{black}{\bxi=\x}} 
		\!\!&\!\!\le\!\!&  \!\!
		\sum_{i=2}^{l\textcolor{black}{-1}
		}  \sum_{k=l}^n \Bigg| \int_0^1 d\lambda  \int_{\R^{(n-1)d}}d (\y_{1:i-1},\y_{i+1:n})\\
		\!\!&\!\!\!\!& \!\!
		\int_{\R^d }d\y_i\Bigg\{\big(\Psi_{i,(l,1),k}^{(s,\y_{1:i-1},\y_{i+1:n}),(t,\x)}(\y_i)\big): 
		D_{\y_k}D_{\y_i} u(s,\y_{1:l-1},\btheta_{s,t}^{l:n}(\textcolor{black}{\x})+\lambda(\y_{l:n}-\btheta_{s,t}^{l:n}(\textcolor{black}{\x})))
		\Bigg\} \Bigg|,\notag
	\end{eqnarray*}
	where ``$\ : \ $" stands for the double tensor contraction. We now use the Schwarz theorem to exchange the order of the differentiation \textcolor{black}{operators} acting on the PDE solution and then integrate by parts to obtain
	\begin{eqnarray}
	|D_{\x_1}I_{\textcolor{black}{l,31}}^{\textcolor{black}{\bxi}}(s,\x)|_{\textcolor{black}{\bxi=\x}} 
	\!\!&\!\!\le\!\!&  \!\!
	\sum_{i=2}^{l\textcolor{black}{-1}
	}  \sum_{k=l}^n \Bigg| \int_0^1 d\lambda  \int_{\R^{(n-1)d}}d (\y_{1:i-1},\y_{i+1:n}) \label{strong_SANS_NOM}\\
	\!\!&\!\!\!\!& \!\!
	\int_{\R^d }d\y_i\Bigg\{D_{\y_i}\big(\Psi_{i,(l,1),k}^{(s,\y_{1:i-1},\y_{i+1:n}),(t,\x)}(\y_i)\big): 
	D_{\y_k} u(s,\y_{1:l-1},\btheta_{s,t}^{l:n}(\textcolor{black}{\x})+\lambda(\y_{l:n}-\btheta_{s,t}^{l:n}(\textcolor{black}{\x})))
	\Bigg\} \Bigg|.\notag
	\end{eqnarray}
	\textcolor{black}{
		Let us now explain why such an expression is well designed. From the very definition of the ``$\Psi$'' term \textcolor{black}{in \eqref{strong_GROSSE_DEF}}, one can see that the additional contribution $(\cdot-\btheta_{s,t}(\textcolor{black}{\x}))_{k}$ is, thanks to Lemma  \textcolor{black}{\ref{strong_LEMME_DENS}}, precisely designed to smoothen the time singularity coming from the differentiation \textcolor{black}{w.r.t.} the variables $\x_l$ of the semigroup  (recall that $k\ge l$). Also, we have from Lemma \ref{strong_SMOOTH_OPERATOR_EFFECTS} that the contribution of the transport term (degenerate part of the operator) therein is, up to a \textcolor{black}{multiplicative} constant, bounded by $\sum_{j=i}^{n} (|\cdot-\btheta_{s,t}(\textcolor{black}{\x}))_{j}|^{\beta_j}  + |(\cdot-\btheta_{s,t}(\textcolor{black}{\x}))_{j-1}|^{1+\eta}$. Therefore, it follows from  Lemma \textcolor{black}{\ref{strong_LEMME_DENS}} and \A{T$_{\beta}$} that this term allows to smooth\textcolor{black}{en} the differentiation coming from the Schwarz theorem and the integration by parts \textcolor{black}{(pay attention that for any $i$ in $\leftB 2,l-1\rightB$, the differential operator $D_i$ now acts on the term $\Psi_{i,\cdot}^{\cdot}$)}. The main idea consists now in absorbing the additional singularity (of order $1/2$) coming from the differentiation w.r.t. the variable ``$\x_1$''  thanks to the $\alpha_i^k$-H\"older regularity of $(D_ku)_i$, provided the exponent $\alpha_i^k$ is large enough (see Remark \ref{rem_holdexpo} below). This is done by putting the two above terms ``in duality'' w.r.t. the variable $ \y_i$, \textcolor{black}{within the framework of  Besov spaces}.}
	
	Recall \textcolor{black}{indeed} that $\textcolor{black}{C
		^{\alpha_i^k}(\R^d,\R)=B_{\infty,\infty}^{\alpha_i^k}(\R^d,\R)} $, where from now on the notation $B_{p,q}^s $ stands for a Besov space with associated indexes $p,q,s $ (see Triebel \cite{trie:83} and \textcolor{black}{the \textcolor{black}{reminder} at the beginning of Section \ref{sec:esti:besov:norm}}\textcolor{black}{)}.
	The indexes $p,q$ \textcolor{black}{refer to} the integrability parameters and $s $ \textcolor{black}{to} the smoothness one. A classical fact is that 
	$B_{\infty,\infty}^{\alpha_i^k} $ and $B_{1,1}^{-\alpha_i^k} $ can be put in duality, see e.g. 
	Proposition 3.6 in Lemari\'e-Rieusset (\cite{lemar:02}). \textcolor{black}{Indeed, with the notations therein $B_{\infty,\infty}^{\textcolor{black}{\alpha_i^k}} $ is the dual of $B_{1,1}^{-\alpha_i^k} =\tilde B_{1,1}^{-\alpha_i^k} $ where $\tilde B_{1,1}^{-\alpha_i^k} $ denotes the closure of the Schwartz class ${\mathcal S} $ in $ B_{1,1}^{-\alpha_i^k}$} (see also Theorem 4.1.3 in Adams Hedberg \cite{adam:hedb:96} for the density of ${\mathcal S} $ in $ B_{1,1}^{-\alpha_i^k}$).
	Exploiting this fact, we then derive from \eqref{strong_SANS_NOM} and \textcolor{black}{the multi-linearity of the tensors involved} that:
	\begin{eqnarray}\label{strong_ALMOST_LAST_EQ_CTR_DERIV}
	\!\!&\!\!\!\!&\!\!|\textcolor{black}{D_{\x_1}I_{\textcolor{black}{l,31}}^{\textcolor{black}{\bxi}}(s,\x)}|_{\textcolor{black}{\bxi=\x}}  \notag\\
	\!\!&\!\!\leq \!\!&C \sum_{i=2}^{l \textcolor{black}{-1}
	}  \sum_{k=l}^n  \int_0^1 d\lambda \int_{\R^{(n-1)d}} \!\! d (\y_{1:i-1},\y_{i+1:n})\notag\\
	\!\!&\!\!\!\!&\!\!\times  
	\Big\|
	D_i\Psi_{i,(l,1),k}^{(s,\y_{1:i-1},\y_{i+1:n}),(t,\x)}(\cdot) \! 
	\Big \|_{\textcolor{black}{ B_{1,1}^{-\alpha_i^k}}}  \Big\|  \y_i\mapsto D_{k}u(s,\y_{1:l-1},\btheta_{s,t}^{l:n}(\textcolor{black}{\x})\!+\!\lambda(\y_{l:n}-\btheta_{s,t}^{l:n}(\textcolor{black}{\x}))) \Big \|_{\textcolor{black}{ B_{\infty,\infty}^{\alpha_i^k}}} \! 
	\notag\\
	\!\!&\!\! \leq \!\!&\!\! C \! \sum_{i=2}^{l\textcolor{black}{-1}
	}  \sum_{k=l}^n  \! \int_{\R^{(n-1)d}}\!\!\! d (\y_{1:i-1},\y_{i+1:n}) 
	\bigg\| 
	D_i\Psi_{i,(l,1),k}^{(s,\y_{1:i-1},\y_{i+1:n}),(t,\x)}(\cdot)
	\bigg\|_{\textcolor{black}{ B_{1,1}^{-\alpha_i^k}}}
	\sup_{\z_j, j\in \leftB1,n\rightB,j\neq i} \bigg\|D_{k}u(s,\z_{1:i-1},\cdot,\z_{i+1:n}) \bigg\|_{\textcolor{black}{ B_{\infty,\infty}^{\alpha_i^k}}}
	\!\!. \notag
	\end{eqnarray}

	\textcolor{black}{To handle this term and conclude the proof of the main theorem}, we now need the following results whose proofs are postponed to the next \textcolor{black}{section}:
	
	\begin{lem}\label{strong_LEMME_CTR_BESOV} Let $l \in \leftB 2,n\rightB$, $i \in \leftB \textcolor{black}{2},l\textcolor{black}{-1}
		\rightB$ and $ k\in \leftB l,n\rightB$ and let $\Psi_{i,(l,1),k}^{(s,\y_{1:i-1},\y_{i+1:n}),(t,\x)}: \R^d \to \textcolor{black}{(\R^d)^{\otimes 4}}$ be the function defined by \eqref{strong_GROSSE_DEF}. There exist $C:=C(\A{A},\textcolor{black}{T})>0$, $\textcolor{black}{\bar C:=\bar C(C)} $,  $\alpha_i^k:=(1+\frac{\eta}{4})/(2i-1) 
		$,  $\gamma_i^k 
		:= \frac 12 +\eta (i-\frac 32)
		$ such that
		\begin{eqnarray*}
			\bigg\| \bigg[\textcolor{black}{D_i}\Psi_{i,(l,1),k}^{(s,\y_{1:i-1},\y_{i+1:n}),(t,\x)}(\cdot)\bigg]\bigg\|_{\textcolor{black}{ B_{1,1}^{\textcolor{black}{-\alpha_i^k}}}} \leq \textcolor{black}{\bar C}  (s-t)^{-\frac 32
				+ \gamma_i^k} \hat q_{c\setminus i}(t,s,\x,(\y_{1:i-1},\y_{i+1:n})),
		\end{eqnarray*}
		where \textcolor{black}{$c=C^{-1}$} and with the notations of Proposition \ref{strong_THE_PROP}, \textcolor{black}{exploiting as well equation \eqref{strong_CTR_GRAD_WGPOINTS} of Remark \ref{REM_WITH_FLOWS_LIN_AND_NON}},
		\begin{eqnarray}
		\label{strong_DEF_HAT_SETMINUS}
		\hat q_{c\setminus i}(t,s,\x,(\y_{1:i-1},\y_{i+1:n})):=\int_{\R^d }d\y_i\hat p_c^\x(t,s,\x,\y) =\prod_{j\in \leftB 1,n\rightB, j\neq i}^n {\mathcal N}_{c(s-t)^{2j-1}}\big( (\btheta_{s,t}(\x)-\y)_j\big),
		\end{eqnarray}
		denoting  for $a>0,\ z\in \R^d$, by ${\mathcal N}_{a}(z)=(2\pi a)^{-d/2}\exp\big(-|z|^2/(2a)\big) $ the standard Gaussian density \textcolor{black}{on} $\R^d $ with covariance matrix $a{\mathbf I}_d$. 
	\end{lem}
	
	\begin{lem}\label{strong_LEMME_CTR_HOLDER} Let $u$ be the solution of \eqref{strong_GENERIC_PDE}. There exists $C:= C(\A{A})>0$ 
		such that for all $i \leq k \in \leftB 2,n\rightB^2$ and $\alpha_i^k=(1+\frac{\eta}{4})/(2i-1)
		$,
		\begin{equation}\label{strong_SENSI_IK}
		\sup_{\y_j, j\in \leftB1,n\rightB,j\neq i} \bigg\|\textcolor{black}{D_{k}u}\big (s,\y_{1:i-1},\cdot,\y_{i+1:n}\big) \bigg\|_{\textcolor{black}{ B_{\infty,\infty}^{\alpha_i^k}}} \leq C(\|{\mathbf D} u\|_\infty + 
		\|{\mathbf D} D_{1}u \|_\infty) .
		\end{equation}
	\end{lem}
	\textcolor{black}{
		\begin{REM}[About the H\"older exponent $\alpha_i^k$, $2\le i\le k \le n$.]\label{rem_holdexpo} We emphasize that upper bounds of the H\"older exponents $(\alpha_i^k)_{2\le i\le k \le n}$ can be derived from the analysis of the H\"older modulus of the partial application
			$\textcolor{black}{\y_{\textcolor{black}{i}}\mapsto (D_k\tilde G^{\bxi} f)_i}(s,\y_{\textcolor{black}{i}})$, where $\tilde G^\bxi f(s,\y):=\int_{s}^T dv \tilde P_{v,s}^\bxi f(v,\y) $ is the Green kernel involving the source $f$ in the indicated equation (see Appendix \ref{SEC_PREUVE_HOLD_MODULUS_G_FROZEN} or Section 5.1 in \cite{chau:hono:meno:18} for similar issues).\\ 
			Let us briefly sketch why these upper bounds are large enough to absorb the additional time singularity of order $1/2$ coming from the differentiation w.r.t. the variable $\x_1$. First note that the computations done in Appendix \ref{SEC_PREUVE_HOLD_MODULUS_G_FROZEN} yield for each $(i\le k)$ in $\leftB 2 , n\rightB^2$ the upper bound $[1-(1-\beta_k)(k-1/2)]/[i-1/2]$. Recall then from \A{T${}_\beta $} that $\beta_k\in \big(\textcolor{black}{(}2k-2\textcolor{black}{)}/\textcolor{black}{(}2k-1\textcolor{black}{)},1\big] $. Thus, if $\beta_k=\big([2k-2]/[2k-1]\big)^+:= \textcolor{black}{(}2k-2\textcolor{black}{)}/\textcolor{black}{(}2k-1\textcolor{black}{)}+\varepsilon $ for some $0<\varepsilon<<1$ meant to be very small, we obtain $(1-\beta_k)(k-1/2) = \textcolor{black}{(1/2)^-:=1/2-\varepsilon}$. Therefore, the smoothing effect of the gradient of the solution of the PDE in the direction ``$\x_k$'' w.r.t. the $i^{\rm{th}}$ variable  is designed to smooth\textcolor{black}{en} a time singularity of order $\textcolor{black}{\alpha_i^k(i-1/2)}\le [\textcolor{black}{(1/2)^+}/(i-1/2)] \times (i-1/2) = \textcolor{black}{(1/2)^+}$.
		\end{REM}
	}

	\textcolor{black}{To control the term $ |D_{\x_1}I_{\textcolor{black}{l,32}}^{\textcolor{black}{\bxi}}(s,\x)|_{\textcolor{black}{\bxi=\x}}$ in \eqref{strong_JUST_BEFORE_DUALITY}, we
		will need the following auxiliary lemma whose proof is postponed to Appendix \ref{PROOF_REV_TAYLOR}}.
	
	\begin{lem}[H\"older moduli of the gradients after the differentiation index through reverse Taylor expansion]\label{Lemme_Taylor_reverse}
		There is a constant $C>0$ s.t, for all $l\in \leftB 2,n\rightB, \ k\in \leftB l,n\rightB$,  $(\y,\bxi) \in \R^{nd}\times \R^{nd}$, 
		$s\in [t,T]$ :
		\begin{equation*}
		\big | D_{{l}}u(s, \y)  -D_{{l}} u(s,\y_{1:l-1}, \btheta_{s,t}^{l:n}(\bxi)) \big |
		\leq
		C(\|\mathbf Du\|_\infty+[(D_l u)_l(s,\cdot)]_{\alpha_l^l})\Big(|(\btheta_{s,t}(\bxi)-\y)_l|^{\alpha_l^l} +\sum_{k=l+1}^n |(\btheta_{s,t}(\bxi)-\y)_k|^{\zeta_l}\Big),
		\end{equation*}
		where $\zeta_l:=(1+ \eta/ 4)/(2l+ \eta/ 4) $.
	\end{lem}
	\textcolor{black}{\textcolor{black}{Now}, from Proposition \ref{strong_THE_PROP}, Lemma \ref{strong_SMOOTH_OPERATOR_EFFECTS} and Lemma \ref{Lemme_Taylor_reverse},
		\begin{eqnarray}
		&&|D_{\x_1}I_{\textcolor{black}{l,32}}^{\textcolor{black}{\bxi}}(s,\x)|_{\textcolor{black}{\bxi=\x}}\notag\\
		&\le& \frac{\textcolor{black}{\bar C}[(D_{{l-1}}\gF_l)_{l-1}(s,\cdot))]_\eta}{(s-t)^l} \int_{\R^{nd}} d\y \hat p_{C^{-1}}^\x(t,s,\x,\y) |(\btheta_{s,t}(\x)-\y)_{l-1}|^{1+\eta} \notag\\
		&&\times (\|\mathbf Du\|_\infty+[(D_l u)_l(s,\cdot)]_{\alpha_l^l})
		\big(|(\btheta_{s,t}(\x)-\y)_l|^{\alpha_l^l}+\sum_{k=l+1}^n |(\btheta_{s,t}(\x)-\y)_k|^{\zeta_l}\big)\notag\\
		&\le & C(\|\mathbf Du\|_\infty+[(D_l u)_l(s,\cdot)]_{\alpha_l^l})\big(\sum_{k=l}^n (s-t)^{-l+(1+\eta)(l-\frac 32)+\alpha_l^l (l-\frac 12)\I_{k=l}+\zeta_l(k-\frac 12)\I_{k\in \leftB l+1,n\rightB}}\big) ,\label{PREAL_HOLD_MOD_HD}
		\end{eqnarray}
		using as well Lemma \ref{strong_LEMME_DENS} \textcolor{black}{for the last inequality}.} \textcolor{black}{Hence, the above equation yields:
		\begin{equation}
		\label{strong_ALMOST_LAST_EQ_CTR_DERIV_OTHER_TERM}
		|D_{\x_1}I_{\textcolor{black}{l,32}}^{\textcolor{black}{\bxi}}(s,\x)|_{\textcolor{black}{\bxi=\x}}\le C\big(\|\mathbf D u\|_\infty+\|(D_l u)_l(s,\cdot)\|_{ B_{\infty,\infty}^{\alpha_l^l}} \big)\sum_{k=l}^n (s-t)^{-l+(1+\eta)(l-\frac 32)+\alpha_l^l (l-\frac 12)\I_{k=l}+\zeta_l (k-\frac 12)\I_{k\in \leftB l+1,n\rightB}} .
		\end{equation}
		Observe now that the time exponents in the previous r.h.s. are integrable. In particular, for each $l\in \leftB 2,n\rightB $ 
		\begin{eqnarray}
		-l+(1+\eta)(l-\frac 32)+\alpha_l^l (l-\frac 12)&=&-(l-\frac 32)(1- (1+\eta))-\frac 32+ \alpha_l^l(l-\frac 12)\ge  \eta(l-\frac 32) -\frac 32 +\frac{1+\frac \eta 4}{2l-1}(l-\frac 12)\notag\\
		&\ge & -\frac 32+\eta(l-\frac 32)+\frac 12=-\frac 32+\gamma_l^k,\label{EXPO_HD_1}
		\end{eqnarray}
		with the notations of Lemma \ref{strong_LEMME_CTR_BESOV}}.
	Also, for $k\in \leftB l+1,n\rightB $:
	\begin{eqnarray}
	-l+(1+\eta)(l-\frac 32)+\zeta_l (k-\frac 12)&=&-\frac 32+ \eta(l-\frac 32)+\frac{1+\frac \eta 4}{2l+\frac \eta 4}(k-\frac 12)\ge  -\frac 32+\eta(l-\frac 32)  +\frac 12 \frac{1+\frac \eta 4}{l+\frac \eta 8}(l+\frac 12)\notag \\
	&\ge & -\frac 32+\eta(l-\frac 32)+\frac 12=-\frac 32+\gamma_l^k.\label{EXPO_HD_2}
	\end{eqnarray}
	\textcolor{black}{Observe carefully from the above computations that the exponent $\zeta_l $ precisely allows to recover a smoothing effect in time of order strictly greater than $1/2$. This perfectly fits the smoothing effects observed for the other contributions, see Remark \ref{rem_holdexpo} above}.

	We can then deduce from \textcolor{black}{\eqref{strong_ALMOST_LAST_EQ_CTR_DERIV} and \eqref{strong_ALMOST_LAST_EQ_CTR_DERIV_OTHER_TERM}}, using  Lemma \ref{strong_LEMME_CTR_BESOV} and Lemma \ref{strong_LEMME_CTR_HOLDER} 
	that
	\begin{equation}
	\label{strong_CTR_DX1I_3L}
	|D_{\x_1}I_{\textcolor{black}{l,3}}^{\textcolor{black}{\bxi}}(s,\x)|_{\textcolor{black}{\bxi=\x}}  
	\leq C \sum_{i=2}^{l
	}  \sum_{k=l}^n   (s-t)^{-\frac 32 + \gamma_i^k} \Big(\|{\mathbf D}u\|_\infty +
	\|{\mathbf D} D_{1}u\|_\infty \Big),
	\end{equation}
	which are integrable terms since $\gamma_i^k >  1/2$. With the notations of \eqref{strong_GROS_DECOUP}, \eqref{strong_DECOUP_I}, we eventually derive from \eqref{strong_CTR_DX1I_3L}, \eqref{strong_CTR_DEBUT_I_12l}, \eqref{strong_CTR_DEBUT_I_11l} that there exists $\gamma:=\gamma(\A{A})>0$ such that:
	\begin{equation}\label{strong_LAST_CONTROL_I}
	\bigg|\int_{t}^T ds D_{\x_1}I_l^{\textcolor{black}{\bxi}}(s,\x)\bigg|_{\textcolor{black}{\bxi=\x}}\le C T^\gamma \Big(\|{\mathbf D}u\|_{\infty} +\|{\mathbf D}D_1 u\|_{\infty}\Big).
	\end{equation}

	\begin{proof}[\textcolor{black}{\textbf{Final proof of Theorem \ref{strong_THM_DER_PDE}}}.]
		Bringing together \eqref{strong_LAST_CONTROL_I} and \eqref{strong_LE_CTR_SUR_H_1L} yields for all $l\in \leftB 1,n\rightB $ and  $(t,\x)\in [0,T] $:
		\begin{equation}\label{strong_AF_D1_GRAD}
		|D_{\x_l}D_{\x_1} u(t,\x)|\le C (T^\gamma+T^{\delta} )\Big( \|{\mathbf D} u\|_{\infty}+\|{\mathbf D}D_1 u\|_{\infty}\Big).
		\end{equation}
		It is clear that the previous analysis can be reproduced without differentiating w.r.t. $\x_1 $, leading to improved singularity exponents (see also the proof of Lemma \ref{strong_LEMME_CTR_HOLDER} which somehow exactly explicit these computations). We therefore get:
		\begin{equation}\label{strong_AF_GRAD}
		|D_{\x_l} u(t,\x)|\le C (T^{\gamma'}+T^{\delta'} )\Big( \|{\mathbf D} u\|_{\infty}+\|{\mathbf D}D_1 u\|_{\infty}\Big),
		\end{equation}
		for some positive exponents $\gamma', \delta' $ (with $\gamma'>\gamma,\ \delta '>\delta $).
		
		Taking the time-space supremum in the l.h.s of \eqref{strong_AF_D1_GRAD} and \eqref{strong_AF_GRAD},
		recalling as well that $T$ is meant to be small, i.e. s.t. $4 CT^{\delta \wedge \gamma}\le  1/2 $, we derive:
		\begin{equation*}
		\textcolor{black}{\|{\mathbf D} u\|_\infty}+\textcolor{black}{ \|{\mathbf D}D_1 u\|_\infty} \le 2C (T^\gamma+T^{\delta} ).
		\end{equation*}
		This concludes the proof.  
	\end{proof}
	\mysection{Estimates in Besov norm}\label{sec:esti:besov:norm}
	This section is dedicated to the proofs of the main technical results needed to obtain  Theorem \ref{strong_THM_DER_PDE}. Namely, we prove the Besov estimates of Lemmas \ref{strong_LEMME_CTR_BESOV} and \ref{strong_LEMME_CTR_HOLDER}. We first start by recalling some definitions/characterizations on Besov spaces from Section 2.6.4 of Triebel \cite{trie:83}. For $\alpha \in \R, q\in (0,+\infty] ,p \in (0,\infty] $, $B_{p,q}^\alpha(\R^d):=\{f\in {\mathcal S}'(\R^d): \|f\|_{{\mathcal H}_{p,q}^\alpha}<+\infty \} $ where ${\mathcal S}(\R^d) $ stands for the Schwartz class and 
	\begin{equation}
	\label{strong_THERMIC_CAR_DEF}
	\|f\|_{{\mathcal H}_{p,q}^\alpha}:=\|\varphi(D) f\|_{L^p(\R^d)}+ \Big(\int_0^1 \frac {dv}{v} v^{(m-\frac \alpha 2)q}    \|\partial_v^m h_{v}\star f\|_{L^p(\R^d)}^q \Big)^{\frac 1q},
	\end{equation}
	with $\varphi \in C_0^\infty(\R^d)$ (smooth function with compact support)  s.t. $\varphi(0)\neq 0 $, $\varphi(D)f := (\varphi \hat f)^{\vee} $ where $\hat f$ and $(\varphi \hat f)^\vee $ respectively denote the Fourier transform of $f$ and the inverse  Fourier transform of $\varphi \hat f $. The parameter $m$ is an integer s.t. $m> \alpha/ 2 $ and for $v>0$, $z\in \R^d $, $h_{v}(z):=\frac{1}{(2\pi v)^{ d/2}}\exp\left(-\frac{|z|^2}{2v} \right)$ is the usual heat kernel of $\R^d$. We point out that the quantities in \eqref{strong_THERMIC_CAR_DEF} are well defined for $q<\infty $. The modifications for $q=+\infty $ are obvious and can be written passing to the limit.

	Observe that the quantity $\|f\|_{{\mathcal H}_{p,q}^s} $, where the subscript ${\mathcal H} $ stands to indicate the dependence on the heat-kernel, depends on the considered function $\varphi $ and the chosen $m\in \N$. It also defines a quasi-norm on $B_{p,q}^s(\R^d) $. The previous definition of $B_{p,q}^\alpha(\R^d) $ is known as the thermic characterization of Besov spaces and is particularly well adapted to our current framework.
	By abuse of notation we will write as soon as this quantity is finite $\|f\|_{{\mathcal H}_{p,q}^\alpha}=:\|f\|_{B_{p,q}^\alpha} $.

	\subsection{Proof of Lemma \ref{strong_LEMME_CTR_BESOV}} 
	We will here exploit the thermic characterization of Besov spaces (see Chapter 2.6.4 in \cite{trie:83}) which is also recalled above.
	From \eqref{strong_THERMIC_CAR_DEF},
	we are thus led to estimate, for any $l \in \leftB 2,n\rightB$, $i \in \leftB 2,l-1\rightB$ and $ k\in \leftB l,n\rightB$:
	\textcolor{black}{
		\begin{eqnarray*}
			\label{strong_BUT_BESOV}
			\|\varphi (D) \textcolor{black}{D_i}\Psi_{i,(l,1),k}^{(s,\y_{1:i-1},\y_{i+1:n}),(t,\x)}(\cdot)\|_{L^1(\R^d,\R)} + \int_0^{\textcolor{black}{\textcolor{black}{1}}} \frac{dv}{v}v^{1+\textcolor{black}{\frac{\alpha_i^k}2 }}\|\partial_ v h_{v}\star \textcolor{black}{D_i}\Psi_{i,(l,1),k}^{(s,\y_{1:i-1},\y_{i+1:n}),(t,\x)}(\cdot)\|_{L^1(\R^d,\R)},
		\end{eqnarray*}
		for a $C^\infty $ compactly supported function $\varphi $ \textcolor{black}{s.t. $\varphi(0)\neq 0 $}. \textcolor{black}{\textcolor{black}{From the definition of $\Psi_{i,(l,1),k}^{(s,\y_{1:i-1},\y_{i+1:n}),(t,\x)} $ in \eqref{strong_GROSSE_DEF}}, using  Lemmas \ref{strong_LEMME_DENS} and \ref{strong_SMOOTH_OPERATOR_EFFECTS}, we easily deduce
			\begin{eqnarray*}
				&&\|\varphi (D) \textcolor{black}{D_i}\Psi_{i,(l,1),k}^{(s,\y_{1:i-1},\y_{i+1:n}),(t,\x)}(\cdot)\|_{L^1(\R^d,\R)}\\
				&\le& \int_{\R^d } dz\Big|\int_{\R^d}d\y_i D_{\y_i}\widehat \varphi (z-\y_i) \cdot \Psi_{i,(l,1),k}^{(s,\y_{1:i-1},\y_{i+1:n}),(t,\x)}(\y_i) \Big|  \\
				&\le&\textcolor{black}{\bar C} (s-t)^{-1/2}\Big( (s-t)^{\frac 32\beta_2} + (s-t)^{\frac 12(1+\eta)} \Big)\hat q_{c\setminus i}(t,s,\x,(\y_{1:i-1},\y_{i+1:n})).
			\end{eqnarray*}
		}
		Let us now focus on the second term in the above definition. We split therein the time integral into two parts writing:}
	\begin{eqnarray}
	&&\int_0^{(s-t)^{\rho_{i,k}}} dvv^{\frac{ \textcolor{black}{\alpha_i^k}}2 }\|\partial_v h_{v}\star\textcolor{black}{D_i}\Psi_{i,(l,1),k}^{(s,\y_{1:i-1},\y_{i+1:n}),(t,\x)}(\cdot)\|_{L^1(\R^d,\R)}\notag\\
	&& 
	+\int_{(s-t)^{\rho_{i,k}}}^{\textcolor{black}{\textcolor{black}{1}}} dvv^{\frac{\textcolor{black}{\alpha_i^k}}2 }\|\partial_v h_{v}\star  \textcolor{black}{D_i}\Psi_{i,(l,1),k}^{(s,\y_{1:i-1},\y_{i+1:n}),(t,\x)}(\cdot)\|_{L^1(\R^d,\R)}\label{strong_DECOUP_HK_INT}\\
	&&=:{\mathbf {Lower}} \Big[\textcolor{black}{D_i}\Psi_{i,(l,1),k}^{(s,\y_{1:i-1},\y_{i+1:n}),(t,\x)}\Big]+{\mathbf {Upper}} \Big[\textcolor{black}{D_i}\Psi_{i,(l,1),k}^{(s,\y_{1:i-1},\y_{i+1:n}),(t,\x)}\Big],\notag 
	\end{eqnarray}
	for a parameter $\rho_{i,k}>0 $ to be specified. \textcolor{black}{The term ${\mathbf {Upper}} $ corresponding to the upper-part  of the integral w.r.t. $v$ does not involve singularities. We will use this fact to calibrate the associated parameter $\rho_{i,k} $ in order to match the integrability constraint
		\begin{eqnarray}
		\label{strong_RES_CALIBRATION_BETA_I}
		{\mathbf {Upper}}\Big[\textcolor{black}{D_i}\Psi_{i,(l,1),k}^{(s,\y_{1:i-1},\y_{i+1:n}),(t,\x)}\Big]\le \frac{\textcolor{black}{\bar C}}{(s-t)^{
				1+\frac 12 -{\gamma_i^k}
		}}\hat q_{c\setminus i}(t,s,\x,(\y_{1:i-1},\y_{i+1:n})),
		\end{eqnarray}
		where $\hat q_{c\setminus i}$ has been defined in \eqref{strong_DEF_HAT_SETMINUS} and $\gamma_i^k>1/2 $ in order to obtain a time integrable singularity.
		For this term, we will only use crude upper-bounds on the derivatives of the heat-kernel and the coefficients satisfying \A{T${}_\beta $}. On the other hand, the contribution ${\mathbf {Lower}} $ in \eqref{strong_DECOUP_HK_INT} precisely contains the singularities w.r.t. $v$. It is therefore crucial to use there suitable cancellation tools. The point will then be to prove that the associated estimates are compatible with the upper-bound in equation \eqref{strong_RES_CALIBRATION_BETA_I}}.

	We now write:
	\begin{eqnarray*}
		&&{\mathbf {Upper}}\Big[\textcolor{black}{D_i}\Psi_{i,(l,1),k}^{(s,\y_{1:i-1},\y_{i+1:n}),(t,\x)}\Big]\\
		&=&\int_{(s-t)^{\rho_{i,k}}}^{\textcolor{black}{\textcolor{black}{1}}} dvv^{\frac{\textcolor{black}{\alpha_i^k}}2 }\|\partial_vh_{v}\star  \textcolor{black}{D_i}\Psi_{i,(l,1),k}^{(s,\y_{1:i-1},\y_{i+1:n}),(t,\x)}(\cdot)\|_{L^1(\R^d,\R)}\\
		&=&\int_{(s-t)^{\rho_{i,k}}}^{\textcolor{black}{\textcolor{black}{1}}}dvv^{\frac{\textcolor{black}{\alpha_i^k}}2 }\int_{\R^d}dz\Big |\int_{\R^d}d\y_i \textcolor{black}{D_i}\partial_vh_{v}(z-\y_i) \Psi_{i,(l,1),k}^{(s,\y_{1:i-1},\y_{i+1:n}),(t,\x)}(\y_i)  \Big|.
	\end{eqnarray*}
	\textcolor{black}{Recall from the definition in \eqref{strong_GROSSE_DEF} that $\Psi_{i,(l,1),k}^{(s,\y_{1:i-1},\y_{i+1:n}),(t,\x)}(\y_i)$ is $(\R^{d})^{\otimes 4} $-valued. To proceed with the computations we assume w.l.o.g. for the rest of the proof  that $d=1$ to avoid tensor notations for simplicity.}
	Writing explicitly the function $\Psi_{i,(l,1),k}^{(s,\y_{1:i-1},\y_{i+1:n}),(t,\x)}(\y_i)$  leads to:
	\begin{eqnarray*}
		&&{\mathbf {Upper}}\Big[\textcolor{black}{D_i}\Psi_{i,(l,1),k}^{(s,\y_{1:i-1},\y_{i+1:n}),(t,\x)}\Big]\\
		&\le& \int_{(s-t)^{\rho_{i,k}}}^{\textcolor{black}{\textcolor{black}{1}}} dv v^{\frac{\textcolor{black}{\alpha_i^k}}2 }  \int_{\R^d} dz 
		\Bigg|\int_{\R^d}d\y_i\textcolor{black}{D_i}\partial_v h_{v}(z-\y_i)\Big(D_{\x_1}D_{\x_l}\tilde p^{\bxi}(t,s,\x,\y)  \Big[\left((\btheta_{s,t}(\bxi) - \y)_{k}\right) \\
		&&
		\Big(\gF_{i}(s,\y_{1:l-1},\btheta_{s,t}^{l:n}(\bxi))-\gF_{i}(s,\btheta_{s,t}(\bxi))
		-D_{i-1}\gF_{i}(s,\btheta_{s,t}(\bxi))\big(\y-\btheta_{s,t}(\bxi)\big)_{i-1}\Big]\Big)\bigg| \Bigg|_{\bxi=\x}.
	\end{eqnarray*}
	%
	From Lemma \ref{strong_SMOOTH_OPERATOR_EFFECTS} and \textcolor{black}{Proposition \ref{strong_THE_PROP}}, we derive \textcolor{black}{that} there \textcolor{black}{exist} $C:= C(\A{A},T)>0,\ \textcolor{black}{\bar C:=\bar C(C)}$ such that \textcolor{black}{introducing $\hat q_c(t,s,\x,\y)=\hat p_c^\x(t,s,\x,\y)$, $c=C^{-1} $}:
	\begin{eqnarray*}
		&&{\mathbf {Upper}}\Big[\textcolor{black}{D_i}\Psi_{i,(l,1),k}^{(s,\y_{1:i-1},\y_{i+1:n}),(t,\x)}\Big]\\
		&\le& \textcolor{black}{\bar C}\int_{(s-t)^{\rho_{i,k}}}^{\textcolor{black}{\textcolor{black}{1}}} dvv^{\frac{\textcolor{black}{\alpha_i^k}}2 } \int_{\R^d} dz 
		\int_{\R^d} d\y_i \frac{h_{cv}(z-\y_i)}{v^{\textcolor{black}{\frac 32 }}}\frac{\hat q_c(t,s,\x,\y)}{(s-t)^{(l-\frac 12)+\frac 12}} |(\btheta_{s,t}(\x)-\y)_k|\\
		&& \times \Bigg\{  
		\sum_{j=i}^{l-1}\Big\{ |(\btheta_{s,t}(\x) - \y)_j|^{\beta_j}\Big\} + |(\btheta_{s,t}(\x) - \y)_{i-1}|^{1+\eta}\Bigg\}\\
		&\underset{\textcolor{black}{k\ge l}}{\le}& \textcolor{black}{\bar C}\int_{(s-t)^{\rho_{i,k}}}^{\textcolor{black}{\textcolor{black}{1}}} dvv^{\frac{\textcolor{black}{\alpha_i^k}}2 } \int_{\R^d} dz 
		\int_{\R^d} d\y_i \frac{h_{cv}(z-\y_i)}{v^{\textcolor{black}{\frac 32}}}\frac{\hat q_c(t,s,\x,\y)}{(s-t)^{\frac 12}}
		\left\{ \sum_{j=i}^{l-1}(s-t)^{\beta_j(j-\frac 12)} + (s-t)^{(1+\eta)(i-\frac 32)}\right\}\\
		&\underset{\textcolor{black}{{\rm Fubini}}}{\le}& \textcolor{black}{\bar C}\hat q_{c\setminus i}(t,s,\x,(\y_{1:i-1},\y_{i+1:n}))\int_{(s-t)^{\rho_{i,k}}}^{\textcolor{black}{\textcolor{black}{1}}} dvv^{-\frac 32+\frac{\alpha_i^k} 2 } (s-t)^{-\frac12}
		\left\{ \sum_{j=i}^{l-1}(s-t)^{\beta_j(j-\frac 12)} + (s-t)^{(1+\eta)(i-\frac 32)}\right\}\\
		&\le& \textcolor{black}{\bar C}\hat q_{c\setminus i}(t,s,\x,(\y_{1:i-1},\y_{i+1:n})) (s-t)^{[-\frac 12+\frac{\alpha_i^k} 2]\rho_{i,k}-\frac 12 }\big((s-t)^{\beta_i(i-\frac 12)}+(s-t)^{(1+\eta)(i-\frac 32)}\big),
	\end{eqnarray*}
	recalling that \textcolor{black}{the lower bound of $\beta_j(j-1/2) $ is increasing for the last inequality (recall indeed that we assumed that $\beta_j \in\left( \frac{2j-2}{2j-1},1 \right]$)}, \textcolor{black}{up to modification of $C,\bar C$ in the previous inequalities}. We now want to \textcolor{black}{choose the threshold $\rho_{i,k} $ in order to} match the integrability condition   in \eqref{strong_RES_CALIBRATION_BETA_I}. This amounts to write: 
	$$-1- \frac 12 + \gamma_i^k=[-\frac 12+\frac{\alpha_i^k} 2]\rho_{i,k}-\frac 12\textcolor{black}{+}\Big( \beta_i(i-\frac 12)\wedge (1+\eta)(i-\frac 32)\Big).$$
	\textcolor{black}{Since this condition should hold for any $\beta_i \in \left(\frac{2i-2}{2i-1},1\right]$ and since the parameter $\eta \in (0,1)$ is small (see \A{${\mathbf H}_{\eta} $}) we have  $\beta_i(i-1/2)\wedge (1+\eta)(i-3/2) = (1+\eta)(i-3/2)$}. The above condition rewrites:
	$$-1-\frac 12+\gamma_i^k=[-\frac 12+\frac{\alpha_i^k} 2]\rho_{i,k} -\frac 12\textcolor{black}{+}(1+\eta)(i-\frac 32).$$
	Our global integrability constraint associated with \textcolor{black}{the} $i^{\rm th} $ variable in the $k^{\rm th} $ derivative writes:
	\begin{equation}
	\label{strong_EXP_THERMIC_HD_V}
	-1-\frac 12+\gamma_i^k=[-\frac 12+\frac{\alpha_i^k} 2]\rho_{i,k}-\frac 12\textcolor{black}{+} (1+\eta)(i-\frac32)
	>-1,
	\end{equation}
	which gives
	\begin{eqnarray}
	\label{strong_RHO_I_K}
	\rho_{i,k}<\frac{(1+\eta)(2i-3)+1}{1-\alpha_i^k}.
	\end{eqnarray}
	
	It therefore remains to check that such a choice is compatible with the time integral part \textcolor{black}{for} $v\in [0,(s-t)^{\rho_{i,k}}] $
	in the thermic characterisation of the Besov norm, \textcolor{black}{see \eqref{strong_DECOUP_HK_INT}}. We point out that for this term it is absolutely essential to get rid of the exponent $v^{-\textcolor{black}{3/2}}$ coming from the upper-bound of the thermic heat-kernel, i.e. $\textcolor{black}{\textcolor{black}{D_i}\partial_{v}}h_{v}(z-\y_i) $. In order to get an integrable singularity in $v$, we need to decrease the crude upper-bound on $\textcolor{black}{D_i}\partial_{v}h_{v}(z-\y_i) $. This is done through cancellation techniques exploiting the smoothness properties of $\Psi_{i,(l,1),k}^{(s,\y_{1:i-1},\y_{i+1:n}),(t,\x)} $.

	To investigate  $ {\mathbf {Lower}}\Big[\textcolor{black}{D_i}\Psi_{i,(l,1),k}^{(s,\y_{1:i-1},\y_{i+1:n}),(t,\x)}\Big]$ let us first recall from  the definition in \eqref{strong_GROSSE_DEF} that for each $k\in \leftB l,n\rightB $:
	\textcolor{black}{\begin{eqnarray}
		\label{strong_SPECIAL_PSI_I_L}
		&&\Psi_{\textcolor{black}{i},(l,1),k}^{(s,\y_{1:\textcolor{black}{i}-1},\y_{\textcolor{black}{i}+1:n}),(t,\x)}(\y_{\textcolor{black}{i}})
		\nonumber \\
		&=&\textcolor{black}{\bigg\{} D_{\x_l} D_{\x_1} \tilde p^{\bxi}(t,s,\x,\y) 
		\\
		&&\textcolor{black}{\times} \Big[   \Big(\gF_{\textcolor{black}{i}}(s,\y_{1:l-1},\btheta^{l:n}_{s,t}(\bxi))- \gF_i(s,\btheta_{s,t}(\bxi))
		- D_{i-1}\gF_i(s,\btheta_{s,t}(\bxi))\big(\y-\btheta_{s,t}(\bxi)\big)_{i-1}\Big)\Big] (\y-\btheta_{s,t}(\bxi))_{k}^*\bigg\}_{\textcolor{black}{\bxi=\x}}.\nonumber
		\end{eqnarray}}
	\textcolor{black}{Let us now specify the dependence w.r.t. $\y_i $ of the previous expression in function of the considered indexes $ l\in \leftB 2, n\rightB, i\in \leftB \textcolor{black}{2,l-1}\rightB, k\in \leftB l,n\rightB $. This will be useful to develop corresponding adapted cancellation arguments.}
	
	\textcolor{black}{Observe first that the  dependence in $\y_i$  in \eqref{strong_SPECIAL_PSI_I_L} appears for \textcolor{black}{all} $ l\in \leftB 2, n\rightB, i\in \leftB \textcolor{black}{2,l-1}\rightB, k\in \leftB l,n\rightB $ through the term
		$\textcolor{black}{D_{\x_l}D_{\x_1}} \tilde p^\bxi(t,s,\x,\y) $.}
	
	\textcolor{black}{For the term into brackets, since $i\le l-1$, then $k>\textcolor{black}{i}$ and the only bracket term containing $\y_{\textcolor{black}{i}}$ 
		is the one associated with $\gF_i(\textcolor{black}{\y_{1:l-1},\btheta^{l:n}_{s,t}(\bxi)}) $}.

	\textcolor{black}{
		\\
		With} the notations of \eqref{strong_DECOUP_HK_INT}, we write:
	\begin{eqnarray}
	\label{strong_DECOUP_AVANT_CANCELLATION}
	&&{\mathbf {Lower}}\Big[\textcolor{black}{D_i}\Psi_{i,(l,1),k}^{(s,\y_{1:i-1},\y_{i+1:n}),(t,\x)}\Big]\notag\\
	&=&\int_0^{(s-t)^{\rho_{i,k}}} dv v^{\frac{\textcolor{black}{\alpha_i^k}}2} \int_{\R^d} dz|\int_{\R^d}  d\y_i \textcolor{black}{D_i}\partial_v h_{cv}(z-\y_i)
	(\Psi_{i,(l,1),k}^{(s,\y_{1:i-1},\y_{i+1:n}),(t,\x)}(\y_i)-\Psi_{i,(l,1),k}^{(s,\y_{1:i-1},\y_{i+1:n}),(t,\x)}(z))|\notag\\
	&=:&\int_0^{(s-t)^{\rho_{i,k}}} \!\!\!dv v^{\frac{\textcolor{black}{\alpha_i^k}}2} \int_{\R^d} dz |({\mathscr T}_{1,i,(l,1),\textcolor{black}{k}}^{(s,\y_{1:i-1},\y_{i+1:n}),(t,\x)}+{\mathscr T}_{2,i,(l,1),\textcolor{black}{k}}^{(s,\y_{1:i-1},\y_{i+1:n}),(t,\x)}) \big(v,z\big)|,
	\end{eqnarray}
	where: 
	\begin{eqnarray}
	&&
	{\mathscr T}_{1,i,(l,1),\textcolor{black}{k}}^{(s,\y_{1:i-1},\y_{i+1:n}),(t,\x)} \big(v,z\big)\label{strong_DECOUP_CAR_THERMIC_T1}
	\\
	&:=&\int_{\R^d} d\y_i \textcolor{black}{D_i}\partial_v h_{v}(z-\y_i)\Big(D_{\x_1}D_{\x_l} \tilde p^{\bxi}(t,s,\x,\y) 
	\notag\\
	&&\textcolor{black}{ \times} 
	\Big[  
	\Big(\gF_{i}(s,\y_{1:l-1},\btheta_{s,t}^{l:n}(\bxi))-\gF_{i}(s,\y_{1:i-1},z,\y_{i+1:l-1},\btheta_{s,t}^{l:n}(\bxi))\Big) \textcolor{black}{ (\y-\btheta_{s,t}(\bxi))_k^*}
	\notag
	\Big]\textcolor{black}{\Big)_{\bxi=\x}} ,
	\end{eqnarray}
	with a slight abuse of notation when $i=l-1$ and 
	\begin{eqnarray}
	\label{strong_DECOUP_CAR_THERMIC_T2}
	\!\!&\!\!\!\!&\!\!{\mathscr T}_{2,i,(l,1),k}^{(s,\y_{1:i-1},\y_{i+1:n}),(t,\x)}\big(v,z\big)\\
	\!\!&\!\!:=\!\!&\!\!\int_{\R^d} d\y_i \textcolor{black}{D_i}\partial_v h_{v}(z-\y_i)\bigg\{\Big[D_{\x_1}D_{\x_l} \tilde p^{\bxi}(t,s,\x,\y)-D_{\x_1}D_{\x_l}\tilde p^{\bxi}(t,s,\x,\y_{1:i-1},z,\y_{i+1:n})\Big] 
	\notag \\
	\!\!&\!\!\!\!&\!\! \!\! \textcolor{black}{ \times}
	\Big[
	\Big(\!\gF_{i}(s,\y_{1:i-1},z,\y_{i+1:l-1},\btheta_{s,t}^{l:n}(\bxi))-\gF_{i}(s,\btheta_{s,t}(\bxi)) -D_{i-1}\gF_{i}(s,\btheta_{s,t}(\bxi))(\y-\btheta_{s,t}(\bxi))_{i-1}\!\Big)\textcolor{black}{ (\y-\btheta_{s,t}(\bxi))_k^*}
	\Big]\bigg\}_{\textcolor{black}{\!\!\bxi=\x}}\!\! \!\!.\notag
	\end{eqnarray}
	Write now from \eqref{strong_DECOUP_CAR_THERMIC_T2}, \textcolor{black}{Proposition \ref{strong_THE_PROP} and Lemma \ref{strong_SMOOTH_OPERATOR_EFFECTS}}:
	\begin{eqnarray*}
		\!\!&\!\!\!\!&\!\!|{\mathscr T}_{2,i,(l,1),k}^{(s,\y_{1:i-1},\y_{i+1:n}),(t,\x)}\big(v,z\big)|\\
		\!\!&\!\!\le \!\!&\!\!\textcolor{black}{\bar C}\int_{\R^d} d\y_i \frac{h_{cv}(z-\y_i)}{v^{\textcolor{black}{\frac 32}}}\int_0^1 d\lambda  \frac{\hat q_c(t,s,\x,\y_{1:i-1},z+\lambda (\y_i-z), \y_{i+1:n})}{(s-t)^{(i- \frac{1}{2})+(l- \frac{1}{2}) + \frac 12}}|(\y-\btheta_{s,t}(\x))_k|  |\y_i-z| \\
		\!\!&\!\!\!\!&\!\!\times\Big (\Big|\gF_{i}(s,\y_{1:i-1},z,\y_{i+1:l-1},\btheta_{s,t}^{l:n}(\x))-\gF_{i}(s,\btheta_{s,t}(\x)) -D_{i-1}\gF_{i}(s,\btheta_{s,t}(\x))(\y-\btheta_{s,t}(\x))_{i-1} \Big|\Big)\\
		\!\!&\!\!\le\!\! &\!\!\textcolor{black}{\bar C}\int_{\R^d} d\y_i \frac{h_{cv}(z-\y_i)}{\textcolor{black}{v}}\int_0^1 d\lambda  \frac{\hat q_c(t,s,\x,\y_{1:i-1},z+\lambda (\y_i-z), \y_{i+1:n})}{(s-t)^{(i- \frac{1}{2})+\frac 12
		}} 
		\\
		\!\!&\!\!\!\!&\!\!\times
		\left\{ |z-(\btheta_{s,t}(\x))_i|^{\beta_i}+\sum_{j=i+1}^{l-1}(s-t)^{\beta_j(j-\frac 12)} + (s-t)^{(1+\eta)(i-\frac 32)}\right\},
	\end{eqnarray*}
	\textcolor{black}{where for the second inequality, we used that for $k\ge l>i,\ |(\y-\btheta_{s,t}(\x))_k|(s-t)^{-(l-1/2)}\le |(\y-\btheta_{s,t}(\x))_k|(s-t)^{-(k-1/2)}  $ which can be absorbed by the $k^{\rm th} $ variables of $\hat q_c $}.
	
	Writing now for any  $\lambda \in [0,1]  $,
	$$|z-\btheta_{s,t}(\x)_i|\le \lambda  |z-\y_i|+|z +\lambda (\y_i-z)-(\btheta_{s,t}(\x))_i|,$$
	we thus derive
	\begin{eqnarray}
	\!\!&\!\!\!\!&\!\!|{\mathscr T}_{2,i,(l,1),k}^{(s,\y_{1:i-1},\y_{i+1:n}),(t,\x)}\big(v,z\big)|\notag \\
	\!\!&\!\!\le \!\!&\!\!  \textcolor{black}{\bar C}\int_{\R^d} d\y_i \frac{h_{cv}(z-\y_i )}{\textcolor{black}{v}}\int_0^1 d\lambda  \hat q_c(t,s,\x,\y_{1:i-1},z+\lambda (\y_i-z), \y_{i+1:n})
	\notag\\
	\!\!&\!\!\!\!&\!\! \times 
	\Big(v^{\frac{\beta_i}2}(s-t)^{-(i -\frac 12) - \frac 12}\!+\!(s-t)^{-(i-\frac 12)- \frac 12\!+\!\beta_i(i-\frac 12) }
	\!+\!\!\!\sum_{j=i+1}^{l-1}\!\! (s-t)^{-(i-\frac 12)- \frac 12\!+\!\beta_j(j-\frac 12) } + (s-t)^{-(i-\frac 12)- \frac 12\!+\!(1+\eta)(i-\frac 32)}\Big)
	\notag\\
	\!\!&\!\!\le\!\!&\!\!  \textcolor{black}{\bar C}\hat q_{c\setminus i}(t,s,\x,\y_{1:i-1}, \y_{i+1:n})\int_0^1 d\lambda \int_{\R^d}d\y_i h_{cv}(z-\y_i) {\mathcal N}_{c(s-t)^{2i-1}}(z+\lambda (\y_i-z)-(\btheta_{s,t}(\x))_i) \notag\\
	\!\!&\!\!\!\!&\!\!\times v^{\textcolor{black}{-1}} \Big(v^{\frac{\beta_i}2}(s-t)^{-(i -\frac 12)- \frac 12}+
	(s-t)^{-(i-\frac 12)- \frac 12+(1+\eta)(i-\frac 32)}\Big),
	\label{strong_CTR_T2_SU}
	\end{eqnarray}
		recalling for the last inequality that for any $j$ in $\leftB i,l-1\rightB$, $\beta_j (j- 1/2)>j-1>i- 3/2 $ and $\eta$ is supposed to be a small 
	parameter. 
	
	\textcolor{black}{Let us now consider the first contribution in the r.h.s. in \eqref{strong_SPECIAL_PSI_I_L}},
	\begin{eqnarray}
	|{\mathscr T}_{1,i,(l,1),k}^{(s,\y_{1:i-1},\y_{i+1:n}),(t,\x)}\big(v,z\big)|\!\!&\!\!\le \!\!&\!\!  \textcolor{black}{\bar C} \int_{\R^d}d\y_i \frac{h_{cv}(z-\y_i)}{v^{\textcolor{black}{\frac 32}}} \frac{\hat q_c(t,s,\x,\y)}{(s-t)^{(l-\frac 12) + \frac 12}}|z-\y_i|^{\beta_i} |(\y-\btheta_{s,t}(\x))_k| \notag \\
	\!\!&\!\!\le \!\!&\!\!  \textcolor{black}{\bar C}\int_{\R^d}d\y_{i} \frac{h_{cv}(z-\y_i)}{v^{\textcolor{black}{\frac 32}-\frac{\beta_i}{2}}} \frac{\hat q_c(t,s,\x,\y)}{(s-t)^{\frac 12}} .\label{strong_PREAL_T1_SU}
	\end{eqnarray}

	From \eqref{strong_CTR_T2_SU} and \eqref{strong_PREAL_T1_SU} we derive, with the notation introduced in \eqref{strong_DEF_HAT_SETMINUS}:
	\begin{eqnarray*}
		&&\|\partial_v h_{v}\star\textcolor{black}{D_i} \Psi_{i,(l,1),k}^{(s,\y_{1:i-1},\y_{i+1:n}),(t,\x)}\|_{L^1(\R^d,\R)}\\
		& \le &\textcolor{black}{\bar C}\hat q_{c\setminus i}(t,s,\x,(\y_{1:i-1},\y_{i+1:n}))\Bigg\{\frac{1}{v^{\textcolor{black}{\frac32}-\frac{\beta_i}{2}}(s-t)^{\frac 12}} +\Bigg(\frac{v^{\textcolor{black}{-1+}\frac{\beta_i}{2}}}{(s-t)^{i }}    +\frac{v^{\textcolor{black}{-1}}}{(s-t)^{i-(1+\eta)(i-\frac 32)}} \Bigg)\\
		&& \times \int_0^1 d\lambda \int_{\R^d} dz \int_{\R^d} d\y_i h_{cv}(z-\y_i) {\mathcal N}_{c(s-t)^{2i-1}}(z+\lambda (\y_i-z)-(\btheta_{s,t}(\x))_i)\Bigg\}\\
		&\le &\textcolor{black}{\bar C}\hat q_{c\setminus i}(t,s,\x,(\y_{1:i-1},\y_{i+1:n}))\Bigg(\frac{1}{v^{\textcolor{black}{\frac 32}-\frac{\beta_i}{2}}(s-t)^{\frac 12}} +\frac{v^{\textcolor{black}{-1+}\frac{\beta_i}{2}}}{(s-t)^{i }}    +\frac{v^{\textcolor{black}{-1}}}{(s-t)^{i-(1+\eta)(i-\frac 32)}}\Bigg),
	\end{eqnarray*}
	using the change of variable $(w_1, w_2)=(z-\y_i,z+\lambda (\y_i-z)-(\btheta_{s,t}(\x))_i) $ for the last inequality.

	For $\rho_{i,k} $ chosen as in \eqref{strong_RHO_I_K} one gets \textcolor{black}{from the definition in \eqref{strong_DECOUP_HK_INT}}:
	\begin{eqnarray*}
		&&{\mathbf {Lower}} \Big[\textcolor{black}{D_i}\Psi_{i,(l,1),k}^{(s,\y_{1:i-1},\y_{i+1:n}),(t,\x)}\Big]\\
		\!&\!\!\le \!\!&\! \textcolor{black}{\bar C}\hat q_{c\setminus i}(t,s,\x,(\y_{1:i-1},\y_{i+1:n}))\int_0^{(s-t)^{\rho_{i,k}}} \frac{dv}{v}v^{\frac{ \alpha_{i}^k}{2}}
		\Bigg(\frac{1}{v^{\frac{1-\beta_i}{2}}(s-t)^{\frac 12}} +\frac{v^{\frac{\beta_i}{2}}}{(s-t)^{i }}    +\frac{1}{(s-t)^{i-(1+\eta)(i-\frac 32)}}\Bigg)\\
		\!&\!\!=: \!\!&\!\textcolor{black}{\bar C}\hat q_{c\setminus i}(t,s,\x,(\y_{1:i-1},\y_{i+1:n}))  {\mathcal B}_{1,\rho_{i,k}}(t,s).
	\end{eqnarray*}
	It therefore remains to prove that, if 
	${\mathcal B}_{1,\rho_{i,k}}(t,s)\le (s-t)^{-1- 1/2 + \gamma_i^k} $ for $ \gamma_i^k>0 $ then:
	\begin{equation}
	\label{strong_REG_SUFF_TO_CHECK_LB}
	-1-\frac 12+ \gamma_i^k >-1 \textcolor{black}{\iff  \gamma_i^k>\frac 12}.
	\end{equation}
	Write:
	\begin{eqnarray}
	{\mathcal B}_{1,\rho_{i,k}}(t,s)&\le& \int_0^{(s-t)^{\rho_{i,k}}} dv \Bigg(\frac{v^{-\frac 32+\frac{\alpha_i^k+\beta_i}2}}{(s-t)^{\frac 12}}+ \frac{v^{-1+\frac{\alpha_i^k+\beta_i}{2}}}{(s-t)^{i}}+ \frac{v^{-1+\frac{\alpha_i^k}{2}}}{(s-t)^{\frac 32- \eta(i-\frac 32)}}\Bigg)\notag\\
	&\le & C \Big (   (s-t)^{\rho_{i,k}\frac{\alpha_i^k+\beta_i-1}{2}-\frac 12}+ (s-t)^{\rho_{i,k}\frac{\alpha_i^k+\beta_i}{2}-i} + (s-t)^{\rho_{i,k}\frac{\alpha_i^k}2-\frac 32+ \eta(i-\frac 32)}\Big),
	\label{strong_BD_BIK}
	\end{eqnarray}
	\textcolor{black}{assuming for a while that $\alpha_i^k $ is s.t. $\alpha_i^k+\beta_i>1 $ to guarantee the integrability of the first above integrand (this condition will be fulfilled for the choice 
		below).}
	Let us check condition \eqref{strong_REG_SUFF_TO_CHECK_LB} is satisfied. Actually, the first two terms \textcolor{black}{in \eqref{strong_BD_BIK}} yield negligible contributions. Recall indeed from \textcolor{black}{Remark \ref{rem_holdexpo}} that the parameter $\alpha_i^k $ must be chosen so that $\alpha_i^k<\frac{1-(1-\beta_k)(k-1/2)}{i-\frac 12}:=\bar \alpha_i^k$. Since $\beta_k\in (\frac{2k-2}{2k-1},1] $, $\bar \alpha_i^k=\frac{2- (1-\beta_k)(2k-1)}{2i-1}> \frac{3-2k +(2k-2)}{2i-1}=\textcolor{black}{\frac{1}{2i-1}}$. Thus, 
	\begin{equation}\label{strong_def_alpha}
	\alpha_i^k=\frac{1+\frac {\eta}4}{2i-1},
	\end{equation} 
	is an admissible choice (\textcolor{black}{recall indeed that we chose $\eta<\inf_{j\in \leftB 2,n\rightB}\{\beta_j-\frac{2j-2}{2j-1} \} $ in assumption \A{H${}_\eta $}}). 
	It gives in particular that $\alpha_i^k+\beta_i>\frac{1+\eta/4+(2i-2)}{2i-1}=1+\frac{\eta/4}{2i-1} $. Hence, 
	\begin{eqnarray*}
		&&\rho_{i,k}\frac{\alpha_i^k+\beta_i-1}{2}-\frac 12>-\frac 12
	\end{eqnarray*}
	which already provides a regularizing term in time for the first term in the r.h.s. of \eqref{strong_BD_BIK}. Now from \eqref{strong_RHO_I_K}, $\rho_{i,k}<\frac{(1+\eta)(2i-3)+1}{1-\alpha_i^k}=:\bar \rho_{i,k} $. The previous choice for $\alpha_i^k $ gives 
	$$\bar \rho_{i,k}  =\frac{2i-2+\eta(2i-3)}{\frac{2i-1-(1+\frac{\eta}{4})}{2i-1}}=(2i-1)\frac{2i-2+\eta(2i-3)}{2i-2-\frac{\eta}{4}}>(2i-1) ,$$
	and 
	\begin{equation}
	\rho_{i,k}=2i-1 \label{DEF_RHO_I_K}  
	\end{equation}
	is an admissible choice. \textcolor{black}{We therefore get for the exponent of the second term in \eqref{strong_BD_BIK}}:
	\begin{eqnarray*}
		&&\rho_{i,k}\frac{\alpha_i^k+\beta_i}{2}-i>\frac{1}{2}\Big(\rho_{i,k}(1+\frac{\frac{\eta}{4}}{2i-1})  -2i \Big)>-\frac 12.
	\end{eqnarray*}
	
	Eventually, \textcolor{black}{for the  exponent of the third contribution in \eqref{strong_BD_BIK}}, for the previous choice of $\rho_{i,k}=2i-1 $, we get
	\begin{eqnarray}\label{SPOT_32}
	\rho_{i,k}\frac{\alpha_i^k}2-\frac 32+ \eta(i-\frac 32)=\frac 12 (2i-1)\frac{1+\frac{\eta}{4}}{2i-1}-\frac 32+\eta(i-\frac 32)=-1+\frac \eta 8+\eta (i-\frac 32)>-1,
	\end{eqnarray}
	which means that criterion \eqref{strong_REG_SUFF_TO_CHECK_LB} is indeed satisfied, even though if this last contribution is rather critical in order to obtain the required smoothing effect with $\gamma_i^k=1/2+ 
	\eta(i-3/2) $. This concludes the proof of Lemma \ref{strong_LEMME_CTR_BESOV}.
	\hfill $\textcolor{black}{\square} $

	\subsection{Proof of Lemma \ref{strong_LEMME_CTR_HOLDER}} 
	\label{strong_SECTION_AVEC_PSI_SIMPLIFIE}
	%

	We now tackle the Besov estimate of the derivative of the solution of \eqref{strong_GEN_PDE}. 
	\textcolor{black}{Fix $t\in [0,T] $ and  $(\x_{1:i-1},\x_{i+1:n})\in \R^{(n-1)d} $}.
	From the thermic characterization  of Besov spaces \textcolor{black}{recalled in equation \eqref{strong_THERMIC_CAR_DEF}}, we actually have
	to control:
	\begin{eqnarray}
	\|D_{k}u(t,\x_{1,i-1},\cdot,\x_{i+1,n})\|_{\textcolor{black}{\textcolor{black}{B_{\infty,\infty}^{\alpha_i^k}}}}&=&
	\|D_{k}u(t,\x_{1,i-1},\cdot,\x_{i+1,n})\|_{\textcolor{black}{\textcolor{black}{C^{\alpha_i^k}}}}\notag\\
	&:=&\|D_{k}u(t,\x_{1:i-1},\cdot,\x_{i+1:n})\|_\infty+[D_k u(t,\x_{1,i-1},\cdot,\x_{i+1,n})]_{\alpha_i^k}.\label{strong_DEV_NORM}
	\end{eqnarray}
	Observe first that $\|D_{k}u(t,\x_{1:i-1},\cdot,\x_{i+1:n})\|_\infty\le \|{\mathbf D}u\|_\infty $. Hence, we can focus on the H\"older modulus in \eqref{strong_DEV_NORM}. For  given $\x_i,\ z\in \R^d $ we want to control the difference:
	\begin{eqnarray}\label{strong_NABLA_XK}
	D_{\x_k}u(t,\x_{1:i-1},\x_i,\x_{i+1:n})-D_{\x_k}u(t,\x_{1:i-1},z,\x_{i+1:n})
	\end{eqnarray}
	through the expansion of the gradients given by  \eqref{strong_GROS_DECOUP}. Two cases then arise: the system is globally in the \textit{off-diagonal} regime, i.e. the spatial points $\x_i $ and $z$ are \textit{far} w.r.t. their corresponding time scale (there exists $c_0$ such that $c_0|\x_i-z| \ge (T-t)^{i- 1/2} $ or equivalently $c_0|\x_i-z|^{2/(2i-1)}\ge (T-t)$); the system is globally in the \emph{diagonal} regime, i.e. the spatial points $\x_i $ and $z$ are \textit{close} w.r.t. their corresponding time scale ($c_0|\x_i-z|^{2/(2i-1)}< (T-t)$). 
	
	Since in the global \textit{off-diagonal} regime the spatial points are \textit{far}, it is  not expectable to control suitably the expansion of the gradients around their difference.
	In this case, it is in fact more natural to expand each gradient term thanks to \eqref{strong_GROS_DECOUP} taking as freezing point the associated spatial argument, i.e. $\bxi=\x$ for the first gradient and denoting by $\x'=(\x_{1:i-1},z,\x_{i+1:n}),\ \bxi'=\x' $ for the second one. This allows to take advantage of the underlying smoothing properties in time of the gradient (cf. Section \ref{SUBSUB_OFF_DIAG} below\footnote{note that in this case we have that for any $s$ in $(t,T]$ the off-diagonal regime holds.}).
	
	On the other hand, in the \emph{global diagonal} regime (when $c_0|\x_i-z|^{2/(2i-1)}  \le (T-t) $), we are again faced with a regime dichotomy. Note indeed that, expanding the gradients in \eqref{strong_NABLA_XK} from \eqref{strong_GROS_DECOUP}, we have to deal with a time integral associated with the source and perturbative terms. We are hence again led to separate, within this time integral, the \emph{local} off-diagonal and diagonal regimes, i.e. w.r.t. the time integration variable $s$. We hence introduce  the time set $\mathcal{S}_i = \{s \in [t,T]: (s-t) \leq c_0|\x_i-z|^{2/(2i-1)}\}$ (for the same previous constant $c_0$) which  corresponds to the \emph{local} off-diagonal regime  and the complementary set $\mathcal{S}_i^c = \{s \in [t,T]: (s-t) > c_0|\x_i-z|^{2/(2i-1)}\} $ which corresponds to the \emph{local} diagonal one. 
	
	As above, in the \emph{local off-diagonal} regime, we will not expand the difference of the gradients and we will only use their underlying smoothing properties in time, working thus with their expansion around two different freezing points associated with the corresponding spatial arguments, as suggested by \eqref{strong_GROS_DECOUP} (see again Section \ref{SUBSUB_OFF_DIAG}).
	
	Concerning the \emph{local diagonal} regime, the proximity of the spatial points suggests to expand the gradients through a Taylor expansion. Starting again from their corresponding representation of \eqref{strong_GROS_DECOUP}, it is natural  to consider similar spatial freezing points. Such a strategy indeed yields to only consider spatial sensitivities of the underlying Gaussian \textit{proxy} (see  Section \ref{SUBSUB_DIAG}). Observe that keeping the two distinct freezing points would  lead to investigate the full sensitivity between two different \textit{proxys}, including the sensitivity of the corresponding covariance matrix and generator. Such an investigation appears to be quite involved. Furthermore, we did not succeed to make it work. 
	
	With our approach, we are led to expand one of the gradients in \eqref{strong_NABLA_XK} around two different freezing points. 
	Such a strategy was already used in the companion paper \cite{chau:hono:meno:18} and leads to consider an additional boundary term arising precisely  from the change of freezing point (see Section \ref{SUBSUB_BOUNDARY}). Namely, we will expand the term $D_{\x_k}u(t,\x_{1:i-1},\x_i,\x_{i+1:n})$ with \eqref{strong_GROS_DECOUP} taking $\bxi=\x $, whereas we will expand differently the contribution $D_{\x_k}u(t,\x_{1:i-1},z,\x_{i+1:n}) $, depending on the considered (local) regime (off-diagonal or diagonal) for the current running time. With the previous notations, this term will be expanded as in \eqref{strong_GROS_DECOUP} around the freezing point $\bxi'=\x' $ in the local off-diagonal regime and around $\tilde \bxi'=\x$ in the local diagonal one. Denoting by $t_0=t+c_0 |\x_i-z|^{2/(2i-1)}  
	$ the transition time between the two regimes, we actually have: 
	\begin{eqnarray}
	u(t,\x')
	&=& \int_{t}^Tds\bigg( \I_{{\mathcal S}_i}
	\Big[ \tilde P_{s,t}^{ \bxi'}f(s,\cdot)\Big](\x') + \I_{{\mathcal S}_i^c}
	\Big[ \tilde P_{s,t}^{\tilde \bxi'}f(s,\cdot)\Big](\x')\bigg)
	+\Big[\tilde P_{t_0,t}^{\bxi'} u(t_0, \cdot)\Big](\x')-\Big[\tilde P_{t_0,t}^{\tilde \bxi'}u(t_0,\cdot)\Big] (\x')\notag\\
	&&+\int_{t}^T ds\bigg(\I_{{\mathcal S}_i} 
	\Big[ \tilde P_{s,t}^{\bxi'}\Big((L_s-\tilde L_s^{ \bxi'})u\Big)(s,\cdot)\Big](\x') +
	\I_{{\mathcal S}_i^c} 
	\Big[ \tilde P_{s,t}^{\tilde\bxi'}\Big((L_s-\tilde L_s^{\tilde \bxi'})u\Big)(s,\cdot)\Big](\x')\bigg).
	\label{strong_INTEGRATED_DIFF_BXI_UNICITE_FORTE}
	\end{eqnarray}
	We refer to Appendix \ref{THE_APP_PARAM} below for a proof of expansion \eqref{strong_INTEGRATED_DIFF_BXI_UNICITE_FORTE} (see also Section 2.4 of  the \textcolor{black}{``Detailed  guide to the proof''}  in  \cite{chau:hono:meno:18}). We again emphasize that in comparison with \eqref{strong_GROS_DECOUP}, the term in the second line in the r.h.s. of the above equation is the price to pay to consider different freezing points associated with the corresponding local off-diagonal and diagonal regimes.\\
	
	From now on, we will assume w.l.o.g. that $c_0$ is a constant meant to be \textit{small} (see Section \ref{SUBSUB_BOUNDARY}, Lemma \ref{strong_lem_mean_covariance} and its proof). We also suppose that $|\x'-\x|\le 1 $ since otherwise the \textit{global off-diagonal} regime holds, and the analysis of Section \ref{SUBSUB_OFF_DIAG} applies.\\
	
	Starting from \eqref{strong_NABLA_XK}, we expand $\textcolor{black}{D_{\x_k}u(t,\x)}=D_{\x_k}u(t,\x_{1:i-1},\x_i,\x_{i+1:n}) $ using \eqref{strong_GROS_DECOUP} with $\bxi=\x $ and $D_{\x_k}u(t,\x_{1:i-1},z,\x_{i+1:n}) $  
	differentiating 
	\eqref{strong_INTEGRATED_DIFF_BXI_UNICITE_FORTE}
	w.r.t. $\x_k $ and setting then $ \bxi'=\x', \tilde \bxi'=\x $. We rewrite: 
	\begin{eqnarray}\label{strong_DECOUP_TERMES}
	D_{\x_k}u(t,\x)-D_{\x_k}u(t,\x')&=:&\big( D_{\x_k}u(t,\x)-D_{\x_k}u(t,\x') \big)|_{{\mathcal S}_i}
	+\big(D_{\x_k}u(t,\x)-D_{\x_k}u(t,\x') \big)|_{{\mathcal S}_i^c} \notag\\
	&&+\big(D_{\x_k}u(t,\x)-D_{\x_k}u(t,\x') \big)|_{\partial {\mathcal S}_i}, 
	\end{eqnarray}
	where, \textcolor{black}{with the notation of  \eqref{strong_GROS_DECOUP}:}
	\begin{equation}
	\label{strong_DEF_MOD_HOLD_SUR_SI}
	\big( D_{\x_k}u(t,\x)-D_{\x_k}u(t,\x') \big)|_{{\mathcal S}_i}
	:=\int_{{\mathcal S}_i} ds   \Big\{ \left[H_k^\bxi(s,\x)-H_k^{\bxi'}(s,\x')\right]+ \left[I_k^\bxi(s,\x)-I_k^{\bxi'}(s,\x')\right]\Big\}_{\textcolor{black}{(\bxi,\bxi')=(\x,\x')}},
	\end{equation}
	corresponds to the difference of the previous expansions on the off-diagonal regime,
	\begin{equation}
	\label{strong_DEF_MOD_HOLD_SUR_SI_C}
	\big( D_{\x_k}u(t,\x)-D_{\x_k}u(t,\x') \big)|_{{\mathcal S}_i^c}
	:=\int_{{\mathcal S}_i^c} ds   \Big\{ \left[H_k^\bxi(s,\x)-H_k^{\tilde \bxi'}(s,\x')\right]+ \left[I_k^\bxi(s,\x)-I_k^{\tilde \bxi'}(s,\x')\right]\Big\}_{\textcolor{black}{(\bxi,\tilde \bxi')=(\x,\x)}},
	\end{equation}
	is the contribution of the diagonal regime and
	\begin{equation}
	\label{strong_DEF_MOD_HOLD_SUR_SI_BORD}
	\big(D_{\x_k}u(t,\x)-D_{\x_k}u(t,\x') \big)|_{\partial{\mathcal S}_i}
	:=   \Big\{D_{\x_k} \tilde P_{t_0,t}^{\bxi'} u(t_0, \x')-D_{\x_k}\tilde P_{t_0,t}^{\tilde \bxi'} u(t_0, \x') \Big\}_{\textcolor{black}{(\bxi',\tilde \bxi')=(\x',\x)}},
	\end{equation}
	is the resulting boundary term.
	This last term, arising from the change of freezing point, is particularly delicate to analyze. 
	\subsubsection{Off-diagonal estimates: control of \eqref{strong_DEF_MOD_HOLD_SUR_SI}.} 
	\label{SUBSUB_OFF_DIAG}
	
	On the time set ${\mathcal S}_i  $, we cannot expect some regularization from the difference of the transition densities so that we bluntly estimate the terms appearing  in \eqref{strong_DEF_MOD_HOLD_SUR_SI}, writing:
	\begin{equation}
	| D_{\x_k}u(t,\x)-D_{\x_k}u(t,\x')|\Big|_{{\mathcal S}_i}
	\leq 
	\int_{{\mathcal S}_i} ds \bigg( \big| H_k^\bxi(s,\x)  \big| \Big|_{\bxi=\x}
	\!\!+\big| H_k^{\bxi'}(s,\x')  \big| \Big|_{\bxi'=\x'}
	\!\!+\big| I_k^\bxi(s,\x)\big|\Big|_{\bxi=\x}
	\!\!+\big| I_k^{\bxi'}(s,\x')\big|\Big|_{\bxi'=\x'}\bigg).
	\label{strong_TERME_SI}
	\end{equation}
	Those terms can then be handled  following the previous analysis performed in Theorem \ref{strong_THM_DER_PDE} and Lemma \ref{strong_LEMME_CTR_BESOV},  observing here that, w.r.t. the previous proofs, the above terms are not differentiated w.r.t. $\x_1 $. This improves the exponents of  the time singularities of $1/2$. Similarly to \eqref{strong_CTR_DX1_HL_1}, \eqref{strong_CTR_DX1_HL_2}, \eqref{strong_CTR_DX1_HL_31} and \eqref{strong_CTR_DX1_HL_32} this therefore yields for the terms $H_k^{\bxi}(s,\x) $:  
	\begin{equation}
	\label{strong_BOUNDS_FOR_H_OFF_DIAG}
	\Big| H_k^\bxi(s,\x)  \Big| \bigg|_{\bxi=\x}
	\le C(s-t)^{-1+\delta_i^k} \big(\textcolor{black}{\|\mathbf D u\|_\infty}+ \|{\mathbf D}D_1 u\|_\infty
	\big),
	\end{equation}
	with $\textcolor{black}{1>}\delta_i^k\textcolor{black}{\ge  1/2+\eta/2}>1/2 $.
	
	Reproducing the arguments that led to equations \eqref{strong_CTR_DEBUT_I_11l}, \eqref{strong_CTR_DEBUT_I_12l}, \textcolor{black}{\eqref{strong_ALMOST_LAST_EQ_CTR_DERIV}-\eqref{strong_ALMOST_LAST_EQ_CTR_DERIV_OTHER_TERM}} and the statement of Lemma \ref{strong_LEMME_CTR_BESOV}, exploiting again that there is now no differentiation w.r.t. $\x_1 $ we get with the notations of \textcolor{black}{\eqref{strong_HOLDER_MODULUS}}:
	\begin{equation}
	\label{strong_BOUNDS_FOR_I_OFF_DIAG}
	\Big| I_k^\bxi(s,\x)  \Big| \bigg|_{\bxi=\x}
	\le 
	C(s-t)^{-1+\gamma_i^k}\big(\textcolor{black}{\|{\mathbf D}u\|_\infty}+ \sup_{\bar s\in [0,T]}[\textcolor{black}{(D_{k}u)_i}(\bar s,\cdot)]_{\textcolor{black}{\textcolor{black}{{\alpha_i^k}}}}).
	\end{equation}
	Similar bounds hold for $\Big| H_k^{\bxi'}(s,\x')  \Big| \bigg|_{\bxi'=\x'}$ and $\Big| I_k^{\bxi'}(s,\x')  \Big| \bigg|_{\bxi'=\x'}$.
	Hence, from \eqref{strong_BOUNDS_FOR_H_OFF_DIAG} and \eqref{strong_BOUNDS_FOR_I_OFF_DIAG},
	\begin{eqnarray}
	\!\!&\!\!\!\!&\!\!\int_{{\mathcal S}_i} ds \Bigg( \Big| H_k^\bxi(s,\x)  \Big| \bigg|_{\bxi=\x}
	+\Big|H_k^{\bxi'}(s,\x')\Big| \bigg|_{\bxi'=\x'}
	+ \Big| I_k^\bxi(s,\x)\Big|\bigg|_{\bxi=\x}
	+\Big| I_k^{\bxi'}(s,\x')\Big| \bigg|_{\bxi'=\x'}
	\Bigg)
	\notag \\
	\!\!&\!\!\le\!\!&\!\! C\Big(\int_{t}^{(t+c_0|\x_i-z|^{\frac{2}{2i-1}})\wedge T}
	\!\!\!\!\!\!
	ds\Big(\textcolor{black}{(s-t)^{-1+\delta_i^k}}\big( \|{\mathbf D}u\|_\infty \!+\!
	\|{\mathbf D}D_1 u\|_\infty 
	\big)
	\!+\!(s-t)^{-1+\gamma_i^k}(
	\|{\mathbf D}u\|_\infty \!+\!\!\sup_{\bar s\in [0,T]}[\textcolor{black}{\big(\textcolor{black}{D_{k}u}\big)_i}(\bar s,\cdot)]_{\textcolor{black}{{\alpha_i^k}}})\Big)\notag \\
	\!\!&\!\!\le\!\!&\!\! C\Big(|\x_i-z|^{\frac{2\delta_i^k}{2i-1}}\!\!\wedge (T-t)^{\delta_i^k}(\|{\mathbf D}u\|_\infty \!+\!
	\|{\mathbf D}D_1 u\|_\infty
	)
	+|\x_i-z|^{\frac{2\gamma_i^k}{2i-1}}\wedge (T-t)^{\gamma_i^k}(\textcolor{black}{\|{\mathbf D}u\|_\infty}+\!\!\sup_{\bar s\in [0,T]}[
	\big(D_{k}u\big)_i(\bar s,\cdot)]_{\textcolor{black}{{\alpha_i^k}}})\Big)\notag\\
	\!\!&\!\!\le\!\!&\!\! C|\x_i-z|^{\alpha_i^k}T^{\delta'} \Big(\textcolor{black}{\|{\mathbf D}u\|_\infty}+ \|{\mathbf D}D_1 u\|_\infty 
	+\sup_{\bar s\in [0,T]}[\textcolor{black}{\big(D_{k}u\big)_i}(\bar s,\cdot)]_{\textcolor{black}{{\alpha_i^k}}} \Big),\label{strong_ESTIM_MOD_HOLDER_HD}
	\end{eqnarray}
	for some $\delta':=\delta'(\A{A})>0 $, recalling for the last inequality that $\gamma_i^k=1/2+\eta(i-3/2) $ so that $2\gamma_i^k/(2i-1)=[1+2\eta(i-3/2)]/(2i-1)>(1+ \eta/ 4)/(2i-1)=\alpha_i^k $ (see also the statements of Lemmas \ref{strong_LEMME_CTR_BESOV} and \ref{strong_LEMME_CTR_HOLDER}) and similarly for the contributions involving $ \delta_{i}^k>1/2$. 
	We eventually get from \eqref{strong_ESTIM_MOD_HOLDER_HD} and \eqref{strong_TERME_SI}:
	\begin{eqnarray}
	|\textcolor{black}{D_{\x_k}u}(t,\x)-D_{\x_k}u(t,\x')|\Big|_{{\mathcal S}_i}
	\!\!&\!\!\le\!\!&\!\! CT^{\delta'}\Big(\textcolor{black}{\|{\mathbf D}u\|_\infty}+
	\|{\mathbf D}D_1 u\|_\infty
	+\sup_{\bar s\in [0,T]}[\textcolor{black}{\big(D_{k}u\big)_i}(\bar s,\cdot)]_{\textcolor{black}{{\alpha_i^k}}} \Big).\label{strong_CTR_PRELIM_OFF_DIAG_SI}
	\end{eqnarray}
	\subsubsection{Diagonal estimates: control of the term \eqref{strong_DEF_MOD_HOLD_SUR_SI_C}.} 
	\label{SUBSUB_DIAG}
	\textcolor{black}{We consider here the difference $D_{\x_k}u(t,\x)-D_{\x_k}u(t,\x') $ of \eqref{strong_NABLA_XK} on ${\mathcal S}_i^c $}. In that case the points $\x_i $ and $\x_i' $ are close w.r.t. the characteristic time scale of the $i^{\rm th} $ variable and the main idea consists in controlling the difference between the frozen densities at $\bxi=\tilde \bxi' = \x $ with starting points $\x$ and $\x' $ respectively. 
	Precisely, recalling that $\x$ and $\x' $ only differ in the $i^{\rm th} $ component,  we can write:
	\begin{equation}\label{strong_DVP_TAYLOR_DIFF_TILDEP}
	D_{\x_k} \tilde p^{\bxi} (t,s,\x,\y)-D_{\x_k} \tilde p^{\bxi} (t,s,\x',\y)=-\int_{0}^1 d\lambda D_{\x_i}\Big(D_{\x_k} \tilde p^\bxi(t,s,\x+\lambda (\x'-\x) ,\y)\Big) \cdot (\x'-\x)_i.
	\end{equation}
	From Lemma \ref{strong_LEMME_SG} 
	we thus derive:
	\begin{equation}\label{strong_DVP_TAYLOR_DIFF_TILDEP_2}
	|D_{\x_k} \tilde p^{\bxi} (t,s,\x,\y)-D_{\x_k} \tilde p^{\bxi} (t,s,\x',\y)|\le \frac{C |(\x'-\x)_i|}{(s-t)^{(i-\frac 12)+(k-\frac 12)}}\int_{0}^1  d\lambda  \hat p_{C^{-1}}^\bxi(t,s,\x+\lambda (\x'-\x) ,\y).
	\end{equation}
	Now, from the definition of $\hat p_{C^{-1}}$ in Proposition \ref{strong_THE_PROP}, recalling as well from \eqref{strong_AFFINE_FLOW} that $\x \mapsto  \m_{s,t}^{\bxi}(\x):=\m_{s,t}^{(t,\bxi)}(\x)$ is affine, we get:
	\begin{eqnarray*}
		&&| \hat p_{C^{-1}}^{\bxi}(t,s,\x+\lambda (\x'-\x),\y)|\notag\\
		&\le& \frac{C}{(s-t)^{\frac{n^2d}2}}\exp( -c(s-t)|\T_{s-t}^{-1}(\m_{s,t}^{\bxi} (\x+\lambda(\x'-\x))-\y)|^2)\notag\\
		&\le& \frac{C}{(s-t)^{\frac{n^2d}2}}\exp(c(s-t)|\T_{s-t}^{-1} \textcolor{black}{\gR^{(t,\bxi)}(s,t)}(\x-\x')|^2 )\exp( -\frac{c}2(s-t)|\T_{s-t}^{-1}(\m_{s,t}^{\bxi} (\x)-\y)|^2).
	\end{eqnarray*}
	
	Using the rescaling arguments of the proof of Proposition \ref{strong_THE_PROP} on the resolvent \textcolor{black}{(see equation \eqref{strong_RESCALED_RES})}, we then get
	$(s-t)^{ 1/2}|\T_{s-t}^{-1} \textcolor{black}{\gR^{(t,\bxi)}(s,t)}(\x-\x')| \le C (s-t)^{ 1/2}|\T_{s-t}^{-1}(\x-\x')|=C(s-t)^{-i+1/2} |(\x'-\x)_i|\le C$, from the very definition of ${\mathcal S}_i^c $.
	Hence, 
	\begin{eqnarray*}
		| \hat p_{C^{-1}}^{\bxi}(t,s,\x+\lambda (\x'-\x),\y)|\le \frac{C}{(s-t)^{\frac{n^2d}2}}\exp( -\frac{c}2(s-t)|\T_{s-t}^{-1}(\m_{s,t}^\bxi (\x)-\y)|^2),
	\end{eqnarray*}
	so that, \textcolor{black}{from \eqref{strong_DVP_TAYLOR_DIFF_TILDEP_2} and recalling that on ${\mathcal S }_i^c,\ |(\x'-\x)_i|/(s-t)\le C(|(\x'-\x)_i|/(s-t))^{\alpha_i^k}$},  the following important control holds:
	\begin{equation}
	|D_{\x_k} \tilde p^{\bxi} (t,s,\x,\y)-D_{\x_k} \tilde p^{\bxi} (t,s,\x',\y)|\le \frac{C |(\x'-\x)_i|^{\alpha_i^k}}{(s-t)^{\alpha_i^k(i-\frac 12)+(k-\frac 12)}} \hat p_{C^{-1}}^\bxi(t,s,\x,\y).
	\label{strong_CTR_RECUP_SING_IK_DIAG}
	\end{equation}
	
	Write now from \eqref{strong_DEF_MOD_HOLD_SUR_SI_C}, recalling that $\tilde \bxi'=\bxi $,
	\begin{equation}
	\big |\textcolor{black}{D_{\x_k}u}(t,\x)-D_{\x_k}u(t,\x') \big |\Big|_{{\mathcal S}_i^c}
	:= \Bigg| \Big(\int_{{\mathcal S}_i^c} ds \left[H_k^\bxi(s,\x)-H_k^{\bxi}(s,\x')\right]+\int_{{\mathcal S}_i^c} ds \left[I_k^\bxi(s,\x)-I_k^{\bxi}(s,\x')\right]\Big)_{\textcolor{black}{\bxi=\x}}
	\Bigg|,
	\label{strong_TERME_SI_C}
	\end{equation}
	and let us discuss how the terms $ H_{k}^\bxi(s,\x)-H_k^\bxi(s,\x'), \ I_{k}^\bxi(s,\x)-I_k^\bxi(s,\x')$ in the above equation can be handled.
	%

	We first focus on the term $ I_{k}^\bxi(s,\x)-I_{k}^\bxi(s,\x')$ in \eqref{strong_TERME_SI_C}. This contribution, associated with the degenerate components of perturbed operator, is again the most delicate to handle. From the definitions in $\eqref{strong_DECOUP_I}$ we are led to control the sum $\sum_{\ell=1}^3 [I_{\textcolor{black}{k,\ell}}^\bxi(s,\x)-I_{\textcolor{black}{k,\ell}}^\bxi(s,\x')]$. 
	For the terms $ I_{\textcolor{black}{k,1}}^\bxi(s,\x)-I_{\textcolor{black}{k,1}}^\bxi(s,\x'),\ I_{\textcolor{black}{k,2}}^\bxi(s,\x)-I_{\textcolor{black}{k,2}}^\bxi(s,\x')$  we are going to reproduce the analysis leading to \eqref{strong_CTR_DEBUT_I_11l}, \eqref{strong_CTR_DEBUT_I_12l}.
	Observe first that the above terms do not involve $D_{\x_1} $, therefore we gain a singularity of order $1/2 $ w.r.t. the indicated equations \eqref{strong_CTR_DEBUT_I_11l}, \eqref{strong_CTR_DEBUT_I_12l}. On the other hand, the difference of the derivatives of the frozen densities w.r.t. $\x_k$ can be handled with \eqref{strong_CTR_RECUP_SING_IK_DIAG}. 
	This leads to: 
	\begin{eqnarray*}
		\!\!&\!\!\!\!&\!\!\Bigg|\sum_{\ell=1}^2\int_{t+\textcolor{black}{c_0}|(\x'-\x)_i|^{\frac {2}{2i-1}}}^T   ds \Big(I_{\textcolor{black}{k,\ell}}^\bxi(s,\x)-I_{\textcolor{black}{k,\ell}}^\bxi(s,\x')\Big)\Bigg|_{\textcolor{black}{\bxi=\x}}\notag \\
		\!\!&\!\!\le\!\!&\!\! C|(\x'-\x)_i|^{\alpha_i^k}\|{\mathbf D} u\|_\infty \sum_{j=k}^n \int_{t}^{T} ds(s-t)^{-(k-\frac 12)- \alpha_i^k (i-\frac 12)}\Big((s-t)^{\beta_j(j-\frac 12)}+(s-t)^{(1+\eta)(j-\frac 12)} \Big),
	\end{eqnarray*}
	\textcolor{black}{changing the summation variables from \eqref{strong_CTR_DEBUT_I_11l} for notational simplicity}.
	
	From the very definition of $\alpha_i^k=(1+\eta/4)/(2i-1) $ in Lemma \ref{strong_LEMME_CTR_BESOV}
	and the specific choice of $\eta $ in assumption \A{A} \textcolor{black}{(see \A{${\mathbf H}_\eta $}), \textcolor{black}{which yields that $\beta_j(j-1/2)>j-1+\eta(j-1/2) $}),} we derive 
	\begin{eqnarray}\label{strong_DIFF_I_12_IK}
	\Bigg|\sum_{\ell=1}^2\int_{t+\textcolor{black}{c_0}|(\x'-\x)_i|^{\frac {2}{2i-1}}}^T   ds \Big(I_{\textcolor{black}{k,\ell}}^\bxi(s,\x)-I_{\textcolor{black}{k,\ell}}^\bxi(s,\x')\Big)\Bigg|_{\textcolor{black}{\bxi=\x}}
	\le C|(\x'-\x)_i|^{\alpha_i^k}\|{\mathbf D} u\|_\infty T^\delta,
	\end{eqnarray}
	for some $\delta>0 $.
	
	From the previous analysis it is therefore sufficient to focus on the tricky term, namely  $ I_{\textcolor{black}{k,3}}(s,\x)$ 
	introduced in \eqref{strong_DECOUP_I}. We begin the proof considering first  $I_{\textcolor{black}{k,3}}(s,\x)$. 
	Exploiting as well Lemma \ref{strong_LEMME_SG} for a centering argument w.r.t. the $k^{\rm th}$ variable, we write: 
	\begin{eqnarray*} 
		\!\!&\!\!\!\!&\!\!I_{\textcolor{black}{k,3}}^\bxi(s,\x)-I_{\textcolor{black}{k,3}}^\bxi(s,\x')\\
		\!\!&\!\!=\!\!&\!\! \sum_{\ell = 2}^{\textcolor{black}{k}} \int_{\R^{nd}} d\y \textcolor{black}{\bigg \langle} \Big(\gF_{\ell}(s,\y_{1:k-1},\btheta_{s,t}^{k:n}(\bxi))-\gF_{\ell}(s,\btheta_{s,t}(\bxi))-D_{\ell-1}\gF_{\ell}(s,\btheta_{s,t}(\bxi))\big(\y-\btheta_{s,t}(\bxi)\big)_{\ell-1}\Big) \textcolor{black}{,} \\
		\!\!&\!\!\!\!&\!\! 
		\Big(D_{\y_\ell}u(s,\y)-D_{\y_\ell}u(s,\y_{1:k-1},\btheta_{s,t}^{k:n}(\bxi)) \Big)\textcolor{black}{\bigg \rangle} \big(D_{\x_k}\tilde p^{\bxi}(t,s,\x,\y)-D_{\x_k}\tilde p^{\bxi}(t,s,\x',\y)\big) .
	\end{eqnarray*}
	Let us reproduce now the arguments used in Section \ref{strong_SECTION_CTR_SENSI} to handle $I_{\textcolor{black}{k,3}}^\bxi $ (see e.g. the computations  from equation \eqref{strong_CTR_P1_I3L_1} to \eqref{strong_JUST_BEFORE_DUALITY}). 
	\textcolor{black}{For $\ell \in \leftB 2,k-1\rightB $}, expanding with the Taylor formula the difference $\Big(D_{\y_\ell}u(s,\y)-D_{\y_\ell}u(s,\y_{1:k-1},\btheta_{s,t}^{k:n}(\bxi)) \Big)$, using the Schwarz theorem to exchange the order of differentiations\footnote{Recall indeed that what we are able to control is precisely the H\"older moduli of the derivatives $D_{\y_m}u(s,\cdot) $ w.r.t. \textcolor{black}{the} variables $\ell\le m$.}, \textcolor{black}{and integrating by parts as well},  we obtain:
	\begin{eqnarray}
	&&I_{\textcolor{black}{k,3}}^\bxi(s,\x)-I_{\textcolor{black}{k,3}}^\bxi(s,\x')\notag\\
	\!&\!\!= \!\!&\!\sum_{\ell = 2}^{\textcolor{black}{k-1}} \sum_{m=k}^n \int_0^1 d\lambda \int_{\R^{nd}} d\y 
	\textcolor{black}{D_{\y_\ell}}\bigg[\Big(\gF_{\ell}(s,\y_{1:k-1},\btheta_{s,t}^{k:n}(\bxi))-\gF_{\ell}(s,\btheta_{s,t}(\bxi))-D_{\ell-1}\gF_{\ell}(s,\btheta_{s,t}(\bxi))\big(\y-\btheta_{s,t}(\bxi)\big)_{\ell-1}\Big)\notag\\
	&&  \big(D_{\x_k} \tilde p^{\bxi}(t,s,\x,\y)-D_{\x_k} \tilde p^{\bxi}(t,s,\x',\y)\big) (\y-\btheta_{s,t}(\bxi))_{m} \bigg] 
	D_{\y_m} u\big(s,\y_{1:k-1},\btheta_{s,t}^{k:n}(\bxi) + \lambda(\y-\btheta_{s,t}(\bxi))_{k:n}\big)\notag\\
	&&\textcolor{black}{+\int_{\R^{nd}}d\y \bigg[\Big(\gF_{k}(s,\y_{1:k-1},\btheta_{s,t}^{k:n}(\bxi))-\gF_{k}(s,\btheta_{s,t}(\bxi))-D_{k-1}\gF_{\ell}(s,\btheta_{s,t}(\bxi))\big(\y-\btheta_{s,t}(\bxi)\big)_{k-1}\Big)}\notag\\
	&&\textcolor{black}{   \big(D_{\x_k} \tilde p^{\bxi}(t,s,\x,\y)-D_{\x_k} \tilde p^{\bxi}(t,s,\x',\y)\big)\bigg] 
		\Big(D_{\y_k}u(s,\y)-D_{\y_k}u(s,\y_{1:k-1},\btheta_{s,t}^{k:n}(\bxi)) \Big)}\notag\\
	&\textcolor{black}{=:}& \textcolor{black}{\Delta_{k,31}^\bxi(s,\x,\x')+ \Delta_{k,32}^\bxi(s,\x,\x')}.\label{THE_DECOUP_FOR_BESOV_DUALITY_HOLDER_3}
	\end{eqnarray}
	Let us now modify slightly the definition of \eqref{strong_GROSSE_DEF} and introduce:
	for all $\ell\in\leftB 2,k\textcolor{black}{-1}
	\rightB, m\in \leftB k,n\rightB$, $(\y_{1:\ell-1},\y_{\ell+1:n})\in \R^{(n-1)d},(t,\textcolor{black}{\z},\x)\in[0,T]\times \textcolor{black}{(\R^{nd})^2} $, \textcolor{black}{$s\in (t,T] $} and \textcolor{black}{$\y_i \in \R^d$}:
	\begin{eqnarray}
	&&\Psi_{\ell,k,m}^{(s,\y_{1:\ell-1},\y_{\ell+1:n}),(t,\z,\x)}(\y_{\textcolor{black}{\ell}})\notag\\
	&=&\bigg[D_{\textcolor{black}{\x_k}} \tilde p^{\bxi}(t,s,\textcolor{black}{\z},\y) 
	\\
	&&\otimes \Big(  (\gF_{\textcolor{black}{\ell}}(\textcolor{black}{s},\y_{1:k-1},\btheta^{k:n}_{s,t}(\bxi))-\gF_{\textcolor{black}{\ell}}(\textcolor{black}{s},\btheta_{s,t}(\bxi))
	\notag
	- D_{\textcolor{black}{\ell}-1}\gF_{\textcolor{black}{\ell}}(\textcolor{black}{s},\btheta_{s,t}(\bxi))\big(\y-\btheta_{s,t}(\bxi)\big)_{\textcolor{black}{\ell}-1})\Big)((\textcolor{black}{\y}-\btheta_{s,t}(\bxi))_{\textcolor{black}{m}})^*\bigg]_{\textcolor{black}{\bxi=\x}}\!\!.\label{strong_GROSSE_DEF_BI_INDEX}
	\end{eqnarray}
	\textcolor{black}{
		Write now with the notations of \textcolor{black}{\eqref{strong_GROSSE_DEF_BI_INDEX}}}:
	\begin{eqnarray*}
		&& (\textcolor{black}{\Delta_{k,31}^\bxi(s,\x,\x')})_{\textcolor{black}{\bxi=\x}}
		\nonumber \\
		&\textcolor{black}{:=}&
		\sum_{\ell = 2}^{\textcolor{black}{k-1}} \sum_{m=k}^n \int_0^1 d\lambda \int_{\R^{nd}} d\y 
		\textcolor{black}{D_{\y_\ell}}\bigg[
		\Psi_{\ell,k,m}^{(s,\y_{1:\ell-1},\y_{\ell+1:n}),(t,\x,\textcolor{black}{\x})}(\y_\ell)
		-\Psi_{\ell,k,m}^{(s,\y_{1:\ell-1},\y_{\ell+1:n}),(t,\x',\textcolor{black}{\x})}(\y_\ell)
		\bigg] 
		\nonumber \\
		&&
		\textcolor{black}{  D_{\y_m} u\big(s,\y_{1:k-1},\btheta_{s,t}^{k:n}(\textcolor{black}{\x}) + \lambda(\y-\btheta_{s,t}(\textcolor{black}{\x}))_{k:n}\big)}.
	\end{eqnarray*}
	
	Thus, \textcolor{black}{we derive similarly to \textcolor{black}{\eqref{strong_ALMOST_LAST_EQ_CTR_DERIV}}}: 
	\begin{eqnarray}
	\label{strong_PREAL_DIAG_I3k}
	&&|\textcolor{black}{\Delta _{k,31}^\bxi(s,\x,\x')}|_{\textcolor{black}{\bxi=\x}} 
	\notag\\
	& \leq & \sum_{\ell = 2}^{\textcolor{black}{k-1}} \sum_{m=k}^n   \int_{\R^{(n-1)d}}d (\y_{1:\ell-1},\y_{\ell+1:n})
	\Bigg\{\bigg\| \textcolor{black}{D_{\ell}\Big[ }\Psi_{\ell,k,m}^{(s,\y_{1:\ell-1},\y_{\ell+1:n}),(t,\x,\textcolor{black}{\x})}(\cdot)-\Psi_{\ell,k,m}^{(s,\y_{1:\ell-1},\y_{\ell+1:n}),(t,\x',\textcolor{black}{\x})}(\cdot)\textcolor{black}{\Big]}\bigg\|_{\textcolor{black}{ B_{1,1}^{-\alpha_\ell^m}}}
	\notag\\
	&& \hspace*{2cm}
	\times 
	\textcolor{black}{\sup_{\z_j, j\in \leftB1,n\rightB,j\neq \ell} \bigg\|
		\textcolor{black}{D_mu}(s,\z_{1:\ell-1},\cdot, \z_{\ell+1:n})\bigg\|_{\textcolor{black}{ B_{\infty,\infty}^{\alpha_\ell^m}}}}  \Bigg\},
	\end{eqnarray}
	\textcolor{black}{where $  \int_{\R^{(n-1)d}}d (\y_{1:\ell-1},\y_{\ell+1:n})$ means that we integrate over $  \y_{1:\ell-1} $ and $  \y_{\ell+1:n}$.}
	To conclude, we  need the following appropriate version of Lemma \ref{strong_LEMME_CTR_BESOV} to handle the Besov norm with negative exponent in the above r.h.s. Its proof is postponed to the next section.

	\begin{lem} \label{strong_LEMME_CTR_DIF_BESOV}
		Let $k \in \leftB 2,n\rightB$, $\ell \in \leftB 2,k\textcolor{black}{-1}\rightB$ and $ m\in \leftB k,n\rightB$ and let $\Psi_{\textcolor{black}{\ell,k,m}}^{(s,\y_{1:\ell-1},\y_{\ell+1:n}),(t,\textcolor{black}{\z},\x)}: \R^d \to \R^d$ be the function defined by \textcolor{black}{}\eqref{strong_GROSSE_DEF_BI_INDEX}. There exist $C:=C(\A{A})>0$ and    $\gamma_\ell^m := \gamma_\ell^m(\A{A}) :=1/2+ \eta (\ell-3/2)>1/2$ such that
		\begin{eqnarray*}
			&&\bigg\|\textcolor{black}{D_{\ell}\Big[  \Psi_{\ell,k,m}^{(s,\y_{1:\ell-1},\y_{\ell+1:n}),(t,\x,\textcolor{black}{\x})}(\cdot)-\Psi_{\ell,k,m}^{(s,\y_{1:\ell-1},\y_{\ell+1:n}),(t,\x',\textcolor{black}{\x})}(\cdot)\Big]}\bigg\|_{\textcolor{black}{ B_{1,1}^{-\alpha_\ell^m}}}\\
			& \leq & C  \hat q_{c\setminus \ell}(t,s,\x,(\y_{1:\ell-1},\y_{\ell+1:n}))  (s-t)^{-1-(i-\frac 12) \alpha_i^k+ \gamma_\ell^m} 
			|(\x-\x')_i|^{\alpha_i^k}
			,
		\end{eqnarray*} 
		with $\hat q_{c\setminus \ell}(t,s,\x,(\y_{1:\ell-1},\y_{\ell+1:n})) $ as in \eqref{strong_DEF_HAT_SETMINUS}.
	\end{lem}
	Again, for the specific choice of $\alpha_i^k=(1+ \eta/ 4)/(2i-1) $ performed in the proof of Lemma \ref{strong_LEMME_CTR_BESOV}, we eventually derive from Lemma \ref{strong_LEMME_CTR_DIF_BESOV} and \eqref{strong_PREAL_DIAG_I3k}
	that:
	\begin{align}
	\int_{{\mathcal S}_i^c} ds |\textcolor{black}{\Delta_{k,31}^\bxi(s,\x,\x')}|_{\textcolor{black}{\bxi=\x}}
	\le& C \Bigg(\sum_{\ell = 2}^{\textcolor{black}{k-1}} \sum_{m=k}^n \sup_{s\in [0,T]}\|  \textcolor{black}{\big( D_m u\big)_\ell(s,\cdot)}\|_{\textcolor{black}{ B_{\infty,\infty}^{\alpha_\ell^m}}} \int_{t}^T ds (s-t)^{-1-(i-\frac 12) \alpha_i^k+ \gamma_\ell^m}\Bigg)|(\x-\x')_i|^{\alpha_i^k}\notag\\
	\le & CT^\delta \sup_{m\in \leftB k,n\rightB, \ell\in \leftB 1,\textcolor{black}{k}\rightB, s\in [0,T]}\| \textcolor{black}{\big( D_m  u\big)_\ell}(s,\cdot)\|_{\textcolor{black}{ B_{\infty,\infty}^{\alpha_\ell^m}}}|(\x-\x')_i|^{\alpha_i^k}, \label{strong_THE_CTR_I3_SI_C}
	\end{align}
	for some $\delta:= 3\eta/8>0 $. \textcolor{black}{On the other hand, from the definition in \eqref{THE_DECOUP_FOR_BESOV_DUALITY_HOLDER_3}, using as well Proposition \ref{strong_THE_PROP}, Lemmas \ref{strong_SMOOTH_OPERATOR_EFFECTS}, \ref{Lemme_Taylor_reverse} and \ref{strong_LEMME_DENS}, we get similarly to \eqref{PREAL_HOLD_MOD_HD}}:
	\textcolor{black}{
		\begin{eqnarray}
		|\Delta_{k,32}^\bxi(s,\x,\x')|_{\textcolor{black}{\bxi=\x}}&\le& \frac{C[(D_{{k-1}}\gF_k)_{k-1}(s,\cdot))]_\eta}{(s-t)^{(k-\frac 12)+(i-\frac 12)}} |(\x-\x')_i|\int_{\R^d} \hat p_{C^{-1}}^\x(t,s,\x,\y) |(\btheta_{s,t}(\x)-\y)_{k-1}|^{1+\eta} \notag\\
		&&\times (\|\mathbf Du\|_\infty+[(D_k u)_k(s,\cdot)]_{\alpha_k^k})
		\big(|(\btheta_{s,t}(\x)-\y)_k|^{\alpha_k^k}+\sum_{m=k+1}^n |(\btheta_{s,t}(\x)-\y)_m|^{\zeta_k}\big)\notag\\
		&\le & C(\|\mathbf Du\|_\infty+[(D_k u)_k(s,\cdot)]_{\alpha_k^k})|(\x-\x')_i|^{\alpha_i^k}\notag\\
		&&\times \big(\sum_{m=k}^n (s-t)^{-(k-\frac 12)-(i-\frac 12)\alpha_i^k+(1+\eta)(k-\frac 32)+\alpha_k^k (k-\frac 12)\I_{m=k}+\zeta_k(m-\frac 12)\I_{m\in \leftB k+1,n\rightB}}\big),\label{PREAL_HOLD_MOD_HD_BIS}
		\end{eqnarray}
		where we have expanded the difference of the Gaussian densities and exploited the fact that on the considered time set the diagonal regime holds for those densities.}
	\textcolor{black}{From \eqref{PREAL_HOLD_MOD_HD_BIS} and recalling \eqref{EXPO_HD_1} and \eqref{EXPO_HD_2} we finally get:}
	\begin{eqnarray}
	\label{strong_THE_CTR_I3_SI_C_APRES_DIAG}
	\textcolor{black}{\int_{{\mathcal S}_i^c} ds} |\textcolor{black}{\Delta_{k,32}^\bxi(s,\x,\x')}|_{\textcolor{black}{\bxi=\x}}&\le& C(\|\mathbf Du\|_\infty+[(D_k u)_k(s,\cdot)]_{\alpha_k^k})|(\x-\x')_i|^{\alpha_i^k} \sum_{m=k}^n\int_{t}^T ds (s-t)^{-\frac 32+ \gamma_k^m}\notag\\
	&\le &CT^\delta(\|\mathbf Du\|_\infty+[(D_k u)_k(s,\cdot)]_{\alpha_k^k})|(\x-\x')_i|^{\alpha_i^k},
	\end{eqnarray}
	for $\delta $ small enough.
	Combining the estimates \eqref{strong_THE_CTR_I3_SI_C} and \eqref{strong_THE_CTR_I3_SI_C_APRES_DIAG} together with \eqref{strong_DIFF_I_12_IK} we eventually derive
	\begin{equation}
	\Bigg|\int_{{\mathcal S}_i^c} ds \left[I_k^\bxi(s,\x)-I_k^{\bxi}(s,\x')\right] \Bigg|_{\textcolor{black}{\bxi=\x}}
	\le T^\delta C|(\x'-\x)_i|^{\alpha_i^k}\left( \|{\mathbf D} u\|_\infty + \sup_{m\in \leftB k,n\rightB, \ell\in \leftB 1,\textcolor{black}{k}\rightB, s\in [0,T]}\textcolor{black}{[}\textcolor{black}{\big( D_m  u\big)_\ell}(s,\cdot)\textcolor{black}{]}_{\textcolor{black}{{\alpha_\ell^m}}}\right)\label{strong_label_pourri_eqinter}.
	\end{equation}

	Now, the term $H_{k}^\bxi(s,\x)-H_k^\bxi(s,\x') $ in \eqref{strong_TERME_SI_C} 
	(non-degenerate variables) can be handled reproducing the same previous arguments for $I_{k}^\bxi(s,\x)-I_k^\bxi(s,\x')  $,
	exploiting \eqref{strong_CTR_RECUP_SING_IK_DIAG} and following the computations performed for $H_k$ in the proof of Theorem \ref{strong_THM_DER_PDE} (\textcolor{black}{see e.g. \eqref{strong_LE_CTR_SUR_H_1L}}).
	From the definition of ${\mathcal S}_i^c $ we obtain:
	\begin{eqnarray}\label{strong_DIFF_H_IK}
	&&\Bigg|\int_{t+\textcolor{black}{c_0}|(\x'-\x)_i|^{\frac {2}{2i-1}}}^T   ds \Big(H_{k}^\bxi(s,\x)-H_k^\bxi(s,\x')\Big)\Bigg|_{\textcolor{black}{\bxi=\x}}\notag \\
	&\le& C|(\x'-\x)_i|^{\alpha_i^k}\big(\|{\mathbf D} D_{1}u\|_\infty +\textcolor{black}{\|{\mathbf D}u\|_\infty}\big)\int_{t}^{T} ds(s-t)^{-1- \alpha_i^k (i-\frac 12)+\delta_i^k}\notag\\
	&\le& C|(\x'-\x)_i|^{\alpha_i^k}\big( \|{\mathbf D} D_{1}u\|_\infty+\textcolor{black}{\|{\mathbf D}u\|_\infty}\big) T^\delta,
	\end{eqnarray}
	for some $\delta>0 $ recalling for the last inequality that, from the bound following \eqref{strong_BOUNDS_FOR_H_OFF_DIAG}, $\delta_i^k\ge 1/2+\eta/2 $, \textcolor{black}{so that $-\alpha_i^k(i-1/2)+\delta_i^k=- (1+\eta/4)/2+\delta_i^k >0$}. The arguments needed to control this term are actually those already  exploited in \cite{chau:17} when $n=2$.

	Gathering equations 
	\eqref{strong_label_pourri_eqinter} and \eqref{strong_DIFF_H_IK}, we finally derive with the notations of \eqref{strong_TERME_SI_C}:
	\begin{equation}
	|\textcolor{black}{D_{\x_k}}u(t,\x)-D_{\x_k}u(t,\x')|\Big|_{{\mathcal S}_i^c}
	\le C T^\delta 
	|\textcolor{black}{(\x'-\x)_i}|^{\alpha_i^k}\Big( \|{\mathbf D} D_{1} u\|_\infty+\textcolor{black}{\|{\mathbf D}u\|_\infty}+\sup_{2\le \textcolor{black}{\ell}\le m\le n, s\in [0,T]}\| \textcolor{black}{\big( D_m   u\big)_\ell}(s,\cdot)\|_{\textcolor{black}{ B_{\infty,\infty}^{\alpha_\ell^m}}}\Big). 
	\label{strong_LE_LABEL_PRELIM_DIAG}
	\end{equation}
	\subsubsection{Discontinuity term associated with the regime time change: control of the term \eqref{strong_DEF_MOD_HOLD_SUR_SI_BORD}.} 
	\label{SUBSUB_BOUNDARY}
	We here aim at handl\textcolor{black}{ing}
	\begin{eqnarray*}
		&&\big(\textcolor{black}{D_{\x_k}}u(t,\x)-D_{\x_k}u(t,\x') \big)|_{\partial{\mathcal S}_i}
		=   \Big\{D_{\x_k} \tilde P_{t_0,t}^{\bxi'} u(t_0, \x')-D_{\x_k}\tilde P_{t_0,t}^{\tilde \bxi'} u(t_0, \x') \Big\}_{\textcolor{black}{(\bxi',\tilde \bxi')=(\x',\x)}},
	\end{eqnarray*}
	which we will actually handle like the off-diagonal components. Recall here that the transition time $t_0 = t + c_0\textcolor{black}{|(\x-\x')_i|^{2/(2i-1)}}
	$. 
	From Lemma \ref{strong_LEMME_SG} (\textcolor{black}{cancellation argument}), we write:
	\begin{eqnarray*}
		\!\!&\!\!\!\!&\!\! 
		D_{\x_k} \tilde P_{t_0,t}^{\bxi'} u(t_0, \x')-D_{\x_k}\tilde P_{t_0,t}^{\tilde \bxi'} u(t_0, \x') 
		\\
		\!\!&\!\!=\!\!&\!\! \int_{\R^{nd}}d\y D_{\x_k} \tilde p^{\bxi'}(t,t_0,\x',\y) [u(t_0,\y)-u(t_0,\y_{1:k-1},(\btheta_{t_0,t}(\x'))_{k:n})] \notag\\
		\!\!&\!\!\!\!&\!\!-\int_{\R^{nd}}d\y D_{\x_k} \tilde p^{\textcolor{black}{\tilde \bxi'}}(t,t_0,\x',\y) [u(t_0,\y)-u(t_0,\y_{1:k-1},(\m^{\textcolor{black}{\tilde \bxi'}}_{t_0,t}(\x'))_{k}, \btheta_{t_0,t}(\x')_{k+1:n})] .\notag
	\end{eqnarray*}
	We now split the above contribution into three terms:
	\begin{eqnarray}
	D_{\x_k} \tilde P_{t_0,t}^{\bxi'} u(t_0, \x')-D_{\x_k}\tilde P_{t_0,t}^{\tilde \bxi'} u(t_0, \x')  
	= ({\mathscr B}_1^{\bxi'\!, \tilde \bxi'}+{\mathscr B}_2^{\bxi'\!, \tilde \bxi'}+{\mathscr B}_3^{\bxi'\!, \tilde \bxi'})(t_0,\x'), \label{strong_decomp_diff_discontinu}
	\end{eqnarray}
	where
	\begin{eqnarray}
	{\mathscr B}_1^{\bxi'\!, \tilde \bxi'}(t_0,\x')\!\!&\!\!:=\!\!&\!\!
	\Bigg\{\int_{\R^{nd}} d\y D_{\x_k} \tilde p^{\bxi'}(t,t_0,\x',\y) [u(t_0,\y)-u(t_0,\y_{1:k},(\btheta_{t_0,t}(\x'))_{k+1:n})] 
	\notag\\
	\!\!&\!\!\!\!&\!\!-\int_{\R^{nd}}d\y D_{\x_k} \tilde p^{\textcolor{black}{\tilde \bxi'}}(t,t_0,\x',\y) \Big[u(t_0,\y)-u(t_0,\y_{1:k},( \btheta_{t_0,t}(\x'))_{k+1:n})
	\Big]  \label{TB_1}
	\Bigg\},
	\end{eqnarray}

	\begin{eqnarray}
	&&{\mathscr B}_2^{\bxi'\!, \tilde \bxi'}(t_0,\x')
	\nonumber \\
	&:=&\!\!\Bigg\{\Big[\int_{\R^{nd}} d\y D_{\x_k} \tilde p^{\bxi'}(t,t_0,\x',\y) \Big[u(t_0,\y_{1:k},(\btheta_{t_0,t}(\x'))_{k+1:n})-u(t_0,\y_{1:k-1},(\btheta_{t_0,t}(\x'))_{k:n})\notag\\
	\!\!&\!\!\!\!&\!\!\qquad -\langle D_{k}u(t_0,\y_{1:k-1},(\btheta_{t_0,t}(\x'))_{k:n}), (\y-\btheta_{t_0,t}(\x'))_k \rangle\Big] \Big]
	\notag\\
	\!\!&\!\!\!\!&\!\!-\Big[\int_{\R^{nd}} d\y D_{\x_k} \tilde p^{\textcolor{black}{\tilde \bxi'}}(t,t_0,\x',\y) [u(t_0,\y_{1:k},( \btheta_{t_0,t}(\x'))_{k+1:n})-u(t_0,\y_{1:k-1},(\m^{\textcolor{black}{\tilde \bxi'}}_{t_0,t}(\x'))_{k}, \btheta_{t_0,t}(\x')_{k+1:n})\notag\\
	\!\!&\!\!\!\!&\!\!\qquad -\langle D_{k}u(t_0,\y_{1:k-1},(\m_{t_0,t}^{\textcolor{black}{\tilde \bxi'}}(\x'))_{k}, \btheta_{t_0,t}(\x')_{k+1:n}), (\y-\m_{t_0,t}^{\textcolor{black}{\tilde \bxi'}}(\x'))_k \rangle] \Big]   
	\Bigg\}\label{TB_2},
	\end{eqnarray}
	\begin{eqnarray}
	{\mathscr B}_3^{\bxi'\!, \tilde \bxi'}(t_0,\x')&:=&\!\!\Bigg\{\int_{\R^{nd}}d\y D_{\x_k} \tilde p^{\bxi'}(t,t_0,\x',\y)\langle D_{k}u(t_0,\y_{1:k-1},(\btheta_{t_0,t}(\x'))_{k:n}), (\y-\btheta_{t_0,t}(\x'))_k \rangle  
	\notag\\
	\!\!&\!\!\!\!&\!\!-\int_{\R^{nd}} d\y D_{\x_k} \tilde p^{\textcolor{black}{\tilde \bxi'}}(t,t_0,\x',\y)\langle D_{k}u(t_0,\y_{1:k-1},(\m_{t_0,t}^{\textcolor{black}{\tilde \bxi'}}(\x'))_{k}, \btheta_{t_0,t}(\x')_{k+1:n}), (\y-\m^{\textcolor{black}{\tilde \bxi'}}_{t_0,t}(\x'))_k \rangle 
	\Bigg\}\textcolor{black}{.}\notag\\
	\label{TB_3}
	\end{eqnarray}

	We now exploit  the H\"older regularity of $D_{k}u$ \textcolor{black}{w.r.t.} the $k^{\rm th}$ variable to control the terms in ${\mathscr B}_2^{\bxi'\!, \tilde \bxi'}(t_0,\x') $ defined in \eqref{TB_2}. 
	Let us first write from the previous decomposition:
	\begin{eqnarray*}
		&&|{\mathscr B}_{2,1}^{\bxi'\!, \tilde \bxi'}(t_0,\x')|
		\nonumber \\
		&:=&\Big |\int_{\R^{nd}} d\y D_{\x_k} \tilde p^{\bxi'}(t,t_0,\x',\y) \Big[u(t_0,\y_{1:k},(\btheta_{t_0,t}(\x'))_{k+1:n})-u(t_0,\y_{1:k-1},(\btheta_{t_0,t}(\x'))_{k:n})
		\notag\\
		&&
		-\langle D_{k}u(t_0,\y_{1:k-1},(\btheta_{t_0,t}(\x'))_{k:n}), (\y-\btheta_{t_0,t}(\x'))_k \rangle\Big]  \Big |\notag\\
		&=& \int_{\R^{nd}}d\y \int_0^1 d\lambda \Big |D_{\x_k} \tilde p^{\bxi'}(t,t_0,\x',\y)\Big | [D_k u(t_0,\y_{1:k-1},\cdot, (\btheta_{t_0,t}(\x') )_{k+1:n})]_{\alpha_k^k}|(\y-\btheta_{t_0,t}(\x') \big )_k|^{1+\alpha_k^k}.
	\end{eqnarray*}
	From Proposition \ref{strong_THE_PROP}, we thus derive:
	\begin{eqnarray}\label{strong_Besov_D2k_u}
	\!\!&\!\!\!\!&\!\!|{\mathscr B}_{2,1}^{\bxi'\!, \tilde \bxi'}(t_0,\x')|_{\textcolor{black}{(\bxi',\tilde \bxi')=(\x',\x)}}
	\nonumber \\
	\!\!&\!\!\leq\!\! &\!\! C( t_0-t)^{\alpha_k^k(k-\frac 12)} \!\!\!\!\sup_{\z_j , j\in \leftB 1,n\rightB,\ j\neq k}[ D_ku(t_0, \z_{1:k}, \cdot,\z_{k+1:n})] _{\alpha_k^k}
	\int_{\R^{(n-1)d}} \!\!\!
	d\y_{1:k-1} d\y_{k+1:n}
	\hat q_{C^{-1} \setminus k}(t,s,\x,(\y_{1:k-1},\y_{k+1:n}))
	\nonumber \\
	\!\!&\!\!=\!\!&\!\!  C(t_0-t)^{\frac 12+\frac \eta 8} [ (D_ku)_k(t_0,  \cdot)] _{\alpha_k^k}, 
	\end{eqnarray}
	recalling from \eqref{strong_def_alpha} that $\alpha_k^k=(1+\eta/4)/(2k-1) $.
	The same arguments readily give:
	\begin{eqnarray}\label{strong_Besov_D2k_u_bis}
	&&|{\mathscr B}_{2,2}^{\bxi'\!, \tilde \bxi'}(t_0,\x')|
	\nonumber \\
	&:=&\Big |\int_{\R^{nd}}\!\!\!\!d\y D_{\x_k} \tilde p^{\textcolor{black}{\tilde \bxi'}}(t,t_0,\x',\y) [u(t_0,\y_{1:k},( \btheta_{t_0,t}(\x'))_{k+1:n})-u(t_0,\y_{1:k-1},(\m^{\textcolor{black}{\tilde \bxi'}}_{t_0,t}(\x'))_{k}, \btheta_{t_0,t}(\x')_{k+1:n})\notag\\
	\!\!&\!\!\!\!&\!\!-\langle D_{k}u(t_0,\y_{1:k-1},(\m_{t_0,t}^{\textcolor{black}{\tilde \bxi'}}(\x'))_{k}, \btheta_{t_0,t}(\x')_{k+1:n}), (\y-\m_{t_0,t}^{\textcolor{black}{\tilde \bxi'}}(\x'))_k \rangle] \Big) \Big |\notag\\
	\!\!&\!\!\leq \!\!&\!\!  C(t_0-t)^{\frac 12+\frac \eta 8} [ (D_ku)_k(t_0,  \cdot)] _{\alpha_k^k}.
	\end{eqnarray}
	\textcolor{black}{Observe also that for this term we do not need to exploit the specific choice $(\bxi',\tilde \bxi')=(\x',\x) $ for the freezing point.}
	
	Let us now deal with the contribution ${\mathscr B}_1^{\bxi'\!,\tilde \bxi'}(t_0,\x') $ in \eqref{TB_1}. Observe from this definition that this term is non-zero if and only if $k<n $.
	Write then 
	\begin{eqnarray}\label{strong_petit_label_trankilou}
	&&|{\mathscr B}_1^{\bxi'\!,\tilde \bxi'}(t_0,\x')|_{\textcolor{black}{(\bxi',\tilde \bxi')=(\x',\x)}}\notag\\
	\!\!&\le&\!\!\int_{\R^{nd}}d\y |D_{\x_k} \tilde p^{\bxi'}(t,t_0,\x',\y)|_{\textcolor{black}{ \bxi'=\x'}} |u(t_0,\y)-u(t_0,\y_{1:k},(\btheta_{t_0,t}(\x'))_{k+1:n})| \notag\\
	\!\!&\!\!\!\!&\!\!+\int_{\R^{nd}}d\y |D_{\x_k} \tilde p^{\textcolor{black}{\tilde \bxi'}}(t,t_0,\x',\y)|_{\textcolor{black}{\tilde \bxi'=\x}} |u(t_0,\y)-u(t_0,\y_{1:k},( \btheta_{t_0,t}(\x'))_{k+1:n})| \notag\\
	\!\!&\!\!\leq \!\!&\!\! C\|{\mathbf D} u\|_\infty\bigg(\int_{\R^{nd}} \frac{d\y}{(t_0-t)^{k-\frac 12}}\hat p_{C^{-1}}^{\textcolor{black}{\x'}}(t,t_0,\x',\y)|(\btheta_{t_0,t}(\x')-\y)_{k+1:n}| \\
	\!\!&\!\!\!\!&\!\!+\int_{\R^{nd}}\!\! \frac{d\y}{(t_0-t)^{k-\frac 12}}\hat p_{C^{-1}}^{\textcolor{black}{\x}}(t,t_0,\x',\y) \big ( |(\m_{t_0,t}^\x(\x')-\y)_{k+1:n}|\!+\!|(\m_{t_0,t}^\x(\x')-\btheta_{t_0,t}(\x'))_{k+1:n}| \big ) \!\bigg ).\notag
	\end{eqnarray}
	\textcolor{black}{
		To deal with the last contribution in the r.h.s., we will need some auxilliary lemmas already used in \cite{chau:hono:meno:18} for Schauder estimates. Namely, analogously to Lemmas 1 and 3 therein, we have the following result.
		\begin{lem}\label{strong_lem_mean_covariance}
			There exists $\vartheta= \vartheta(\A{A}) \in (0,1)$ s.t. for all $j \in \leftB \textcolor{black}{1} , n \rightB$, $\x\in \R^{nd}, \x'=(\x_{1:i-1},z,\x_{i+1:n})\in \R^{nd}$:
			\begin{equation}\label{strong_ineq_m_theta}
			|(\m_{t_0,t}^\x(\x')-\btheta_{t_0,t}(\x'))_{j}| \leq \textcolor{black}{c_0^\vartheta \d^{2j-1}(\x,\x') :=}  
			c_0^\vartheta |(\x-\x')_i|^{\frac{2j-1}{2i-1}}.
			\end{equation} 
	\end{lem}}
	\textcolor{black}{For the sake of completeness a proof is provided in Appendix \ref{strong_APP} below}.
	Recalling that $t_0=t+c_0 |(\x-\x')_i|^{ 2/(2i-1)} $ with $|(\x-\x')_i|\le 1 $, we obtain
	\begin{eqnarray*}
		\frac{\|{\mathbf D} u\|_\infty}{(t_0-t)^{k-\frac 12}} |(\m_{t_0,t}^\x(\x')-\btheta_{t_0,t}(\x'))_{k+1:n}| 
		&\le& \frac{\|{\mathbf D} u\|_\infty}{c_0^{k-\frac 12}|(\x-\x')_i|^{\frac{2k-1}{2i-1}}} c_0^\vartheta |(\x-\x')_i|^{\frac{2(k+1)-1}{2i-1}}
		\nonumber \\
		& \leq & \|{\mathbf D} u\|_\infty c_0^{\vartheta- (k-\frac 12)
		} |(\x-\x')_i|^{\frac{2}{2i-1}}\\
		&\leq& \|{\mathbf D} u\|_\infty c_0^{\vartheta- (k-\frac 12)
		} |(\x-\x')_i|^{\alpha_i^k}.
	\end{eqnarray*}
	Plugging the above control in \eqref{strong_petit_label_trankilou} and \textcolor{black}{using as well Lemma \ref{strong_LEMME_DENS} and Proposition \ref{strong_THE_PROP}}, we obtain:
	\begin{eqnarray}\label{strong_ineq_interlemmebesov}
	|{\mathscr B}_1^{\bxi'\!,\tilde \bxi'}(t_0,\x')|_{\textcolor{black}{(\bxi',\tilde \bxi')=(\x',\x)}}\!\!&\le&\!\! C\|{\mathbf D} u\|_\infty \left\{2(t_0-t) +  c_0^{\vartheta- (k-\frac 12)
	} |(\x-\x')_i|^{\alpha_i^k}\right\}.
	\end{eqnarray}
	Let us eventually control the term ${\mathscr B}_3^{\bxi'\!,\tilde \bxi'}(t_0,\x') $ defined in \eqref{TB_3} which we rewrite in the following way: 
	\begin{eqnarray*}\label{strong_Delta_mathscr_R2}
		&&{\mathscr B}_3^{\bxi'\!,\tilde \bxi'}(t_0,\x')
		\nonumber \\
		\!\!&\!\!=\!\!&\!\! \Bigg \{\int_{\R^{nd}}d\y  D_{\x_k} \tilde p^{\bxi'}(t,t_0,\x',\y) \Big \langle  \big [ D_{k}u(t_0,\y_{1:k-1},(\btheta_{t_0,t}(\x'))_{k:n})-D_{k}u(t_0,\btheta_{t_0,t}(\x')) \big ] , 
		(\y-\btheta_{t_0,t}(\x'))_k \Big \rangle  \notag \\
		\!\!&\!\!\!\!&\!\!- \int_{\R^{nd}} d\y D_{\x_k} \tilde p^{\textcolor{black}{\tilde \bxi'}}(t,t_0,\x',\y) \Big \langle \big[ D_{k}u(t_0,\y_{1:k-1},(\m_{t_0,t}^{\textcolor{black}{\tilde \bxi'}}(\x'))_{k}, \btheta_{t_0,t}(\x')_{k+1:n})\notag \\
		\!\!&\!\!\!\!&\!\! \quad - D_{k}u(t_0,(\m_{t_0,t}^{\textcolor{black}{\tilde \bxi'}}(\x'))_{1:k}, \btheta_{t_0,t}(\x')_{k+1:n}) \big] (\y-\m^{\textcolor{black}{\tilde \bxi'}}_{t_0,t}(\x'))_k \Big \rangle  
		\Bigg \}
		\nonumber \\
		\!\!&\!\!\!\!&\!\!+\Bigg \{ \int_{\R^{nd}}d\y  D_{\x_k} \tilde p^{\bxi'}(t,t_0,\x',\y)
		\nonumber \\
		\!\!&\!\!\!\!&\!\! \quad
		\Big \langle [D_{k}u(t_0,\btheta_{t_0,t}(\x'))-D_{k}u(t_0,(\m_{t_0,t}^{\textcolor{black}{\tilde \bxi'}}(\x'))_{1:k}, \btheta_{t_0,t}(\x')_{k+1:n})   , (\y-\btheta_{t_0,t}(\x'))_k \Big \rangle  
		\Bigg \} \nonumber \\
		\!\!&\!\!\!\!&\!\!- \Bigg \{\int_{\R^{nd}} d\y D_{\x_k} \tilde p^{\textcolor{black}{\tilde\bxi'}}(t,t_0,\x',\y) 
		\Big \langle  D_{k}u(t_0,(\m_{t_0,t}^{\textcolor{black}{\tilde \bxi'}}(\x'))_{1:k}, \btheta_{t_0,t}(\x')_{k+1:n}) , (\y-\m^{\textcolor{black}{\tilde \bxi'}}_{t_0,t}(\x'))_k \Big \rangle 
		\nonumber \\
		\!\!&\!\!\!\!&\!\! \quad -\int_{\R^{nd}}d\y  D_{\x_k} \tilde p^{\bxi'}(t,t_0,\x',\y) \Big \langle D_{k}u(t_0,(\m_{t_0,t}^{\textcolor{black}{\tilde \bxi'}}(\x'))_{1:k}, \btheta_{t_0,t}(\x')_{k+1:n})   , (\y-\btheta_{t_0,t}(\x'))_k \Big \rangle  
		\Bigg \}, \nonumber 
	\end{eqnarray*}
	where, thanks to Proposition \ref{strong_Prop_moment_D2_tilde_p} the last contribution is actually 0. For the first and second contributions in the above r.h.s. we have, thanks to the H\"older regularity of $\x_{1:k} \mapsto D_k u(\cdot, \x_{1:k}, \cdot)$, Proposition \ref{strong_THE_PROP} and Lemma  \ref{strong_lem_mean_covariance}:
	\begin{eqnarray}\label{strong_ineq_Delta_12}
	&&|{\mathscr B}_3^{\bxi'\!,\tilde \bxi'}(t_0,\x')|_{\textcolor{black}{(\bxi',\tilde \bxi')=(\x',\x)}}
	\nonumber \\
	&\leq&\!\!  C\int_{\R^{nd}}d\y  \left\{\hat p_{C^{-1}}^{\textcolor{black}{\x'}}(t,t_0,\x',\y) +  \hat p_{C^{-1}}^{\textcolor{black}{\x}}(t,t_0,\x',\y)\right\} \sum_{j=1}^{k-1} [(D_ku)_j(t_0,\cdot)]_{\alpha_j^k} (t_0-t)^{\alpha_j^k(j-\frac 12)} \notag\\
	\!\!&\!\!\!\!&\!\!
	+C  \int_{\R^{nd}}d\y  \hat p_{C^{-1}}^{\textcolor{black}{\x'}}(t,t_0,\x',\y) \sum_{j=1}^{k} [(D_ku)_j(t_0,\cdot)]_{\alpha_j^k} (t_0-t)^{\alpha_j^k(j-\frac 12)}\notag\\
	\!\!&\!\!\leq\!\!&\!\!  C \Big(\sum_{j=1}^{k} [(D_ku)_j(t_0,\cdot)]_{\alpha_j^k}\Big) (t_0-t)^{\frac 12+\frac \eta 8}.
	\end{eqnarray}
	Thus, plugging estimates \eqref{strong_Besov_D2k_u},  \eqref{strong_Besov_D2k_u_bis},  \eqref{strong_ineq_interlemmebesov} and \eqref{strong_ineq_Delta_12} in \eqref{strong_decomp_diff_discontinu} we deduce that
	\begin{eqnarray}\label{strong_ineq_Delta_12_inter}
	&&|D_{\x_k} \tilde P_{t_0,t}^{\bxi'} u(t_0, \x')-D_{\x_k}\tilde P_{t_0,t}^{\tilde \bxi'} u(t_0, \x')|_{\textcolor{black}{(\bxi',\tilde \bxi')=(\x',\x)}}\notag\\
	&\leq & 
	C\textcolor{black}{\Big[}\|{\mathbf D} u\|_\infty \left\{(t_0-t) +  c_0^{\vartheta- (k-\frac 12)
	} |(\x-\x')_i|^{\alpha_i^k}\right\} 
	+ \Big(\sum_{j=1}^{k}[(D_ku)_j(t_0,\cdot)]_{\alpha_j^k}\Big) (t_0-t)^{\frac 12+\frac \eta 8}\textcolor{black}{\Big]}.
	\end{eqnarray}
	
	From the definition of $t_0=t+c_0 |(\x-\x')_i|^{2/(2i-1)}$ and $\alpha_{i}^k=(1+\eta/4)/(2i-1)$, recalling that $t_0-t$ is small as well (i.e. $t_0-t\le C (t_0-t)^{1/2+\eta/8} $), we obtain from \eqref{strong_ineq_Delta_12_inter}:
	\begin{eqnarray*}
		&&|D_{\x_k} \tilde P_{t_0,t}^{\bxi'} u(t_0, \x')-D_{\x_k}\tilde P_{t_0,t}^{\tilde \bxi'} u(t_0, \x')|_{\textcolor{black}{(\bxi',\tilde \bxi')=(\x',\x)}}\\
		&\leq & C\Bigg\{c_0^{\frac 12+\frac \eta 8} \Big( \sum_{j=1}^{k} [\textcolor{black}{\big(D_k  u\big)_j(t_0,\cdot)} ]_{\textcolor{black}{{\alpha_j^k}}}\Big) +  (c_0^{\vartheta-  (k-\frac 12)
		}+c_0^{\frac 12+\frac \eta 8} )\|{\mathbf D} u\|_\infty   \Bigg\} |(\x-\x')_i|^{\alpha_i^k}.
	\end{eqnarray*}
	Thus, \textcolor{black}{from \eqref{strong_DEF_MOD_HOLD_SUR_SI_BORD}}, there exists $\tilde \delta := \tilde \delta(\A{A}) \textcolor{black}{\in (0,1)}$ such that
	\begin{equation}\label{strong_COOL}
	\Big |\big( \textcolor{black}{D_{\x_k}u}(t,\x)-D_{\x_k}u(t,\x') \big)|_{\partial{\mathcal S}_i} \Big |
	\leq |(\x-\x')_i|^{\textcolor{black}{\alpha_i^k}}   C (c_0^{-\tilde \delta^{\textcolor{black}{-1}} } \|{\mathbf D} u\|_\infty   + c_0^{\tilde \delta} \textcolor{black}{\max_{2\leq \ell \leq m \leq n} [ \textcolor{black}{\big(  D_m u\big)_\ell(t_0,\cdot) }]_{\textcolor{black}{ {\alpha_\ell^m}}} }
	).
	\end{equation}

	%
	%
	%

	\paragraph{Conclusion: control of \eqref{strong_DECOUP_TERMES}.} 
	
	Plugging \eqref{strong_CTR_PRELIM_OFF_DIAG_SI}, \eqref{strong_LE_LABEL_PRELIM_DIAG} and \eqref{strong_COOL} into \eqref{strong_DECOUP_TERMES}, \eqref{strong_NABLA_XK}, we eventually derive that for some positive $\tilde \delta := \tilde \delta(\A{A}) >0$:
	\begin{equation}
	\|\textcolor{black}{D_{\x_k}u}(t,\x_{1:i-1},\cdot,\x_{i+1,n})\|_{\textcolor{black}{ B_{\infty,\infty}^{\alpha_i^k}}}
	\le C\Big(c_0^{-\tilde \delta^{\textcolor{black}{-1}} }\|{\mathbf D} u\|_\infty+ \|{\mathbf D} D_{1} u\|_\infty+ \big( c_0^{\tilde \delta } + T^{\tilde \delta}\big) \sup_{s\in [0,T],  2\le \ell\le m\le n}\| \textcolor{black}{\big( D_m u\big)_\ell(s,\cdot)}\|_{\textcolor{black}{ B_{\infty,\infty}^{\alpha_\ell^m}}}\Big). \nonumber  \label{strong_CTR_BESOV_FINAL_AVANT_SOMMATION}
	\end{equation}
	The main point to close our circular argument then consists  in taking  the supremums \textcolor{black}{w.r.t. $\x_{1:i-1},\x_{i+1:n} $}, $i,k $ and $t\in [0,T] $ on the l.h.s. and to tune the constant $c_0$ and the terminal time $T$ in order to obtain $C \big( c_0^{\tilde \delta } + T^{\tilde \delta}\big) \leq  1/2$. We then derive that for all $2\le i\le k\le n $:
	\begin{eqnarray}
	\sup_{t\in [0,T]}\|\textcolor{black}{\big(D_{k}u\big)_i(t,\cdot)}\|_{\textcolor{black}{ B_{\infty,\infty}^{\alpha_i^k}}}\le   C  \{\|{\mathbf D} u\|_\infty+\|{\mathbf D} D_{1} u\|_\infty \}, \label{strong_CTR_BESOV_FINAL_APRES_SOMMATION}
	\end{eqnarray}
	which concludes the proof of Lemma \ref{strong_LEMME_CTR_HOLDER}.\hfill $\textcolor{black}{\square} $.
	
	\subsection{Proof of of Lemma \ref{strong_LEMME_CTR_DIF_BESOV}}
	
	We follow the proof of Lemma \ref{strong_LEMME_CTR_BESOV}, concentrating on the case $\ell\le k-1 $, the specific case $\ell=k $ could be treated similarly considering the slightly different cancellation terms already discussed in Lemma \ref{strong_LEMME_CTR_BESOV}. The quantity to estimate is now:
	\begin{equation}\label{strong_DIFF_DES_PSI}
	\textcolor{black}{D_\ell}\bigg[\Big[\Psi_{\ell,k,m}^{(s,\y_{1:\ell-1},\y_{\ell+1:n}),(t,\x,\textcolor{black}{\x})}(\cdot)-\Psi_{\ell,k,m}^{(s,\y_{1:\ell-1},\y_{\ell+1:n}),(t,\x',\textcolor{black}{\x})}(\cdot)\bigg].
	\end{equation}
	Splitting the thermic part of the Besov norm as in \eqref{strong_DECOUP_HK_INT}, we obtain the same kind estimate for the \textcolor{black}{non-singular in time} part. Indeed, we point out that the difference \eqref{strong_DIFF_DES_PSI}  does not involve $D_{\x_1} $, therefore we gain a singularity of order $1/2 $ w.r.t. 
	equation \eqref{strong_RES_CALIBRATION_BETA_I}. On the other hand, the difference of the derivatives of the frozen densities w.r.t. $\x_k$ can be handled with \eqref{strong_CTR_RECUP_SING_IK_DIAG}. Choosing $ \rho_{\ell,m}=2\ell -1$ as in the proof of Lemma \ref{strong_LEMME_CTR_BESOV}, and recalling that $\alpha_i^k(i- 1/2)=(1+\eta/4)/2$ (see \eqref{strong_def_alpha}), it is plain to check that:
	\begin{eqnarray}\label{strong_ESTI_DE_U_DES_DIFF}
	&&{\mathbf {Upper}}\bigg[\textcolor{black}{D_\ell\Big[}\Psi_{\ell,k,m}^{(s,\y_{1:\ell-1},\y_{\ell+1:n}),(t,\x,\textcolor{black}{\x})}(\cdot)-\Psi_{\ell,k,m}^{(s,\y_{1:\ell-1},\y_{\ell+1:n}),(t,\x',\textcolor{black}{\x})}(\cdot)\textcolor{black}{\Big]}\bigg]
	\nonumber \\
	&\le& \frac{\tilde C|(\x-\x')_i|^{\alpha_i^k}}{(s-t)^{1+\alpha_i^k(i-\frac 12) -{\gamma_\ell^m}}}\hat q_{c\setminus \ell}(t,s,\x,(\y_{1:\ell-1},\y_{\ell+1:n})),
	\end{eqnarray}
	where from \A{${\mathbf H} _\eta$}, $1+\alpha_i^k(i- 1/2) -{\gamma_\ell^m}<1 $.
	Turning to the \textcolor{black}{singular in time} contribution of the thermic part of the Besov norm of \eqref{strong_DIFF_DES_PSI} we decompose with the notations of Lemma \ref{strong_LEMME_CTR_BESOV} (\textcolor{black}{see e.g. \eqref{strong_DECOUP_CAR_THERMIC_T1}, \eqref{strong_DECOUP_CAR_THERMIC_T2} which exhibit an additional spatial derivative}):
	\begin{eqnarray*}
		&&{\mathbf {Lower}}\bigg[\textcolor{black}{D_\ell \Big[}\Psi_{\ell,k,m}^{(s,\y_{1:\ell-1},\y_{\ell+1:n}),(t,\x,\textcolor{black}{\x})}(\cdot)-\Psi_{\ell,k,m}^{(s,\y_{1:\ell-1},\y_{\ell+1:n}),(t,\x',\textcolor{black}{\x})}(\cdot)\textcolor{black}{\Big]}\bigg]\\
		&=:&\int_0^{(s-t)^{\rho_{\ell,m}}} \textcolor{black}{\frac{dw}{w} w^{\frac{\alpha_\ell^m}2+1}} \int_{\R^d} d\tilde z |({\mathscr T}_{1,\ell,k,m}^{(s,\y_{1:\ell-1},\y_{\ell+1:n}),(t,\x,\x')}+{\mathscr T}_{2,\ell,k,m}^{(s,\y_{1:\ell-1},\y_{\ell+1:n}),(t,\x,\x')}) \big(w,\tilde z\big)|,
	\end{eqnarray*}
	where  $\ell<k\leq m$ and:
	\begin{eqnarray}
	\!\!&\!\!&\!\!{\mathscr T}_{1,\ell,k,m}^{(s,\y_{1:\ell-1},\y_{\ell+1:n}),(t,\x,\x')} \big(w,\tilde z\big)\label{strong_DECOUP_CAR_THERMIC_T1_MODIF}
	\\
	\!\!&\!\!=\!\!&\!\!\int_{\R^d}d\y_\ell \textcolor{black}{D_\ell}\partial_w h_w(\tilde z-\y_\ell )\big(D_{\x_k} \tilde p^{\bxi}(t,s,\x,\y)-D_{\x_k} \tilde p^{\bxi}(t,s,\x',\y)\big)_{\textcolor{black}{\bxi=\x}}
	\notag\\
	\!\!&\!\!\!\!&\!\!  \big \langle \gF_{\ell}(s,\y_{1:k-1},\btheta_{s,t}^{k:n}(\textcolor{black}{\x}))-\gF_{\ell}(s,\y_{1:\ell-1},\tilde z,\y_{\ell+1:k-1},\btheta_{s,t}^{k:n}(\textcolor{black}{\x}))\big)
	, (\y-\btheta_{s,t}(\textcolor{black}{\x}))_m \big \rangle ,\notag
	\end{eqnarray}
	with a slight abuse of notation when $\ell=k-1$ and 
	\begin{eqnarray}
	\label{strong_DECOUP_CAR_THERMIC_T2_MODIF}
	\!\!&\!\!\!\!&\!\!{\mathscr T}_{2,\ell,k,m}^{(s,\y_{1:\ell-1},\y_{\ell+1:n}),(t,\x,\textcolor{black}{\x'})} \big(w,\tilde z\big)\\
	\!\!&\!\!=\!\!&\!\!\int_{\R^d}d\y_\ell \textcolor{black}{D_\ell}\partial_w h_w(\tilde z-\y_\ell)\Bigg[\Big(D_{\x_k} \tilde p^{\bxi}(t,s,\x,\y)-D_{\x_k}\tilde p^{\bxi}(t,s,\x,\y_{1:\ell-1},\tilde z,\y_{\ell+1:n})\Big)\notag\\
	\!\!&\!\!\!\!&\!\!-\Big(D_{\x_k} \tilde p^{\bxi}(t,s,\x',\y)-D_{\x_k}\tilde p^{\bxi}(t,s,\x',\y_{1:\ell-1},\tilde z,\y_{\ell+1:n})\Big)  
	\Bigg]_{\textcolor{black}{\bxi=\x}}\notag \\
	\!\!&\!\!\!\!&\!\! \Big\langle \gF_{\ell}(s,\y_{1:\ell-1},\tilde z,\y_{\ell+1:k-1},\btheta_{s,t}^{k:n}(\textcolor{black}{\x}))-\gF_{\ell}(s,\btheta_{s,t}(\textcolor{black}{\x})) -D_{\ell-1}\gF_{\ell}(s,\btheta_{s,t}(\textcolor{black}{\x}))(\y-\btheta_{s,t}(\textcolor{black}{\x}))_{\ell-1}, 
	(\y-\btheta_{s,t}(\textcolor{black}{\x}))_m\Big \rangle .\notag
	\end{eqnarray}
	
	Note now that, when proceeding first as in \eqref{strong_DVP_TAYLOR_DIFF_TILDEP}, \eqref{strong_DVP_TAYLOR_DIFF_TILDEP_2} and control\textcolor{black}{ling} then  the \textcolor{black}{associated} difference as in \eqref{strong_DECOUP_CAR_THERMIC_T2}, we get thanks to \eqref{strong_CTR_RECUP_SING_IK_DIAG} that:
	\begin{eqnarray*}
		\!\!&\!\!\!\!&\!\!\Bigg|\Big(D_{\x_k} \tilde p^{\bxi}(t,s,\x,\y)-D_{\x_k}\tilde p^{\bxi}(t,s,\x,\y_{1:\ell-1},\tilde z,\y_{\ell+1:n})\Big)
		\!-\!\Big(D_{\x_k} \tilde p^{\bxi}(t,s,\x',\y)-D_{\x_k}\tilde p^{\bxi}(t,s,\x',\y_{1:\ell-1},\tilde z,\y_{\ell+1:n})\Big)  \Bigg|\\
		\!\!&\!\!\leq \!\!&\!\!\frac{\tilde C}{(s-t)^{\ell-\frac 12 + i-\frac 12 + k-\frac 12}}\int_0^1 d \lambda  \hat p_{C^{-1}}^\bxi(t,s,\x,\y_{1:\ell-1},\tilde z\!+ \! \lambda (\y_\ell\!-\!z), \y_{\ell+1:n})|(\x'\!-\!\x)_i| |\tilde z\!-\!\y_\ell|\\
		\!\!&\!\!\leq \!\!&\!\!\frac{\tilde C}{(s\!-\!t)^{(\ell-\frac 12) + \alpha_i^k(i-\frac 12) + (k-\frac 12)}} \!\int_0^1\!\! d \lambda  \hat p_{C^{-1}}^\bxi(t,s,\x,\y_{1:\ell-1},\tilde z \!+\! \lambda (\y_\ell \!-\!z), \y_{\ell+1:n})|(\x'\!-\!\x)_i|^{\alpha_i^k} |\tilde z \!- \!\y_\ell|,
	\end{eqnarray*}
	using the fact we are in the diagonal regime in the last inequality. With this control at hand, together with estimate \eqref{strong_CTR_RECUP_SING_IK_DIAG}, to handle the contributions $({\mathscr T}_{1,\ell,k,m}^{(s,\y_{1:\ell-1},\y_{\ell+1:n}),(t,\x,\textcolor{black}{\x'})}+{\mathscr T}_{2,\ell,k,m}^{(s,\y_{1:\ell-1},\y_{\ell+1:n}),(t,\x,\textcolor{black}{\x'})}) \big(w,\tilde z\big)$, we can mimic the proof of the estimation of the contributions 
	$({\mathscr T}_{1,i,(l,1),k}^{(s,\y_{1:i-1},\y_{i+1:n}),(t,\x)}+{\mathscr T}_{2,i,(l,1),k}^{(s,\y_{1:i-1},\y_{i+1:n}),(t,\x)}) \big(v,z\big)$ done in Lemma \ref{strong_LEMME_CTR_BESOV} to obtain 
	\begin{eqnarray*}
		&&{\mathbf {Lower}}\bigg[\textcolor{black}{D_\ell\Big[}\Psi_{\ell,k,m}^{(s,\y_{1:\ell-1},\y_{\ell+1:n}),(t,\x, \textcolor{black}{\x})}(\cdot)-\Psi_{\ell,k,m}^{(s,\y_{1:\ell-1},\y_{\ell+1:n}),(t,\x',\textcolor{black}{\x})}(\cdot)\textcolor{black}{\Big]}\bigg] 
		\nonumber \\
		&\le& \tilde C\hat q_{c\setminus \ell}(t,s,\x,(\y_{1:\ell-1},\y_{\ell+1:n}))|(\x'-\x)_i|^{\alpha_i^k} 
		\int_0^{(s-t)^{\rho_{\ell,m}}} \textcolor{black}{\frac{dw}{w} w^{\frac{\alpha_{\ell}^m}{2}+1}} 
		\nonumber \\
		&&\Bigg\{\frac{1}{w^{\textcolor{black}{\frac 32}-\frac{\beta_\ell}{2}}(s-t)^{\alpha_i^k(i-\frac 12)}} +\frac{w^{\textcolor{black}{-1+\frac{\beta_\ell}{2}}}}{(s-t)^{\ell-\frac 12 + \alpha_i^k(i-\frac 12)}}    +\frac{w^{\textcolor{black}{-1}}}{(s-t)^{\ell-\frac 12+ \alpha_i^k(i-\frac 12)-(1+\eta)(\ell-3/2)}}\Bigg\}\\
		& \leq & \tilde C\hat q_{c\setminus \ell}(t,s,\x,(\y_{1:\ell-1},\y_{\ell+1:n})) (s-t)^{-1- (i-\frac 12)\alpha_i^k+ \gamma_\ell^m } |(\x'-\x)_i|^{\alpha_i^k},
	\end{eqnarray*}
	where $\gamma_\ell^m =  1/2 + \eta(\ell- 3/2)$. From \A{${\mathbf H} _\eta$} and the very definition of $\alpha_i^k$ we hence have $1-\alpha_i^k(i- 1/2) +{\gamma_\ell^m}<1$ which, together with \eqref{strong_ESTI_DE_U_DES_DIFF}, concludes the proof. \qed

	\appendix
	\mysection{Sensitivity  results for  the mean: Proof of Lemma \ref{strong_lem_mean_covariance}}
	\label{strong_APP}
	In order to prove Lemma \ref{strong_lem_mean_covariance}, we first
	need to establish some  controls on the sensitivity of the flows, see Lemma \ref{strong_lem_theta_theta_bis} below.  Those results are obtained under the sole assumption \A{A} and remain valid for the mollification procedure of the coefficients considered in \A{AM}. 
	We will then proceed to the final proof  of  Lemma \ref{strong_lem_mean_covariance} in Section \ref{strong_sec:ProofLemma_flow}.
	
	For our analysis, we now introduce the spatial homogeneous distance, which basically reflects the various scales of the system already seen e.g. in Proposition 
	\ref{strong_THE_PROP}. Namely, for $(\x,\x')\in \R^{nd} $, we define:
	\begin{equation}
	\label{DIST_HOMO}
	\d(\x,\x'):=\sum_{i=1}^n |(\x-\x')_i|^{\frac{1}{2i-1}}.
	\end{equation}
	The distance is homogeneous in the sense that, for any $\lambda>0,\ \d(\lambda^{- 1/2}\T_\lambda \x,\lambda^{-1/2}\T_\lambda \x')=\lambda^{1/2}\d(\x,\x') $.
	
	\subsection{A first sensitivity result for the flow}\label{strong_sec:ProofLem_d_theta}
	\begin{lem}[Control of the flows]\label{strong_lem_theta_theta_bis} Under \A{A}, there exists $C:=C(\A{A},T)$ s.t. for all spatial points $(\x,\x')\in (\R^{nd})^{2},\ \textcolor{black}{\d(\x,\x')\le 1}, \ 0\le t<s\le T\textcolor{black}{\le 1}$ and $i \in \leftB 1,n\rightB$:
		\begin{equation*}
		|(\btheta_{s,t}(\x)-\btheta_{s,t}(\x'))_i| \leq 
		C\Big( (s-t)^{i-\frac 12}+\d^{2i-1}(\x-\x') \Big).
		\end{equation*}
	\end{lem}
	The flow, $\btheta_{s,t}$ is, somehow, \textcolor{black}{locally} ``almost\textquotedblright \- Lipschitz continuous in space w.r.t. the homogeneous distance $\d$,  up to a time additive term.
	This time contribution is a consequence of the non-Lipschitz continuity of the drif $\gF$.
	The analysis which was already done for $\gF$ Lipschitz continuous  in Proposition 4.1 of \cite{meno:17}, and Appendix A.1 in \cite{chau:hono:meno:18} with different H\"older regularity of $\gF$. Actually, as we consider a \textit{smoother} drift than in \cite{chau:hono:meno:18}, the following lemma \textcolor{black}{can be seen as} a by-product of Lemma 12 therein. 
	For the sake of completeness, we provide the corresponding, and more direct, analysis below.
	
	\begin{proof}
		

		The analysis mainly relies on Gr\"onwall type arguments coupled with suitable mollification techniques, because $\gF $ is not Lipschitz continuous,  and appropriate Young inequalities in order to make the intrinsic scales associated to the spatial variables appear.
		\textcolor{black}{Let $\delta \in \R^n$ be the vector whose entries $\delta_{i} >0$ correspond to the mollification parameter of the drift $\gF_i $ for $i\in \leftB 2,n\rightB $. Namely,
			for all $\ v\in [0,T],\ \z\in \R^{nd} $, $i\in \leftB 2,n\rightB $, we define
			\begin{equation}
			\label{strong_DEF_CONV_ADHOC}
			\gF_i^\delta(v,\z^{i-1:n}):=\gF_i(v,\cdot)\star \rho_{\delta_{i}}(\z)=\int_{\R^{d}}dw \gF_i(v,\z_{i-1},\z_i-w,\z_{i+1} \textcolor{black}{,\hdots} ,\z_n)\rho_{\delta_{i}}(w),
			\end{equation} 
			with  $\rho_{\delta_{i}}(w):=(1/\delta_{i}^{d})\rho\left(w/\delta_{i} \right) $ where $ \rho:\R^{d}\rightarrow \R_{+}$ is a usual mollifier, namely $\rho$ has compact support and $\int_{\R^{d}} \rho(\z)d\z=1 $. }
		Eventually, we write $\gF^\delta(v,\z):=(\gF_1 (v,\z),\gF_2^\delta(v,\z), \textcolor{black}{\hdots},\gF_n^\delta(v,\z)) $. 
		With a slight abuse of notation in the previous definitions, since the first component $\gF_1$ is not mollified. 
		The  sublinearity of $\gF_1$ is actually enough to obtain the desired control.
		\\
		
		To be at the \textit{good} current time scale for the contributions associated with the mollification, we pick $\delta_{i}$ in order to have $C:=C(\A{A},T)>0$ s.t. for all $\z \in \R^{nd}$, $u \in [t,s]$:
		\begin{equation}\label{strong_approxrescaleddrfit}
		\Big|(s-t)^{\frac 12}\T_{s-t}^{-1}\Big(\gF(u,\z)-\gF^\delta(u,\z)\Big)\Big|\leq C(s-t)^{-1}.
		\end{equation}
		By the previous definition of $\gF^\delta$ in \eqref{strong_DEF_CONV_ADHOC}, 
		identity \eqref{strong_approxrescaleddrfit} is equivalent to:
		\begin{equation}\label{strong_eq:controleT1}
		\textcolor{black}{\sum_{i=2}^n (s-t)^{\frac 12-i}  \delta_{i}^{\frac{2i-2}{2i-1}} 
			\leq C(s-t)^{-1}.}
		\end{equation}
		Hence, we choose from now on, for each $i\in \leftB 2,n \rightB $:
		\begin{equation}\label{strong_choice_deltaij}
		\delta_{i}= (s-t)^{(i-\frac 32)\frac{2i-1}{2i-2}}.
		\end{equation}
		
		Next, let us control the last components of the flow.
		By the definition of $\btheta_{s,t}$ in \eqref{strong_DYN_DET_SMOOTH}, we get:
		\begin{eqnarray*}
			&&|(\btheta_{s,t}(\x)-\btheta_{s,t}(\x'))_n|
			\nonumber \\
			&\le& |(\x-\x')_n|
			+
			\int_{t}^s dv \Big ( |\gF_n^\delta (v,\btheta_{v,t}(\x))- \gF_n^\delta (v,\btheta_{v,t}(\x'))|
			\nonumber \\
			&&+ |\gF_n^\delta (v,\btheta_{v,t}(\x))- \gF_n (v,\btheta_{v,t}(\x))|
			+ |\gF_n^\delta (v,\btheta_{v,t}(\x'))- \gF_n (v,\btheta_{v,t}(\x'))|
			\Big) 
			\\
			&
			\leq& 
			|(\x-\x')_n|
			+ C
			\int_{t}^s dv \Big ( 
			\big  |\big (\btheta_{v,t}(\x)-\btheta_{v,t}(\x') \big )_{n-1} \big|
			+   \delta_{n}^{-1+\frac{2n-2}{2n-1}} \big  |\big (\btheta_{v,t}(\x)-\btheta_{v,t}(\x') \big )_n \big|
			\Big)
			+(s-t)  \delta_{n}^{\frac{2n-2}{2n-1}},
		\end{eqnarray*}
		observing for the last inequality that since $\beta_n>(2n-2)/(2n-1) $ and $\delta_n $ is meant to be small, $\delta_{n}^{\beta_n}\le \delta_n^{(2n-2)/(2n-1)} $.
		
		Hence by Gr\"onwall's lemma, we get:
		\begin{eqnarray}\label{strong_ineq_theta_x_x_n}
		\!\!&\!\!\!\!&\!\!|(\btheta_{s,t}(\x)-\btheta_{s,t}(\x'))_n|
		\nonumber \\
		\!\!&\!\!\le\!\!&\!\! C \exp \Big(C (s-t)  \delta_{n}^{-1+\frac{2n-2}{2n-1}} \Big)
		\Big(|(\x-\x')_n| + (s-t)  \delta_{n}^{\frac{2n-2}{2n-1}}  
		+\int_{t}^s dv 
		\big  |\big (\btheta_{v,t}(\x)-\btheta_{v,t}(\x') \big )_{n-1} \big| \Big)
		\nonumber \\
		\!\!&\!\!\le\!\!&\!\! C \exp \Big(C (s-t)^{\frac 1 2} \Big)
		\Big(|(\x-\x')_n| + (s-t)^{n-\frac 12} 
		+\int_t^s dv \big  |\big (\btheta_{v,t}(\x)-\btheta_{v,t}(\x') \big )_{n-1} \big|  \Big),
		\end{eqnarray}
		using 
		\textcolor{black}{\eqref{strong_choice_deltaij}} for the last inequality.
		For the $(n-1)^{\rm {th}}$ component, the situation is quite different in the sense that we have to handle the non-Lipschitz continuity of $\gF^\delta_{n-1}$ in its $n^{\rm {th}}$ variable. 
		Write:
		\begin{eqnarray}\label{strong_ineq_theta_x_x_n1}
		\!\!&\!\!\!\!&\!\!|(\btheta_{s,t}(\x)-\btheta_{s,t}(\x'))_{n-1}|
		\nonumber \\
		\!\!&\!\!\le\!\!&\!\! C \exp \Big (C (s-t)  \delta_{n-1}^{-1+\frac{2(n-1)-2}{2(n-1)-1}} \Big)
		\Big(|(\x-\x')_{n-1}| + (s-t)  \delta_{n-1}^{\frac{2(n-1)-2}{2(n-1)-1}}  
		\nonumber \\
		\!\!&\!\!\!\!&\!\!
		+\int_{t}^sdv  \big \{
		\big  |\big (\btheta_{v,t}(\x)-\btheta_{v,t}(\x') \big )_{n-2} \big| +  |\big (\btheta_{v,t}(\x)-\btheta_{v,t}(\x') \big )_{n} \big|^{\textcolor{black}{\beta_n}}  \big \}\Big)
		\nonumber \\
		\!\!&\!\!\le\!\!&\!\! C \exp(C (s-t)^{\frac 1 2} )
		\Big(|(\x-\x')_{n-1}| + (s-t)^{n-\frac 32} 
		+\int_{t}^sdv \bigg \{
		\big  |\big (\btheta_{v,t}(\x)-\btheta_{v,t}(\x') \big )_{n-2} \big| 
		\nonumber \\
		\!\!&\!\!\!\!&\!\!
		+|(\x-\x')_n|^{\textcolor{black}{\beta_n}} + (v-t)^{\textcolor{black}{\beta_n(n-\frac 12 )}}
		+
		\Big (\int_t^vdw\big  |\big (\btheta_{w,t}(\x)-\btheta_{w,t}(\x') \big )_{n-1} \big|  \Big )^{\textcolor{black}{\beta_n}}\bigg \} \bigg)
		\nonumber \\
		\!\!&\!\!\le\!\!&\!\! C \exp(C (s-t)^{\frac 1 2} )
		\bigg(|(\x-\x')_{n-1}| + (s-t)^{n-\frac 32} 
		+\int_{t}^s dv\bigg \{
		\big  |\big (\btheta_{v,t}(\x)-\btheta_{v,t}(\x') \big )_{n-2} \big| 
		\nonumber \\
		\!\!&\!\!\!\!&\!\!
		+|(\x-\x')_n|^{\textcolor{black}{\frac{2n-3}{2n-1}}} 
		+
		\Big (\int_t^vdw \big  |\big (\btheta_{w,t}(\x)-\btheta_{w,t}(\x') \big )_{n-1} \big|  \Big )^{\textcolor{black}{\beta_n}}\bigg \} \bigg),
		\end{eqnarray}
		from our choice of $\delta_{n-1} $ in \eqref{strong_choice_deltaij} for the second inequality. We also exploited for the last inequality that, since under \A{T${}_\beta $}, $\beta_n>(2n-2)/(2n-1) $, ${\beta_n(n-\frac 12)}>n-1>n-3/2 $ and $0\le t<s\le T $ where $T$ is \textit{small}, then $(v-t)^{\beta_n(n-1/2)}\le (v-t)^{n-3/2} $. Also, since $\d(\x,\x')\le 1 $, the same arguments yield $|(\x-\x')_n|^{\beta_n}\le |(\x-\x')_n|^{(2n-2)/(2n-1)}\le  |(\x-\x')_n|^{(2n-3)/(2n-1)} $.
		
		From  \eqref{strong_ineq_theta_x_x_n1}, which still holds true replacing $s$ by any $\tilde s\in [t,s] $, we deduce that taking the supremum over $\tilde s\in [t,s] $:
		\begin{eqnarray*}
			&&\sup_{\tilde s\in [t,s]}|(\btheta_{\tilde s,t}(\x)-\btheta_{\tilde s,t}(\x'))_{n-1}|
			\nonumber \\
			&\le& C \exp(C ( s-t)^{\frac 1 2} )
			\Big(|(\x-\x')_{n-1}| + ( s-t)^{n-\frac 32} 
			+\int_{t}^{ s} dv \bigg \{
			\big  |\big (\btheta_{v,t}(\x)-\btheta_{v,t}(\x') \big )_{n-2} \big| 
			\nonumber \\
			&&
			+ |(\x-\x')_n|^{\textcolor{black}{\frac{2n-3}{2n-1}}} 
			+
			\Big (\int_t^vdw\big  |\big (\btheta_{w,t}(\x)-\btheta_{w,t}(\x') \big )_{n-1} \big|  \Big )^{\textcolor{black}{\beta_n}}\bigg \} \bigg).
		\end{eqnarray*} 
		Taking then the supremum in \textcolor{black}{$w\in [t,s]$ in the above integral}, we obtain:
		\begin{eqnarray} 
		&&\sup_{\tilde s\in [t,s]}|(\btheta_{\tilde s,t}(\x)-\btheta_{\tilde s,t}(\x'))_{n-1}| 
		\nonumber \\
		&\le& C \exp(C (s-t)^{\frac 1 2} )
		\Big(|(\x-\x')_{n-1}| + (s-t)^{n-\frac 32} 
		+\int_{t}^s dv
		\big  |\big (\btheta_{v,t}(\x)-\btheta_{v,t}(\x')) \big )_{n-2} \big| 
		\nonumber \\
		&&
		+ |(\x-\x')_n|^{\frac{2n-3}{2n-1}}  +
		\sup_{w\in [t,s]} |\big (\btheta_{w,t}(\x)-\btheta_{w,t}(\x')) \big )_{n-1}\big|^{\textcolor{black}{\beta_n}} \textcolor{black}{(s-t)^{\beta_n+1}} \Big )  \bigg)
		.\label{FLOW_SUP_1}
		\end{eqnarray}
		From Young's inequality we now derive:
		\begin{eqnarray*}
			\sup_{w\in [t,s]} |\big (\btheta_{w,t}(\x)-\btheta_{w,t}(\x')) \big )_{n-1}\big|^{\textcolor{black}{\beta_n}} \textcolor{black}{(s-t)^{\beta_n+1}}&\le&\frac 12 \sup_{w\in [t,s]} |\big (\btheta_{w,t}(\x)-\btheta_{w,t}(\x')) \big )_{n-1}\big|(s-t)+C (s-t)^{\frac{1}{1-\beta_n}}\\
			&\le & \frac 12\sup_{w\in [t,s]} |\big (\btheta_{w,t}(\x)-\btheta_{w,t}(\x')) \big )_{n-1}\big|
			+C (s-t)^{n-\frac 32} ,
		\end{eqnarray*}
		recalling for the last inequality that $s-t\le 1 $, and since $\beta_n>(2n-2)/(2n-1), 1-\beta_n<1/(2n-1)$, we also have $(s-t)^{1/(1-\beta_n)}<(s-t)^{2n-1}<(s-t)^{n- 3/2} $. Plugging the above control into \eqref{FLOW_SUP_1}, we obtain:
		\begin{eqnarray}
		&&\frac 12\sup_{\tilde s\in [t,s] }|(\btheta_{\tilde s,t}(\x)-\btheta_{\tilde s,t}(\x'))_{n-1}|
		\\
		&\le& C \exp(C (s-t)^{\frac 1 2} )
		\Big(|(\x-\x')_{n-1}| + (s-t)^{n-\frac 32} 
		+\int_{t}^s dv
		\big  |\big (\btheta_{v,t}(\x)-\btheta_{v,t}(\x')) \big )_{n-2} \big| 
		\nonumber 
		+ |(\x-\x')_n|^{\frac{2n-3}{2n-1}}  
		\Big)\label{strong_I_N_1_2}.
		\end{eqnarray}
		
		We explicitly see from \eqref{strong_I_N_1_2} that each entry of the difference of the starting points appears at its intrinsic scale for the homogeneous distance $\d$ \textcolor{black}{introduced in \eqref{DIST_HOMO}}.
		
		Plugging the above inequality into \eqref{strong_ineq_theta_x_x_n} we derive:
		\begin{eqnarray*}
			&&|(\btheta_{s,t}(\x)-\btheta_{s,t}(\x'))_n|
			\nonumber \\
			\!\!&\!\!\leq\!\!& \!\!C \exp \Big(C (s-t)^{\frac 1 2} \Big)
			\Big(|(\x-\x')_n| + (s-t)^{n-\frac 12} + |(\x-\x')_{n-1}| (s-t) 
			\nonumber \\
			\!\!&\!\!\!\!&\!\!+ 
			|(\x-\x')_n|^{\frac{2n-3}{2n-1}} (s-t) +\! \int_{t}^s\!\! dv \int_{t}^v\! dw
			\big  |\big (\btheta_{w,t}(\x)-\btheta_{w,t}(\x')) \big )_{n-2} \big|   \Big)\nonumber\\
			\!\!&\!\!\leq\!\!& \!\!C \exp \Big(C (s-t)^{\frac 1 2} \Big)
			\Big(|(\x-\x')_n| + (s-t)^{n-\frac 12} + |(\x-\x')_{n-1}|^{\frac{2n-1}{2n-3}}  
			+ \!\int_{t}^s\!\! dv \int_{t}^v \!dw
			\big  |\big (\btheta_{w,t}(\x)-\btheta_{w,t}(\x')) \big )_{n-2} \big|   \Big),
		\end{eqnarray*}
		using again the Young inequalities $|(\x-\x')_n|^{(2n-3)/(2n-1)}(s-t)\le C(|(\x-\x')_n|+(s-t)^{n- 1/2}) $ and $|(\x-\x')_{n-1}|(s-t)\le C\big(|(\x-\x')_{n-1}|^{(2n-1)/(2n-3)}+(s-t)^{n- 1/2}\big) $ for the last inequality.
		Iterating the procedure,  we get:
		\begin{equation}\label{strong_CTR_N}
		|(\btheta_{s,t}(\x)-\btheta_{s,t}(\x'))_n|
		\le C
		\Bigg ( (s-t)^{n-\frac 12}+\sum_{j=2}^n |(\x-\x')_j|^{\frac{2n-1}{2j-1}}
		+ \int_ t^{v_n=s}\!\!\! dv_{n-1} \hdots \int_t^{v_{2}} \!\!\! dv_{1} \big| \big(\btheta_{v_1,t}(\x')-\btheta_{v_1,t}(\x)\big)_{1}\big|\Bigg). \nonumber
		\end{equation}
		Anagolously, for $i\in \leftB 2,n\rightB  $, we obtain:
		\begin{equation}\label{strong_CTR_I}
		|(\btheta_{s,t}(\x)-\btheta_{s,t}(\x'))_i|
		\le C
		\Bigg ( (s-t)^{i-\frac 12}+\sum_{j=2}^n |(\x-\x')_j|^{\frac{2i-1}{2j-1}}
		+ \int_ t^{v_i=s}\!\!\! dv_{i-1} \hdots \int_t^{v_{2}} \!\!\! dv_{1} \big| \big(\btheta_{v_1,t}(\x')-\btheta_{v_1,t}(\x)\big)_{1}\big|\Bigg). \nonumber
		\end{equation}
		\begin{REM}\label{strong_EXT_TEMP_GR}
			Observe that equations \eqref{strong_CTR_N} and \eqref{strong_CTR_I} are available for any fixed time $s\in [t,T] $. 
		\end{REM}

		The first term, i.e. for $i=1 $ is controlled slightly differently. In other words, for any $\tilde s\in [t,s] $, write:
		\begin{equation*}
		\label{strong_CTR_1}
		|(\btheta_{\tilde s,t}(\x)-\btheta_{\tilde s,t}(\x'))_1|
		\le |(\x-\x')_1|+C \sum_{j=1}^n\int_t^{\tilde s}dv |(\btheta_{v,t}(\x)-\btheta_{v,t}(\x'))_j|^{\textcolor{black}{\beta_j}},
		\end{equation*}
		which in turn implies from \eqref{strong_CTR_I}, Remark \ref{strong_EXT_TEMP_GR} and convexity inequalities:
		\begin{eqnarray*}
			&&\sup_{\tilde s\in [t,s]}|(\btheta_{\tilde s,t}(\x)-\btheta_{\tilde s,t}(\x'))_1|
			\nonumber\\
			&\le& |(\x-\x')_1|+C\Big( (s-t) \sup_{v\in [t,s]}|(\btheta_{v,t}(\x)-\btheta_{v,t}(\x'))_1|^{\beta_1} 
			+\sum_{j=2}^n\int_t^sdv |(\btheta_{v,t}(\x)-\btheta_{v,t}(\x'))_j|^{\beta_j}\Big)\nonumber\\
			&\le &|(\x-\x')_1|+\bigg( (s-t) \Big[\frac 12 \sup_{v\in [t,s]}|(\btheta_{\tilde s,t}(\x)-\btheta_{\tilde s,t}(\x'))_1|+ C\Big] \nonumber\\
			&&
			+ C\sum_{j=2}^n  (s-t) \Big( (s-t)^{j-\frac 12}+\sum_{k=2}^n |(\x-\x')_k|^{\frac{2j-1}{2k-1}}
			+(s-t)^{j-1} \sup_{v\in [t,s]}|(\btheta_{v,t}(\x)-\btheta_{v,t}(\x'))_1|\Big)^{\textcolor{black}{\beta_j}}\bigg).
		\end{eqnarray*}
		Recalling $\beta_j>(2j-2)/(2j-1) $ and $0\le t<s\le T\le 1 $, $\d(\x,\x')\le 1 $, we get:
		\begin{eqnarray}
		&&\frac 12 \sup_{\tilde s\in [t,s]}|(\btheta_{\tilde s,t}(\x)-\btheta_{\tilde s,t}(\x'))_1|
		\nonumber 
		\\
		&\le &|(\x-\x')_1|
		\\
		&&+C\bigg( (s-t) 
		+\!\sum_{j=2}^n  (s-t) \Big( (s-t)^{\frac 1 2} \!+\sum_{k=2}^n |(\x-\x')_k|^{\frac{1}{2k-1}}
		+(s-t)^{(j-1)\textcolor{black}{\beta_j}} \!\!\sup_{v\in [t,s]}|(\btheta_{v,t}(\x)-\btheta_{v,t}(\x'))_1|^{\textcolor{black}{\beta_j}}\Big)\bigg).\nonumber
		\label{strong_THE_CTR_1_AFTER_YOUNG}
		\end{eqnarray}
		Write now from the Young inequality:
		\begin{equation*}
		C (s-t)^{1+(j-1)\textcolor{black}{\beta_j}} \sup_{v\in [t,s]}|(\btheta_{v,t}(\x)-\btheta_{v,t}(\x'))_1|^{\textcolor{black}{\beta_j}}\le C(s-t)+\frac 14 \sup_{v\in [t,s]}|(\btheta_{v,t}(\x)-\btheta_{v,t}(\x'))_1|.
		\end{equation*}
		We eventually derive from \eqref{strong_THE_CTR_1_AFTER_YOUNG} that:
		\begin{equation*}
		\sup_{v\in [t,s]}|(\btheta_{v,t}(\x)-\btheta_{v,t}(\x'))_1|\le C\big((s-t)^{\frac 12}+\d(\x,\x')\big),
		\end{equation*}
		which gives the statement for $i=1$.
		Plugging now this inequality into \eqref{strong_CTR_I}, we get for each $i\in \leftB 2,n \rightB  $:
		\begin{eqnarray*}
			|(\btheta_{s,t}(\x)-\btheta_{s,t}(\x'))_i|
			&\le& C
			\big( (s-t)^{i-\frac 12}+\d^{2i-1}(\x,\x')+(s-t)^{i-1}\sup_{v\in [t,s]}|(\btheta_{v,t}(\x)-\btheta_{v,t}(\x'))_1|\big)
			\nonumber \\
			&\le& C\Big( (s-t)^{i-\frac 12}+\d^{2i-1}(\x,\x')+(s-t)^{i-1}\big( (s-t)^{\frac 12}+\d(\x,\x')\big)\Big)\\
			&\le& C\big( (s-t)^{i-\frac 12}+\d^{2i-1}(\x,\x')\big),
		\end{eqnarray*}
		using again for the last identity the Young inequality to derive that $(s-t)^{i-1}\d(\x,\x') \le C\big((s-t)^{i- 1/2}+\d^{2i-1}(\x,\x') \big) $. The proof is complete.
	\end{proof}
	
	\subsection{Sensitivity  results for  the mean: final proof of Lemma \ref{strong_lem_mean_covariance}}
	\label{strong_sec:ProofLemma_flow}
	
	Again through the analysis, we assume w.l.o.g. that $\d(\x,\x')\le 1$. 
	The control is done with a distinction of two contributions to handle.
	\begin{equation}\label{strong_decomp_tilde_m_theta}
	\m_{s,t}^{\x}(\x')-  \btheta_{s,t}(\x')=[ \m_{s,t}^{\x}(\x')-  \btheta_{s,t}(\x)]+   [\btheta_{s,t}(\x)-  \btheta_{s,t}(\x')].
	\end{equation}
	By the \textit{proxy} definition in \eqref{strong_FROZ_MOL_FOR_NO_SUPER_SCRIPTS}, we deduce that the mean value of $\tilde  \X^{m, \bxi}_v$, $ \m_{v,t}^{\bxi}$ is s.t.
	\begin{equation}
	\m_{s,t}^{\x}(\x')-  \btheta_{s,t}(\x)= \x'- \x+ \int_ t^s dv
	D\gF(v,\btheta_{v,t}(\x))[\m_{v,t}^\x(\x')-\btheta_ {v,t}(\x)].
	\end{equation}
	The sub-triangular structure of $D\gF$ yields that  for each $i \in \leftB2,n\rightB$:
	\begin{equation*}
	\big (\m_{s,t}^{\x}(\x')-  \btheta_{s,t}(\x) \big )_i= \x'_i- \x_i +\int_ t^s dv
	D_{i-1}\gF_i(v,\btheta_{v,t}(\x))[\m_{v,t}^\x(\x')_{i-1}-\btheta_ {v,t}(\x)_{i-1}]
	.
	\end{equation*}
	Also, since  $\m_{v,t}^{\x}(\x')_1= \x_1'+\int_t^s dv \gF_1(v,\btheta_{v,t}(\x))$, so we obtain that $[\m_{v,t}^\x(\x')_{1}-\btheta_ {v,t}(\x)_{1}]=\x'_1-\x_1 $,  we  then obtain by iteration that:
	\begin{equation*}
	\big (\m_{s,t}^{\x}(\x')-  \btheta_{s,t}(\x) \big )_i 
	= \x'_{i}-\x_{i}
	+ \sum_{k=2}^i \Big [ \int_ t^{v_i=s}\!\!\! dv_{i-1} \hdots \int_t^{v_{k}} \!\!\! dv_{k-1} \prod_{j=k}^{i}   D_{j-1}\gF_{j}(v_j,\btheta_{v_j,t}(\x)) \Big ] [\x'_{k-1}-\x_{k-1}],
	\end{equation*}
	with the convention that for $i=1$, $ \sum_{k=2}^i=0$. From the above control, equation \eqref{strong_decomp_tilde_m_theta} and the dynamics of the flow, and because  the starting points are the same,  the contributions involving differences of the spatial points ($\x'-\x)$ or flows only appear in iterated time integrals, we obtain: 
	\begin{eqnarray*}
		&&|\big (\m_{s,t}^{\x}(\x')-  \btheta_{s,t}(\x') \big )_i|
		\\
		&\le &
		\bigg |\sum_{k=2}^i \Big [ \int_ t^{v_i=s}\!\!\! dv_{i-1} \hdots \int_t^{v_{k}} \!\!\! dv_{k-1} \prod_{j=k}^{i}   D_{j-1}\gF_{j}(v_j,\btheta_{v_j,t}(\x)) \Big ] [\x'_{k-1}-\x_{k-1}] \bigg |
		\\
		&&
		+
		\int_{t}^s  |\gF_i(v,\btheta_{v,t}(\x))- \gF_i (v,\btheta_{v,t}(\x'))| dv
		\\
		&\leq& C \Big(\sum_{k=2}^{i-1} (s-t)^{i-k} |\x_k-\x_k'|
		+ \int_{t}^s dv  \Big(\sum_{j=i}^n 
		\big  |\big (\btheta_{v,t}(\x)-\btheta_{v,t}(\x') \big )_{j} \big|^{\textcolor{black}{\beta_j}}
		+\big  |\big (\btheta_{v,t}(\x)-\btheta_{v,t}(\x') \big )_{i-1} \big| \Big)
		\Big).
	\end{eqnarray*}
	We derive from the previous Lemma \ref{strong_lem_theta_theta_bis} (control of the flows) recalling again that $\beta_j>(2j-2)/(2j-1) $ and $\d(\x,\x')\le 1, 0\le t<s\le T\le 1 $:
	\begin{eqnarray*}
		\!\!&\!\!&\!\!|\big (\m_{s,t}^{\x}(\x')-  \btheta_{s,t}(\x') \big )_i|
		\nonumber \\
		\!\!&\!\!\le\!\!&\!\! 
		C \bigg(\sum_{k=2}^{i-1} (s-t)^{i-k} |\x_k-\x_k'|
		\!+ 
		(s-t)^{\frac{2i-2}{2}+1}
		\!+ \d^{2i-2}(\x,\x')(s-t)
		\!+ \big((s-t)^{(i-1)-\frac 12} \!+\d^{2(i-1)-1}(\x,\x')\big)(s-t) 
		\bigg).
	\end{eqnarray*}
	In particular, for $s=t_0=t+c_0\d^2(\x,\x')$ with $c_0<1$, the previous equation yields:
	\begin{equation*}
	|\big (\m_{t_0,t}^{\x}(\x')-  \btheta_{t_0,t}(\x') \big )_i|
	\le 
	C   \Big( c_0  \d^{2i-1}(\x,\x')  
	+ (c_0^{i}+c_0)\d^{2i-1}(\x,\x')+(c_0^{i- \frac 12}+c_0)\d^{2i-1}(\x,\x') \Big ),
	\end{equation*}
	using again $\d(\x,\x') \leq 1$ for the middle term. After summing and by convexity inequalities, we eventually deduce:
	\begin{equation*}
	\d\big (\m_{s,t}^{\x}(\x'),  \btheta_{s,t}(\x') \big ) \leq  Cc_0^{\frac 1{2n-1}} \d(\x,\x').
	\end{equation*}
	This concludes the proof of Lemma \ref{strong_lem_mean_covariance}. \hfill $\square $

	\mysection{Parametrix expansion with different freezing points}
	\label{THE_APP_PARAM}
	In this section we show how the parametrix expansion \eqref{strong_INTEGRATED_DIFF_BXI_UNICITE_FORTE} involving different freezing points can be derived. This can actually been done from the Duhamel formulation up to an additional discontinuity term. Restarting from  \eqref{strong_eq:repPDE} we can indeed rewrite from the Markov property that for given $(t,\x')\in [0,T]\times \R^{nd} $ and any $ r\in (t,T], \bxi'\in \R^{nd}$:
	\begin{eqnarray}
	u(t,\x')&=&\Big[ \tilde P_{r,t}^{\bxi'} u(r,\cdot )\Big](\x')+ \int_t^r ds \Big[ \tilde P_{s,t}^{\bxi'} f(s,\cdot)\Big](\x')+\int_t^{r} ds \Big[\tilde P_{s,t}^{\bxi'}\Big((L_s-\tilde L_s^{\bxi')})u\Big)(s,\cdot)\Big](\x').
	\end{eqnarray}
	Differentiating the above expression in $r\in (t,T]$ yields for any $\bxi'\in \R^{nd} $:
	\begin{eqnarray}
	\label{DER_DUHAMEL}
	0=\partial_r \Big[  \tilde P_{r,t}^{\bxi'} u(r,\cdot )\Big](\x')+  \Big[ \tilde P_{r,t}^{\bxi'} f(r,\cdot)\Big](\x')+ \Big[\tilde P_{r,t}^{\bxi'}\Big((L_r-\tilde L_r^{\bxi')})u\Big)(r,\cdot)\Big](\x').
	\end{eqnarray}
	Denoting by $t_0\in (t,T]$ the time at which we change the freezing point and
	integrating \eqref{DER_DUHAMEL} on $[t,t_0]$ for a first given $\bxi' $ and between $[t_0 ,T]$ with a possibly different $\tilde \bxi'  $ yields:
	\begin{eqnarray}
	0&=&\Big[\tilde P_{t_0,t}^{\bxi'} u(t_0, \cdot)\Big](\x')-u(t,\x')+\int_{t}^{t_0}ds \Big[  \tilde P_{s,t}^{\bxi'}f(s,\cdot)\Big](\x')+\int_t^{t_0} ds \Big[ \tilde P_{s,t}^{\bxi'}\Big((L_s-\tilde L_s^{\bxi'})u\Big)(s,\cdot)\Big](\x')\notag\\
	&&+\Big[\tilde P_{T,t}^{\tilde \bxi'} u(T, \cdot)\Big](\x')-\Big[\tilde P_{t_0,t}^{\tilde \bxi'}u(t_0,\cdot)\Big] (\x')+\int_{t_0}^Tds \Big[ \tilde P_{s,t}^{\tilde \bxi'}f(s,\cdot)\Big](\x')
	+\int_{t_0}^T ds \Big[ \tilde P_{s,t}^{\tilde\bxi'}\Big((L_s-\tilde L_s^{\tilde \bxi'})u\Big)(s,\cdot)\Big](\x')\notag.
	\end{eqnarray}
	Recalling that $u(T,\cdot)=0 $ (terminal condition), 
	the above equation rewrites:
	\begin{eqnarray*}
		&&u(t,\x')
		\nonumber \\
		&=&\int_{t}^Tds\bigg(\I_{s\le t_0} \Big[ \tilde P_{s,t}^{ \bxi'}f(s,\cdot)\Big](\x') +\I_{s> t_0} \Big[ \tilde P_{s,t}^{\tilde \bxi'}f(s,\cdot)\Big](\x')\bigg)
		+\Big[\tilde P_{t_0,t}^{\bxi'} u(t_0, \cdot)\Big](\x')-\Big[\tilde P_{t_0,t}^{\tilde \bxi'}u(t_0,\cdot)\Big] (\x')\notag\\
		&&+\int_{t}^T ds\bigg(\I_{s\le t_0}\Big[ \tilde P_{s,t}^{\bxi'}\Big((L_s-\tilde L_s^{ \bxi'})u\Big)(s,\cdot)\Big](\x') +\I_{s>t_0}\Big[ \tilde P_{s,t}^{\tilde\bxi'}\Big((L_s-\tilde L_s^{\tilde \bxi'})u\Big)(s,\cdot)\Big](\x')\bigg).
	\end{eqnarray*}
	We see that for $\bxi ' \neq\tilde  \bxi ' $ we have an additional discontinuity term deriving from the change of freezing point along the time variable. 
	Eventually, the above equation precisely gives \eqref{strong_INTEGRATED_DIFF_BXI_UNICITE_FORTE}, recalling that for $t_0=t+c_0 |(\x_i-z)|^{2/(2i-1)}$, $\I_{s\le t_0}=\I_{{\mathcal S}_i} $.

	\mysection{Auxiliary results concerning the multi-scale Gaussian densities, their derivatives and some related objects}
	\label{RES_DENS_SUM}
	
	In order to be self contained, we gather in this section the proof of some results related  to the Gaussian dynamics in \eqref{strong_FROZ_MOL_FOR_NO_SUPER_SCRIPTS}. Namely, we provide a complete proof of Proposition \ref{strong_PROP_SCALE_COV} and some auxiliary related results used throughout the previous proofs.  We here freely use the notations of Section
	\ref{strong_SEC_GAUSS_DENS}.

	\subsection{About the objects appearing in the multi-scale density}
	
	\subsubsection{Good scaling properties of the covariance matrix: proof of Proposition \ref{strong_PROP_SCALE_COV}}
	\label{PROOF_PROP_4}
	We recall the correspondence between the notations of \cite{dela:meno:10} and those of the current article.
	
	\begin{trivlist}
		\item[-] Notations and Assumptions from \cite{dela:meno:10}. Consider the Gaussian process with dynamics
		\begin{equation}\label{ALINEAR}\tag{LIN}
		d\mathbf G_t =\mathbf L_t dt +B\Sigma_t dW_t
		\end{equation}
		where $ (\Sigma_t)_{t\in[ 0,T]}$ is a measurable deterministic $\R^d\otimes \R^d $ valued family s.t. $\mathbf A_t:= \Sigma_t \Sigma_t^* $ has uniformly non-degenerate spectrum, i.e. there exists $\Lambda \ge 1 $ s.t. for any $t\in [0,T],\ {\rm Spec}(\mathbf A_t)\in [\Lambda^{-1},\Lambda] $,  and the measurable deterministic $\R^{nd}\otimes \R^{nd}$ valued family $(\mathbf L_t)_{t\in[ 0,T]} $ is such that for any $t\in [0,T] $:
		\begin{equation}\label{GEN_STRUCT_L}
		\mathbf L_t=\left(\begin{array}{cccccc} U_t^{1,1}& U_t^{1,2}&\cdots & \cdots&U_t^{1,n}\\
		\alpha_t^{1} & U_t^{2,2} &\cdots & \cdots& U_t^{2,n}\\
		\mathbf {0}_{d,d}&\alpha_t^{2}&U_t^{3,3}& \cdots &U_{t}^{3,n}\\
		\vdots & \ddots& \ddots &\ddots & \vdots\\
		\mathbf {0}_{d,d}& \cdots & \mathbf {0}_{d,d}&\alpha_t^{n-1}& U_{t}^{n,n}
		\end{array}\right),
		\end{equation}
		where the $(\alpha_{t}^{i})_{i\in \leftB 1, n-1\rightB} $  and $(U_{t}^{i,j})_{i\in \leftB 1,n\rightB, j\in \leftB i,n\rightB} $ are $\R^d\otimes \R^d $ valued.
		
		This is assumption \A{$\mathbf A^{{\rm linear}}$} in \cite{dela:meno:10}. Proposition 3.4 of that reference states that, whenever for all $i\in \leftB 1,n-1\rightB $ and $t\in [0,T] $, $\alpha_t^i $ belongs to $\mathcal E_i $ (closed convex subset of $GL_d(\R) $),  there exists a constant $c\ge 1$ depending on $\mathcal E_i $, $\Lambda, \kappa $ s.t. $\max_{i\in \leftB 1,n-1\rightB}\sup_{t\in [0,T]}|\alpha_t^i|\le \kappa $, $n$, $d$ such that the Gaussian process $(\mathbf G_t)_{t\in[0,T]} $   introduced in \eqref{ALINEAR} satisfies a good scaling property with constant $c$ in the sense of Definition 3.2 of \cite{dela:meno:10}. Precisely, denoting by $\big(\gR(s,t)\big)_{0\le t,s\le T } $ the resolvent matrix associated with $(\mathbf L_t)_{t\ge 0} $, the covariance matrix 
		$$\mathbf K_t=\int_0^tds \gR_{t,s}B \mathbf A_s B^*\gR_{t,s}^*  $$
		of the random variable $\mathbf G_t $ satisfies that for any $\y\in \R^{nd} $, 
		$$c^{-1}t^{-1}|\T_t \y|^2\le \langle \mathbf K_t\y,\y\rangle \le c t^{-1}|\T_t \y|^2, $$
		or equivalently, \textcolor{black}{for all $t\in (0,T] $}: 
		$$ c^{-1}t|\T_t^{-1} \y|^2\le \langle \mathbf K_t^{-1}\y,\y\rangle \le c t|\T_t^{-1} \y|^2.$$
		\item[-] Derivation of the Proposition \ref{strong_PROP_SCALE_COV} from the previous results of \cite{dela:meno:10}. From the dynamics of the process $(\tilde \X_v^{(\tau,\bxi)})_{v\in [t,T]} $ given in \eqref{strong_FROZ_MOL_FOR_NO_SUPER_SCRIPTS} and starting from $\x $ at time $t$, see also the associated integrated expression in \eqref{strong_INTEGRATED_FLOW}, it can be seen that the covariance matrix of the random variable $\tilde \X_v^{(\tau,\bxi)} $ writes, with the notations of Section \ref{strong_SEC_GAUSS_DENS}:
		$$\tilde \K_{v,t}^{(\tau,\bxi)}:=\int_t^vdu \textcolor{black}{\tilde  \gR}^{(\tau,\bxi)}(v,u) Ba(u,\btheta_{u,\tau}(\bxi))B^*\textcolor{black}{\tilde \gR}^{(\tau,\bxi)}(v,u)^* , 
		$$
		as given in the statement of Proposition 4. This covariance matrix also corresponds to the one of a Gaussian process with dynamics \eqref{ALINEAR} setting for fixed $0\le t<s\le T $ and for any $r\in [0,s-t]$ 
		\begin{equation}\label{PART_R}
		{\mathbf L}_r=\left (\begin{array}{ccccc}\0_{d,d} & \cdots & \cdots &\cdots  & \0_{d,d}\\
		D_{\z_1}\gF_{2}(t+r,\btheta_{t+r,\tau}(\bxi)) & \0_{d,d} &\cdots &\cdots &\0_{d,d}\\
		\0_{d,d} & D_{\z_2} \gF_{3}(t+r,\btheta_{t+r,\tau}^{2:n}(\bxi))& \0_{d,d}& \0_{d,d} &\vdots\\
		\vdots &  \0_{d,d}                     & \ddots & \ddots &\vdots\\
		\0_{d,d} &\cdots &     \0_{d,d}      & D_{\z_{n-1}}\gF_{n}(t+r,\btheta_{t+r,\tau}^{n-1:n}(\bxi)) & \0_{d,d}
		\end{array}\right) ,
		\end{equation}
		and
		$$ \Sigma_r=\sigma(t+r,\btheta_{t+r,\tau}(\bxi)).$$
		Since the resolvent $(\gR(r,v))_{v\in [0,s-t]}$ associated with $(\mathbf L_r)_{r\in [0,s-t]} $ writes $\gR(r,v)= \tilde \gR^{(\tau,\bxi)}(t+r,t+v)$, one readily derives that the covariance matrix $ \mathbf K_{s-t}$ of $\mathbf G_{t-s}=\gR(t-s,0)\x+\int_{0}^{t-s}\gR(t-s,r)B\Sigma_r dW_r$ coincides with $\tilde \K_{s,t}^{(\tau,\bxi)} $. Since $\mathbf K_{s-t}=\tilde \K_{s,t}^{(\tau,\bxi)} $ satisfies a good scaling property, this proves Proposition 4. 
	\end{trivlist}

	\subsubsection{Scaling Properties of the resolvent.}
	We here aim at proving the following control. There exists $\hat C_1:=\hat C_1(\A{A},T)$ s.t. 
	\begin{equation}\label{THE_PROOF_BD_SCALED_RES}
	\forall \bzeta \in \R^{nd}, |\big[\widehat{\tilde  \gR}^{(\tau,\bxi),s,t}(1,0)\big]^* \bzeta|\le\hat C_1 |\bzeta|.
	\end{equation}
	
	We will proceed following the arguments of Lemma 3.6 in \cite{dela:meno:10} specifying how they can apply to the current setting. Let us restart from the previous linear Gaussian dynamics. Namely,
	for fixed $t\in [0,T]$,  let $(\mathbf G_s)_{s\in [0,t]} $ be as in \eqref{ALINEAR} and  introduce the \textit{scaled process} $\hat {\mathbf G}_s^t=t^{1/2}\T_t^{-1} \mathbf G_{s t} ,\ s\in [0,1]$. Intuitively, from the previously established \textit{good scaling property of $\mathbf G $},  all the components of the process $\hat {\mathbf G}^t $ actually evolve at a \textit{macro} scale, i.e. its covariance matrix is of order  one at time $s=1$.
	
	It is then easily checked that $\hat {\mathbf G}_s^t $ satisfies \eqref{ALINEAR}, i.e. $ d\hat {\mathbf G}_s^t=\hat {\mathbf L}_s^t\hat {\mathbf G}_s^t+\hat \Sigma_s^t d\hat W_s^t$ with:
	$$ \hat {\mathbf L}_s^t=t\T_{t}^{-1}\mathbf L_{st}\T_t ,\ \hat \Sigma_s^t=\Sigma_{st},\ \hat W_s^t=t^{-1/2}W_{st},\ s\in [0,1] .$$
	From the general structure of $\mathbf L$ in \eqref{GEN_STRUCT_L}, and the definition of the scale matrix $\T_t$, one easily derives that $ |\hat {\mathbf L}_{s}^t| \le (1\vee T^n) |\mathbf L_{st}|$. More specifically, for the special structure considered in \eqref{PART_R}, i.e. without upper-triangular part, we even get $\hat {\mathbf L}_{s}^t=  \mathbf L_{st}$. In any case, we obtain
	$|\hat {\mathbf L}_{s}^t| \le (1\vee T^n)\kappa $ as soon as $\sup_{s\in [0,1] }| \mathbf L_{st}| \le \kappa$.
	
	It is then clear that denoting by $\hat \gR^t $ the resolvent associated with $ (\hat {\mathbf L}_{s}^t)_{\textcolor{black}{s\in [0,1]}}$ it holds that there exists $\hat C_1:=\hat C_1(\A{A},T) $ s.t. for all $s_0,s_1\in [0,1] $, $|\hat \gR^t(s_1,s_0)|\le \hat C_1 $. On the other hand, direct computations also yield that
	$$\hat \gR^t(s_1,s_0)=\T_t^{-1}\gR(s_1 t,s_0 t) \T_t\iff \gR(s_1 t,s_0 t)=\T_t\hat \gR^t(s_1,s_0)\T_t^{-1} \Longrightarrow [\gR(t,0)]^*=\textcolor{black}{\T_t^{-1}} [\hat \gR^t(1,0)]^*\textcolor{black}{\T_t},$$
	where $\gR$ stands for the resolvent associated with $\mathbf L $. The final bound \eqref{THE_PROOF_BD_SCALED_RES} on the rescaled resolvent associated with the frozen process $((s-t)^{1/2}\T_{s-t}^{-1}\tilde \X_v^{(\tau,\bxi),(t,\x)})_{v\in [t,s]} $ is eventually derived from the same previous correspondence exhibited to prove the good scaling property of Proposition 4 in the previous paragraph.

	\subsection{Control of the H\"older modulus of the frozen Green kernel.}
	\label{SEC_PREUVE_HOLD_MODULUS_G_FROZEN}
	We aim here at proving, with the notations of Section \ref{strong_SECTION_CTR_SENSI} that the H\"older regularity index $\textcolor{black}{\alpha_i^k}$ of $\y_i \mapsto D_{\y_k} \tilde G^{i,\bxi} f(s,\y_i):=D_{\y_k} \tilde G^\bxi f(s,\y)=D_{\y_k}\int_s^T dr\tP_{r,s} f(s,\y)$, with $i\le k$, can be taken equal to $\alpha_i^k=\frac{1-(k- 1/2)(1-\beta_k)}{i- 1/2}  $.
	
	First, we get from \eqref{strong_SMOOTHING_EFFECT} in Lemma \ref{strong_LEMME_SG} that:
	\begin{equation*}
	|D_{\y_k} \tilde G^{i,\bxi} f(s,\y_i)|=|\int_s^T dr\int_{\R^{nd}}d\z D_{\y_k}\tilde p^\bxi(s,r,\y,\z)f(s,\z) |\le C\sum_{j=k}^n\int_{s}^T dr (r-s)^{-(k-\frac 12)+\beta_j(j-\frac 12)}.
	\end{equation*} 
	Recall from Assumption \A{T${}_\beta $} that $\beta_j(j- 1/2)\ge \beta_k(k- 1/2)  $ for $j\in \leftB k,n\rightB $. The above equation thus yields:
	\begin{equation}\label{PREAL_HOLDER_HD_I}
	\begin{split}
	|D_{\y_k} \tilde G^{i,\bxi} f(s,\y_i)|&=|\int_s^T dr \int_{\R^{nd}}d\z D_{\y_k}\tilde p^\bxi(s,r,\y,\z)f(s,\z) |
	\\
	&\le C\int_{s}^T dr (r-s)^{-(k-\frac 12)+\beta_k(k-\frac 12)}\\
	&\le C(T-s)^{1-(k-\frac 12)(1-\beta_k)}.
	\end{split}
	\end{equation} 
	From this first bound we can then first draw a control of the H\"older modulus in the so-called \textit{off-diagonal case} w.r.t. 
	the current time scale of the considered variable $i$ in the oscillator chain. Namely, for \textcolor{black}{$\y_i,\y_i'\in \R^d $} s.t. $(T-s)/2\le  |\y_i-\y_i'|^{1/(i- 1/2)}$, equation \eqref{PREAL_HOLDER_HD_I} readily gives:
	\begin{align}\label{bd_hd}
	|D_{\y_k} \tilde G^{i,\bxi} f(s,\y_i)-D_{\y_k} \tilde G^{i,\bxi} f(s,\y_i')|
	&\le |D_{\y_k} \tilde G^{i,\bxi} f(s,\y_i)|+|D_{\y_k} \tilde G^{i,\bxi} f(s,\y_i')|\notag\\
	&\le 2C(T-s)^{1-(k-\frac 12)(1-\beta_k)}
	\nonumber \\
	&\le \tilde C|\y_i-\y_i'|^{\frac{1-(k-\frac 12)(1-\beta_k)}{i-\frac 12}}.
	\end{align}
	Assuming now that $(T-s)/2 \ge  |\y_i-\y_i'|^{1/(i-1/2}$, we split the integrals as follows. Denoting with a slight abuse of notation $\y':=(\y_{1:i-1},\textcolor{black}{\y_i'},\y_{i+1:n}) $,
	\begin{eqnarray*}
		&&|D_{\y_k} \tilde G^{i,\bxi} f(s,\y_i)-D_{\y_k} \tilde G^{i,\bxi} f(s,\y_i')|\notag\\
		&\le& \Big|\int_s^{s+|\y_i-\y_i'|^{\frac 1{i-\frac 12}}} dr \int_{\R^{nd}}d\z D_{\y_k}\tilde p^\bxi(s,r,\y,\z)f(s,\z) \Big|+\Big|\int_s^{s+|\y_i-\y_i'|^{\frac 1{i-\frac 12}}} dr \int_{\R^{nd}}d\z D_{\y_k}\tilde p^\bxi(s,r,\y',\z)f(s,\z) \Big|\notag\\
		&&+\Big| \int_{0}^1 d\lambda \int_{s+|\y_i-\y_i'|^{\frac 1{i-\frac 12}}}^T\int_{\R^{nd}}d\z D_{\y_k}D_{\y_i} \tilde p^\bxi(s,r,\y'+\lambda (\y-\y'),\z) \cdot (\y_i-\y_i') f(s,\z)\Big|.
	\end{eqnarray*}
	For the two first term, we have a \textit{local off-diagonal regime} within the global \textit{diagonal} one. Applying \eqref{PREAL_HOLDER_HD_I} with $T$ replaced by $s+|\y_i-\y_i'|^{ 1/(i-1/2)} $, we get:
	\begin{eqnarray*}
		&&|D_{\y_k} \tilde G^{i,\bxi} f(s,\y_i)-D_{\y_k} \tilde G^{i,\bxi} f(s,\y_i')|
		\nonumber \\
		&\le& C |\y_i-\y_i'|^{\frac{1-(k-\frac 12)(1-\beta_k)}{i-\frac 12}}
		+\Big| \int_{0}^1\!  d\lambda \int_{s+|\y_i-\y_i'|^{\frac 1{i-\frac 12}}}^T\! \textcolor{black}{dr} \int_{\R^{nd}}\!\!d\z D_{\y_k}D_{\y_i} \tilde p^\bxi(s,r,\y'+\lambda (\y-\y'),\z)  f(s,\z)\Big| \ |\y_i-\y_i'|.
	\end{eqnarray*} 
	Using again (2.19) from the new Lemma \ref{strong_LEMME_SG}, one eventually gets:
	\begin{eqnarray}\label{bd_d}
	&&|D_{\y_k} \tilde G^{i,\bxi} f(s,\y_i)-D_{\y_k} \tilde G^{i,\bxi} f(s,\y_i')|
	\nonumber \\
	&\le& C |\y_i-\y_i'|^{\frac{1-(k-\frac 12)(1-\beta_k)}{i-\frac 12}}
	+C\Big| \sum_{j=k}^n\int_{s+|\y_i-\y_i'|^{\frac 1{i-\frac 12}}}^T dr (r-s)^{-(k-\frac 12)-(i-\frac 12)+\beta_j(j-\frac 12)} \Big| \ |\y_i-\y_i'|\notag\\
	&\le& C |\y_i-\y_i'|^{\frac{1-(k-\frac 12)(1-\beta_k)}{i-\frac 12}}+C\int_{s+|\y_i-\y_i'|^{\frac 1{i-\frac 12}}}^T dr (r-s)^{-(i-\frac 12)-(k-\frac 12)(1-\beta_k)} \ |\y_i-\y_i'|\notag\\
	&\le&C \Big(|\y_i-\y_i'|^{\frac{1-(k-\frac 12)(1-\beta_k)}{i-\frac 12}}+(r-s)^{-(i-\frac 12)+1-(k-\frac 12)(1-\beta_k)}\Big|_{r=s+|\y_i-\y_i'|^{\frac 1{i-\frac 12}}}^{r=T}|\y_i-\y_i'|\Big)\notag\\
	&\le & C|\y_i-\y_i'|^{\frac{1-(k-\frac 12)(1-\beta_k)}{i-\frac 12}}.\label{FINAL_BD_HOLD_G_D}
	\end{eqnarray} 
	Equations \eqref{bd_hd} and \eqref{FINAL_BD_HOLD_G_D}  precisely give the claim.
	Since $D_{\y_k} \tilde G^{i,\bxi} f(s,\y_i)$ is bounded (see \eqref{PREAL_HOLDER_HD_I}), any $\alpha_i^k\le \frac{1-(k- 1/2)(1-\beta_k)}{i- 1/2}  $ actually fits.

	\section{Reverse Taylor formula}
	\label{PROOF_REV_TAYLOR}

	\begin{proof}[Proof of Lemma \ref{Lemme_Taylor_reverse}]
		We assume here, for the sake of simplicity and without loss of generality, that $d=1$ (scalar case). 
		When $d>1$, the proof below can be reproduced componentwise.
		Let us decompose the expression around the variables for which Lemma \ref{strong_LEMME_CTR_HOLDER} applies 
		and those for which it does not.
		Namely, we write:
		\begin{eqnarray} 
		&& D_{l} u(s, \y) -D_{l} u(s, \y_{1:l-1},\btheta_{s,t}^{l:n}(\bxi) )
		\nonumber \\
		&=&\big(D_{l} u(s, \y) -D_{l} u(s, \y_{1:l-1},\btheta_{s,t}^l(\bxi) ,\y_{l+1:n})\big)
		+\big(D_{l} u(s, \y_{1:l-1},\btheta_{s,t}^l(\bxi) ,\y_{l+1:n})-D_{l} u(s, \y_{1:l-1},\btheta_{s,t}^{l:n}(\bxi) )\big)
		\nonumber \\
		&=:&(\Delta_l^1+\Delta_l^2)(s,\y,\bxi)
		\end{eqnarray}
		We readily get:
		\begin{equation}\label{THE_FIRST_CTR_HOLD_APRES}
		|\Delta_l^1(s,\y,\bxi)|\le C [(D_l u)_l(s,\cdot)]_{\alpha_l^l}|(\y-\btheta_{s,t}(\bxi))_l|^{\alpha_l^l}
		\end{equation}
		which is controlled from Lemma \ref{strong_LEMME_CTR_HOLDER}. On the other hand, setting for any $\mathbf a\in \R^{(n-l+1)d} $, $\bXi_{l,\y,s}(\mathbf a):= D_l u(s,\y_{1:l-1},\mathbf a)$, and $\d:=\sum_{k=l+1}^n|(\btheta_{s,t}(\bxi)-\y)_k| $
		we now have for any $\delta_l>0$
		\begin{eqnarray}
		\label{def_mathcal_U}
		\Delta_l^2(s,\y,\bxi) &=&\bXi_{l,\y,s}\big((\btheta_{s,t}(\bxi))_l,\y_{l+1:n}\big)-\bXi_{l,\y,s}\big(\btheta_{s,t}^{l:n}(\bxi)\big)\nonumber \\
		&=&\int_0^1 d \mu \Big(
		\{D_{l} u(s, \y_{1:l-1},(\btheta_{s,t}(\bxi))_l,\y_{l+1:n} ) -D_{l} u(s,\y_{1:l-1},(\btheta_{s,t}(\bxi))_l+\mu\d^{\delta_l},\y_{l+1:n})  \}
		\nonumber \\
		&&+\{D_{l} u(s, \y_{1:l-1},(\btheta_{s,t}(\bxi))_l+\mu\d^{\delta_l },\btheta_{s,t}^{l+1:n}(\bxi)) -D_{l} u(s,\y_{1:l-1}, \btheta_{s,t}^{l:n} (\bxi) )  \}
		\nonumber \\
		&&+\{D_{l} u(s,\y_{1:l-1},(\btheta_{s,t}(\bxi))_l+\mu\d^{\delta_l},\y_{l+1:n})   -D_{l} u(s, \y_{1:l-1},(\btheta_{s,t}(\bxi))_l+\mu\d^{\delta_l },\btheta_{s,t}^{l+1:n}(\bxi))\}\Big)\nonumber \\
		&=:& \sum_{\ell=1}^3 \Delta_{l}^{2,\ell}(s,\y,\bxi) .
		\end{eqnarray}
		The  first two terms can be dealt directly.
		\begin{equation} \label{mathcal_U_L_12}
		|  \Delta_l^{2,1}(s,\y,\bxi)  |+|  \Delta_l^{2,2}(s,\y,\bxi)  |
		\leq 
		2 [(D_l u)_l(s,\cdot)]_{\alpha_l^l} \d^{\delta_l \alpha_l^l}.  
		\end{equation}
		For $\Delta_l^{2,3}(s,\y,\bxi)$, we use an explicit reverse Taylor expansion which yields together with the smoothness assumption of $\gF_i $ in \A{A}:
		\begin{eqnarray} \label{mathcal_U_L_3}
		| \Delta_l^{2,3}(s,\y,\bxi)|
		&=&\d^{-\delta_l} 
		\Big |\Big [ u(s,\y_{1:l-1}, (\btheta_{s,t}(\bxi))_l+\d^{\delta_l},\y_{l+1:n}) -u(s,\y_{1:l-1}, (\btheta_{s,t}(\bxi))_l,\y_{l+1:n}) 
		\nonumber \\
		&&-\big(u(s,\y_{1:l-1}, (\btheta_{s,t}(\bxi))_l+\d^{\delta_l},(\btheta_{s,t}(\bxi))^{l+1:n}) -u(s,\y_{1:l-1}, (\btheta_{s,t}(\bxi))_l,(\btheta_{s,t}(\bxi))^{l+1:n})\big)
		\Big ] \Big |
		\nonumber \\ 
		&\leq& 2
		\|\mathbf D u\|_{\infty}
		\d^{1-\delta_l},
		\end{eqnarray}
		using the Lipschitz property of $u$ w.r.t. the variables $l+1$ to $n $.
		Taking $\delta_l$ s.t. $\delta_l \alpha_l^l=1-\delta_l^l$, which implies that $\delta_l= (1+\alpha_l^l)^{-1}$, gives
		in \eqref{mathcal_U_L_12} and \eqref{mathcal_U_L_3} a global bound of order $2 (\|Du\|_\infty+[(D_l u)_l(s,\cdot)]_{\alpha_l^l})
		\d^{\delta_l \alpha_l^l} $. Since we now recall from Lemma \ref{strong_LEMME_CTR_HOLDER} that  $\alpha_l^l=(1+\eta/4)/(2l-1) $ so that 
		$$\delta_l\alpha_l^l= \frac{1}{\frac{2l-1+1+\frac \eta 4}{2l-1}}\frac{1+\frac{\eta}4}{2l-1}=\frac{1+\frac \eta 4}{2l+\frac \eta 4} =:\zeta_l.$$ 
		We then write from \eqref{def_mathcal_U} and the definiton of $\d$ that:
		\begin{equation*}
		| \Delta_l^{\textcolor{black}{2}}(s,\y,\bxi)|\le 2 (\|Du\|_\infty+[(D_l u)_l(s,\cdot)]_{\alpha_l^l}) \sum_{k=l+1}^n |(\btheta_{s,t}(\bxi)-\y)_k|^{\zeta_l},
		\end{equation*}
		which together with \eqref{THE_FIRST_CTR_HOLD_APRES} gives the result. 
	\end{proof}

	
	\section*{Acknowledgments.}
	For the first \textcolor{black}{a}uthor, this work has been partially supported by the ANR project ANR-15-IDEX-02. For the third author, the article was prepared within the framework of a subsidy granted to the HSE by the Government of the Russian Federation for the implementation of the Global Competitiveness Program.


	\bibliographystyle{alpha}
	\bibliography{bibli}

\begin{thebibliography}{CdRMP20b}

\bibitem[ABM20]{athreya2020}
S.~Athreya, O.~Butkovsky, and L.~Mytnik.
\newblock Strong existence and uniqueness for stable stochastic differential
  equations with distributional drift.
\newblock {\em Ann. Probab.}, 48(1):178--210, 01 2020.

\bibitem[AH96]{adam:hedb:96}
D.~R. Adams and L.~I. Hedberg.
\newblock {\em Function spaces and Potential Theory}.
\newblock Springer, 1996.

\bibitem[Bas97]{bass:97}
R.~F. Bass.
\newblock {\em Diffusions and {E}lliptic {O}perators}.
\newblock Springer, 1997.

\bibitem[BC03]{bass:chen:03}
R.~Bass and Z.~Q. Chen.
\newblock Brownian motion with singular drift.
\newblock {\em Ann. Probab.}, 31-- 2:791--817, 2003.

\bibitem[BCLP10]{bram:cupi:lanc:prio:09}
M.~Bramanti, G.~Cupini, E.~Lanconelli, and E.~Priola.
\newblock Global {$L^p$} estimates for degenerate {O}rnstein-{U}hlenbeck
  operators.
\newblock {\em Math. Z.}, 266(4):789--816, 2010.

\bibitem[BCLP13]{bram:cupi:lanc:prio:13}
M.~Bramanti, G.~Cupini, E.~Lanconelli, and E.~Priola.
\newblock Global {$L^p$} estimates for degenerate {O}rnstein-{U}hlenbeck
  operators with variable coefficients.
\newblock {\em Math. Nachr.}, 286(11-12):1087--1101, 2013.

\bibitem[BFGM19]{beck:flan:gubi:maur:19}
Lisa Beck, Franco Flandoli, Massimiliano Gubinelli, and Mario Maurelli.
\newblock Stochastic odes and stochastic linear pdes with critical drift:
  regularity, duality and uniqueness.
\newblock {\em Electronic Journal of Probability}, 24(0), 2019.

\bibitem[BZ11]{bram:zhu:11}
M.~Bramanti and M.~Zhu.
\newblock ${L}^p$ and {S}chauder estimates for nonvariational operators
  structured on {H}\"ormander vector fields with drift.
\newblock {\em Analysis of {PDE}s}, 2011.

\bibitem[CC18]{canni:chouk:18}
G.~Cannizzaro and K.~Chouk.
\newblock Multidimensional sdes with singular drift and universal construction
  of the polymer measure with white noise potential.
\newblock {\em Ann. Probab.}, 46(3):1710--1763, 05 2018.

\bibitem[CdR17]{chau:17}
P.-E. Chaudru~de Raynal.
\newblock Strong existence and uniqueness for degenerate {SDE} with {H\"older}
  drift.
\newblock {\em Annales de l'Institut Henri Poincar\'e, Probabilit\'es et
  Statistiques}, 53(1):259--286, February 2017.

\bibitem[CdR18]{chau:16}
P.-E. Chaudru~de Raynal.
\newblock Weak regularization by stochastic drift: result and counter example.
\newblock {\em Discrete and Continuous Dynamical Systems (Series A)},
  38--3:1269--1291, 2018.

\bibitem[CdRHM20]{chau:hono:meno:18}
P-E. Chaudru~de Raynal, I.~Honor\'e, and S.~Menozzi.
\newblock {Sharp Schauder Estimates for some Degenerate Kolmogorov Equations}.
\newblock {\em ar{X}iv:1810.12227 . To appear in Ann. Scuola Normale Superiore
  di Pisa, Scienze}, 2020.

\bibitem[CdRM17]{chau:meno:17}
P-E. Chaudru~de Raynal and S.~Menozzi.
\newblock {Regularization effects of a noise propagating through a chain of
  differential equations: an almost sharp result}.
\newblock {\em arXiv:1710.03620. To appear in Trans. American Math. Society:
  https://doi.org/10.1090/tran/7947}, 2017.

\bibitem[CdRM19]{chau:meno:stable}
P.-E. Chaudru~de Raynal and S.~Menozzi.
\newblock On multidimensional stable-driven stochastic differential equations
  with {B}esov drift.
\newblock {\em arXiv 1907.12263}, 2019.

\bibitem[CdRMP20a]{chau:meno:prio:20}
P.-E. Chaudru~de Raynal, S.~Menozzi, and E.~Priola.
\newblock Schauder estimates for drifted fractional operators in the
  supercritical case.
\newblock {\em Journal of Functional Analysis}, 278(8):108425, 2020.

\bibitem[CdRMP20b]{chau:meno:prio:00}
P.-E. Chaudru~de Raynal, S.~Menozzi, and E.~Priola.
\newblock Weak well-posedness of multidimensional stable driven {SDE}s in the
  critical case.
\newblock {\em Stochastics and Dynamics}, 0(0):2040004, 2020.

\bibitem[CG16]{catellier_averaging_2012}
R.~Catellier and M.~Gubinelli.
\newblock Averaging along irregular curves and regularisation of {ODEs}.
\newblock {\em {S}toch. {P}roc and {A}ppl.}, 126--8:2323--2366, 2016.

\bibitem[CZZ17]{CZZ17}
Z.-Q. {Chen}, X.~{Zhang}, and G.~{Zhao}.
\newblock {Well-posedness of supercritical SDE driven by L\textbackslash'evy
  processes with irregular drifts}.
\newblock {\em arXiv:1709.04632}, Sep 2017.

\bibitem[Dav07]{da:07}
A.~M. Davie.
\newblock {Uniqueness of Solutions of Stochastic Differential Equations}.
\newblock {\em International Mathematics Research Notices}, 2007, 01 2007.

\bibitem[DD15]{delarue_rough_2015}
F.~Delarue and R.~Diel.
\newblock Rough paths and 1d {SDE} with a time dependent distributional drift:
  application to polymers.
\newblock {\em Probability Theory and Related Fields}, pages 1--63, 2015.

\bibitem[DF14]{delarue_transition_2014}
F.~Delarue and F.~Flandoli.
\newblock The transition point in the zero noise limit for a 1d {Peano}
  example.
\newblock {\em Discrete and Continuous Dynamical Systems}, 34(10):4071--4083,
  April 2014.

\bibitem[DM10]{dela:meno:10}
F.~Delarue and S.~Menozzi.
\newblock Density estimates for a random noise propagating through a chain of
  differential equations.
\newblock {\em Journal of Functional Analysis}, 259--6:1577--1630, 2010.

\bibitem[DPL89]{dipe:lion:89}
R.~Di~Perna and P.~L. Lions.
\newblock Ordinary {D}ifferential {E}quations, transport theory and {S}obolev
  {S}paces.
\newblock {\em Inv. Math.}, 98:511--547, 1989.

\bibitem[FF11]{fedr:flan:11}
E.~Fedrizzi and F.~Flandoli.
\newblock Pathwise {U}niqueness and {C}ontinuous {D}ependence for {SDE}s with
  {N}onregular {D}rift.
\newblock {\em Stochastics}, 83-3:241--257, 2011.

\bibitem[FFPV17]{fedr:flan:prio:vove:17}
E.~Fedrizzi, F.~Flandoli, E.~Priola, and J.~Vovelle.
\newblock Regularity of stochastic kinetic equations.
\newblock {\em Electronic Journal of Probability}, 22:$\sharp $ 48, 2017.

\bibitem[FGP10]{flandoli_well-posedness_2010}
F.~Flandoli, M.~Gubinelli, and E.~Priola.
\newblock Well-posedness of the transport equation by stochastic perturbation.
\newblock {\em Inventiones Mathematicae}, 180(1):1--53, 2010.

\bibitem[FIR17]{flandoli_multidimensional_2017}
F.~Flandoli, E.~Issoglio, and F.~Russo.
\newblock Multidimensional stochastic differential equations with
  distributional drift.
\newblock {\em Transactions of the American Mathematical Society},
  369(10.1090/tran/6729):1665--1688, 2017.

\bibitem[Fla11a]{flandoli_random_2011}
F.~Flandoli.
\newblock {\em Random perturbation of {PDEs} and fluid dynamic models}, volume
  2015 of {\em Lecture {Notes} in {Mathematics}}.
\newblock Springer, Heidelberg, 2011.
\newblock Lectures from the 40th Probability Summer School held in Saint-Flour,
  2010.

\bibitem[Fla11b]{flan:11}
F.~Flandoli.
\newblock Regularizing properties of brownian paths and a result of davie.
\newblock {\em Stochastics and Dynamics}, 11(02n03):323--331, 2011.

\bibitem[Fri64]{frie:64}
A.~Friedman.
\newblock {\em Partial differential equations of parabolic type}.
\newblock Prentice-Hall, 1964.

\bibitem[Fri08]{friedman_partial_2008}
A.~Friedman.
\newblock {\em Partial Differential Equations of Parabolic Type}.
\newblock Dover Publications, April 2008.

\bibitem[GO13]{gradinaru_existence_2013}
M.~Gradinaru and Y.~Offret.
\newblock Existence and asymptotic behaviour of some time-inhomogeneous
  diffusions.
\newblock {\em Annales {I}nstit. {H}. Poincar\'e, Probabilit\'es et
  Statistiques}, 49(1):182--207, 2013.

\bibitem[GS69]{gikh:skor:69}
I.I. Gikhman and A.V. Skorokhod.
\newblock {\em Introduction to the theory of random processes}.
\newblock Saunders {M}ath. {B}ooks, 1969.

\bibitem[HM16]{huan:meno:15}
L.~Huang and S.~Menozzi.
\newblock {A {P}arametrix {A}pproach for some {D}egenerate {S}table {D}riven
  {SDE}s}.
\newblock {\em {A}nnales {I}nstit. {H}. {P}oincar{\'e}}, 52(4):1925--1975,
  2016.

\bibitem[HMP19]{HUANG2019162}
L.~Huang, S.~Menozzi, and E.~Priola.
\newblock Lp estimates for degenerate non-local kolmogorov operators.
\newblock {\em Journal de Math\'ematiques Pures et Appliqu\'ees}, 121:162 --
  215, 2019.

\bibitem[H{\"o}r67]{hormander_hypoelliptic_1967}
L.~H{\"o}rmander.
\newblock Hypoelliptic second order differential equations.
\newblock {\em Acta Mathematica}, 119:147--171, 1967.

\bibitem[HP20]{harang2020cinfinity}
F.A. Harang and N.~Perkowski.
\newblock ${C}^\infty$ regularization of {ODE}s perturbed by noise, 2020.

\bibitem[HWZ20]{HAO2020139}
Z.~Hao, M.~Wu, and X.~Zhang.
\newblock Schauder estimates for nonlocal kinetic equations and applications.
\newblock {\em Journal de Math\'ematiques Pures et Appliqu\'ees}, 140:139 --
  184, 2020.

\bibitem[Kol34]{kolm:33}
A.~N. Kolmogorov.
\newblock {Z}uf{\"a}llige {B}ewegungen (zur {T}heorie der {B}rownschen
  {B}ewegung).
\newblock {\em Ann. of Math.}, 2-35:116--117, 1934.

\bibitem[KP10]{kryl:prio:10}
N.~V. Krylov and E.~Priola.
\newblock Elliptic and parabolic second-order {PDE}s with growing coefficients.
\newblock {\em Comm. Partial Differential Equations}, 35(1):1--22, 2010.

\bibitem[KR05]{kryl:rock:05}
N.~Krylov and M.~R\"ockner.
\newblock Strong solutions of stochastic equations with singular time dependent
  drift.
\newblock {\em Prob. {T}heory {R}el. {F}ields}, 131:154--196, 2005.

\bibitem[Kry96]{kryl:96}
N.~V. Krylov.
\newblock {\em Lectures on elliptic and parabolic equations in {H}\"older
  spaces}.
\newblock Graduate {S}tudies in {M}athematics 12. AMS, 1996.

\bibitem[Kry01]{kry:01}
N.~V. Krylov.
\newblock The heat equation in lq((0,t),lp)-spaces with weights.
\newblock {\em SIAM Journal on Mathematical Analysis}, 32(5):1117--1141, 2001.

\bibitem[Lor05]{lore:05}
L.~Lorenzi.
\newblock Estimates of the derivatives for a class of parabolic degenerate
  operators with unbounded coefficients in {${\mathbb R}^N$}.
\newblock {\em Annali della Scuola Normale Superiore di Pisa. Classe di
  Scienze. Serie V}, 2, 01 2005.

\bibitem[LR02]{lemar:02}
P.-G. Lemari\'e-Rieusset.
\newblock {\em Recent developments in the Navier-Stokes problem}.
\newblock CRC Press, 2002.

\bibitem[LZ19]{ling:zhao:00}
C.~Ling and G.~Zhao.
\newblock Nonlocal elliptic equation in h\"older space and the martingale
  problem, 2019.

\bibitem[Mar20]{mar:20}
L.~Marino.
\newblock Schauder estimates for degenerate stable {K}olmogorov equations.
\newblock {\em Bulletin des Sciences Math\'ematiques}, 162:102885, 2020.

\bibitem[Men11]{meno:10}
S.~Menozzi.
\newblock Parametrix techniques and martingale problems for some degenerate
  {K}olmogorov equations.
\newblock {\em Electronic {C}ommunications in {P}robability}, 17:234--250,
  2011.

\bibitem[Men18]{meno:17}
S.~Menozzi.
\newblock Martingale problems for some degenerate {K}olmogorov equations.
\newblock {\em Stoc. Proc. Appl.}, 128-3:756--802, 2018.

\bibitem[MP14]{miku:prag:14}
R.~Mikulevicius and H.~Pragarauskas.
\newblock {On the {C}auchy problem for integro-differential operators in
  {H}{\"o}lder classes and the uniqueness of the martingale problem}.
\newblock {\em Potential Anal.}, 40(4):539--563, 2014.

\bibitem[MS67]{mcke:sing:67}
H.~P. McKean and I.~M. Singer.
\newblock Curvature and the eigenvalues of the {L}aplacian.
\newblock {\em J. Differential Geometry}, 1:43--69, 1967.

\bibitem[NO02]{NUALART2002103}
D.~Nualart and Y.~Ouknine.
\newblock Regularization of differential equations by fractional noise.
\newblock {\em Stochastic Processes and their Applications}, 102(1):103 -- 116,
  2002.

\bibitem[Pri12]{priola_pathwise_2012}
E.~Priola.
\newblock Pathwise uniqueness for singular {SDEs} driven by stable processes.
\newblock {\em Osaka Journal of Mathematics}, 49(2):421--447, June 2012.

\bibitem[Pri15]{prio:15}
E.~Priola.
\newblock On weak uniqueness for some degenerate {SDE}s by global {$L^p$}
  estimates.
\newblock {\em Potential Anal.}, 42(1):247--281, 2015.

\bibitem[Pri18]{priola2018}
E.~Priola.
\newblock Davie's type uniqueness for a class of sdes with jumps.
\newblock {\em Ann. Inst. H. Poincar\'e Probab. Statist.}, 54(2):694--725, 05
  2018.

\bibitem[Pri19]{priola2019davies}
E.~Priola.
\newblock On {D}avie's uniqueness for some degenerate {SDE}s, 2019.

\bibitem[Sha16]{shaposhnikov_2016}
A.~V. Shaposhnikov.
\newblock Some remarks on {D}avie's uniqueness theorem.
\newblock {\em Proceedings of the Edinburgh Mathematical Society},
  59(4):1019--1035, 2016.

\bibitem[Str08]{stroock_partial_2008}
D.~W. Stroock.
\newblock {\em Partial differential equations for probabilists}, volume 112 of
  {\em Cambridge {Studies} in {Advanced} {Mathematics}}.
\newblock Cambridge University Press, Cambridge, 2008.

\bibitem[SV79]{stro:vara:79}
D.W. Stroock and S.R.S. Varadhan.
\newblock {\em Multidimensional diffusion processes}.
\newblock Springer-Verlag Berlin Heidelberg New-York, 1979.

\bibitem[SW20]{sha:wre:20}
A.~Shaposhnikov and L.~Wresch.
\newblock Pathwise vs. path-by-path uniqueness, 2020.

\bibitem[Tri83]{trie:83}
H.~Triebel.
\newblock {\em Theory of function spaces, II}.
\newblock Birkhauser, 1983.

\bibitem[TT90]{tala:tuba:90}
D.~Talay and L.~Tubaro.
\newblock Expansion of the global error for numerical schemes solving
  sto\-chastic differential equations.
\newblock {\em Stoch. Anal. and App.}, 8-4:94--120, 1990.

\bibitem[Ver80]{veretennikov_strong_1980}
A.~Y. Veretennikov.
\newblock Strong solutions and explicit formulas for solutions of stochastic
  integral equations.
\newblock {\em Matematicheski Sbornik. Novaya Seriya}, 111(153)(3):434--452,
  480, 1980.

\bibitem[Ver83]{veretennikov_stochastic_1983}
A.~Y. Veretennikov.
\newblock Stochastic equations with diffusion that degenerates with respect to
  part of the variables.
\newblock {\em Izvestiya Akademii Nauk SSSR. Seriya Matematicheskaya},
  47(1):189--196, 1983.

\bibitem[WZ16]{wang:zhan:16}
F.Y. Wang and X.~Zhang.
\newblock Degenerate {SDE} with {H}older-{D}ini drift and non-{L}ipschitz noise
  coefficient.
\newblock {\em SIAM J. Math. Anal.}, 48(3):2189--2226, 2016.

\bibitem[Zha10]{zhang_well-posedness_2010}
X.~Zhang.
\newblock Well-posedness and large deviation for degenerate {SDEs} with
  {Sobolev} coefficients.
\newblock {\em Revista Matematica Iberoamericana}, 29(1), 2010.

\bibitem[Zha18]{zhan:16}
X.~Zhang.
\newblock Stochastic {H}amiltonian flows with singular coefficients.
\newblock {\em Sci. China Math.}, 61(8):1353--1384, 2018.

\bibitem[Zvo74]{zvonkin_transformation_1974}
A.~K. Zvonkin.
\newblock A transformation of the phase space of a diffusion process that will
  remove the drift.
\newblock {\em Mat. Sb. (N.S.)}, 93(135):129--149, 152, 1974.

\bibitem[ZZ17]{ZZ17}
X.~{Zhang} and G.~{Zhao}.
\newblock {Heat kernel and ergodicity of SDEs with distributional drifts}.
\newblock {\em arXiv:1710.10537}, 2017.

\end{thebibliography}

\end{document}